%% file: Maitre-arxiv.tex
\documentclass[12pt]{amsbook}

\usepackage{amsbooka}

\usepackage{minitoc}

\input{usepackage.tex}
\input{styles.tex}

\makeatletter
\@addtoreset{section}{part}
\makeatother  
\begin{document}
\doparttoc

\title[Dynamics and unlikely Intersections]{Homogeneous dynamics\\ and Unlikely intersections.}

\author{R. Richard}
\curraddr{DPMMS, University of Cambridge}
\email{\href{mailto:rodolphe.richard@normalesup.org}{rodolphe.richard@normalesup.org}}
\author{A. Yafaev}
\address{University College London, Gower Street, London WC1E 6BT, United Kingdom}
\email{\href{mailto:yafaev@math.ucl.ac.uk}{yafaev@math.ucl.ac.uk}}
\author{T. Zamojski}
\address{École polytechnique fédérale de Lausanne, SB -- MATHGEOM-- TAN
EPFL, Station 8, CH–1015 Lausanne, Suisse}
\email{\href{mailto:thomas.zamojski@gmail.com}{thomas.zamojski@gmail.com}}


\maketitle
\tableofcontents



\selectlanguage{english}
\chapter*{Introduction}
\input{./0-Intro-SMF/Fusion-IntroductionSMF.tex}

\selectlanguage{english}
\part{{Inner Galois equidistribution \except{toc}{\\} in $S$-Hecke orbits \author{R. Richard \and A. Yafaev}}}


\input{./1-APZ/Fusion-APZ.tex}

\selectlanguage{english}
\part{Limit distributions \except{toc}{\\}of\except{toc}{\\} translated pieces of leaves\except{toc}{\\} in\except{toc}{\\} $S$-arithmetic homogeneous spaces. 
\author{R. Richard and T. Zamojski}}

\numberwithin{equation}{section}
\setcounter{tocdepth}{1}

\setcounter{tocdepth}{2}
\setcounter{secnumdepth}{4}

\begin{abstract}
We describe and characterize the limit distributions of translates of a  bounded open ``piece of orbit'', a non necessarily closed orbit, of a reductive Lie subgroup, non necessarily algebraic, in a space of $S$-arithmetic lattices. This is accomplished under a mild assumption of ``analytic stability'' on the sequence of translates. A notable improvement is in not assuming that the reductive subgroup or its centralizer are algebraic over the rational numbers. Moreover, it is also not necessary to assume that the initial orbit is of finite measure or even closed (that is we handle leaves). This article thus provides important generalizations of previously known results and opens ways to new applications in number theory, two of which are mentioned. 
\end{abstract}

%
%
%

\renewcommand{\N}{\mathbf{N}} 
\renewcommand{\M}{\mathbf{{M}}} 
\renewcommand{\Mpp}{{M^\ddagger}}
\renewcommand{\Hom}{\mathbf{Hom}}

\input{./2-Richard-Zamojski/Fusion-Richard-Zamojski.tex}
\part{Résultat géométrique sur les représentations de groupes réductifs sur un corps ultramétrique \author{R. Richard}.}
\selectlanguage{french}
\setcounter{tocdepth}{0}
\renewcommand{\SL}{\mathrm{SL}}
\renewcommand{\GL}{\mathrm{GL}}
\renewcommand{\Ad}{\mathrm{Ad}}
\input{3-Resultat-ultrametrique-these-V/RichardUltrametrique.tex}


\renewcommand{\Nm}[1]{\left\|{#1}\right\|}
\renewcommand{\NM}[1]{\left|\!\left|\!\left|{#1}\right|\!\right|\!\right|}
\renewcommand{\lie}[1]{\mathfrak{#1}}
\renewcommand{\ciso}{\equiv}
\renewcommand{\GL}{GL{}}
\renewcommand{\SL}{SL{}}
\renewcommand{\d}{{\mathbf{d}}}
\renewcommand{\Z}{{\mathbf{Z}}}
\renewcommand{\Id}{\mathrm{Id}}
\renewcommand{\Hom}{\mathrm{Hom}}
\renewcommand{\R}{\mathbf{R}}
\renewcommand{\N}{\mathbf{N}}
\renewcommand{\C}{{\mathbf{C}}}
\renewcommand{\1}{\mathbf{1}}
\renewcommand{\A}{\mathbf{A}}
\renewcommand{\Tr}{\mathbf{Tr}}
\renewcommand{\Ad}{\mathrm{Ad}}
\renewcommand{\ad}{\mathrm{ad}}
\renewcommand{\der}{\mathrm{der}}
\renewcommand{\abs}[1]{\left|#1\right|}
\renewcommand{\Q}{\mathbf{Q}}
\renewcommand{\an}{\mathrm{an}}
\renewcommand{\N}{\mathbf{N}}

\setcounter{theorem}{0}
\selectlanguage{english}

\part{On narrowness\except{toc}{\\} for translated algebraic probabilities\except{toc}{\\} in $S$-arithmetic homogeneous spaces \author{R. Richard}\except{toc}{}}

{
\let\thefootnote\relax\footnotetext
{
\author{Rodolphe \textsc{Richard}}
\address{IRMAR, Bâtiment 22-23, université de Rennes 1, Campus de Beaulieu, 35000~Rennes.}
{\email{\href{mailto:Rodolphe.RICHARD@Normalesup.org}{Rodolphe.RICHARD@Normalesup.org}}}
}
}


%
\parttoc

\input{./4-these-VI/Bombay-Lemma-A.tex}
\bibliography{./Bibliography/ThisVolumeandPhD,./Bibliography/0,./Bibliography/1,./Bibliography/2,./2-Richard-Zamojski/bib,./Bibliography/3,./3-Resultat-ultrametrique-these-V/ultrametrique}
\bibliographystyle{SMF-Files/smfalpha}

\end{document}

%% file: usepackage.tex
\usepackage{amsmath}
\usepackage{amssymb}
\usepackage{amstext}
\usepackage{amscd}
\usepackage[matrix,arrow,ps]{xy}

\usepackage[frenchb,english]{babel}
\usepackage{enumitem}
\usepackage{greekctr}

\usepackage[T1]{fontenc}
\usepackage[utf8]{inputenc}
\usepackage{amsmath}
\usepackage{amsthm}
\usepackage{amssymb}
\usepackage{stmaryrd} 
\usepackage{mathrsfs}
\usepackage{pdfpages}

\usepackage{tikz-cd}

\usepackage{geometry} 
\usepackage{graphicx} 
\usepackage{extarrows}

\usepackage{enumitem}

\usepackage{newunicodechar}
\DeclareFontEncoding{LS1}{}{}
\DeclareFontSubstitution{LS1}{stix}{m}{n}
\DeclareFontFamily{LS1}{stixscr}{\skewchar\font127 }
\DeclareFontShape{LS1}{stixscr}{m}{n} {<->s*[.7] stix-mathscr}{}
\newunicodechar{◌}{{\usefont{LS1}{stixscr}{m}{n}\symbol{\string"E3}}}

\DeclareUnicodeCharacter{B7}{\cdot}
\DeclareUnicodeCharacter{B7}{\times}

\usepackage{calc}
\usepackage{graphicx}
\makeatletter
\newlength{\temp@wc@width}
\newlength{\temp@wc@height}
\newcommand{\widecheck}[1]{%
  \setlength{\temp@wc@width}{\widthof{$#1$}}%
  \setlength{\temp@wc@height}{\heightof{$#1$}}%
  #1\hspace{-\temp@wc@width}%
  \raisebox{\temp@wc@height+2pt}[\heightof{$\widehat{#1}$}]%
     {\rotatebox[origin=c]{180}{\vbox to 0pt{\hbox{$\widehat{\hphantom{#1}}$}}}}%
}
\makeatother

\usepackage{mathrsfs}
\usepackage{enumitem}

	\usepackage{amsfonts}	%
	\usepackage{graphicx}

%% file: styles.tex
\input{./1-APZ/1-styles.tex}

\input{./2-Richard-Zamojski/2-styles.tex}
\input{./3-Resultat-ultrametrique-these-V/3-styles.tex}

\newcommand{\latin}[1]{\emph{#1}}

\newcommand{\sous}{\backslash}
\renewcommand{\tens}{\otimes}

\newcommand{\cf}{cf.}

\newcommand{\cov}{\mathrm{cov}}
\newcommand{\sys}{\mathrm{sys}}


%% file: 1-APZ/1-styles.tex
\newtheorem{teo}{Theorem}[section]
\newtheorem{prop}[teo]{Proposition}
\newtheorem{lem}[teo]{Lemma}
\newtheorem{cor}[teo]{Corollary}
\newtheorem{conj}[teo]{Conjecture}

\newtheorem{defi}[teo]{Definition}

\newcommand{\Hom}{\mbox{Hom}}

\newcommand{\GSp}{{\rm GSp}}
\newcommand{\GL}{{\rm GL}}
\newcommand{\SL}{{\rm SL}}
\newcommand{\Sh}{{\rm Sh}}

\newcommand{\Gal}{{\rm Gal}}

\newcommand{\der}{{\rm der}}

\newcommand{\CC}{{\mathbb C}}
\newcommand{\RR}{{\mathbb R}}
\newcommand{\ZZ}{{\mathbb Z}}
\newcommand{\QQ}{{\mathbb Q}}

\newcommand{\HH}{{\mathbb H}}

\newcommand{\Hcal}{{\mathcal H}}

\newcommand{\SSS}{{\mathbb S}}
\newcommand{\AAA}{{\mathbb A}}

\newcommand{\lto}{\longrightarrow}

\def\Fh{\mathfrak{h}}
\def\Fu{\mathfrak{u}}

\def\Fm{\mathfrak{m}}
\def\Fz{\mathfrak{z}}

\newcommand{\cF}{\mathcal{F}}
\newcommand{\cA}{{\mathcal A}}

\newcommand{\cZ}{{\mathcal Z}}

\newcommand{\ol}{\overline}

\renewcommand{\contentsname}{Table des mati\`eres}

\newcommand{\lie}{{\mbox{Lie}}}

\renewcommand{\contentsname}{Table of contents.}

\newcommand{\wt}{\widetilde}

\newcommand{\tens}{\otimes}
\usepackage{mathtools}
\DeclareMathOperator{\Radu}{{\mathrm{Rad}^{\mathrm{u}}}}
\DeclarePairedDelimiterX{\Nm}[1]{\lVert}{\rVert}{#1}

%% file: 2-Richard-Zamojski/2-styles.tex
\newtheorem{theorem}{Theorem}
\newtheorem{proposition}[theorem]{Proposition}

\newtheorem{lemma}[theorem]{Lemma}
\newtheorem{corollary}[theorem]{Corollary}
\theoremstyle{definition}
\newtheorem{definition}[theorem]{Definition}

\theoremstyle{remark}

\DeclareMathOperator{\convolution}{\star}
\newcommand{\Prob}{\mathrm{Prob}}
\newcommand{\U}{\mathbf{U}}
\newcommand{\G}{\mathbf{{G}}}
\newcommand{\C}{\mathbf{C}}

\global\let\restriction|

\newcommand{\A}{\mathbf{A}}
\newcommand{\R}{\mathbf{R}} 
\newcommand{\Q}{\mathbf{Q}} 
\newcommand{\Z}{\mathbf{Z}} 
\renewcommand{\L}{\mathbf{{L}}} 
\newcommand{\K}{\mathbf{K}} 
\renewcommand{\H}{\mathbf{{H}}}
\newcommand{\GammaN}{{\Gamma_{N(L)}}}
\newcommand{\g}{\mathfrak{g}}
\newcommand{\h}{\mathfrak{h}}
\renewcommand{\l}{\mathfrak{l}}
\newcommand{\w}{\mathfrak{w}}
\newcommand{\F}{\mathbf{F}} 

\newcommand{\RatQ}{\mathscr{R}_\Q}
\newcommand{\RatG}{\mathscr{R}_\Gamma}
\newcommand{\Lscr}{\mathscr{R}_\Gamma}
\newcommand{\Supp}{\textbf{Supp}}
\newcommand{\ad}{\mathbf{ad}}
\newcommand{\Ad}{\mathbf{Ad}}
\newcommand{\Cj}{\text{Conj}}
\newcommand{\Stab}{\mathrm{Stab}}
\renewcommand{\det}{\mathrm{det}}
\newcommand{\alg}{\mathrm{alg}}

\newcommand{\ciso}{\simeq}

\newcommand{\leftexp}[2]{{\vphantom{#2}}^{#1}{#2}}
\newcommand{\abs}[1]{{\left|{#1}\right|}}

\newcommand{\eps}{\varepsilon}
\newcommand{\Lpp}{{L^\ddagger}}

%% file: 3-Resultat-ultrametrique-these-V/3-styles.tex
\newcommand{\kk}{\mathbf{k}}

\newcommand{\Ik}{\mathscr{I}_\kk}
\newcommand{\an}{\textrm{an}}
\newcommand{\gl}{\lie{gl}}
\newcommand{\Id}{\mathrm{Id}}
\newtheorem{theoreme}{Théorème}[section]

\newcommand{\Tr}{\mathbf{Tr}}
\newcommand{\z}{\lie{z}}

\makeatletter
\def\renewtheorem#1{%
  \expandafter\let\csname#1\endcsname\relax
  \expandafter\let\csname c@#1\endcsname\relax
  \gdef\renewtheorem@envname{#1}
  \renewtheorem@secpar
}
\def\renewtheorem@secpar{\@ifnextchar[{\renewtheorem@numberedlike}{\renewtheorem@nonumberedlike}}
\def\renewtheorem@numberedlike[#1]#2{\newtheorem{\renewtheorem@envname}[#1]{#2}}
\def\renewtheorem@nonumberedlike#1{  
\def\renewtheorem@caption{#1}
\edef\renewtheorem@nowithin{\noexpand\newtheorem{\renewtheorem@envname}{\renewtheorem@caption}}
\renewtheorem@thirdpar
}
\def\renewtheorem@thirdpar{\@ifnextchar[{\renewtheorem@within}{\renewtheorem@nowithin}}
\def\renewtheorem@within[#1]{\renewtheorem@nowithin[#1]}
\makeatother

\renewtheorem{proposition}{Proposition}[section]
\newcommand{\NM}[1]{{\left\vert\kern-0.25ex\left\vert\kern-0.25ex\left\vert #1 
    \right\vert\kern-0.25ex\right\vert\kern-0.25ex\right\vert}}
\newtheorem{corollaire}{Corollaire}[section]
\newtheorem{lemme}{Lemme}[section]
\newcommand{\Aff}{\mathbf{Aff}}

%% file: 0-Intro-SMF/Fusion-IntroductionSMF.tex

Since the 1990s applications of results in homogeneous dynamics to problems in Number Theory and 
Diophantine Geometry have
become one of the important trends in Number Theory.
A key role in the subject is played by Ratner's theorems on rigidity of unipotent flows on homogeneous
spaces. These theorems prove a conjecture by Raghunathan, generalised by Dani and Margulis. Ratner's proofs use earlier work of Dani and Margulis.
A summary of Ratner's results can be found in \cite{0-Ra}.

Various generalisations of Ratner's theorems have been obtained by 
Eskin, Margulis, Mozes, Shah, Tomanov and many others. We refer to
this general area as Ratner's theory.

In 1987, Margulis proved the Oppenheim conjecture on Diophantine 
approximation by quadratic forms.  Ratner proved her theorems shortly after. 
Eskin, Mozes and Shah proved Manin's
conjecture on counting integer points for a wide range of reductive homogeneous spaces.

In 2000 Cornut and Vatsal applied Ratner's theory to some problems in the Arithmetic of Elliptic Curves
(Mazur's conjecture on the non-triviality of Heegner points, see \cite{0-Co}
and \cite{0-Va}). Roughly at the same time Clozel and Ullmo have applied, in~\cite{Ullmospeciale1},
Ratner's theory to the Andr\'e-Oort conjecture on the distribution of special points on subvarieties
of Shimura varieties. It is this last avenue of research that we develop further in this volume.

This volume consists of four papers:
\begin{enumerate}
\item
\emph{Inner Galois equidistribution in $S$-Hecke orbits} by Richard and Yafaev (\cite{0-RY});
\item
\emph{Limit distributions of translated pieces of possibly irrational leaves in $S$-arithmetic 
homogeneous spaces} by Richard and Zamojski (\cite{0-RZ});
\item
\emph{R\'esultat g\'eometrique sur les repr\'esentations de groupes r\'eductifs sur un corps
ultram\'etrique} by Richard (\cite{0-R});
\item
\emph{On narrowness for translated algebraic probabilities in $S$-arithmetic homogeneous spaces} by Richard (\cite{LemmaA}).
\end{enumerate}

Central to the volume is the article \cite{0-RZ} which contains new developments in homogeneous dynamics. Main theorems of \cite{0-RZ} go beyond the results of 
Eskin, Mozes, Shah and Margulis.
 This article uses technical results of papers~\cite{0-R} and
\cite{LemmaA} extracted from the author's thesis. 

The paper~\cite{0-RY} contains applications of main results of \cite{0-RZ}
to Diophantine problems, namely
to some problems
of unlikely intersections in Shimura varieties. 

The four articles can be read independently from each other. 
For instance, the reader more interested in Diophantine applications,  may read~\cite{0-RY}.
Those interested in homogeneous dynamics only may read~\cite{0-RZ,LemmaA}.
In addtion to proving the main equidistribution theorem used in
\cite{0-RY}, paper \cite{0-RZ} develops the general theory of $S$-arithmetic homogeneous spaces and proves a large number of results concerning these spaces which may be of independent interest. Note also that we sometimes give several alternative proofs of the same lemma or proposition.

Paper~\cite{0-R} contains results relevant to
those interested in the geometry of reductive groups.
We also believe that technical results of \cite{0-R} and \cite{LemmaA} to be of independent interest. The reader may also be interested in a consise 
account of the Bruhat-Tits theory contained in the paper \cite{0-R}.




We now outline the contents of the two major papers \cite{0-RZ} and \cite{0-RY}
of the volume.

The paper \cite{0-RZ} is Ergodic Theoretic.
It develops the theory of equidistribution of certain measures on $S$-arithmetic homogeneous spaces.
Let us describe the general situation. Let $G$ be a semisimple algebraic group over $\QQ$ and $S$ a finite set of places of $\QQ$.
Let $\QQ_S:= \prod_{p \in S} \QQ_p$.
Consider an $S$-adic Lie subgroup $H$ of $G(\QQ_S)$, let $\Omega$ 
an open subset of $H$ and $\nu$ the Haar probability of $H$
normalised so that $\nu(\Omega)=1$.
Let $\Gamma$ be an $S$-arithmetic lattice in $G(\QQ_S)$ and
consider the space $\Gamma \backslash G(\QQ_S)$. 
Denote $\mu_{\Omega}$ the pushforward of $\nu$ to 
$\Gamma \backslash G(\QQ_S)$.
Consider the sequence of translated measures $\mu_n =  \mu_{\Omega} g_n$ with $g_n \in G(\QQ_S)$.
The aim of \cite{0-RZ} is to describe the limit measures of $\mu_n$.
One defines an `$S$-Ratner's class' $R_S$ of subgroups of $G$. The definition is technical and is given and fully explained in \cite{0-RZ}.
The first main result of \cite{0-RZ} is that the limit measure
$\mu_{\infty}$ is \emph{roughly} as follows (for precise statements and explanations we refer to the paper \cite{0-RZ} itself).
There exists a subgroup $L$ of $G$ of $S$-Ratner class,
there exists an open (precisely described) subgroup $L^{++}$ 
of $L(\QQ_S)$, an element $n_{\infty}$ of the normaliser of $L$ and an element $g_{\infty}$ of $G(\QQ_S)$ such that
$$
\mu_{\infty} = \mu_{L^{++}} \cdot n_{\infty} \star \mu_{\Omega} \cdot g_{\infty} 
$$
where $\star$ denotes the convolution.
Furthermore, it is proven in \cite{0-RZ} that the convergence of
measures (or equidistribution) is ``inner'': up to extracting a subsequence
the supports of $\mu_n$ are contained in the support $\mu_{\infty}$
for all $n$ large enough.
This property is essential for the arithmetic applications in \cite{0-RY}.

We now turn to the contents of the paper \cite{0-RY}.
The paper \cite{0-RY} applies the results of \cite{0-RZ} outlined above to problems on unlikely 
intersections in Shimura varieties, namely the Zilber-Pink conjecture
which is a vast generalisation of the Andr\'e-Oort conjecture.
Recall that the Andr\'e-Oort conjecture is the statement that 
the Zariski closure of a set of special points is a finite union of special subvarieties.
Precise statements and survey of recent methods and results can be found in the volume \cite{0-PS}.
There are currently three approaches to the Zilber-Pink type of problems.
One is purely algebro-geometric and relies on Galois properties of special points and 
geometric properties of subvarieties (in particular their property of being `stable' under Hecke correspondences).
Another approach relies on equidistribution (Ratner's theory). The two approaches 
combine together to produce a proof of the Andr\'e-Oort conjecture under the Generalised Riemann Hypothesis.
We refer to Yafaev's article in \cite{0-PS}  and references therein for more details.
Another very fruitful approach that emerged about 10 years ago is that involving o-minimality (a branch of 
Model Theory). T. Scanlon's article in \cite{0-PS} provides a good introduction to this approach.
This approach sometimes allows to obtain complete results. A major achievement is a proof of the Andr\'e-Oort
conjecture unconditionally for all Shimura varieties of abelian type. This result should be attributed to a large number of contributors.
It seems that in the recent years o-minimality has overshadowed the equidistribution approach to the 
Zilber-Pink type of problems.
However, the equidistribution approach, when applicable yields much finer results than just a characterisation of the Zariski closure of a certain set of points. 
In the paper \cite{0-RY}, the equidistribution approach to
a special case of the Zilber-Pink conjecture, namely the Andr\'e-Pink-Zannier conjecture, is developed.
Simply put, the conjecture predicts that the Zariski closure of a subset of the Hecke orbit of a point in a Shimura variety  is a finite union of weakly special subvarieties.
There are several characterisations of the weakly special subvarieties: a Hodge theoretic one (in terms of Shimura data), a differential geometric one (they are exactly the totally geodesic subvarieties) and a bi-algebraic one, particularly relevant to the o-minimal approach.
We refer to \cite{0-KUY} for a thorough exposition of these notions.
Very loosely speaking, weakly special subvarieties are homogeneous varieties embedded into 
the ambient Shimura variety via a morphism arising from an inclusion of Shimura data.

It should be noted that in \emph{any} approach to Zilber-Pink type of problems, Galois action on `special' (or `atypical')
points plays a central role. Obtaining precise information about this action and the growth of Galois orbits is the most difficult part of any strategy. A large section of \cite{0-RY} is devoted to analysis of the Galois action on points of $S$-Hecke orbits (see below) which we believe to be of independent interest.

The paper \cite{0-RY} concerns itself with the $S$-Andr\'e-Pink-Zannier conjecture  (see precise statement below) - a special case where one considers Hecke orbits restricted to a finite set $S$ of primes. This assumption is essential in order to apply equidistribution results of \cite{0-RZ}. 
To state our results precisely, we need to introduce some technical notations.
 Let $(G,X)$ be a Shimura datum (see \cite{0-Deligne}) and $K$ be a compact open subgroup 
 of $G(\AAA_f)$. 
 The following double coset space  is the set of complex points of the corresponding Shimura variety:
 $$
 Sh_K(G,X) = G(\QQ)\left\backslash X\times G(\AAA_f)\right/ K.
 $$
 We recall that $Sh_K(G,X)$ admits a canonical model over a certain
 explicitly described number field $E(G,X)$ called the reflex field of the Shimura datum $(G,X)$.

Let $x\in X$ and $g \in G(\AAA_f)$.
We denote by~$[x,g]$ the image  of~$(x,g)\in X\times G(\AAA_f)$ in
$Sh_K(G,X)$. Note that in this introduction we use slightly simpler 
(but less precise) notations than in the paper \cite{0-RY}.
The Hecke orbit of $[x,g]$ is the set~$H(x)=\{[x,g]|g\in G(\AAA_f)\}$.
The Andr\'e-Pink-Zannier conjecture is the statement that irreducible components 
of the Zariski closure of any subset of $H(x)$ are weakly special subvarieties.
For Shimura varieties of abelian type and components of dimension one, this
statement was proved by Orr in \cite{0-Orr} using the theory of o-minimality.
Orr's approach is currently not generalisable 
to higher dimensional subvarieties or Shimura varieties of exceptional type.
The approach in \cite{0-RY} is very different from Orr's - it relies on 
recent results in homogeneous dynamics described above.

Let us now be more specific.
The $S$-Hecke orbit of $[x,t]$ is 
~$H_S(x)=\{[x,tg]|g\in G(\QQ_S)\}$.

The precise statement of the problem we are interested in is
as follows:

\begin{conj}[$S$-APZ conjecture]\label{Conjecture 1}
The Zariski closure~$\overline{\Sigma}$ of a subset~$\Sigma\subseteq H(x)$ is a finite union of weakly special subvarieties.
\end{conj}

An essential Galois-theoretic assumption we must make is the $S$-Shafarevich assumption
(see next section)
which is the natural generalisation of the Shafarevich conjecture 
proved by Faltings for abelian varieties. In particular this assumtion is
true for all Shimura varieties of abelian type.

Let $s = [x,t]$ be a point of $Sh_K(G,X)$ and 
 $E \subset \CC$ be a field of  definition of $s$.
Let $M \subset G$ be the Mumford-Tate group of~$x$.
We have the following (see section 3 of \cite{0-RY} for details):
\begin{enumerate}

\item \label{defimonorep} 
There exists a continuous $S$-adic monodromy representation
$$
{\rho_{x,S}} \colon Gal(\overline E / E) \rightarrow M(\AAA_f) \cap K \cap G(\QQ_S).
$$
which has the  property that for any $g \in G_S$ and $\sigma \in Gal(\overline{E} / E)$,
\begin{equation}\label{Galois carac}
\sigma( [x,tg] ) = [x, \rho(\sigma) \cdot tg].
\end{equation}
\item \label{defimonogroup}
The~\emph{$S$-adic monodromy group} is
$U_S=\rho_{x,S}\left(Gal(\overline{E} / E)\right)$.
It is a compact $S$-adic Lie subgroup of~$M(\AAA_f)\cap G_S\simeq M(\QQ_S)$.
\item The algebraic~\emph{$S$-adic monodromy group}, denoted~$H_S \subset M_S$, is $H_S=\overline{U_S}^{\text{Zar}}$.
\end{enumerate}
We now introduce the following definitions:
we say that the point~$s$ over~$E$ is
\begin{enumerate}
\item \label{S-MT}
	 of $S$-\emph{Mumford-Tate type} if~$U_S$ is open in 
	 $M(\AAA_f) \cap G_S$;
\item \label{S-Sha}
	 of \emph{$S$-Shafarevich type} if for every finite extension~$F$ of $E$, there are only finitely many $F$-rational points in~$H_S(s)$;
\item \label{S-Tate}
	 of \emph{$S$-Tate type} if~$M$ and~$H^0_S$ (neutral component of $H_S$) have the same centraliser in~$G_S$.
\end{enumerate}
We say that the point~$s$ over~$E$
\begin{enumerate}\setcounter{enumi}{3}
\item \label{S-semisimple} satisfies $S$-\emph{semisimplicity} if $H_S$ is a reductive group;
\item \label{Algebraicity} satisfies~$S$-\emph{algebraicity} if the subgroup~$U_S$ of~$H_S(\QQ_S)$ is open. 
\end{enumerate}

These properties are heavily dependent on the choice of the field $E$, particularly on whether $E$ is of 
finite type over $\QQ$ or not. 
In section 3.4 of \cite{0-RY} we 
give examples of various interesting phenomena
that occur when one varies $E$.
Note that a point of~$S$-Mumford-Tate type is obviously of~$S$-Tate type. These properties hold 
notably for a special point, by definition of a canonical model.
We prove that $S$-Tate type property~\eqref{S-Tate} implies the algebraicity of~$U_S$ if~$E$ is of finite type. For such an~$E$, the~$S$-Mumford-Tate type property~\eqref{S-MT} is equivalent to its \emph{a priori} weaker variant:~$H_S^0=M_{\QQ_S}$.

The main results of \cite{0-RY} are proved under the $S$-Shafarevich hypothesis~\eqref{S-Sha} which  is implied by
(and in fact not far from being equivalent to) the $S$-Tate hypothesis~\eqref{S-Tate} together with the $S$-semisimplicity assumption~\eqref{S-semisimple}. 
More precisely (\cite{0-RY}, Proposition 3.7), the $S$-Shafarevich
property is equivalent to $S$-semisimplicity together with the property that
the centraliser of $U_S$ in $G_S$ is compact modulo the centraliser of $M$ in $G_S$.
In our opinion these results are of independent interest and may prove useful elsewhere.

Let $L$ be a group in $R_S$ ($S$-Ratner class as above). An $S$-real weakly special submanifold is a subset of~$Sh_K(G,X)$ of the  following form (for some~$L$ in~$R_S$ and~$(x,g)\in X\times G(\AA_f)$),
\[
Z_{L,(x,g)}=\left\{[lx,g]\middle|l\in L(\RR)^+\right\}.
\] 
We have a natural identification
\[
(\Gamma\cap L(\RR)^+)\backslash L(\RR)^+/(L(\RR)^+\cap C)\simeq Z_{L,(x,g)}
\]
where~$\Gamma$ is an arithmetic subgroup of~$G$ depending on~$g$ and~$C$ a maximal compact subgroup of~$G^{der}(\RR)$ depending on~$x$. We obtain a probability measure~$\mu_Z$ on~$Z=Z_{L,(x,g)}$ by pushing forward a Haar measure on $L(\RR)^+$ and normalising. 

We now describe the main results of \cite{0-RY}. Let $s$ be a point of 
$Sh_K(G,X)$ and $E$ be its field of definition.
As explained before we must assume that $s$ and $E$ are of $S$-Shafarevich type.
Consider a sequence ~$(s_n)_{n\geq0}$  of points in the~$S$-Hecke orbit~$H_S(s)$ of~$s$, and denote 
$\mu_n = \frac{1}{|Gal(\overline{E}/E)|} \sum_{\zeta\in Gal(\overline{E}/E)\cdot s_n } \delta_{\zeta}$ 
the sequence 
of discrete probability measures attached to the Galois orbits of the~$s_n$.
Applying results of \cite{0-RZ} to our particular situation, we show that 
after possible extracting a subsequence,
the sequence of measures $(\mu_n)$ weakly converges to a measure 
attached to a real weakly special submanifold. Furthermore, we show that that the 
equidistribution is inner in the sense explained above.
From this we deduce that the $S$-Andr\'e-Pink-Zannier conjecture and a much finer statement about the topological closure of a subset of an $S$-Hecke orbit. 
Namely we show that the topolocial closure of a subset of an $S$-Hecke orbit is a
finite union of $S$-real weakly special submanifolds.
Under a stronger assumption of the $S$-Mumford-Tate conjecture (which holds for example when these points are special),
we prove that these submanifolds are actually weakly special subvarieties.

Finally one  easily proves a converse statement below
which shows that the $S$-Shafarevich property assumption is essential and optimal.
Keeping the above notations, 
assume that for any sequence~$(s_n)_{n\geq0}$ in~$H_S(s)$, for any 
finite extension~$F$ of~$E$, there is an extracted 
subsequence for which the associated measure~$\mu_n$ converges weakly
to a limit~$\mu_\infty$ and furthermore the convergence is inner.
Then~$s$ is of~$S$-Shafarevich type.

%% file: 1-APZ/Fusion-APZ.tex

\section{Introduction.}

This paper concerns itself with certain special cases of the Zilber-Pink conjecture on unlikely intersections
in Shimura varieties and some of its natural generalisations.

Let $(G,X)$ be a Shimura datum, let $K$ be a compact open subgroup of $G(\AAA_f)$, and let
\[
\Sh_K(G,X) = G(\QQ)\backslash X\times G(\AAA_f)/K
\]
be the associated Shimura variety. 
We refer to \cite{1-Milne} and references therein for definitions and facts related to Shimura varieties. Throughout the paper we always assume that $G$ is the 
generic Mumford-Tate group on $X$. This is a standard convention in the theory of Shimura varieties.
We will write a point of $\Sh_K(G,X)$ as $s=\overline{(h,t)}$ - the double class of the element $(h,t) \in X \times G(\AAA_f)$.

The \emph{Hecke orbit} of a point of $\Sh_K(G,X)$ is defined as follows.

\begin{defi}[Hecke orbit] Let~$s=\ol{(h, t)}$ be a point of~$\Sh_K(G,X)$. We define the \emph{Hecke orbit} of the point~$s$ in~$\Sh_K(G,X)$ to be the set
$$
\Hcal(s) = \left\{ \ol{(h, t\cdot g)}~\middle|~g \in G(\AAA_f) \right\}\subseteq \Sh_K(G,X).
$$
\end{defi}

We refer to \cite{1-Moonen} for a definition of \emph{weakly special subvarieties}\footnote{A concise characterisation of \emph{weakly special subvarieties}, the one studied in \cite{1-Moonen}, is the following: they are the complex algebraic subvarieties of~$\Sh_K(G,X)$ which are the  images of totally geodesic submanifolds of the symmetric space~$X\times\{t\}$, with some~$t\in G(\AAA_f)$. Each such subvariety is isomorphic to a component of some Shimura variety, that is, an arithmetic quotient of a Hermitian symmetric domain. Weakly special subvarieties containing smooth special points are called special subvarieties. A special subvariety is a 
component of the image of a morphism of Shimura varieties induced by a morphism of Shimura data.
In a certain sense a weakly special subvariety is a translate of a special subvariety.} of $\Sh_K(G,X)$ and related notions. 
Several equivalent definitions (different in flavour) are given in \cite{1-U1}, \cite{1-UY3} and \cite{1-UY2}. 
Note that a single point is trivially a weakly special subvariety.

The Andr\'e-Pink-Zannier conjecture is the following statement.

\begin{conj}[Andr\'e-Pink-Zannier] \label{APZ}
Irreducible components of the Zariski closure of any subset~$\Sigma$ of $\Hcal(s)$ are weakly special subvarieties.
\end{conj}

This conjecture was formulated by Andr\'e for curves in Shimura varieties in \cite{1-A}, then by Pink 
for arbitrary subvarieties of mixed Shimura varieties in \cite{1-Pinky} and independently
by Zannier (unpublished). We therefore refer to this conjecture as the
Andr\'e-Pink-Zannier conjecture.

For \emph{curves} in Shimura varieties of abelian type, conjecture \ref{APZ} was proved by Orr in \cite{1-Orr}.
Orr also proves in \cite{1-Orr} that conjecture \ref{APZ} is a special case of the Zilber-Pink conjecture on unlikely intersections in Shimura varieties. 

In this paper we deal with the `$S$-Hecke orbit' version of Conjecture~\ref{APZ}. In our situation, we actually derive much stronger conclusions then those of conjecture \ref{APZ}.

\subsection{The $S$-Andr\'{e}-Pink-Zannier conjecture}
We consider the following weaker notion of a Hecke orbit.

\begin{defi}[$S$-Hecke orbits]\label{DefiS-Hecke}
 Let $S$ be a finite set of prime numbers. 
\begin{enumerate}
\item
Write~$(g_\ell)_\ell$ an element of~$G(\AAA_f)$ viewed as a restricted product indexed by primes~$\ell$. We denote~$G_S$ the subgroup of~$G(\AAA_f)$  consisiting of elements~$(g_\ell)_\ell$ such that $g_\ell = 1$ for~$\ell\notin S$.\\ As~$S$ is a finite set, we may identify~$G_S$ with~$\prod_{\ell\in S} G(\QQ_\ell)$, or equivalently with~$G(\QQ_S)$ where~$\QQ_S=\prod_{\ell\in S} \QQ_\ell$.
\item 
For a point~$s=\ol{(h, t)}$ of $\Sh_K(G,X)$, we define the \emph{$S$-Hecke orbit} of~$s$ to be the subset
$$
\Hcal_S(s) = \left\{ \ol{(h, t\cdot g)}~\middle|~ g \in G_S \right\}\subseteq \Hcal(s).
$$
\end{enumerate}
\end{defi}

Let $E$ be a field of definition of $s$. The fact that Hecke correspondences are defined over the reflex field $E(G,X)$ shows that
points of $\Hcal(s)$ and $\Hcal_S(s)$ are defined over $\overline{E}$.

We introduce the following definition which will play a central role in what follows.

\begin{defi}[$S$-Shafarevich property] \label{SSha}
The point $s$ is said to satisfy the $S$-Shafarevich property or to be
of $S$-Shafarevich type if for every finite extension $F \subset \overline{E}$
of $E$, $\Hcal_S(s)$ contains only finitely many points defined over $F$.
\end{defi}

We immediately observe that the $S$-Shafarevich property is invariant by 
finite morphisms induced by the inclusions of compact open subgroups
$K' \subset K$.
The conclusions of Conjectures \ref{APZ} and \ref{APS} as well as those 
of our main theorems \ref{Theoreme1} and \ref{Theoreme2} are also invariant by 
replacing $K$ by a subgroup of finite index.
Therefore, throughout the paper we always assume \emph{the group $K$ to be 
\underline{neat}}.
More precisely, we choose the group $K$ as follows.
\begin{enumerate}
\item $K$ is a product $K = \prod K_p$ of compact open subgroups 
$K_p \subset G(\QQ_l)$.
\item Fix $l \geq 3$ and $l \notin S$. We assume that 
$K_l$ is contained in the group of elements congruent to the identity modulo $l$ 
with respect to some faithful representation $G \subset \GL_n$.
\end{enumerate}

As noted just above these assumptions cause no loss of any generality but
avoid many annoying technicalities. These assumptions will be kept throught the paper.

We will examine this property in great detail in section \ref{SectionGalois}. In particular,
we will obtain a group theoretic characterisation of this property.

We now formulate the corresponding weaker form of Conjecture~\ref{APZ} which 
will be the main object of study in this paper.

\begin{conj}[Andr\'e-Pink-Zannier conjecture for $S$-Hecke orbits] \label{APS}
Irreducible components of the Zariski closure of any subset of $\Hcal_S(s)$ are weakly special subvarieties.
\end{conj}
In this paper we prove results in the direction of  Conjecture~\ref{APS}. There are some known cases of this conjecture, all are implied by our results. In particular,
\begin{itemize}
 \item that Conjecture~\ref{APS} holds whenever~$(G,X)$ is of abelian type (proved in \cite{1-OrrThesis} by a different method)
 \item that Conjecture~\ref{APS} holds if~$s\in \Sh_K(G,X)$ is a special point (this is a special case of the 'weak Andr\'e-Oort conjecture' proven in \cite{1-KY});
\end{itemize}
 
The conclusion of our main theorem \ref{Theoreme1} holds under the assumption that the $S$-Shafarevich property (\ref{SSha}) holds.
We actually prove equidistribution results that are much stronger than conclusions of Conjecture \ref{APS}.

\subsection{Galois monodromy properties}\label{introgalois}

Let $E \subset \CC$ be a field 
over which the point $s = \overline{(h,t)}$ is defined. 
Let $M \subset G$ denote the Mum\-ford-Tate group of~$h$.
We assume that $E$ contains the reflex field of the Shimura datum $(M,M(\RR) \cdot h)$.

\begin{defi}\label{defiintro}
\begin{enumerate}
\item \label{defimonorep} 
There exists a continuous ``$S$-adic monodromy representation'' 
$$
{\rho_{h,S}} \colon \Gal(\ol E / E) \lto M(\AAA_f) \cap K \cap G_S.
$$
which has the  property that for any $g \in G_S$ and $\sigma \in \Gal(\ol E / E)$,
\begin{equation}\label{Galois carac}
\sigma\left(\ol{(h,tg)}\right) = \ol{(h, \rho(\sigma) \cdot tg)}.
\end{equation}
\item \label{defimonogroup}
The~\emph{$S$-adic monodromy group}, which we will denote by~$U_S$, is defined as the image of~$\rho_{h,S}$:
\[
U_S=\rho_{h,S}\left(\Gal(\ol E / E)\right).
\]
This is a compact $S$-adic Lie subgroup of~$M(\AAA_f)\cap G_S\simeq M(\QQ_S)$.
\item The algebraic~\emph{$S$-adic monodromy group}, denoted~$H_S$, is defined as the algebraic envelope in~$M_{\QQ_S}$ of the subgroup~$U_S$ of~$M(\QQ_S)$:
\[
H_S=\ol{U_S}^{\text{Zar.}}
\]
This is an algebraic group over~$\QQ_S$. 
We write~$H^0_S$ for its neutral component.
\end{enumerate}
\end{defi}
\begin{defi}\label{ProprietesGaloisiennes} Let~$s$ be a point of~$\Sh_K(G,X)$ and~$E$ a field of definition of~$s$ such that the associated $S$-adic monodromy representation~\[\rho:\Gal\left(\ol{E}/E\right)\to M(\AAA_f)\cap G_S\] is defined. 
We say that the point~$s$ over~$E$ is
\begin{enumerate}
\item \label{S-MT}
	 of $S$-\emph{Mumford-Tate type} if~$U_S$ is open in 
	 $M(\AAA_f) \cap G_S$;
\item \label{S-Sha}
	 of \emph{$S$-Shafarevich type} if for every finite extension~$F$ of $E$, there are only finitely many $F$-rational points in~$\Hcal_S(s)$;
\item \label{S-Tate}
	 of \emph{$S$-Tate type} if~$M$ and~$H^0_S$ have the same centraliser in~$G_S$:
$$
Z_{G_S}(H_S^0) = Z_{G_S}(M).
$$
\end{enumerate}
We say that the point~$s$ over~$E$
\begin{enumerate}\setcounter{enumi}{3}
\item \label{S-semisimple} satisfies $S$-\emph{semisimplicity} if $H_S$ is a reductive group;
\item \label{Algebraicity} satisfies~$S$-\emph{algebraicity} if the subgroup~$U_S$ of~$H_S(\QQ_S)$ is open. 
\end{enumerate}
This latter is equivalent to the Lie algebra of~$U_S$ being algebraic (in the sense of Chevalley, cf. \cite[\S7]{BorelLAG}), as a~$\QQ_S$-Lie subalgebra of the Lie algebra of~$M_{\QQ_S}$.
\end{defi}

We will examine the interrelations between these properties in detail 
 in section \ref{SectionGalois}. We will also provide examples where these
various assumptions hold and where they do not.
In \ref{SectionGalois}, we will make apparent that these properties are very heavily dependent on the choice of the field $E$, particularly on whether $E$ is of 
finite type over $\QQ$ or not.

Note that a point of~$S$-Mumford-Tate type is obviously of~$S$-Tate type. These properties hold 
notably for a special point, by definition of a canonical model.
We will prove that $S$-Tate type property~\eqref{S-Tate} implies the algebraicity of~$U_S$ if~$E$ is of finite type. As a consequence, for such~$E$, the~$S$-Mumford-Tate type property~\eqref{S-MT} is equivalent to its \emph{a priori} weaker variant:~$H_S^0=M_{\QQ_S}$.

Our main results are proved under the $S$-Shafarevich hypothesis~\eqref{S-Sha} which as we will see in section \ref{SectionGalois} is implied by
(and in fact not far from being equivalent to) the $S$-Tate hypothesis~\eqref{S-Tate} together with the $S$-semisimplicity assumption~\eqref{S-semisimple}. 
More precisely, the main result of section \ref{SectionGalois} is that the $S$-Shafarevich
property is equivalent to $S$-semisimplicity together with the property that
the centraliser of $U_S$ in $G_S$ is compact modulo the centraliser of $M$ in $G_S$.

\subsection{Notions related to equidistribution.}

Let $s \in \Sh_K(G,X)(E)$. Recall that points of $\Hcal(s)$ are defined over $\ol E$.

To any $z \in \Hcal(s)$ we attach a probability measure with finite support that we define as follows. In what follows, for a point $z$ in $\Sh_K(G,X)$, we refer to the following set as its  Galois orbit:
$$
\Gal\left(\ol E/E\right) \cdot z = \left\{ \sigma(z)\,\middle|\,\sigma \in \Gal\left(\ol E/E\right) \right\}.
$$

\begin{defi} \label{measure}
Let $z = \ol{(x, t)}$ be a point of $\Hcal_S(s)$. We define
\begin{equation}\label{defi mesure}
\mu_z =  \frac{1}{\#\Gal\left(\ol E/E\right) \cdot z}\cdot\sum_{\zeta \in   \Gal\left(\ol E/E\right)\cdot z} \delta_{\zeta}
\end{equation}
where $\delta_{\zeta}$ is the Dirac mass at $\zeta$.
\end{defi}

We introduce a class of groups, arising from the use of Ratner's theorem.

\begin{defi}\label{defi S Ratner class}
A connected $\QQ$-subgroup~$L$ of $G$ is said to be of \emph{$S$-Ratner class} if its Levi subgroups are semisimple and for every~$\QQ$-quasi factor~$F$ of this Levi,~$F(\RR\times \QQ_S)$ is not compact\footnote{Equivalently such factor~$F$ can not be anisotropic simultaneously over~$\RR$ and every of the~$\QQ_v$ for~$v$ in~$S$.}. 

We denote by~$L^\dagger$ the subgroup of~$L(\RR\times\QQ_S)$ generated by the unipotent  elements. Then~$L$ is of $S$-Ratner class if and only if no proper subgroup of~$L$ defined over~$\QQ$ contains~$L^\dagger$.

We denote by~$L(\RR)^{+}$ the neutral component of~$L(\RR)$ with respect to the Archimedean topology. 
\end{defi}

This class is slightly more general\footnote{Actually every Zariski connected subgroup~$L$ with semisimple Levi subgroups will be of~$S$-Ratner class for some~$S$ big enough.} than~\cite[D\'{e}f.~2.1]{1-U} and~\cite[Def.~2.4]{1-UY2}. The latter actually corresponds to the case where~$S$ is empty, in which case we have~$L^\dagger=L(\RR)^+$.

Let~$L$ be subgroup of~$G$ of $S$-Ratner class and~$(h,t)\in X\times G(\AAA_f)$. 
To such data we associate the following subset (actually a real analytic variety):
\[
	Z_{L,(h,t)}=\left\{\ol{(l\cdot h,t)}~\middle|~l\in L(\RR)^+\right\}\subseteq \Sh_K(G,X).
\]
We consider the corresponding generalisation of a notion of weakly special subvariety from~\cite{1-UY2}.

\begin{defi}\label{ws subm}
A \emph{weakly $S$-special real submanifold (or subvariety)}~$Z$ of~$\Sh_K(G,X)$ is a subset of the form~$Z_{L,(h,t)}$ for some subgroup~$L$ of~$G$ of~$S$-Ratner class and~$(h,t)$ in~$X\times G(\AAA_f)$.
\end{defi}
Note that the parametrising map~$L(\RR)^+\xrightarrow{l\mapsto \ol{(l\cdot h,t)}} Z_{L,(h,t)}$ induces a homeomorphism
\[
	\left.
	\left(\Gamma_{tK}\cap L(\RR)^+\right)
	\middle\backslash 
	L(\RR)^+
	\middle/
	(L(\RR)^+\cap K_h)
	\right. 
	\to Z_{L,(h,t)}
\] 
where~$\Gamma_{tK}$ is an arithmetic subgroup of~$G(\QQ)$ depending only on~$tK$
whose  intersection with~$L(\RR)^+$ is an arithmetic subgroup, necessarily a lattice. 
Indeed, there is a canonical right $L(\RR)^+$-invariant probability measure on~$\left(\Gamma_{tK}\cap L(\RR)^+\right)\backslash L(\RR)^+$. We denote~$\mu_{L,(h,t)}$
its direct image in~$Z_{L,(h,t)}$, viewed as a Borel probability measure on~$\Sh_K(G,X)$.
\begin{defi}\label{canonical} Let~$Z$ be a weakly $S$-special real submanifold~$Z$ inside of~$\Sh_K(G,X)$. The \emph{canonical probability~$\mu_Z$ with support~$Z$}, is a measure of the form~$\mu_{L,(h,t)}$ with support~$Z=Z_{L,(h,t)}$.
\end{defi}

The next lemma shows that $\mu_{L,(h,t)}$ is independent of choices.

\begin{lem}
The canonical probability measure~$\mu_Z$ is well defined: it depends only on $Z$ and not on the choice of~$(L,(h,t))$. 
\end{lem}
\begin{proof}
Firstly note that~$Z$ is almost everywhere locally isomorphic to its 
inverse image~$\wt{Z}$ in~$X\times G(\AAA_f)/K$, and~$\mu_Z$ is determined by the corresponding locally finite measure~$\mu_{\wt{Z}}$ on~$\wt{Z}$.
It will suffice to show that~$\mu_{\wt{Z}}$ is intrinsic up to a locally constant scaling factor, the latter being characterised by~$\mu_Z$ being a probability and~$Z$ being connected. Endow~$X$ with a~$G(\RR)$-invariant Riemannian structure, which we extend to~$X\times G(\AAA_f)/K$. Then~$\mu_{\wt{Z}}$ is locally proportional to the volume form of the induced Riemannian structure on~$\wt{Z}$. It suffices to check it for an orbit~$L(\RR)^+\cdot x$ in~$X$.
But the~$L(\RR)^+$-invariant measure on~$L(\RR)^+\cdot x$, the Haar measure,
is unique up a factor, and, as~$L(\RR)^+\leq G(\RR)$ acts by isometries on~$X$, the
Riemannian volume form on~$L(\RR)^+\cdot x$ is a Haar measure.
\end{proof}

\subsubsection{Weakly special subvarieties} The first  statement below is a slight generalisation of~\cite[Prop.~2.6]{1-UY2}. This in fact is a direct consequence of the hyperbolic Ax-Lindemann-Weierstrass theorem proven in \cite{1-KUY}.

\begin{prop}[\cite{1-KUY}, \cite{1-UY}]\label{Ax}
 The Zariski closure of a weakly $S$-special real submanifold 
is a weakly special subvariety.
\end{prop}
The following is a generalisation of an observation\footnote{We quote:~``In the case where~$h$ viewed as a morphism from~$\SSS$ to~$G_\RR$ factors through~$H_{\RR}$, the corresponding real weakly special subvariety has Hermitian structure and in fact is a weakly special subvariety in the usual sense''.} of~\cite[p.~2]{1-UY2}.

\begin{prop}Let~$Z=Z_{L,(h,t)}$ be a weakly $S$-special real submanifold of~$\Sh_K(G,X)$.
If~$L$ is normalised by~$h$, then~$Z$ a weakly special subvariety: it is Zariski closed.
\end{prop}
\subsection{Main theorems} 
We may finally state our main theorems, which give a stronger form of Conjecture~\ref{APS}, at the cost of~$S$-Tate type assumption.

We now state our first main result.

\begin{teo}[Inner Equidistributional $S$-Andr\'{e}-Pink-Zannier] \label{Theoreme1}
Let~$s$ be a point of ~$\Sh_K(G,X)$ defined over a field $E$ such that 
\item the point $s$ is of  $S$-Shafarevich type.

Let~$(s_n)_{n\geq0}$ be a sequence of points in the~$S$-Hecke orbit~$\Hcal_S(s)$ of~$s$, and denote $(\mu_n)_{n\geq0}$ the sequence 
of measures attached to the~$s_n$ as in definition \ref{measure}.

After possibly extracting a subsequence and replacing $E$ by a finite extension, 
there exists a finite set~$\cF$ of weakly $S$-special real submanifolds $Z$ with canonical probability measure~$\mu_Z$  (as defined in Definitions~\ref{ws subm} and~\ref{canonical})  such that 
\begin{enumerate}
 \item  \label{TheoremeLimite}	 the sequence~$(\mu_n)_{n\geq 0}$ tightly converges to a barycentre~$\mu_\infty$ of the~$\mu_Z$,
 \item and for all $n\geq 0$, we have~${\rm Supp}(\mu_n)\subseteq{\rm Supp}(\mu_\infty)=\bigcup_{Z\in\cF}  Z$. \label{TheoremeInner}	
\end{enumerate}

\begin{enumerate}\setcounter{enumi}{2}
 \item If $s$ is of $S$-Mumford-Tate type,
 then every $Z$ in~$\cF$ is a weakly special subvariety.\label{TheoremeMT}
\end{enumerate}
\end{teo}
\noindent We may refer to property~\eqref{TheoremeLimite} as ``equidistribution'' property -- a shorthand for a convergence of measures --
and to property~\eqref{TheoremeInner} by saying that this equidistribution is ``inner''. 

Note that conclusion~\eqref{TheoremeMT} applies to special points, in which case every~$Z$ in~$\cF$ is actually a \emph{special} subvariety.

We deduce from theorem  \ref{Theoreme1} the following theorem, which is more directly 
related to the Andr\'{e}-Pink-Zannier conjecture. Let us stress that, in the deduction process, 
we need not only~\eqref{TheoremeLimite}, but also~\eqref{TheoremeInner}, from 
Theorem~\ref{Theoreme1}.

\begin{teo}[Topological and Zariski $S$-Andr\'{e}-Pink-Zannier] \label{Theoreme2}
Let~$s$ be a point of a Shimura variety~$\Sh_K(G,X)$ defined over a field~$E\subseteq \CC$.
Consider a subset $\Sigma\subseteq \Hcal_S(s)$ of its~$S$-Hecke orbit and denote
$$
\Sigma_E = \Gal\left(\ol E/E\right)\cdot \Sigma 
=
\left\{ \sigma(x)~\middle|~\sigma \in \Gal(\ol E/E), x \in \Sigma \right\}.
$$
Then
\begin{enumerate}
\item \label{Theoreme2-1}
 If $s$ is of $S$-Shafarevich type then the topological closure of~$\Sigma_E$
 is a finite union of weakly $S$-special real submanifolds;

 Furthermore, the Zariski closure of $\Sigma$ is a finite union of weakly special subvarieties.
 
\item \label{Theoreme2-2}
If $s$ is of $S$-Mumford-Tate type, then the  topological closure of~$\Sigma_E$ is a finite union of weakly special subvarieties;

\end{enumerate}
\end{teo}

We will prove that~\eqref{Theoreme2-1} holds whenever
~\eqref{TheoremeLimite} and~\eqref{TheoremeInner} from Theorem~\ref{Theoreme1} hold for sequences in~$\Sigma$. The second statement will then follow of Proposition \ref{Ax}.

When the $S$-Mumford-Tate property holds, the conclusion ~\eqref{TheoremeMT} from Theorem~\ref{Theoreme1} will imply
~\eqref{Theoreme2-2}. 

\subsubsection{A converse statement.}

Let us end with a statement emphasizing the importance of property~\eqref{TheoremeInner} of Theorem~\ref{Theoreme1}. This statement 
makes precise the idea that property~\eqref{TheoremeInner} of Theorem~\ref{Theoreme1} implies the~$S$-Sha\-fa\-rev\-ich property.

This shows that the $S$-Shafarevich property assumption is essential and optimal.

\begin{prop} Let~$s$ be a point in a Shimura variety~$\Sh_K(G,X)$ defined over a field~$E$ and let~$\Hcal_S(s)$ be its~$S$-Hecke orbit.

Assume that for any sequence~$(s_n)_{n\geq0}$ in~$\Hcal_S(s)$, for any 
finite extension~$F$ of~$E$, there is an extracted 
subsequence for which the associated measure~$\mu_n$ converges weakly
to a limit~$\mu_\infty$ in such a way that 
\begin{equation}\label{innerprop}
\forall n\geq 0, {\rm Supp}(\mu_n)\subseteq {\rm Supp}(\mu_\infty).
\end{equation}

Then~$s$ is of~$S$-Shafarevich type.
\end{prop}
\begin{proof} Assume for contradiction that~$s$ is not of~$S$-Shafarevich type.
Then there is a finite extension~$F$ of~$E$ such that there is an infinite
sequence~$(s_n)_{n\geq 0}$ of pairwise distinct~$F$-rational points in~$\Hcal_S(s)$. After possibly extracting a subsequence, we may assume that this 
sequence is convergent or is divergent in~$\Sh_K(G,X)$.

As these~$s_n$ are rational points, the associated measures~$\mu_n$ are Dirac masses. We recall that weak convergence of Dirac masses is induced by
convergence in the Alexandroff compactification, with the point
at infinity corresponding to the zero measure.

If~$(s_n)_{n\geq 0}$ is divergent, so is any subsequence,
and the measure~$\mu_\infty$ will be the~$0$ measure, in which case~\eqref{innerprop} may not hold.

If~$(s_n)_{n\geq 0}$ converges to~$s_\infty$, then~$\mu_\infty$ will be the Dirac measure~$\delta_{s_\infty}$, and~\eqref{innerprop} means that~$(s_n)_{n\geq 0}$ is a stationary sequence, which it cannot be since the~$s_n$ are pairwise distinct.
This yields a contradiction.
\end{proof}

\subsection{Plan of the Article} In Section~\ref{consequence} we explain how to deduce Theorem~\ref{Theoreme2}
from Theorem~\ref{Theoreme1}. 
Section~\ref{SectionGalois} reviews Galois representations,
various properties listed before ($S$-Mumford-Tate, $S$-Shafarevich, $S$-Tate, etc)
and relations between them. We in particular prove useful and practical group-theoretic
characterisation of the $S$-Shafarevich property.
We also provide examples and counterexamples of when the properties do and do not hold depending on the field $E$. 
We believe the contents and results of this section to be of independent interest.

The sections that follow are devoted to the proof of  Theorem~\ref{Theoreme1}.
\begin{itemize}
\item 
Section~\ref{Reduction} explain how to reduce to a situation
falling under the scope of application of~\cite{1-RZ}. It ends 
by invoking~\cite{1-RZ}, which immediatly gives us~\eqref{TheoremeLimite}
of Theorem~\ref{Theoreme1}
\item Section~\ref{SectionRZ}  then discusses how
to get~\eqref{TheoremeInner} of Theorem~\ref{Theoreme1}.
\item Finally Section~\ref{SectionMT} treats the
stronger conclusion we can reach under the $S$-Mumford-Tate hypothesis. 
\end{itemize}

\section*{Acknowledgments}
While working on this article,
both authors were supported by the ERC grant (Project 511343, SPGSV). They gratefully acknowledge ERC's support.
The first named author is grateful to UCL for hospitality.

\newpage
\section{From Inner equidistribution to Topological closures}\label{consequence}

In this section we show how to derive Theorem~\ref{Theoreme2} from Theorem~\ref{Theoreme1}.
The main result is proposition \ref{Sorite} which implies the main theorem of this section:

\begin{teo} \label{implication}
The conclusions of Theorem~\ref{Theoreme1} imply the conclusions of Theorem~\ref{Theoreme2}.
\end{teo}

In order to prove this theorem,
 we need to develop a dimension theory of weakly $S$-special real submanifolds.

\subsection{Dimension and Measure in chains of weakly $S$-special real submanifolds}\label{subsec dimension}

We prove here some standard properties about inclusions of weakly~$S$-special real submanifolds, involving dimension, that we define, and their canonical measure.

\begin{defi} Let~$Z=Z_{L,(h,t)}$ be a weakly $S$-special real sub\-ma\-ni\-fold. Then we define the dimension of~$Z$
as  the codimension of the stabiliser~$K_h\cap L(\RR)$ of~$h$ in~$L(\RR)$. This is also the dimension~$L(\RR)^+\cdot h$, or equivalently~$L(\RR)\cdot h$  in~$X$, as a semialebraic set and as a real analytic variety.
\end{defi}

\begin{lem}\label{lemme canonique}
\begin{enumerate}
\item The dimension of a weakly $S$-special real sub\-ma\-ni\-fold is well defined. If~$Z_{L_1,(h_1,t_1)}=Z_{L_2,(h_2,t_2)}$,
then~
\begin{equation}\label{defi dim}
\dim\left( L_1(\RR)^+\cdot h_1\right)=\dim\left( L_2(\RR)^+\cdot h_2\right).
\end{equation}
\item Let~$Z_1\subsetneq Z_2$ be two $S$-special real submanifolds. 
Then
\begin{equation}
\dim Z_1< \dim Z_2.
\end{equation}
\item  \label{mesure ortho}
Let~$Z_1$ and~$Z_2$ be two weakly $S$-special real submanifolds, such that~$\mu_{Z_1}(Z_2)\neq 0$.
Then~$Z_1\subseteq Z_2$, and~$\mu_{Z_1}(Z_2)= 1$.
\end{enumerate}
\end{lem}\label{lemme chains}

One immediately deduces the following.

\begin{cor}  \label{length}
Let~$Z_1\subsetneq \ldots \subsetneq Z_l$ be a chain of strictly included weakly $S$-special real submanifolds.
Then its length~$l$ satisfies~$l\leq 1+\dim(G)$. 
\end{cor}

From this we deduce the following.

\begin{cor}\label{coro max}
Any non empty collection~$\mathcal{F}$ of weakly $S$-special real submanifolds, partially ordered by inclusion, has maximal elements, and any element of~$\mathcal{F}$ is contained in a maximal element of~$\mathcal{F}$.
\end{cor}
\begin{proof}[Proof of the Corollary~\eqref{coro max}] By induction we may extend every chain in~$\mathcal{F}$ to a maximal one. By \ref{length}, this induction terminates after at most~$1+\dim(G)$ steps. 

The last element of a non empty maximal chain is a maximal element.
Hence any element~$f$, seen as a chain of length one, is part of  a non empty maximal chain. The last element of the latter contains~$f$ and is maximal in~$\mathcal{F}$.
If~$\mathcal{F}$ has an element, this implies that there is a maximal element.
\end{proof}

\begin{proof} Write~$Z_i$ for ~$Z_{L_i,(h_i,t_i)}$ (for~$i=1$ or~$i=2$).

\emph{We assume that the intersection~$Z_1\cap Z_2$ is not empty.}

These two subsets~$Z_1$ and~$Z_2$ of~$\Sh_K(G,X)$ are connected, and hence belong to the same connected component of~$\Sh_K(G,X)$.
This implies, as subsets in~$G(\AAA)$,
\[G(\QQ)\cdot\left( G(\RR)\times t_1K\right)=G(\QQ)\cdot \left(G(\RR)\times t_2K\right).\]
Left translating~$t_1$ with~$\gamma\in G(\QQ)$ and right translating with~$k\in K$ we may assume~$t_1=t_2$.
We have to substitute accordingly~$h_1$ with~$\gamma h_1$ and~$L_1$ with~$\gamma L_1 \gamma^{-1}$.
As we have
\[
\gamma L_1 \gamma^{-1}\cdot \gamma h_1=\gamma\left(L_1\cdot h_1\right)
\]
this does not change the notion of dimension of~$Z_1$.

\emph{We now assume that~$t_1=t_2$, which we will be denote simply~$t$.}

Let~$\Gamma_{tK}$ be the inverse image in~$G(\QQ)$ of~$tKt^{-1}$ with respect to the map~$G(\QQ)\to G(\AAA_f)$. This is the
arithmetic subgroup such that the previous component of~$\Sh_K(G,X)$ belongs to those of~$\Gamma_{tK}\backslash X\times\{t\}$.
We will identify~$X\times\{t\}$ with~$X$ for simplicity. The inverse images of~$Z_1$ and~$Z_2$ in~$X$ are
\(\wt{Z_1}=\Gamma_t\cdot L_1(\RR)^+\cdot h_1\) and \(\wt{Z_2}=\Gamma_t\cdot L_2(\RR)^+\cdot h_2\) respectively.

We may write 
\[
\wt{Z_1}\cap\wt{Z_2}=\Gamma_{tK}\cdot\left(\left( L_1(\RR)^+\cdot h_1\right)\cap \wt{Z_2}\right)
\]
and 
\begin{equation}\label{inter union}
\left(L_1(\RR)^+\cdot h_1\right)\cap \wt{Z_2}=\bigcup_{\gamma\in\Gamma_t} \left(L_1(\RR)^+\cdot h_1\right)\cap\left(\gamma_t\cdot L_2(\RR)^+\cdot h_2\right).
\end{equation}

\emph{Assume first that~$\mu_{Z_1}(Z_2)\neq 0$.}

By our definition of~$\mu_{Z_1}$, the~$\left(\Gamma_{tK}\cap L_1(\RR)^+\right)$-saturated set
\begin{equation}\label{sature}
\left(\Gamma_{tK}\cap L_1(\RR)^+\right)\cdot\left(L_1(\RR)^+\cdot h_1\right)\cap \wt{Z_2}
\end{equation}
is non negligible (Cf. Lemma~\ref{lemme neglect} proven below) in~$L_1(\RR)^+\cdot h_1$ with respect to a Haar measure on 
the homogeneous~$L_1(\RR)^+$-set~$L_1(\RR)^+\cdot h_1$. But this~\eqref{sature} is again 
the countable union~\eqref{inter union}. So there is a~$\gamma$ in~$\Gamma_{tK}$ such
that
\[
\left(L_1(\RR)^+\cdot h_1\right)\cap\left(\gamma_t\cdot L_2(\RR)^+\cdot h_2\right)
\]
is not negligible. This is a real semi-algebraic subset of the real semi-algebraic set~$L_1(\RR)^+\cdot h_1$.
We use a cylindrical cellular decomposition of tis subset. Subset of codimension~$1$ are negligible\footnote{}.
So at least one cell has codimension~$0$. It must have non empty interior. 

The orbit map~$L_1(\RR)^+\to L_1(\RR)^+\cdot h_1$ are open maps. So there is an open subset~$U$
in~$L_1(\RR)^+$ such that~$U\cdot h_1\subseteq \gamma L_2(\RR)^+\cdot h_2$. But~$L_1$ is Zariski
connected, and~$U$ is Zariski dense in~$L_1$. Hence~
\[L_1(\RR)^+\cdot h_1\subseteq (L_1\cdot h_1)(\RR)\subseteq (\gamma L_2\cdot h_2)(\RR).\]
We note that~$L_1(\RR)^+\cdot h_1$ is a connected component of~$(L_1\cdot h_1)(\RR)$. 
Likewise~$\gamma L_2(\RR)^+\cdot h_2$ is a connected component of~$(\gamma L_2\cdot h_2)(\RR)$.
But the connected~$L_1(\RR)^+\cdot h_1$ intersects the component~$\gamma L_2(\RR)^+\cdot h_2$, hence
is contained in it. It follows~$\wt{Z_1}\subseteq \wt{Z_2}$ and finally~$Z_1\subseteq Z_2$.
We have proved the last point of the lemma. 

We now turn to the first point.

\emph{We now assume~$Z_1\subseteq Z_2$ instead of~$\mu_{Z_1}(Z_2)\neq 0$.}

Then certainly~$\mu_{Z_1}(Z_2)\neq0$. We  may, and will, keep the notations above. We have already proved
\[
L_1(\RR)^+\cdot h_1\subseteq \gamma L_2(\RR)^+\cdot h_2
\]
for some~$\gamma$ in~$\Gamma_t$. It follows
\[
	\dim(L_1(\RR)^+\cdot h_1)\leq\dim( \gamma L_2(\RR)^+\cdot h_2)=\dim(L_2(\RR)^+\cdot h_2).
\]
If~$Z_1=Z_2$ we may echange te roles to get a converse comparison, yielding~\eqref{defi dim}: the dimension of~$Z_i$ is well defined.
This was the first point of the lemma. 

It remains to prove the second point.
\emph{We now assume~$Z_1\subsetneq Z_2$.}
We keep our notations. We have proved
\[
L_1(\RR)^+\cdot h_1\subseteq \gamma L_2(\RR)^+\gamma^{-1}\cdot\gamma h_2.
\]
The reverse inclusion does not hold, as it would, easily, imply~$Z_2\subseteq Z_1$. 
We may substitute our base point~$\gamma h_2$ with~$h_1$, as it does belong to the same~$\gamma L_2(\RR)^+\gamma^{-1}$ orbit. We deduce
\[
L_1(\RR)^+\cdot h_1\subsetneq \gamma L_2(\RR)^+\gamma^{-1}\cdot h_1.
\]
Assume by contradiction that both sides have same dimension. The orbit~$L_1(\RR)^+\cdot h_1$ is closed
in~$X$, and \emph{a fortiori} closed in~$L_2(\RR)^+\gamma^{-1}\cdot h_1$. Furthermore~$L_1(\RR)^+\cdot h_1\to L_2(\RR)^+\cdot h_1$, as a map of differential manifolds, is a submersion at at least a point, by equality of dimensions.
It is a submersion everywhere y homogeneity, hence an open map. As a consequence,~$L_1(\RR)^+\cdot h_1$ is not only closed, but open as well in~$L_2(\RR)^+\cdot h_1$. As~$L_2(\RR)^+\gamma^{-1}\cdot h_1$ is connected, we deduce
\[
L_1(\RR)^+\cdot h_1= \gamma L_2(\RR)^+\gamma^{-1}\cdot h_1.
\]
That is a condradiction. This ends our proof.
\end{proof}

We finish this section by proving a lemma used in the proof above.

\begin{lem}\label{lemme neglect} Endow on~$(\Gamma\cap L )\backslash L$ with a Haar measure. Then the inverse of a negligible 
subset in~$(\Gamma\cap L )\backslash L$ is negligible in~$L$, with respect to a Haar measure.
\end{lem}
\begin{proof} Let~$N$ be a negligible subset in~$(\Gamma\cap L )\backslash L$, and~$\wt{N}$ its inverse
image in~$L$. The lemma amounts to proving that
\begin{equation}\label{neglect}
\int_{l\in L} 1_{\wt{N}}~dl=0
\end{equation}
where~$1_{\wt{N}}$ is the characteristic function  of~$\wt{N}$ and~$dl$ is a Haar measure on~$L$.
Let~$K$ be a compact subset in~$L$. And consider, as a real function~$(\Gamma\cap L )\backslash L\to\RR$,
\[
f:(\Gamma\cap L)\cdot l\mapsto \int_{\gamma \in \Gamma\cap L} 1_{K\cap N}(\gamma l)~d\gamma,
\]
where~$d\gamma$ is a Haar measure on~$(\Gamma\cap L )$.
Then we have, see~\cite[VII~\S2.1]{BBK-I-VII},
\[
\int_{l\in L} 1_{\wt{N}\cap K}~dl=\int_{(\Gamma\cap L )\backslash L} f(x) dx
\]
where~$dx$ is the quotient Haar measure (cf. loc. cit.) on~$(\Gamma\cap L )\backslash L$. But the
support of~$f$ is contained in~$N$, hence is negligible. The last integral eveluates as zero.
By choosing increasing compact subsets whose union is~$L$, we, by the monotone convergence~$\lim_K 1_{\wt{N}\cap K}=1_{\wt{N}}$,
deduce~\eqref{neglect}.
\end{proof}

\subsection{Topological and Zariski closures.} \label{topologicalsub}
We place ourselves in the situation of Theorem~\ref{Theoreme2}. In particular we assume that~$s$ is of~$S$-Shafarevich type.

Let $\Sigma$ be as in theorem \ref{Theoreme2}. It is a countable set, and we write it as
~$\Sigma = \{ s_n , n \geq 0 \}$. Let~$(\mu_n)_{n\geq0}$ be the sequence of probability measures attached to~$(s_n)_{n\geq0}$ as in Definition~\ref{measure}. As the~$S$-Shafarevich hypothesis is assumed, we are permitted to
invoke Theorem~\ref{Theoreme1} for any infinite subsequence.

 Our proof of Theorem~\ref{Theoreme2} relies on the following.
 
\begin{prop}\label{Sorite} We consider the following situation.
\begin{itemize}
\item[---] Let~$\mathcal{S}$ be the set of supports of limits of converging subsequences of~$(\mu_n)_{n\geq0}$.
\item[---] Let~$\mathcal{Z}$ be the collection of $S$-special real submanifolds~$Z$ such that the canonical measure~$\mu_Z$ occurs in the decomposition of the limit of a converging subsequence of~$(\mu_n)_{n\geq0}$. 
\item[---] We endow~$\mathcal{Z}$ with the partial order induced by inclusion. Let~$\mathcal{M}$ be the subset of maximal elements in~$\mathcal{Z}$.
\end{itemize}
We have the following.
\begin{enumerate}[label=\roman*)]
\item \label{Scholie1} Every support~$S$ belonging to~$\mathcal{S}$ is a finite union of finitely many weakly $S$-special real submanifolds belonging to~$\mathcal{Z}$. If~$s$ is of $S$-Mumford-Tate type, then the~$Z$ belonging to~$\mathcal{Z}$ are actually weakly special subvarieties.
\item \label{Scholie2} Every element~$Z$ of~$\mathcal{Z}$ is included in a maximal element of~$\mathcal{Z}$, an element belonging to~$\mathcal{M}$.
\item \label{Scholie3} The subset~$\mathcal{M}$ of maximal elements of~$\mathcal{Z}$ is a finite subset.
\item \label{Scholie4} Every~$Z$ in~$\mathcal{Z}$, or~$S$ in~$\mathcal{S}$, is contained in the topological closure of~$\Sigma_E$.
\item \label{Scholie5} All but finitely many elements of~$\Sigma_E$ are in~$\bigcup_{Z\in \mathcal{M}} Z$.
\item \label{Scholie6} The topological closure of $\Sigma_E$ is a finite union 
of weakly $S$-special real manifolds. 
\item \label{Scholie7} The Zariski closure of $\Sigma_E$ is a finite union of weakly special subvarieties.
\end{enumerate}
\end{prop}

To justify the definition of $\mathcal{Z}$, we need to show
 that the~$\mu_Z$ that occur in the sum  with a nonzero coefficient of a limit measure~$\mu$ are defined unambiguously.
 
 This is a consequence of the following:
 
 \begin{lem}
Any finite set of canonical measures~$\mu_Z$ is linearly independent.
\end{lem}
\begin{proof} Consider a linear combination~$\mu=\lambda_1\mu_{Z_1}+\ldots+\lambda_n\mu_{Z_n}$. We may compute~$\mu(Z)$ by using~\ref{lemme canonique}\,\eqref{mesure ortho}. It follows that we recover the coefficient of~$\mu_Z$ as the measure~$\mu(Z)$ minus the coefficients associated with subvarieties in~$Z$. To see it is well defined, we argue by induction on the dimension of~$Z$ to check that we thus obtain only finitely many non zero coefficients, because these agree with the~$\lambda_i$. We refer to~Corollary~\ref{lemme chains} for justification why this induction is legit.

So the coefficient of~$\mu_Z$ in~$\mu$ is uniquely defined.
\end{proof}

We proved that \ref{Theoreme1} implies $(1)$ and $(2)$ of Theorem \ref{Theoreme2}.
We split the proof of ~\ref{Scholie3} into two lemmas below. Lemma~\ref{lem_sh3} is
\ref{Scholie3}.

\begin{proof}[of  Proposition~\ref{Sorite}]

The statement~\ref{Scholie1} is a direct consequence of  Theoreme \ref{Theoreme1}.

The statement \ref{Scholie2} is Corollary~\ref{coro max} from the previous section.

To prove \ref{Scholie3} we prove two lemmas. The statement \ref{Scholie3} is lemma
~\ref{lem_sh3}.

\begin{lem}
The set~$\mathcal{Z}$ is countable.
\end{lem}
\begin{proof} Every element of~$\mathcal{Z}$ can be associated with the data
\begin{itemize}
\item of some group of Ratner class, which is a algebraic subvariety over~$\QQ$ of~$G$, hence belong to a countable class;
\item of some point in the $S$-Hecke orbit~$\Sigma$, which is countable.
\end{itemize}
As there only finitely many possibilities for these data, we can construct at most countably many elements in~$\mathcal{Z}$.
\end{proof}

\begin{lem}\label{lem_sh3}
The set~$\mathcal{M}$ is finite.
\end{lem}
\begin{proof}[Proof of the last claim.] Assume for contradiction that~$\mathcal{M}$ is infinite. It is countable. Hence we can we can arrange its elements as a sequence $\mathcal{M} = (M_n)_{n\geq 0}$ such that the $M_n$ are the distinct maximal (for inclusion) elements of $\cZ$. We arrange~$\cZ$ likewise in a sequence~$(Z_n)_{n\geq0}$.

Define~$S_n=\Sh_K(G,X)\smallsetminus \bigcup_{i<n} M_i$. This is an open subset.
By maximality of the~$M_i$, we have~$\mu_{M_i}(M_j)=0$ whenever~$i\neq j$ by~Lemme~\ref{lemme canonique}\,\eqref{mesure ortho}. Hence~$S_n$ is of full measure for~$\mu_{M_n}$.

We will use a diagonal argument.

By definition, there is a convergent subsequence, say~$(\mu^{(n)}_m)_{m\geq 0}$, of the sequence~$(\mu_m)_{m\geq 0}$ such that its limit, say~$\mu^{(n)}_\infty$, admits~$\mu_{M_n}$ as a component, with some non zero coefficient~$\lambda_n$.
By convergence, there is some~$N_n$ such that for~$m\geq N_n$ we have~$\mu^{(n)}_m(S_n)\geq\lambda_n/2$.
Write~$\mu^{(n)}_{N_n}$. Consequently~${\rm Supp}\nu_n$ is not included in~$M_1\cup\ldots \cup M_{n-1}$. We deduce that no finite union of subsets~$M$ from~$\mathcal{M}$ can support infinitely many of the~$\nu_n$. As any~$Z$ from~$\mathcal{Z}$ is contained in some~$M$ from~$\mathcal{M}$, by~\eqref{Scholie2}, no finite union of such~$Z$  can support infinitely many of the~$\nu_n$. A fortiori no~$S$ from~$\mathcal{S}$  can support infinitely many of the~$\nu_n$.

But, by Theorem~\ref{Theoreme1} we may extract a subsequence from~$(\nu_n)_{n\geq0}$ which is converging, say with limit~$\nu’_\infty$, satisfying the conclusions of Theorem~\ref{Theoreme1}, and
\[
{\rm Supp}(\nu’_n)\subseteq {\rm Supp}(\nu’_\infty)\in\mathcal{S}.
\]
This yields a contradiction.
\end{proof}

The statement~\ref{Scholie4} is obvious: the topological closure of~$\Sigma_E$ is a closed subset containing 
the support of the~$\mu_n$, and hence contains the support of any limit of a subsequence of~$(\mu_n)_{n\geq0}$.

Finally, to prove~\ref{Scholie5},
assume for contradiction that there exists an infinite subsequence $(s_n)$ of points of the set $\Sigma_E$ which are
not in $\bigcup_{Z\in \mathcal{M}}Z$. Let $(\mu_n)$ be the associated sequence of measures as defined in \ref{measure}.
 By theorem \ref{Theoreme1},
after possibly extracting a subsequence, we may assume that $(\mu_n)$ converges to a measure $\mu$ whose support contains ${\rm Supp}(\mu_n)$ for all $n$.
By definition, we have that  
$$s_n\in{\rm Supp}(\mu_n)\subseteq{\rm Supp}(\mu) \subset \bigcup_{S\in \mathcal{S}} S=\bigcup_{Z\in \mathcal{Z}} Z=\bigcup_{Z\in \mathcal{M}} Z.
$$
This contradicts the choice of $s_n$.

The statement~\ref{Scholie6} follows directly from \ref{Scholie5} and
statement \ref{Scholie7} follows from \ref{Scholie6} and the fact that 
the Zariski closure of a weakly $S$-special manifold is a weakly special subvariety (Proposition \ref{Ax}).

We have finished proving proposition~\ref{Sorite} and hence Theorem \ref{implication}.

\end{proof}


\section{Review of the Galois monodromy representations.}\label{SectionGalois}
\subsection{Construction of representations.}

In this section we recall the construction of Galois representations attached to points in Hecke orbits, and then specialise to $S$-Hecke orbits. 

The contents of this section are mostly taken from Section 2 of \cite{1-UY1}. There it
was assumed that $E$ was a number field, however all arguments carry over 
verbatim to an arbitrary field of characteristic zero.

Let $\Sh(G,X)$ be the Shimura variety of infinite level, it is the profinite cover
\[
	\Sh(G,X)(\CC)=\varprojlim_K \Sh_K(G,X)
\]
with respect to the finite maps induced by the inclusions of compact open subgroups.
By Appendix to \cite{1-UY1}, the centre $Z$ of $G$ has the property that 
$Z(\QQ)$ is discrete in $G(\AAA_f)$.
It follows (Theorem 5.28 of \cite{1-Milne}), we have
$$
\Sh(G,X)(\CC)=G(\QQ)\backslash\left(X\times G(\AAA_f)\right).
$$
The scheme $\Sh(G,X)$ is endowed with the right action of $G(\AAA_f)$
which is defined over the reflex field $E(G,X)$.

We let 
$$
\pi \colon \Sh(G,X) \lto \Sh_K(G,X)
$$
be the natural projection.
The Shimura varieties $\Sh(G,X)$ and $\Sh_K(G,X)$ are defined over the reflex field $E(G,X)$ and so is the map $\pi$.

Let $s = \overline{(h,t)}$ be a point of $\Sh_K(G,X)$ defined over a field $E$. 
Lemma 2.1 of \cite{1-UY1} shows that the fibre $\pi^{-1}(s)$ has a 
transitive fixed point free right action of $K$. Explicitly
$$
\pi^{-1}(s) = \overline{(h,tK)}
$$
and the action of $K$ is the obvious one.
Furthermore, $\Gal(\overline{E}/E)$ acts on $\pi^{-1}(s)$ and this action commutes
with that of $K$. This is a consequence of the theory of canonical models of Shimura varieties (see \cite{1-Milne} and \cite{1-Deligne}).
By elementary group theoretic Lemma 2.4 of \cite{1-UY1}, we obtain a
morphism 
$$
\rho_{s} \colon \Gal(\overline{E}/E) \lto K
$$ 
such that the Galois action on $\pi^{-1}(s)$ is described as follows:
$$
\sigma((h, tk)) = (h, k \rho_s(\sigma)).
$$
This representation is continuous since an open subgroup $K'$ of $K$ is of finite index
and hence $\rho_{s}^{-1}(K')$ contains $\Gal(\overline{E}/F)$ for a finite extension
$F/E$.

The representation $\rho_{s}$ has the following fundamental property.
Let $M$ be the Mumford-Tate group of $h$. To $M$ one associates the Shimura datum
$(M,X_M)$ with $X_M = M(\RR) \cdot h$. Let $E_M$ be the reflex field
of $(M,X_M)$ and $E'$ the subfield of $\ol{E}$ generated by $E$ and $E_M$.

\begin{prop}
We have
$$
\rho_{s}(\Gal(\overline{E}/E')) \subset M(\AAA_f) \cap K.
$$
\end{prop}
\begin{proof}
This is Proposition 2.9 of \cite{1-UY1}. This proposition is
stated with $E$ a number field, however the proof goes through
without any changes for an arbitrary field of characteristic zero.
\end{proof}

The representation $\rho_{s}$ describes the action of $\Gal(\overline{E}/E)$ on the 
Hecke orbit $\Hcal(s)$.
Let $s' = \overline{(h, tg)}$ be a point of $\Hcal(s)$.
Consider the point  $\widetilde{s'} = (h, tg)$ of $\pi^{-1}(s)$.
Let $\sigma \in \Gal(\overline{E}/E)$.
Since the action of $G(\AAA_f)$ is defined over $E(G,X)$, we have
$$
\sigma( (h, tg))= (\sigma((h,t))\cdot g = (h, \rho_s(\sigma) t)g = (h, \rho_s(\sigma) tg)
$$
By applying $\pi$ to this relation and using the fact that $\pi$ is defined over $E(G,X)$,
we obtain:
$$
\sigma(\overline{(h, tg)}) = \overline{(h, \rho_s(\sigma) tg)}.
$$

We now describe the representation $\rho_{s,S}$ and Galois action the $S$-Hecke orbit
$\Hcal_S(s)$. 
Recall that $K$ is product of compact open subgroups $K_p$ of $G(\QQ_p)$.
Let $K_S := \prod_{p \in S} K_p$ and $K^S = \prod_{p \notin S} K_p$.
We denote by $p_S$ the projection map 
$$
p_S \colon G(\AAA_f) \lto G_S.
$$
Clearly $p_S(K) = K_S$.

\begin{defi} \label{defSGal}
We define 
$$
\rho_{s,S} := p_S \circ \rho_{s} \colon \Gal(\ol{E}/E) \lto K_S. 
$$ 
\end{defi}

Let 
$$
\Sh_{K^S}(G,X) = G(\QQ) \backslash X \times G(\AAA_f) / K^S.
$$
This is a scheme defined over $E(G,X)$ endowed with a continuous right $G_S$ action 
and a morphism $\pi^S \colon  \Sh(G,X) \lto \Sh_{K^S}(G,X)$ defined over $E(G,X)$.

The maps $\pi^S \colon  \Sh(G,X) \lto \Sh_{K^S}(G,X)$ and
$\pi_S \colon  \Sh_{K^S}(G,X) \lto \Sh_{K}(G,X)$ are defined over $E(G,X)$.
Furthermore,
$$
\pi = \pi_S \circ \pi^S.
$$

Contemplation of these properties and the properties of $\rho_{s,S}$ show the following:

\begin{teo} \label{Sadic}
The morphism 
$$
\rho_{s,S} \colon \Gal(\ol{E}/E) \lto K_S
$$ 
we have constructed above has the following properties:
\begin{enumerate}
\item
The representation $\rho_{s,S}$ is continuous.
\item
Let $s' = \overline{(h, tg)}$ (with $g \in G_S$) be a point in $\Hcal_S(s)$.
Then for any $\Gal(\ol{E}/E)$, we have 
$$
\sigma(s') = \ol{(h, \rho_{s,S}(\sigma)tg)}.
$$
\item
After replacing $E$ by $EE_M$, we have
$$
\rho_{s,S}(\Gal(\ol{E}/E)) \subset M(\AAA_f) \cap K_S = M(\QQ_S) \cap K
$$
(intersection taken inside $G(\AAA_f)$). 
\end{enumerate}
\end{teo}

\subsection{Properties of $S$-Galois representations.} 

In this section we examine in detail the properties \ref{ProprietesGaloisiennes}.

\subsubsection{General assumptions.} \label{assumptions}

For the sake of the ease of reading, we recall the general situation.
Let $S$ be a finite set of places of $\QQ$. Let $(G,X)$ be a Shimura datum,
normalised so that $G$ is the generic Mumford-Tate group on $X$
and $K$ a compact open subgroup of $G(\AAA_f)$ satisfying the following conditions:
\begin{enumerate}
\item $K$ is a product $K = \prod K_p$ of compact open subgroups 
$K_p \subset G(\QQ_l)$.
\item There eists an $l \geq 3$ and $l \notin S$ such that 
$K_l$ is contained in the group of elements congruent to the identity modulo $l$ 
with respect to some faithful representation $G \subset \GL_n$.
In particular, this implies that $K$ is neat.
\end{enumerate}

We let $s = \overline{(h,t)}$ a point of $\Sh_K(G,X)$ defined over a field 
$E$. 
We let $M$ be the Mumford-Tate group of $h$. 
In what follows, for ease of notation, we will write $M$ for $M_{\QQ_S}$
or $M(\QQ_S)$ when it is clear from the contex what is meant.

As we have seen earlier in this section, there exists a continuous ``$S$-adic monodromy representation''
(we have if necessary replaced $E$ with $EE_M$): 
$$
{\rho_{h,S}} \colon \Gal(\ol E / E) \lto M(\AAA_f) \cap K_S =M(\QQ_S)\cap K_S
$$
which has the  property that for any $t \in G_S$ and $\sigma \in \Gal(\ol E / E)$,
\begin{equation}\label{Galois carac}
\sigma\left(\ol{(h,tg)}\right) = \ol{(h, \rho(\sigma) \cdot tg)}.
\end{equation}

We let $U_S \subset M(\QQ_S)$ be the image of $\rho_{(h,S)}$
and $H_S$ the algebraic monodromy group i.e. the Zariski closure 
of $U_S$.
It is immediate that properties \ref{ProprietesGaloisiennes} are invariant by replacing $K$ by an open subgroup
and $E$ by a finite extension (or equivalently $U_S$ by an open subgroup).
After replacing $E$ by a finite extension, we assume that $U_S$ (and hence $H_S$)
are connected.

\subsubsection{The algebraicity property.}

In this section we consider the $S$-algebraicity property. All results of this section
are under the assumption that $E$ is of finite type over $\QQ$.

\subsubsection*{The centre of the $S$-adic monodromy.}

The first result is that when $E$ is of finite type, the $S$-Mumford-Tate property
(and hence the $S$-algebraicity) hold for abelian $S$-adic representations.

\begin{prop}\label{Prop E type fini abelien}
Let~$M^{der}$ be the derived group of~$M$ and~$M^{ab}=M/M^{der}$ its maximal abelian quotient.
We consider the quotient map
\[
\pi:M(\QQ_S)\to M^{ab}(\QQ_S).
\]
\begin{enumerate}
\item \label{Prop E 1} If~$s$ is of $S$-Tate type, then the centre~$Z(U_S)$ of~$U_S$ is contained in the centre of~$M(\QQ_S)$.
\item \label{Prop E 2} If~$s$ is of $S$-Tate type and $S$-simplicity holds, then we have~$\pi(Z(U_S))$ is open in $\pi(U_S)$.
\item \label{Prop E 3} If~$E$ is of finite type over~$\QQ$, then~$\pi(U_S)$ is open in~$M^{ab}(\QQ_S)$.
\end{enumerate}
In particular, when $M$ is abelian (i.e. the point $s$ is special), the $S$-Mumford-Tate
property holds.
\end{prop}
The last remark is a straightforward consequence of the reciprocity law for canonical models of Shimura varieties.
We will rely on this law for the proof of~\eqref{Prop E 3}.
\begin{proof}
Let us prove~\eqref{Prop E 1}. The center of~$U_S$ is the intersection of~$U_S$ with its centraliser~$Z_{G_S}(U_S)$.
By the $S$-Tate property we have~$Z_{G_S}(U_S)=Z_{G_S}(M)$. We also have~$U_S\subseteq M(\QQ_S)$. It follows
\[U_S\subseteq M(\QQ_S)\subseteq M(\QQ_S)\cap Z_{G_S}(M)\]
but the latter is just the centre~$Z(M)$ of~$M$.

We now prove~\eqref{Prop E 2}. 
Let $x \in \pi(U_S)$. Write $x=\pi(y)$ for $y \in U_S$. 
Since semisimplicity holds, there exists an integer $n$ (depending only on $U_S$ but not $y$)
such that 
$$
y^n = z\cdot t
$$
with $z \in  Z_{G_S}(U_S)$ (by ~\eqref{Prop E 1}) and $t \in  U_S^{der} \subset M^{der}(U_S)$.
Thus $x^n = \pi(y^n) = \pi(z) \in \pi(Z(U_S))$. This proves ~\eqref{Prop E 2}.
 
We prove~\eqref{Prop E 3}.
As~$M$ is reductive, the restriction of~$\pi$ to~$Z(M)$ is an isogeny. Since $S$ is a finite set of primes,  it follows that the induced map~$Z(M)(\QQ_S)\to M^{ab}(\QQ_S)$
has finite kernel and an open image.
It will hence suffice to prove that the image of~$U_S\cap Z(M)$ is open in~$M^{ab}(\QQ_S)$. 

In the Shimura variety\footnote{Of dimension~$0$, related to the space of connected components of~$\Sh(G,X)$.}, associated with~$M^{ab}$ every point is a special point.
So is the image of~$s$. The associated Galois representation is
\[
 \pi\circ\rho_{s,S}: \Gal(\overline{E}/E)\to M(\QQ_S)\to M^{ab}(\QQ_S).
\]
By the reciprocity law for the canonical model of the Shimura variety associated with~$M^{ab}$, this representation
factors through 
\[
 \Gal(\overline{\QQ}/E(M,X_M))\to M^{ab}(\QQ_S).
\]
and the image of the latter is open (\cite[composantes connexes]{Deligne}).
As~$E$ is of finite type, the algebraic closure~$F$ of the number field~$E_ME(M,X_M)$ 
in~$E$ is finite. Then the image of~$\Gal(\overline{E}/E)$ in~$\Gal(\overline{\QQ}/E_M)$
is the open subgroup~$\Gal(\overline{\QQ}/F)$. It follows that the image in~$ M^{ab}(\QQ_S)$ of~$\Gal(\overline{E}/E)$
is open  (of index at most~$[F:E_M]$) in that of~$\Gal(\overline{\QQ}/E_M)$. This concludes the proof of the proposition.
\end{proof}

\subsubsection*{$S$-Tate property and $S$-algebraicity.}

We prove here that when $E$ is of finite type over $\QQ$, the $S$-Tate property implies the $S$-algebraicity.
Let~$\Fm$ be the Lie algebra
of~$M_{\QQ_S}$, let~$\Fu$ be the Lie algebra of~$U_S$, a $\QQ_S$-Lie subalgebra 
of~$\Fm$. Let~$H$ be the $\QQ_S$-algebraic envelope of~$U_S$, and~$\Fh$ its Lie algebra.

We refer to \cite{1-Serre} for the following.

\begin{prop} The following are equivalent:
\begin{itemize}
 \item the subgroup~$U_S$ of~$H(\QQ_S)$ is open (for the $S$-adic topology);
 \item the Lie algebra~$\Fu$ of~$U_S$ is an \emph{algebraic} Lie subalgebra of the Lie algebra~$\Fm$ of~$M$, in the sense of Chevalley as in~\cite[II.\S7]{BorelLAG};
 \item the Lie algebras~$\Fu$ of~$U_S$ and~$\Fh$ of~$H$ are the same:~$\Fu=\Fh$.
\end{itemize}
If these properties hold we say that~$U_S$ is an \emph{algebraic Lie subgroup}.
\end{prop}

\begin{prop} Assume that Lie algebras~$\Fu$ is of $S$-Tate type and satisfies~$S$-semisimplicity.
If~$E$ is an extension of finite type of~$\QQ$, then the Lie subgroup~$U_S$ is algebraic in the sense of the previous proposition, that is~$s$ satisfies the~$S$-algebraicity in the sense of definition~\ref{ProprietesGaloisiennes}\,\eqref{Algebraicity}.
\end{prop}

\begin{proof}We first use the $S$-semisimplicity to note that the adjoint action of~$\Fu$ on itself is semisimple,
that is~$\Fu$ is a reductive Lie algebra by the definition used in~\cite[I \S6.4 D\'{e}f.~4]{Lie1}. 
It follows that  we may decompose~$\Fu$ as a direct sum~$\Fu=[\Fu,\Fu]+\Fz$ of its derived Lie algebra~$[\Fu,\Fu]$ and its centre~$\Fz$, by~\cite[Cor. (b) to Prop. 5]{Lie1}. It is enough to show that both~$[\Fu,\Fu]$ and~$\Fz$ are algebraic by~\cite[Cor. 7.7 (1) and (3)]{BorelLAG}. The derived Lie algebra~$[\Fu,\Fu]$ is algebraic by~\cite[Cor. 7.9]{BorelLAG}. 

It remains to prove that~$\Fz$ is an algebraic Lie subalgebra. 
We use the $S$-Tate type property, to note that the centraliser of~$U$ is included
in the  centraliser of~$M$. We infer~$\mathfrak{z}\subseteq\mathfrak{z}_M(\QQ_S)$. As~$E$ is of finite type, 
we may apply Prop.~\ref{Prop E type fini abelien}.
we have seen  that~$U_S$ contain an open subgroup of the centre of~$M(\QQ_S)$. We have
conversely~$\mathfrak{z}\supseteq \mathfrak{z}_M(\QQ_S)$. Finally,~$\mathfrak{z}$ is the algebraic lie subalgebra~$\mathfrak{z}_M$.
\end{proof}
\subsubsection*{N.B.:} This algebraicity statement is similar to Bogomolov's algebraicity result~\cite[Th.~1]{1-Bogo} for abelian varieties. The reduction to the case of abelian Lie algebras is very similar. We rely on $S$-Tate property and the theory of canonical models to treat the abelian case.

For similar algebraicity results see~\cite[133. Th. p.~4]{SerreOEuvres4}, and notably its subsequent Corollary.


\subsubsection{Characterisation of the $S$-Shafarevich property.}
We prove that~$S$-Shafarevich property~\ref{ProprietesGaloisiennes}\,\eqref{S-Sha} is equivalent to the conjonction of the $S$-semi-simplicity property~\ref{ProprietesGaloisiennes}\,\eqref{S-semisimple} and a weakening or the~$S$-Tate property (as defined in~\ref{ProprietesGaloisiennes}\,\eqref{S-Tate}). This is amounts to a group theoretic characterisation of the~$S$-Shafarevich property, which is essential for proving the main theorems of this paper and which, we believe, is also of independent interest.

%



\begin{prop} \label{charSha} Let~$s$ be a point of~$\Sh_K(G,X)$ and~$E$ a field of definition of~$s$ such that the associated $S$-adic monodromy representation~\[\rho:\Gal\left(\ol{E}/E\right)\to M(\AAA_f)\cap G_S\] is defined, and the image~$U_S$ of~$\rho$ is Zariski-connected. 
Let~$Z_{G_S}(M)$ and~$Z_{G_S}(U_S)$ denote the centraliser in~$G_S$ of the Mumford-Tate group~$M$ of~$s$ and of the image~$U_S$ of~$\rho$ respectively.

The point~$s$ is of $S$-Shafarevich type if and only if it satisfies $S$-semisimplicity and, furthermore,~$Z_{G_S}(M) \backslash Z_{G_S}(U_S)$ is compact.

In particular the $S$-Shafarevich property is implied by the conjunction of $S$-semisimplicity and $S$-Tate properties.
\end{prop}
We will split the proof into several steps.
Let us consider a sequence $s_n = \overline{(h,g_n)}$ with $g_n \in G_S$
of points  in the~$S$-Hecke orbit~$\mathcal{H}_S(s)$ of~$s$. 

We start with an easy Lemma about the homogeneous structure of the $S$-Hecke orbit~$\Hcal_S(s)$.
\begin{lem}\label{Lemme Hecke orbite Z}We embed~$Z_G(M)(\QQ)$ into~$G(\AAA_f)$ (at the finite places \emph{only}) and~$G(\QQ)$ into~$G(\AAA)$. Then the application
\begin{multline*}
Z_G(M)(\QQ)\backslash Z_G(M)(\QQ)\cdot G_S\cdot K/K
\\\to \Hcal_S(s)=G(\QQ)\backslash G(\QQ)\cdot (\{h\}\times G_S)\cdot K/K,
\end{multline*}
which maps the double class of~$g$ to~$\ol{(h,g)}$, is a bijection.
\end{lem}
\begin{proof} The surjectivity is immediate, by the very definition of~$\Hcal_S(s)$. We prove the injectivity. We start with the identity~$\ol{(h,g)}=\ol{(h,g')}$. Equivalently there is~$q$ in~$G(\QQ)$ and~$k$ in~$K$ such that
\[
q\cdot(h,g)\cdot k=(h,g').
\]
From~$qh=h$, we infer that~$q$ is in the stabiliser of~$h$. As~$q$ is in~$G(\QQ)$, it belongs to the biggest~$\QQ$-subgroup in the stabiliser of~$h$, which is~$Z_G(M)(\QQ)$. We conclude by observing
\[
q\cdot g\cdot k= g'.\qedhere
\]
\end{proof}

In the following Lemma, we show how the $S$-simplicity hypothesis can be used to work out a group theoretic condition on $g_n$ for~$s_n$ to be defined over $E$.
\begin{lem}\label{stab}
Assume that the $S$-simplicity holds.

There exists a compact subset $C \subset G_S$ such that if a point~$\ol{(h,g)}$ of~$\Hcal_S(s)$ is defined over~$E$, then
$$
g \in Z_{G_S}(U_S)\cdot C.
$$
\end{lem}
\begin{proof} The point~$\ol{(h,g)}$ is defined over~$E$ if and only if, for every element~$u$ of~$U_S$, we have
\[
\overline{(h, u g)} = \overline{(h, g)}.
\]
This means that there exists a $q$ in $G(\QQ)$ and $k$ in $K$, depending on~$u$, such that
\begin{eqnarray*}
q h = h,&\text{ i.e. }&q \in Z_G(M)(\QQ),\\
\text{ and }q u g = g k,&\text{ i.e. }&qu = g k g^{-1}.
\end{eqnarray*}
Thus~$qu$ belongs to the group~$gKg^{-1}$. Any power~$(qu)^n$ belongs to the same group.
As~$q$ centralises~$U_S$, we have~$(qu)^n=u^nq^n$. It follows
\[
q^n=u^{-n}\cdot gk^ng^{-1}.
\]
In particular all powers of~$q$ belongs to the compact set~$U_S\cdot gKg^{-1}$. They also
belong to the discrete set~$Z_G(M)(\QQ)$ in~$G(\AAA_f)$. It must be that~$q$ is a torsion element of~$Z_G(M)(\QQ)$.

Actually all torsion elements of~$Z_G(M)(\QQ)$ satisfy~$q^N=1$ for a uniform order~$N>0$: we may embed~$Z_G(M)$ in a linear group~$GL(D)$ and apply Lemma~\ref{Lemme phi}. We hence have
\begin{equation}\label{eq transporteur}
u^N=q^Nu^N=gk^Ng^{-1},\text{ whence }g^{-1}u^Ng\in K.
\end{equation}
Let~$V=\{u^N|u\in U\}$. This is a neighbourhood of the neutral element in~$U$, as the~$N$-th power map has a non-zero differential at the origin: it is the multiplication by~$N$ map on the Lie algebra. It follows that~$V$ is Zariski dense in~$U$ (recall that~$U$ is Zariski connected).

We deduce from~\eqref{eq transporteur} above that~$g$ belongs to the transporteur, for the conjugation right-action, of~$V$ to~$K$
\[
T=\{t\in G_S|t^{-1}Vt\subseteq K\}.
\]

The Zariski closed subgroup generated by~$V$ is the same as the one generated by~$U$, and is a reductive group by hypothesis. This is the essential hypothesis we need to invoke~\cite[Lemma D.2]{1-RU},  according to which there exists a compact subset $C$ of $G_S$ such that
\[
g\in T \subseteq Z_{G_S}(U_S)\cdot C. 
\]
We are done, but from the fact that actually \emph{loc. cit.} works only for one ultrametric place at a time. But arguing with the projections~$G_p$, $U_p$, $V_p$, $K_p$ and~$T_p$ of~$G_S$, $U_S$, $V$,~$K$ and~$T$ in~$G_p$ at a some place~$p$ in~$S$, we can prove as above that
\[
T_p\subseteq \{t\in G_p|t^{-1}V_p t\subseteq K_p\}\subseteq Z_{G_p}(U_p)\cdot C_p
\] 
for a compact subset~$C_p$ of~$G_p$, and conclude, with~$C$ the product compact set
\[
T\subseteq \prod_{p\in S} T_p\subseteq\prod_{p\in S} Z_{G_p}(U_p)\cdot C_p=Z_{G_S}(U_S)\cdot \prod_{p\in S} C_p=Z_{G_S}(U_S)\cdot C.\qedhere
\]
\end{proof}
The following standard fact was used in the preceding proof.
\begin{lem}\label{Lemme phi} 
Consider a general linear group~$GL(D,\QQ)$ over the field~$\QQ$ of rational numbers. Then there is an integer~$N(D)$ such that for any~$g$ in~$GL(D,\QQ)$ of finite order its power~$g^{N(D)}$ is the neutral element.
\end{lem}
\begin{proof}Let~$g$ be a torsion element. Every complex eigenvalue of~$g$ is some root of unity~$\zeta$. Let~$d$ be the order of~$\zeta$. As the cyclotomic polynomials are irreducible over~$\QQ$, it must be that~$g$ has at least~$\phi(d)$ eigenvalues, the algebraic conjugates of~$\zeta$. We hence have~$\phi(d)\leq D$.

It is known that~$\phi(n)$ diverges to infinity as~$n$ diverges to infinity (one has~$n^{1-\varepsilon}=o(\phi(n))$ for instance).
There is a largest integer~$N$ such that~$\phi(N)\leq D$.

We have necessarily~$d\leq N$. It follows~$d|N!$, and hence~$\zeta^{N!}=1$. The only eigenvalue of~$g^{N!}$ is~$1$.
The power~$g^{N!}$ is unipotent. It is also of finite order, and we must have that~$g^{N!}$ is the neutral element. TWe can take~$N(D)=N!$.
\end{proof}
We now prove one implication in the second part of Proposition~\ref{charSha}.
\begin{lem}
Assume now that $S$-simplicity holds and that 
$Z_{G_S}(M)$ is cocompact in~$Z_{G_S}(U_S)$.
Then~$s$ is of $S$-Shafarevich type.
\end{lem}
\begin{proof}
Let $\overline{(h,g_n)}$ be a sequence of points, in the $S$-Hecke orbit~$\Hcal_S(s)$, defined over a finite extension~$F$ 
of $E$. Our aim is to show this sequence can take at most finitely many distinct values.
After replacing $U_S$ by an open subgroup of finite index, we may assume that~$F=E$, which translates into the property
$$
\forall u\in U_S,~\overline{(h, u\cdot g_n)} = \overline{(h,g_n)}
$$
for all $n$.

Let us write
\begin{subequations}\label{Abbr Z}
\begin{align}
Z&=Z_G(M),&
Z(\QQ_S)&=Z_{G_S}(M),
\end{align}
\begin{equation}
Z(\QQ)=Z_G(M)(\QQ)\subseteq Z(\AAA_f)=Z_G(M)(\AAA_f)
\end{equation}
\end{subequations}

Lemma \ref{stab} shows that elements $g_n$ are contained in 
$T=Z_{G_S}(U_S)\cdot C$ for some compact subset~$C$ of 
$G_S$.
By hypothesis,~$Z_{G_S}(M)\backslash Z_{G_S}(U_S)$ is compact.
By the adelic version of Godements's compactness criterion 
 it is also true that $Z(\QQ) \backslash Z(\AAA_f)$ is compact as well, as~$Z$ is~$\RR$-anisotropic up to the centre of~$G$.

We remark that~$Z(\AAA_f)$ normalises~$Z_{G_S}(U_S)$. As the latter is a place by place product it can be checked place by place: at places in~$S$ the projection of~$Z(\AAA_f)$ is contained in~$Z_{G_S}(M)$ which is itself contained in~$Z_{G_S}(U_S)$; at other places~$Z_{G_S}(U_S)$ has only trivial factors. We hence have a homomorphism
\begin{equation}\label{conversion S vs adelic}
Z_{G_S}(M)\backslash Z_{G_S}(U_S)\longrightarrow{}
Z(\AAA_f)\backslash Z_{G_S}(U_S)\cdot Z(\AAA_f),
\end{equation}
which is surjective, with compact source, hence has compact image.

Incorporating with the compactness of~$Z(\QQ)\backslash Z(\AAA_f)$ we infer the compactness of
\[
Z(\QQ)\backslash Z_{G_S}(U_S)\cdot Z(\AAA_f)
\]
and follows the compactness of the subset
\[
Z(\QQ)\backslash  Z(\AAA_f)\cdot Z_{G_S}(U_S)\cdot C= Z(\QQ)\backslash  Z(\AAA_f)\cdot T
\]
of~$Z(\QQ)\backslash G(\AAA_f)$ (we note that this quotient is separated, as~$Z(\QQ)$ is discrete in~$G(\AAA_f)$).

As~$K$ is open, we deduce that 
\[
 Z(\QQ)\backslash  Z(\AAA_f)\cdot T\cdot K/K
\]
is finite. To sum it up we have
\[
Z(\QQ)\cdot g_n \cdot K\in Z(\QQ)\backslash  Z(\AAA_f)\cdot T\cdot K/K\subseteq  Z(\QQ)\backslash Z(\QQ) T\cdot K/K.
\]
Now, the double coset on the right characterises~$\ol{(h,g_n)}$, by Lemma~\ref{Lemme Hecke orbite Z}.
Finally~$\ol{(h,g_n)}$ can take at most~$\#Z(\QQ)\backslash Z(\QQ) T\cdot K/K$ distinct values.\qedhere
%
%

%
%

\end{proof}

To prove the other inclusion, let us first prove that 
the $S$-Shafarevich property implies $S$-semisimplicity.
This is done in the following lemma.

\begin{lem}
The $S$-Shafarevich property implies the $S$-semisimplicity property, namely that the group $H_S$ is reductive.
\end{lem}
\begin{proof} We can argue place by place: indeed the place by place Shafarevich hypothesis is weaker than~$S$-Shafarevich property, and the reductivity of~$H_S$ can be checked place by place. We may assume for simplicity that~$S$ consists of only one finite place.

We will prove the contrapositive statement, namely that for non reductive~$H_S$ the $S$-Shafarevich property cannot hold.
Let us denote the unipotent radical of~$H$ by~$N_H$. That~$H_S$ is reductive means that~$N_H$ is trivial. We assume it is not the case.

We apply \cite[Proposition 3.1]{BorelTits} to the unipotent subgroup~$N_H \subset G_S$.
There exists a parabolic subgroup $P$ of $G_S$ such: that $N_H$ is
contained in the unipotent radical~$N_P)$ of $P$ and the normaliser of
$N_H$ in $G_S$ is contained in $P$. In particular, $H_S$ is contained in $P$.

By \cite[Prop 8.4.5]{1-Springer} there exists a cocharacter 
$$
y : {\bf G}_m \lto G_S
$$
of $G_S$ over $\QQ_S$  such that
$$
P = \{ g \in G_S : Ad_{y(t)}(g) \text{ converges as } t \lto \infty \} 
$$
and (cf. \cite[Th. 13.4.2(i)]{1-Springer}, \cite[\S2.2 Def. 2.3/Prop. 2.5]{GIT})
$$
N_P = \{ g \in G_S : Ad_{y(t)}(g) \text{ converges to~$e$ as } t \lto \infty \}.
$$
Moreover, the centraliser of~$y$ in~$P$ is a Levi factor~$L$ of~$P$. It follows that for all~$p=\lambda\cdot n$ in~$P$,
with~$l$ in~$L$ and~$n$ in~$\Radu(P)$, the limit~$\lim_{t\to\infty} Ad_{y(t)}(p)$ is the factor~$l$.

We will contradict the $S$-Shafarevich property by showing that the family of the~$\ol{(h,y(t))}$, as~$t$ diverges to infinity,
\begin{itemize}
\item describes infinitely many points in the $S$-Hecke orbit of~$s$,
\item and that these points are all defined on a common finite extension of~$E$.
\end{itemize}

We address the first statement.

By the non triviality of~$N_H$, we may pick an element~$u$ in~$N_H$ distinct from~$e$. It follows that 
the conjugacy class~$C(u)$ of~$u$ in~$G_S$ does not contain~$e$.
As~$Z_G(M)$ centralises~$H_S$ and its subgroup~$N_H$, the orbit map at~$u$ for the 
conjugation action factors through~$G_S/Z_{G_S}(M)$. We have a map
\[
c:G_S/Z_{G_S}(M)\xrightarrow{gZ_{G_S}(M)\mapsto Ad_g(u)} C(u).
\]

We can deduce by contradiction that~$y(t)$ is not bounded modulo~$Z_{G_S}(M)$ as~$t$ diverges to~$\infty$.
Assume not. Then~$y(t)Z_{G_S}(M)$ would have some accumulation point~$gZ_{G_S}(M)$. Hence~$c(y(t)Z_{G_S}(M))$ would have the accumulation point~$c(gZ_{G_S})$, which belongs to~$C(u)$ and hence is distinct from~$e$. This contradicts
the fact that~$c(y(t)Z_{G_S}(M))$ converges to~$e$.

The Hecke orbit of~$s$ can be identified with~$Z_{G_S}(M)(\QQ)\backslash G_S /K_S$ through the quotient of the map~$g\mapsto \ol{(s,g)}$ on~$G_S$. Let us claim that~$Z_{G_S}(M)(\QQ)y(t)K_S$ describes infinitely many cosets. It is sufficient that~$Z_{G_S}(M)y(t)K_S$ does so. If not, then~$Z_{G_S}(M)y(t)$ would be contained in finitely many right~$K_S$ orbits, that is in a bounded set of $Z_{G_S}(M)$-cosets, which cannot be, as we already proved. This proves the statement.

We address the second statement, investigating the field of definition of these $S$-Hecke conjugates of~$s$. We use that the extension of definition of a point~$\ol{(s,g)}$ is associated with the finite quotient~$gU_Sg^{-1} K_S/K_S$ of~$\Gal(\ol{E}/E)$.

As~$U_S$ is topologically of finite type, it will be sufficient to show that~$\# y(t)U_Sy(t)^{-1} K_S/K_S$ is bounded as~$t$ diverges to~$\infty$. It will even be sufficient that~$y(t)U_Sy(t)^{-1}$ remains in a bounded subset~$C$ of~$G_S$, as then we have the bound~$\# y(t)U_Sy(t)^{-1} K_S/K_S\leq \# CK_S/K_S$. 

On~$P$ the family of functions~$Ad_{y(t)}$ converges simply to the projection onto the Levi factor~$L$ of~$P$.
Let us prove the claim that this convergence is uniform on compacts subsets of~$P$, including~$U_S$. As~$Ad_{y(t)}$
acts on factor of the Levi decomposition~$P=L M$ of~$P$ separately, it will suffice to argue for~$L$ and~$\Radu(P)$
separately. 
\begin{itemize}
\item It is immediate for~$L$, on which~$Ad_{y(t)}$ is the identity, independently from~$t$.
\item  We turn to the unipotent group~$N_P$. We may argue at the level of Lie algebras, on which the corresponding action~$ad_{y(t)}$ is linear. 
Yet again we may argue separately, this time with respect to the decomposition into eigenspaces. On a given eigenspace,~$ad_{y(t)}$ acts by a negative power of~$t$, which converges uniformly to~$0$ as~$t$ diverges to~$\infty$ on any bounded subset.
\end{itemize} 
This proves the claim. 

We conclude that~$Ad_{y(t)}$ is a uniformly bounded family on~$U_S$ as~$t$ diverges to~$\infty$: there is a bounded set~$C$ that contains~$Ad_{y(t)}(U_S)$ for big enough~$t$. We obtained bounded subset that we sought. This proves the second statement. 

We have proven that the~$S$-Shafarevich property cannot hold.
%
%
%
%
\end{proof}

We can now conclude the proof of the implication.

\begin{lem}
The $S$-Shafarevich property implies that~$Z_{G_S}(M)$ is cocompact in~$Z_{G_S}(H_S)$.
\end{lem}
\begin{proof}
Assume  that $Z_G(M)\backslash Z_{G_S}(U_S)$ is not compact, and let us disprove the $S$-Shafarevich property.
Possibly substituting~$E$ with a finite extension thereof, we may assume~$U_S\subseteq K$. 
We will prove that
\[\Hcal=\{\ol{(h,z)}|z\in Z_{G_S}(U_S)\}\]
is an infinite set of points defined over~$E$.

Let~$\ol{(h,z)}$ be such a point. For~$\sigma$ in~$\Gal(\ol{E}/E)$ with image~$u$ in~$U_S$ we have
\[
\sigma\left(\ol{(h,z)}\right)=\ol{(h,uz)}=\ol{(h,zu)}=\ol{(h,z)},
\]
by definition of~$U_S$, by the fact that~$z$ commutes with~$U_S$, and that~$U_S$ is contained in~$K$ respectively.
It follows that these points are defined over~$E$.

By Lemma~\ref{Lemme Hecke orbite Z} we have a bijection
\[
\Hcal\simeq Z_G(M)(\QQ)\backslash Z_G(M)(\QQ)\cdot Z_{G_S}(U_S)\cdot K/K.
\]
We need to to prove this is an infinite set. 
As the following map of double quotients
\[
Z(\QQ)\backslash Z(\QQ)\cdot Z_{G_S}(U_S)\cdot K/K
\longrightarrow
Z(\AAA_f)\backslash Z(\AAA_f)\cdot Z_{G_S}(U_S)\cdot K/K
\]
is a surjection, it is sufficient to prove its image is infinite. As~$K$ is compact,
it is enough that
\begin{equation}\label{quotient Z A}
Z(\AAA_f)\backslash Z(\AAA_f)\cdot Z_{G_S}(U_S)
\end{equation}
be unbounded in~$Z(\AAA_f)\backslash G(\AAA_f)$. 

Let us accept for now  that the the map
\begin{equation}\label{closed immersion}
Z(\QQ_S)\backslash G(\QQ_S)\longrightarrow Z(\AAA_f)\backslash G(\AAA_f),
\end{equation}
induced by the closed immersion~$\QQ_S\to\AAA_f$, is itself a closed immersion. Moreover~$Z_{G_S}(U_S)$ is closed in~$G(\QQ_S)$ and contains~$Z(\QQ_S)$. Hence the group~$Z(\QQ_S)\backslash Z_{G_S}(U_S)$ embeds as a closed subset of~$Z(\QQ_S)\backslash G(\QQ_S)$. It is also non compact by hypothesis. Its image in~$Z(\AAA_f)\backslash G(\AAA_f)$ is closed and non compact. It is hence unbounded. But this image is~\eqref{quotient Z A}. This concludes

We prove now that~\eqref{closed immersion} is a closed immersion. It is certainly the case of
\[
(Z\backslash G)(\QQ_S)\longrightarrow
(Z\backslash G)(\AAA_f)
\]
as it induced by the closed immersion~$\QQ_S\to\AAA_f$ (we may embed~$Z/G$ as a closed subvariety in an affine space). Moreover the map
\[
Z(\QQ_S)\backslash G(\QQ_S) \longrightarrow (Z\backslash G)(\QQ_S)
\]
is the kernel of the continuous map~$ (Z\backslash G)(\QQ_S)\to H^1(\QQ_S;Z)$, 
hence a closed immersion. The image~$F$ of a closed subset of~$Z(\QQ_S)\backslash G(\QQ_S)$
in~$(Z\backslash G)(\AAA_f)$ is hence closed. The inverse image of~$F$ by the continuous map~$Z(\AAA_f)\backslash G(\AAA_f)\to (Z\backslash G)(\AAA_f)$ is a fortiori closed. The map~\eqref{closed immersion} is then a closed map.
Clearly, it is furthermore injective. It is finally a closed immersion. This concludes.\qedhere

\end{proof}

This finishes the proof of Proposition \ref{charSha}.

\subsection{Shimura varieties of abelian type.}

In this section we show that when $E$ \emph{is of finite type} over $\QQ$
and $(G,X)$ is a Shimura datum of abelian type, all the properties
except possibly $S$-Mumford-Tate hold.

\begin{prop}
Assume that $\Sh_K(G,X)$ is a Shimura variety of abelian type
and that $E$ is a field of finite type.

Then the $S$-semisimplicity, $S$-Tate (and hence $S$-Shafarevich)
and $S$-algebraicity  hold for all points of 
 $\Sh_K(G,X)(E)$.
\end{prop}

\begin{proof}
By definition of a Shimura variety of abelian type, there exists 
a Shimura subdatum $(G',X) \subset (GSp_{2g}, \HH_g)$
with a central isogeny $\theta \colon G' \lto G$.

There exists a compact open subgroup $K' \subset G'(\AAA_f)$ such that
$\Sh_{K'}(G',X)$ is a subvariety of $\cA_{g,3}$ (the fine moduli scheme 
of abelian vareities with level $3$ structure) and there is a finite morphism
$\Sh_{K'}(G',X) \lto \Sh_K(G,X)$.

For the purposes of proving the $S$-Tate property we may assume that
all our Shimura varieties are defined over $E$.

Let $x = \overline{(s,t)} \in \Sh_{K}(G,X)(E)$ and $x'$ a point 
of $\Sh_{K'}(G',X)$ over a finite extension of $E$.
Without loss of generality in the proof of the $S$-Tate property 
we replace $E$ by this finite extension and hence assume that $s'$ is defined over $E$.

The point $s'$ corresponds to an abelian varietey $A$ defined over $E$.
Let $V_S(A) = \prod_{l \in S} V_l(A)$ be the product of $l$-adic Tate modules attached to $A$ for $l \in S$. The module $V_l(A)$ is endowed 
with an action of $\Gal(\ol{E}/E)$ and with a symplectic action of 
$\GSp_{2g}(\QQ_S)$. Following the arguments of Remark 2.8
of \cite{1-UY1}, we see that the action of $\Gal(\ol{E}/E)$ is given by 
the representation 
$$
\rho'_S \colon \Gal(\ol{E}/E) \lto G'_S \subset \GSp_{2g}(\QQ_S).
$$

Note that in \cite{1-UY1} there the authors suppose the field $E$ to be a number field and $S$ to consist of one prime. However, all arguments adapt verbatim in our situation.

We denote by $M'$ the Mumford-Tate of $x'$ and by $H'_S$
the image of $\rho'_S$.

By Falting's theorem (Tate conjecture for abelian varieties),
 the group $H'_S$ is reductive. Note that Falting's theorem holds for abelian
varieties over finitely generated fields over $\QQ$ (see Chapter  VI of \cite{1-FW}).
 
 Again, by the Tate conjecture, we have 
 $$
 Z_{GSp_{2g, S}}(M') = Z_{GSp_{2g, S}}(H'_S)
 $$
 and therefore 
 $$
 Z_{G'_S}(M') = Z_{G'_S}(H'_S).
 $$
 Let $M$ be the Mumford-Tate group of $x$ and let $H_S$ be Zariski closure of the image 
 of $\rho_S$, the representation of $\Gal(\ol{E}/E)$ attached to $x$.
 
 Recall that we have a central isogeny $\theta \colon G' \lto G$.
 We naturally have
 $$
 M = \theta(M'), \quad H_S = \theta(H'_S).
 $$ 
In particular, $H_S$ is semisimple. 
 
 Since $\theta$ is a central isogeny and hence 
 commutes with conjugation, the equality $Z_{G'_S}(M') = Z_{G'_S}(H'_S)$.
\end{proof}

\subsection{Dependence on the field $E$.}

We have seen that for a Shimura variety of abelian type and
when $E$ is of finite type all properties~\ref{ProprietesGaloisiennes}
(except $S$-Mumford-Tate) hold.

In this section we will show that these properties fail even when $(G,X)$ is
of abelian type when the field $E$ is not of finite type.

Let $A$ be an elliptic curve over $\QQ$ without complex multiplication.
By a celebrated theorem of Serre, the image of the adelic representation of $\Gal(\overline{\QQ}/\QQ)$ attached to $A$ is open in $\GL_2(\widehat{\ZZ})$.
For any $S$, all our properties hold for $A$ over $\QQ$ 
(or any finite type extension of $\QQ$).

Choose a prime $l$ such that $\rho_l$ surjects onto $\GL_2(\ZZ_l)$.
Take $S= \{ l \}$.
Let $B$ be the standard Borel in $\GL_2(\ZZ_l)$ (upper triangular matrices)
and let $E$ be an extension corresponding to the subgroup $\rho_l^{-1}(B)$ of $\Gal(\overline{\QQ}/\QQ)$.
The extension $E$ is not of finite type over $\QQ$.

We immediately see that $S$-Mumford-Tate property does not hold
for $A/E$ (the image of Galois is $B$). The $S$-semisimplicity also fails
since the Zariski closure of $B$ is not reductive.
The $S$-Tate property holds however.
Indeed the centraliser of $B$ is the centre of $\GL_2$ which is of course
also the centraliser of $\GL_2$ in itself.

Finally, let us see directly that the $S$-Shafarevich property fails.
Write $\overline{(s, 1)}$ the point of $\Sh_{\GL_n(\widehat{\ZZ})}(\GL_2,\HH^{\pm})$ corresponding to $A$.

Consider elements
$g_n = \left(\begin{smallmatrix}l^{-n}& 0\\0&1\end{smallmatrix}\right)$ for $n > 0$.

Note that
$$
g_n^{-1} B g_n \subset B
$$

and therefore the sequence of points $\overline{(s,g_n)}$ is defined over 
$E$. This sequence of points is obviously infinite - it corresponds to all elliptic curves isogeneous (over $\CC$) to $A$ by a cyclic isogeny of
degree $l^n$.

We now give an example where $S$-algebraicity fails.
In the previous situation, consider the subgroup of $\GL_2(\ZZ_l)$
defined by 
$\left(\begin{smallmatrix}e^x& x\\0&e^x\end{smallmatrix}\right)$
where $x \in l^2 \ZZ_l$.

Note that this group is Zariski dense in the group
$\left(\begin{smallmatrix}a& b\\0&a\end{smallmatrix}\right)$  where $a,b \in \ZZ_l$.
However it is not open in this group and therefore $S$-algebraicity fails.

Note that in the previous example semi-simplicity did not hold. However,
we can construct an example where $S$-semi-simplicity does hold, but
the $S$-algebraicity fails.

Let $\alpha$ be the Liouville number, 
$\alpha = \sum_{n=0}^{\infty} l^{n!}$. This is an element of $\ZZ_l$ which is transcendental over $\QQ$. We refer to \cite{1-AM} and references therein for more details.

Consider the subgroup 
$\left(\begin{smallmatrix}e^{x}& 0\\0&e^{\alpha x}\end{smallmatrix}\right)$  where $x \in \ZZ_l$ and as before $E$ the extension corresponding to this subgroup. The Zariski closure is the diagonal torus therefore 
the $S$-semisimplicity holds.
However the $S$-algebraicity fails.
This example is analogous to the one given after Theorem 2.3 in \cite{1-Bost}.



\section{The $S$-Arithmetic lift and Equidistribution}\label{Reduction}

In this section we start working towards proving Theorem~\ref{Theoreme1}.
Our aim here is to translate our problem (i.e. Theorem~\ref{Theoreme1}) about the probabilities~$\mu_n$ on~$\Sh_K(G,X)$ into a problem of equidistribution 
on an~$S$-arithmetic homogeneous space of a semisimple algebraic group over~$\QQ$, of the kind studied~\cite{1-RZ}, in order to apply the main theorem thereof. More precisely, we will use the maps~$\pi_G$ of~\eqref{piG} and~$\pi$ of~\eqref{pider}.

We work with a point~$s = \overline{(h,t)}$ of $\Sh_K(G,X)$ and 
its $S$-Hecke orbit~$\Hcal_S(s)$ as in the statement of Theorem~\ref{Theoreme1}.

We also have a sequence of points~$s_n = \ol{(h,t\cdot g_n)}$ of~$\Hcal_S(s)$
with some $g_n \in G_S$.
The elements~$g_n$ are not uniquely defined by the points $s_n$ and will actually be subject to modification, without changing~$s_n$, in the course of the proof.

For any~$g$ in~$G_S$, we may change the representative~$(h,t)$ of~$s$ into the representative~$(h,t\cdot g)$, provided that we change accordingly each~$g_n$ into~$g^{-1}g_n$. Neither the~$S$-Hecke orbit~
$\Hcal_S(s)$, nor the set of points~$\{s_n\}$ are changed by such substitutions.  The group~$U_S$ is left unchanged too. Hence the $S$-Shafarevich property of~$s$ remains valid under this substitution of~$s$ with~$\overline{(h,tg)}$.

Lastly, since the conclusion of Theorem~\ref{Theoreme1} is up to extracting
 a subsequence, we may, whenever necessary, replace~$(g_n)_{n\geq0}$ by a subsequence.

For technical reasons we will need to assume that $H_S \cap G^{der}$
is $\QQ_S$-Zariski connected. 
It is sufficient that~$U_S\cap G^{\der}(\QQ_S)$ be  contained in~$\left( H_S\cap G^{\der}_{\QQ_S}\right)^0(\QQ_S)$.
This can be achieved by passing to a subgroup of finite index in~$U_S$, that is to a finite extension of~$E$.
We note that this assumption will still be satisfied after passing again to a finite extension of~$E$. We
recall that the statement of Theorem~\ref{Theoreme1} is invariant by passing to finite extensions of~$E$.

\subsection{The reductive $S$-arithmetic lift.}

Since the equidistribution theorems of \cite{1-RZ} apply to $S$-arithmetic homogeneous
spaces, we need to ``lift'' our situation, from the base Shimura variety of level~$K$ to such a space.

In this section we construct a map~$\pi_G$  below (see ~\eqref{piG}), from an~$S$-arithmetic homogeneous space of the reductive group~$G$ to~$\Sh_K(G,X)$. We then introduce probabilities~$\wt{\mu_n}'$ which
are ``lifts'' of the probabilities~${\mu_n}$, that is such that we have the compatibility~\eqref{lem compat} by direct image.

\subsubsection{The~$S$-arithmetic map.}
For convenience we identify the sub\-group~$G_S\subseteq G(\AAA_f)$ with its image~$G(\QQ_S)$. Let us consider the ``orbit maps''
\begin{align*}
 \omega_{tK}&:	G_S\xrightarrow{g\mapsto g\cdot t K} G(\AAA_f)/K&&\text{ at the coset }tK,\\
 \text{ and }~
 \omega_h&:G(\RR) \xrightarrow{g \mapsto g \cdot h} X&&\text{ at the point~$h$.}
\end{align*}
Together these induce the following map
\[
 	\ol{\omega_{(h,tK)}}:
 		G(\RR\times \QQ_S)
 			\xrightarrow{ \omega_{h} \times \omega_{tK}}
		X\times G(\AAA_f)
			\xrightarrow{(h,t)\mapsto \ol{(h,t)}}
		\Sh_K(G,X).
\]
Equivalently, for any element~$(g_\RR,g_S)\in G(\RR)\times G(\QQ_S)\simeq G(\RR\times \QQ_S)$,
\[\ol{\omega_{(h,tK)}}(g_\RR,g_S)=\ol{( g_\RR\cdot h, g_S\cdot t)}.\]
\begin{quote}
\item
\paragraph*{Remark}
The Shimura variety~$\Sh_K(G,X)$ has finitely many geometrically connected components.
Each component is a quotient of a the Hermitian symmetric domain~$X^+$ by an arithmetic subgroup of~$G(\QQ)$. The image of~$\ol{\omega_{(h,tK)}}$ 
consists of the union the components this image intersects.\footnote{This is an instance of a ``$S$-adic packet'' of components extracted out of the ``adelic packet'' of all components. For instance, if the Shimura variety is reduced to a class group (and hence finite), viewed as a Galois group, we obtain cosets of the subgroup generated by the Frobenius  from the places in~$S$.}
\end{quote}

The product~$G_S\cdot tKt^{-1}$  is open in~$G(\AAA_f)$. As~$G_S$ is a normal subgroup in~$G(\AAA_f)$, this product is a subgroup. 
We define then following $S$-arithmetic subgroup of $G(\QQ)$
\footnote{We understand ``$S$-arithmetic'' in the sense that there is a faithful $\QQ$-linear representation~$\rho:G\hookrightarrow \GL(n)$, such that the groups~$\rho(G)(\QQ)\cap \GL\left(n,\ZZ[(1/\ell)_{\ell\in S}]\right)$ and~$\rho(\Gamma)$ are commensurable. For~$S=\emptyset$ we recover the usual notion of arithmetic subgroup.
\newline \indent
We will define~$\Gamma$ in~\eqref{defi Gamma}, as an~$S$-arithmetic subgroup of~$G^\der$.}:
\begin{equation}\label{defi Gamma G}
\Gamma_G := G(\QQ)\cap (G_S\cdot tKt^{-1})
\end{equation}
Equivalently we may define, as the intersection, inside of~$G(\AAA)$,
\[
\Gamma_G=\left( G(\RR)\cdot G_S \cdot tKt^{-1} \right)\cap G(\QQ).
\]
This~$\Gamma_G$ depends on~$tK$ and~$S$ though we don't specify it to simplify notation.

The map~$\ol{\omega_{(h,tK)}}$ is left invariant under the action of~$\Gamma_G$: this map factors through a map
\begin{equation}\label{piG}
\pi_G=\pi_{(h,tK)} : \Gamma_G\backslash G(\RR \times \QQ_S) \lto \Sh_K(G,X).
\end{equation}

Let~$K_h$ be the stabiliser of~$h$ in~$G(\RR)$ and~$K_S=K\cap G_S$, so that~$tK_St^{-1}$ is the stabiliser of the coset~$tK$ in~$G_S$.
We may further factor~$\pi_{(h,tK)}$ as the composition
\begin{itemize}
\item of the quotient map
\[
\Gamma_{G} \backslash G(\RR \times \QQ_S)
\to
\Gamma_{G} \backslash \left( G(\RR)\times G(\QQ_S)\right)/\left( K_h\times tK_St^{-1}\right)
\] by the right action of the compact group $K_h \times t K_S t^{-1}$;
\item followed by a closed open immersion into~$\Sh_K(G,X)$, that is  the inclusion of a union of components.
\end{itemize}

\subsubsection{Lift of probabilities}
Recall from~\eqref{defimonogroup} of~Definition~\ref{defiintro} that we denote~$U_S$ denote the~$S$-adic monodromy group associated with~$\overline{(h,t)}$. 
Since $U_S$ is a compact group, it supports a Haar probability.
Let~$\mu_{U_S}$ be \emph{the direct image} of this Haar probability 
in the $S$-arithmetic quotient
\[
	U_S\hookrightarrow G(\RR\times \QQ_S)\to\Gamma_G\backslash G(\RR\times \QQ_S).
\]

The right translate of~$\mu_{U_S}$ by~$g_n$ is denoted by
\begin{equation}\label{lifted1}
	\wt{\mu_n}^\prime=\mu_{U_S}\cdot g_n.
\end{equation}
We now show that the~$\wt{\mu_n}'$ are lifts of the~$\mu_n$.

\begin{lem} Pushing forward the measures~$\wt{\mu_n}'$ from~\eqref{lifted1} along the map~$\pi_G$ from~\eqref{piG} we get
\begin{equation}\label{lem compat}
{\pi_G}_\star(\wt{\mu_n}')=\mu_n.
\end{equation}
\end{lem}
\begin{proof}
The map~$\alpha$ defined by the commutativity of the diagram

 \begin{equation}\label{composee}
 \xymatrix{
	\Gal\left(\ol{E}/E\right)
		\ar[r]^(.65){\rho}
		\ar[drrr]_\alpha
	&
	U_S\ar[drr]^\beta
		\ar[r]^(.30){\sigma\mapsto \Gamma_G\cdot\sigma}
	&
	\Gamma_G\backslash G(\RR\times \QQ_S)
		\ar[r]^{x\mapsto x \cdot g_n}\ar[dr]^{\gamma}
	&
	\Gamma_G\backslash G(\RR\times \QQ_S)
		\ar[d]^{\pi_G}
	\\&&&
	\Sh_K(G,X)
}
\end{equation}
sends~$\sigma$ to~$\ol{(h,\rho(\sigma) t g_n)}$. By the defining properties~\eqref{Galois carac} of~$\rho$ and of~$g_n$, we may rewrite
\(
\alpha(\sigma)=\sigma\cdot \ol{(h,t g_n)}=\sigma\cdot s_n.
\) In other words~$\alpha$ is the orbit map at~$s_n$ for the action of~$\Gal\left(\ol{E}/E\right)$ on~$\Sh_K(G,X)$. In particular its image is the Galois orbit of~$s_n$, which is~${\rm Supp}(\mu_{s_n})$ in the notation of~\eqref{defi mesure}. The map~$\alpha$ is clearly left~$\Gal\left(\ol{E}/E\right)$-equivariant.
Consequently the direct image of the Haar probability, say~$\mu_E$, on~$\Gal\left(\ol{E}/E\right)$ is an invariant probability on~${\rm Supp}(\mu_{s_n})$, necessarily the Haar probability of the transitive~$\Gal\left(\ol{E}/E\right)$-space~${\rm Supp}(\mu_{s_n})$.
The counting probability~$\mu_n=\mu_{s_n}$ is invariant by permutation, hence is~$\Gal\left(\ol{E}/E\right)$-invariant.
It must equal~$\mu_n$ by the uniqueness of the Haar probability:
\[
	\alpha_\star\left(\mu_E\right)=\mu_n.
\]

As~\eqref{composee} is commutative and pushforwards are functorial, we may factor
\[
	\beta_\star\circ \rho_\star=(\beta\circ \rho)_\star=\alpha_\star.
\]
The representation~$\rho$ is a continuous map of compact groups, 
so the direct image~$\rho_\star(\mu_E)$ is the Haar measure on
the image~$\rho(\Gal\left(\ol{E}/E\right))$. This image is~$U_S$
by definition.
It follows that~$\mu_n$ is the direct image of the Haar probability measure~$\rho_\star(\mu_E)$ on~$U_S$ through this~$\beta$. 

We defined~$\mu_{U_S}$ as
the image of the Haar measure of~$U_S$ in the first occurrence in~\eqref{composee} of~$\Gamma_G\backslash G(\RR\times \QQ_S)$, and~$\wt{\mu_n}'$ as the direct image in the second occurrence. By the same functoriality argument as above, the compatibilities
\[
\mu_n=\gamma_\star(\mu_{U_S})={\pi_G}_\star(\wt{\mu_n}')=\beta_\star(\rho_\star(\mu_E))=\alpha_\star(\mu_E),
\]
including the identity~\eqref{lem compat}, follow.
\end{proof}


\subsection{Passing to the derived subgroup}\label{derived}

Theorems from ~\cite{1-RZ} require that the group $G$ be semisimple. 
We will reduce to this case by passing from~$G$ to~$G^{der}$.
In this section,
 we modify our lifting probabilities~$\wt{\mu_n}'=\mu_{U_S}\cdot g_n$ in two ways
 in order to be able to make this assumption and thus recover the setting of ~\cite{1-RZ}.
 
\begin{itemize}
\item Firstly we substitute the translating element~$g_n$ in~\eqref{lifted1} with another one which comes from the derived group~$G^\der(\QQ_S)$, thus constructing 
a probability~$\wt{\mu_n}''$. This first step might require passing to a subsequence and altering~$t$.

\item Secondly, we replace the compact subgroup~$U_S$ of~$G(\QQ_S)$ by a compact subgroup~$\Omega$ of~$G^\der(\QQ_S)$, thus producing~$\wt{\mu_n}'''$. This step may require passing to a finite extension of~$E$.
\end{itemize}


We recall that we beforehand ensured that~$U_S\cap G^{\der}(\QQ_S)$ is $\QQ_S$-Zariski connected.

As, for simplicity reasons, the reference~\cite{1-RZ} deals semisimple groups~$G$, instead of general reductive, we need to carry out the reduction to the semisimple case. A reader uninterested in subtle technical details may skip directly to the next section~\ref{sous section stabilite}.

\subsubsection{Some finite index open subgroups.}
We first note that
\begin{equation}\label{indice fini}
G^{\der}(\QQ_S)\cdot Z(\QQ_S)\text{ is a finite index open subgroup of~$G_S$.}
\end{equation}
\begin{proof}
 As~$S$ is finite, we may work place by place, in which case it is sufficient to refer to~\cite[\S3.2, Cor.~3 p.~122, \S6.4 Cor.~3 p.~320]{PR}.
\end{proof}

In this section we will  use the notation:
\begin{equation}\label{def Gamma Z}
\Gamma_Z=\left(Z(\RR)\cdot \Gamma_G\right)\cap Z(\QQ_S).
\end{equation}
The finiteness of the class group of the torus~$Z$ tells us  that
\[
	(Z(\QQ_S)\cap K) \cdot \Gamma_Z
\]
is an open subgroup of finite index in~$Z(\QQ_S)$. We deduce, combining with~\eqref{indice fini}, that
\begin{equation}\label{indice fini 2}
   G^{\der}(\QQ_S)\cdot	(Z(\QQ_S)\cap K) \cdot \Gamma_Z
\end{equation}
is a finite index open subgroup of~$G_S$.
\subsubsection{Reducing to~$g_n\in G^{\text der.}(\QQ_S)$}
The sequence  of right cosets relative to~\eqref{indice fini 2}, induced by~$(g_n)_{n\geq0}$, that is
\begin{equation}\label{coset gamma Z}
X	\left( g_n\cdot\left(G^{\der}(\QQ_S)\cdot	(Z(\QQ_S)\cap K) \cdot \Gamma_Z\right)\right)_{n\geq 0},
\end{equation}
can, by finiteness of the index below, be decomposed into at most
\begin{equation}
N(G,\Gamma_G)=\left[G^{\text der.}(\QQ_S)\cdot (Z(\QQ_S)\cap K) \cdot \Gamma_Z~:~G_S\right]
\end{equation}
 constant subsequences (this index only depends on~$G$ and~$\Gamma_G$).

After possibly passing to a subsequence, we may assume the sequence~\eqref{coset gamma Z} has a constant value.
This value is the right coset of~$g_0$. Replacing~$t$ by~$tg_0$, we may assume that~$g_0=1$, and that the right coset of~$g_n$
is the neutral coset. Equivalently,  for every~$n\geq0$,
\begin{equation}\label{decompo dzg}
	\exists (d_n,z_n,\gamma_n)\in G^{\text der.}(\QQ_S)\times	(Z(\QQ_S)\cap K) \times \Gamma_Z,~g_n=d_n\cdot z_n\cdot\gamma_n.
\end{equation}
We will see that we may assume~$z_n=\gamma_n=1$, and replace~$g_n$ by~$d_n$,
without interfering with the compatibility~\eqref{lem compat}. 
To achieve this, similarly to~\eqref{lifted1}, we set:
\begin{equation}\label{lifted2}
	\wt{\mu_n}''=\mu_{U_S}\cdot d_n.
\end{equation}
\begin{lem}\label{lemme compat2} The direct image along~$\pi_G$ of the above~$\wt{\mu_n}''$ is given by
\begin{equation}\label{lem compat2}
{\pi_G}_\star(\wt{\mu_n}'')=\mu_n.
\end{equation}
\end{lem}
\noindent Lemma~\ref{lemme compat2} is deduced from~\eqref{lem compat} using following invariance properties.

\begin{lem}\label{LemmeInvariances}
Let~$\mu$ be any probability measure on~$\Gamma_G\backslash G(\RR)\times G_S$.
\begin{subequations}
\begin{enumerate}
\item \label{item1} For any~$\gamma$ in $\Gamma_G\cap Z(\RR\times\QQ_S)$ we have
\begin{equation}\label{invariance 1}
\mu\cdot\gamma=\mu.
\end{equation}
\item \label{item2} For any~$k\in K_h\times K$, we have
\begin{equation}\label{inva 2}
	{\pi_G}_\star(\mu\cdot k)={\pi_G}_\star(\mu).
\end{equation}
\item \label{item3} For every~$g\in G(\QQ_S)$, and every~$ z\in K\cap Z(\QQ_S)$ or~$z\in Z(\RR)$,
\begin{equation}\label{seconde invariance}
{\pi_G}_\star(\mu\cdot z\cdot g)={\pi_G}_\star(\mu\cdot g).
\end{equation}
\end{enumerate}
\end{subequations}
\end{lem}
\begin{proof}[Proof of \eqref{item1}] Note that~$\Gamma_G\cap Z(\RR\times\QQ_S)$ acts trivially on~$\Gamma_G\backslash G(\RR)\times G_S$,
as can be checked pointwise, for some~$\gamma\in Z(\RR\times\QQ_S)$ and~$\gamma\in\Gamma_S$, with
\[
\Gamma_G \cdot(g_\RR,g_S)\cdot \gamma
=
\Gamma_G \cdot \gamma\cdot (g_\RR,g_S)
=
\Gamma_G \cdot(g_\RR,g_S).
\]
By ``transport of structure'' it acts trivially its measure space.
\end{proof}
\begin{proof}[Proof of \eqref{item2}] 
As~$\pi_G$ factors through the right action of~$K_h\times K$, we have
\[
\forall k\in K\times K_h,~
\pi_G(x\cdot k)=\pi_G(x).
\]
Equivalently~${\pi_G}_\star(\delta_x\cdot k)={\pi_G}_\star(\delta_x)$. Concerning~$\mu$ such as in the statement, we may
compute
\begin{align*}
{\pi_G}_\star(\mu\cdot k)&={\pi_G}_\star\left(\int\delta_x\cdot k~\mu(x)\right)=\int{\pi_G}_\star(\delta_x\cdot k )~\mu(x)\\&
=\int{\pi_G}_\star(\delta_x )~\mu(x)={\pi_G}_\star\left(\int\delta_x~\mu(x)\right)={\pi_G}_\star(\mu).
\end{align*}
We used linearity and continuity of~${\pi_G}_\star$ on bounded measures, seen for instance as continuous linear form on continuous functions.\footnote{We recall that our spaces are `polish' (separable and metrizable) and hence they are Radon spaces: Borel probability measures are inner regular (cf \cite[INT Ch. IX, \S3 \no3 Prop.\,3]{BBKINT}).}
\end{proof}
\begin{proof}[Proof of \eqref{item3}]  As~$z\in Z(\RR\times\QQ_S)$,we may substitute~$z\cdot g=g\cdot z$. Replacing~$\mu$ by~$\mu\cdot g$ we may omit omit~$g$.
We have~$z\in Z(\RR)\leq K_h$ or~$z\in Z(\QQ_S)\cap K\leq K$. In either case we may apply~\eqref{inva 2}.
\end{proof}

\begin{proof}[Proof of Lemma~\ref{lemme compat2}] We note from definition~\eqref{def Gamma Z}
that
\[
\Gamma_Z\subseteq Z(\RR)\cdot \left(\Gamma_G\cap Z(\RR\times\QQ_S)\right).
\]
We may decompose accordingly~\[\gamma_n= \zeta_n\cdot c_n\] with~$\zeta_n\in Z(\RR)$ and~$c_n\in\Gamma_G\cap Z(\RR\times\QQ_S)$. We have
\[
\mu_{U_S}\cdot g_n = \mu_{U_S}\cdot d_n\cdot z_n\cdot \zeta_n\cdot c_n=\mu_{U_S}\cdot d_n\cdot z_n\cdot \zeta_n
\]
where we may omit~$c_n$ in the right-hand side thanks to~\eqref{invariance 1}. Applying~\eqref{seconde invariance}
to~$z=\zeta_n$ and to~$z=z_n$, we deduce
\[
{\pi_G}_\star\left(\mu_{U_S}\cdot d_n\cdot z_n\cdot \zeta_n\right)
=
{\pi_G}_\star\left(\mu_{U_S}\cdot d_n\cdot z_n\right)
=
{\pi_G}_\star\left(\mu_{U_S}\cdot d_n\right).
\]
The Lemma then follows from~\eqref{lem compat}.
\end{proof}

We conclude by noting that 
\[
	\ol{(h,tg_n)}=s_n=\ol{(h,td_n)}.
\]
We may follow the same proof as that of Lemma~\ref{lemme compat2}, but with~$\delta_{(h,tg_n)}$ instead of~$\mu_{U_S}$.
We are now reduced to the case~$g_n\in G^{\text der.}(\QQ_S)$. 

\subsubsection{Passing from~$U_S\leq G_S$ to~$\Omega\leq G^{\der}(\QQ_S)$.}
We now turn to the matter of replacing~$U_S$ by a sugbroup~$\Omega$ of~$G^{\der}(\QQ_S)$.

Let 
\begin{equation}\label{U S tilde}
\wt{U_S}:=U_S\cdot (K\cap Z(\QQ_S)).
\end{equation}
Note that this is a compact group.
We define
\begin{equation}\label{defi Omega}
	\Omega= \wt{U_S}\cap G^{\der}(\QQ_S).
\end{equation}


This~$\Omega$ is a compact group, and therefore carries a Haar probability.
Let~$\mu_\Omega$ be the \emph{direct image} of this Haar probability 
in the $S$-arithmetic quotient
\begin{equation}\label{Sar quotient1}
	\Gamma_G\backslash G(\RR\times \QQ_S).
\end{equation}
Let~
\begin{equation}\label{defi Gamma}
\Gamma=\Gamma_G\cap G^{\der}(\RR\times \QQ_S).
\end{equation}
This is an~$S$-arithmetic group in~$G^{\text der.}(\RR\times \QQ_S)$, and a \emph{lattice} by the $S$-arithmetic form of Borel and Harish-Chandra theorem (e.g.~\cite{1-GW} by Go\-de\-ment-Weil,
cf~\cite[\S 5.4]{PR}).
We dropped the indiex~in~$\Gamma$, as this will be the~$S$-arithmetic lattice involved when applying~\cite{1-RZ}, which is denoted by~$\Gamma$ in \emph{loc.\,cit.}

We identify the semisimple arithmetic quotient space
\begin{equation}\label{Sar quotient2}
\Gamma\backslash G^{\text der.}(\RR\times \QQ_S)
\end{equation}
with its image via the natural embedding~$\Gamma\cdot g\mapsto \Gamma_G\cdot g$ into the reductive arithmetic quotient space~\eqref{Sar quotient1}. The support of the probability measure~$\mu_\Omega$ is contained in $\Gamma\backslash G^{\text der.}(\RR\times \QQ_S)$.

We now define
\begin{equation}\label{attempt3}
		\widetilde{\mu_n}'''=\mu_\Omega\cdot g_n,
\end{equation}
as a probability measure on~$\Gamma\backslash G^{\text der.}(\RR\times \QQ_S)$.

Let now
\begin{equation}\label{pider}
	\pi:\Gamma\backslash G^{\text der.}(\RR\times \QQ_S)\to \Sh_K(G,X)
\end{equation}
be the restriction of~$\pi_G$.

We have thus reduced ourselves to the case $G = G^{\der}$.

\subsubsection{Passing to a splitting extension and the lifting property.}
Ensuring that these$~\widetilde{\mu_n}'''$ are still lifts of the~$\mu_n$, as in~\eqref{lem compat2}, will require a bit of extra work.

Recall from~\eqref{indice fini} that~$G^{\text der.}(\QQ_S)\cdot Z(\QQ_S)$ is 
an open subgroup of~$G_S$. 
By its construction in~\eqref{U S tilde}, the group~$\wt{U_S}$ contains an open subgroup of~$Z(\QQ_S)$.
The Lie algebra of~$\wt{U_S}$ hence contains that of~$Z(\QQ_S)$, and is then the sum of the latter with the Lie algebra of~$\Omega$.
Equivalently, the product
\begin{equation}\label{open galois}
	\left(\wt{U_S}\cap Z(\QQ_S)\cap K'\right)\cdot\Omega
\end{equation}
is an open subgroup of~$\wt{U_S}$; which has finite index, as~$U_S$ is compact.
After replacing~$E$ by the finite extension corresponding to this subgroup, we may assume that
\begin{equation}\label{eq 37}
	\wt{U_S}=\left(\wt{U_S}\cap Z(\QQ_S)\cap K'\right)\cdot\Omega,
\end{equation}
which implies
\begin{equation}
	\wt{U_S}\cap Z(\QQ_S)\subseteq Z(\QQ_S)\cap K'\subseteq Z(\QQ_S)\cap K.
\end{equation}
Thus replacing~$E$, we did not change~$\Omega$, nor the associated~$\widetilde{\mu_n}'''$.
We can now prove:


\begin{lem}\label{lemmelift}
The direct images of the measures~$\widetilde{\mu_n}'''$ from~\eqref{attempt3} above, along the map~$\pi$ from~\eqref{pider}, are given by
\begin{equation}\label{lifted3}
\forall n\geq 0,~\pi_\star(\widetilde{\mu_n}''')=\mu_n.
\end{equation}
\end{lem}
This follows from~\eqref{lifted2} and the following equality, proven below,

\begin{equation}\label{lift compat}
\pi_\star(\widetilde{\mu_n}''')={\pi_G}_{\star}(\wt{\mu_n}'').
\end{equation}

\begin{proof}[Proof of~\eqref{lift compat}] Let~$U=\wt{U_S}\cap Z(\QQ_S)$ . Note that the map
\[
\Omega\times U\to \Omega\cdot U=\wt{U_S}.
\]
is a continuous map of compact groups.
It is surjective iun view of~\eqref{eq 37}. The image of the Haar probability measure is a probability which is invariant
under the image of the map. This is hence the Haar probability measure on~$\wt{U_S}$. The direct image measure is actually a convolution of measure. Pushing 
into~$\Gamma_G\backslash G(\RR\times \QQ_S)$, this convolution, in integral form, is defined as
\[
\mu_{\wt{U_S}}=\int_{u\in U} \mu_\Omega\cdot u~du\quad\text{on } \Gamma_G\backslash G(\RR\times \QQ_S)
\]
where the differential notation~$du$ denotes Haar probability measure on~$U$, and~$\mu_{\wt{U_S}}$ is the direct image of the Haar probability on~$\wt{U_S}$ in the~$S$-arithmetic space of~$G^\der$.
Similarly we can prove
\[
\mu_{\wt{U_S}}=\int_{k\in K\cap Z(\QQ_S)} \mu_{U_S}\cdot k~dk
\]
From~\eqref{item3} of Lemma~\ref{LemmeInvariances}, we get, with $dz$ the Haar probability on~$K\cap Z(\QQ_S)$,
\begin{eqnarray*}
{\pi_G}_\star(\mu_{\wt{U_S}}\cdot g_n)
&=&{\pi_G}_\star\left(\left(\int_{z\in K'\cap Z(\QQ_S)} \mu_{U_S}\cdot z~dz\right)\cdot g_n\right)\\
&=&{\pi_G}_\star\left(\int_{z\in K'\cap Z(\QQ_S)} \mu_{U_S}\cdot z\cdot g_n~dz\right)\\
&=&\int_{z\in K'\cap Z(\QQ_S)} {\pi_G}_\star(\mu_{U_S}\cdot z\cdot g_n)~dz\\
&=&\int_{z\in K'\cap Z(\QQ_S)} {\pi_G}_\star(\mu_{U_S}\cdot g_n)~dz\\
&=&{\pi_G}_\star(\mu_{U_S}\cdot g_n)\\
\end{eqnarray*}
In the same manner, we prove
\begin{equation}
{\pi_G}_\star(\mu_{\wt{U_S}}\cdot g_n)
=
{\pi_G}_\star(\mu_{\Omega}\cdot g_n).
\end{equation}
Finally, using Definitions~\eqref{attempt3}, \eqref{pider} we prove~\eqref{lift compat}
\begin{equation}
\pi_\star(\wt{\mu_n}''')={\pi_G}_\star(\mu_{\Omega}\cdot g_n)={\pi_G}_\star(\mu_{\wt{U_S}}\cdot g_n)={\pi_G}_\star(\mu_{U_S}\cdot g_n).\qedhere
\end{equation}
\end{proof}


\subsubsection{Conclusion.}
We now have a sequence~$(\wt{\mu_n}''')_{n\geq0}$ which lifts the sequence~$(\mu_n)_{n\geq0}$, and is the translate~$\mu_\Omega\cdot g_n$ of a probability~$\mu_\Omega$ coming from~$\Omega\leq G^\der(\QQ_S)$ by elements~$g_n$ from~$G^\der(\QQ_S)$, with respect to the semisimple group~$G^\der$. This lifted setting is the setting of~\cite{1-RZ}. We will now need to verify the hypothesis of~\emph{loc.~cit.}

\subsection{The analytic stability hypothesis.}\label{sous section stabilite}

In this section we will further alter the
sequence~$(\wt{\mu_n}''')_{n\geq0}$ in order to to be able to apply the results of \cite{1-RZ} to it. 
 

To apply results of \cite{1-RZ}, the sequence~$(g_n)_{n\geq 0}$ must satisfy a technical ``analytic stability'' hypothesis of~\cite{1-RZ}, which very loosely speaking means that the ``direction'' in which the~$g_n$ diverge is ``not too close to that of the centraliser of~$\Omega$''.

Recall that~$M$ denotes the Mumford-Tate group of $s$, and that~$H_S$ is the $\QQ_S$-algebraic envelope of the
$S$-adic monodromy group~$U_S$.

Since $s$ satisfies the $S$-Shafarevich assumption, by Proposition \ref{charSha},
we know that
\[
Z_{G}(M)(\QQ_S)\backslash Z_{G_S}(H_S).
\]
is compact. Then the closed subspace 
\[
Z_{G^{\der}}(M)(\QQ_S)\backslash Z_{G^{\der}_S}(H_S)
\]
is also compact. We may find a compact subset~$C$ of~$Z_{G^{\der}_S}(H_S)$ such that
\begin{equation}\label{eq 43}
Z_{G^{\der}_S}(H_S)=Z_{G^{\der}}(M)(\QQ_S)\cdot C.
\end{equation}

We will denote~$Z_S = Z_{G_S^{der}}(H_S)$.
Hence
\[
	Z_{G^{\der}(\RR\times\QQ_S)}(\Omega)=G^{\der}(\RR)\times Z_S.
\]

Recall that the centraliser $Z_{G^{der}}(M)$ is $\RR$-anisotropic, 
as it is contained in the centraliser of the point~$h\in X$, which is compact subgroup. It is a fortiori a~$\QQ$-anisotropic group.

We are in a trivial instance of Godement compactness criterion. There is a compact subset $\cF \subset  Z_{G^{der}}(M)(\RR\times\QQ_S)$ (a ``fundamental set'') such that
\begin{equation}\label{Godement}
Z_{G^{der}}(M)(\RR\times\QQ_S)=\left(Z_{G^{der}}(M)\cap \Gamma_S\right)\cdot \cF.
\end{equation}
Combining with~\eqref{eq 43} we get
\begin{equation}\label{Godement 2}
Z_S=\left(Z_{G^{der}}(M)\cap \Gamma_S\right)\cdot F\text{ with }F=\cF\cdot C
\end{equation}
a compact subset of~$Z_S$.

By the results of~\cite{1-RS} and~\cite[Partie~2]{1-R}, there exists a subset~
\[Y=\{1_\RR\}\times Y_S\subseteq G_S,\] where~$1_\RR$ denotes the neutral element of~$G(\RR)$, such that 
\begin{itemize}
\item we have a ``$S$-adic Mostow decomposition''
\begin{equation}\label{Mostow}
G^{\text{der.}}(\RR\times\QQ_S)=Z_{G^{der}}(\Omega)(\RR\times\QQ_S)\cdot (\{1_\RR\}\times Y_S),
\end{equation}
which at finite places becomes
\begin{equation}\label{Mostow2}
G^{\text{der.}}(\QQ_S)=Z_S\cdot Y_S,
\end{equation}
\item any sequence~$(g_n)_{n\geq0}$ in~$\{1_\RR\}\times Y_S$ satisfies the ``analytic stability'' property\footnote{\label{footnote stab 1}Namely the~$y^{-1}$ in~${Y_S}^{-1}$ are such that for any representation~$\rho:G\to GL(N,\QQ_S)$, there is a constant~$c=c({\rho,\Omega})$ such that for any vector~$v\in{\QQ_S}^N$, the action of~$y^{-1}$ one cannot shrink~$\Omega\cdot V$ uniformly by a factor bigger than~$c$: one has~
\begin{equation}\label{eq footnote}
\sup_{\omega\in\Omega} \Nm{y^{-1}\cdot\omega\cdot v}\geq c\cdot\Nm{v}
\end{equation}
with respect, say, to some standard product norm on~$\QQ_S^N$. 
See footnote~\ref{footnote stab 2}.}
of~\cite{1-RZ} with respect to~$\Omega$ (for~$y^{-1}$ in~$Y_S$ as we deal with left arithmetic quotient~$\Gamma\backslash G$). Here we used that the $\QQ_S$-algebraic envelope~$H_S\cap G_S^\der$ of~$\Omega$ is reductive.
\end{itemize}
Moreover we may, in~$G^{\der}(\RR\times\QQ_S)$, multiply~$\{1_\RR\}\times Y_S$ on the left by a compact subset of~$G^{\der}(\QQ_S)\times Z_S$ (and on the right by any compact subset of~$G$) and the product will still satisfy previous properties\footnote{\label{footnote stab 2}We refer to footnote~\ref{footnote stab 1}. Let~$F$ resp.~$C$ be a compact subset of~$Z_S$ resp.~$G$. For~$(\gamma,y,f,v)\in C\times Y_S\times F\times {\QQ_S}^N$ we have~$\gamma^{-1}\cdot y^{-1}\cdot f^{-1}\cdot\omega\cdot v=\gamma^{-1}\cdot y^{-1}\cdot\omega\cdot f^{-1}\cdot v$. Let~$c_1$ and~$c_2$ be the maximum of the operator norm of~$\rho(\gamma)$ resp.~$\rho(f)$ for~$\gamma$ in~$C$ resp.~$f$ in~$F$. For instance one has~$\Nm{v}=\Nm{f\cdot (f^{-1} \cdot v)}\leq c_2\cdot \Nm{f^{-1}\cdot v}$.
We have
\[
\Nm{\gamma^{-1} y^{-1} f^{-1} \omega\cdot v}
=
\Nm{\gamma^{-1} y^{-1} \omega \cdot f^{-1} v}
\geq c_1\cdot\Nm{y^{-1} \omega \cdot f^{-1} v}
\geq c_1 c\cdot \Nm{f^{-1} v}\\
\geq c_1 c c_2\cdot\Nm{v}.
\]
This proves~\eqref{eq footnote} of footnote~\ref{footnote stab 1} for $F\cdot Y_S\cdot C$ with the constant~$c_1\cdot c\cdot c_2$.
}. This is the case of the subset
\begin{equation}\label{def Y'}
Y'=F\cdot \left(\{1_\RR\}\times Y_S\right)\cdot Z_{G^{\der}}(M)(\RR)\subseteq G^{\der}(\RR\times\QQ_S).
\end{equation}
Let us now decompose~$g_n$ according to~\eqref{Mostow2}
\[
\forall n\geq0, \exists z_n\in  Z_{G^{der}}(M)(\QQ_S),\exists y_n\in Y_S,~
	g_n= z_n\cdot y_n,
\]
and then decompose~$z_n$ inside~$Z_{G^{der}}(M)(\RR\times\QQ_S)$ with respect to the decomposition~\eqref{Godement 2},
\[
	z_n= \gamma_n\cdot f_n,\text{ with }\gamma_n\in \Gamma_S\cap Z_{G^{der}}(M)(\RR\times\QQ_S)\text{ and }f_n\in F.
\]
We recall from~\eqref{invariance 1} that
\begin{equation}\label{gamma invariance}
\mu_\Omega\cdot \gamma_n=\mu_\Omega.
\end{equation}

Substituting~$g_n=\gamma_n\cdot f_n\cdot y_n$ with~$f_n \cdot y_n$ we may, and will, assume that~$\gamma_n=1$.
But doing so we lose the property that~$g_n\in G_S$, as~$\gamma_n$ may be non  trivial at the real place.
Instead we define~
\begin{equation}\label{def gn'}
g^\prime_n={\gamma_n}^{-1}\cdot g_n\cdot {\gamma_{n,\RR}},
\end{equation}
where~${\gamma_{n,\RR}}$ is the real place factor of~$\gamma_n$.
We then have~$g^\prime_n\in G_S$ as we ensured it be trivial at the real place. Define, in our final attempt to lift the~$\mu_n$,
\begin{equation}
	\wt{\mu_n}=\mu_\Omega\cdot g^{\prime}_n.
\end{equation}
We observe that definition~\eqref{def gn'} agrees with~\eqref{def Y'}: we have
\begin{equation}\label{g' dans Y'}
\forall n\geq 0, g'_n\in Y'.
\end{equation}

\begin{prop}\label{prop compat}
The sequence~$(g^\prime_n)_{n\geq0}$ satisfies the analytic stability hypothesis
required by~\cite[Theorem~3]{1-RZ}. Furthermore, we have
\begin{equation}\label{final compat}
\pi_\star\left( \mu_\Omega\cdot g^\prime_n\right)=\mu_n.
\end{equation}
\end{prop}

\begin{proof}[Proof of~\ref{final compat}] The analytic stability hypothesis is given by~\eqref{g' dans Y'}. It remains to prove the lifting compatibility~\eqref{final compat}.

We have, using~\eqref{gamma invariance}, and recalling definition~\eqref{attempt3},
\[
\wt{\mu_n}=\mu_\Omega\cdot{\gamma_n}^{-1}\cdot g_n\cdot  {\gamma_{n,\RR}}=\mu_\Omega\cdot g_n\cdot  {\gamma_{n,\RR}}=\wt{\mu_n}'''\cdot{\gamma_{n,\RR}}.
\]
But recall that~${\gamma_{n,\RR}}\in Z_{G^{der}}(M)(\RR)$ a compact group which is \emph{included in the stabiliser~$K_h$ of~$h$ in~$X$}. Similarly to~\eqref{seconde invariance}, we deduce that the measures~$\wt{\mu}_n$ and~$\wt{\mu}_n'''$, though they maybe different
as measures on~$\Gamma_S\backslash G^{\text der.}(\RR\times \QQ_S)$,
will satisfy
\[
{\pi_G}_\star(\wt{\mu}_n)={\pi_G}_\star(\wt{\mu}_n''').
\]
We are done: by recalling Lemma~\ref{lemmelift}, and substituting~$\pi$ for~$\pi_G$, as we are dealing with measures supported on~$\Gamma\backslash G^{\text{der.}}(\RR\times\QQ_S)$.
\end{proof}

\subsection{Invoking the Theorem of Richard-Zamojski.}\label{invocation}
We now are in a position to apply~\cite[Theorem 3]{1-RZ}.
In this section we review the assumptions of the statement of this theorem
and apply the theorem to our situation.

\subsubsection{Reviewing the hypotheses..}

\paragraph{The ambient group~$G$ and its $S$-arithmetic quotient.}

The semisimple algebraic  group~$G$ of~\cite{1-RZ} is the derived subgroup~$G^{\text der.}$ of our
initial~$G$.
The finite set of places considered in~\cite{1-RZ} is our set~$S$
together with the archimedean place. The $S$-arithmetic lattice
is our~$\Gamma$ defined by~\eqref{defi Gamma} and~\eqref{defi Gamma G}. 
This defines the ambient~$S$-arithmetic space~$\Gamma\backslash G^{\text der.}(\RR\times\QQ_S)$
in our notations.

\paragraph{The piece of orbit~$\Omega$.}
The~$\Omega$ of~\cite{1-RZ} is~$\{1_\RR\}\times\Omega$ here. The image~$U_S$ of the representation of the Galois group is a compact~$S$-adic Lie subgroup of~$M(\QQ_S)$. It follows that~$\Omega$, as defined in~\eqref{U S tilde} and~\eqref{defi Omega}, is a bounded~$S$-adic Lie subgroup. 
 We postpone the proof that its $\QQ_S$-algebraic envelope~$H$ is a Zariski connected subgroup of~$G_S$ to~\ref{H Omega} below.

\paragraph{The sequence of translates of $\mu_\Omega$.}
The measure~$\mu_\Omega$ is of the type studied in~\cite{1-RZ}.

The translating elements~$g_n$ of~\cite{1-RZ} are our~$g^\prime_n$.
The ``analytic stability'' hypothesis has been taken care of in the preceding section.

\subsubsection{Applying Theorem of~{\cite{1-R}}}
As a consequence of this analytic stability hypothesis, we may apply
\cite[Th\'{e}or\`{e}me~1.3, Exp.~VI, p.~121]{1-R} with $H = H_S$ and $Y_S = \{ y_n \}$ and $f$ the characteristic function of $\Omega$ to the sequence of translated probabilities
\[
	\wt{\mu_n}=\mu_\Omega\cdot g_n^\prime
\]
on the $S$-arithmetic homogeneous space
\begin{equation}\label{Homog space}
	\Gamma\backslash G^\text{der.}(\RR\times\QQ_S).
\end{equation}
Thus the sequence~$(\wt{\mu_n}^\prime)_{n\geq0}$ is tight: 
any subsequence contains a subsequence converging to a probability measure.

In particular, after possibly extracting a subsequence, we may assume the sequence~$(\wt{\mu_n}^\prime)_{n\geq0}$ is tightly convergent.

We may now invoke~\cite[Theorem~3]{1-RZ}. 

\subsubsection{The envelope of~$\Omega$}\label{H Omega} We recall that the $\QQ_S$-algebraic envelope of~$U_S$, the algebraic monodromy group, is denoted~$H_S$, and that~$H_S$ is a reductive group. Here we will determine the $\QQ_S$-algebraic envelope of~$\Omega$, in terms of that of~$H_S$. We discuss why it is a reductive group, and how to ensure it is Zariski connected.

Firstly the  $\QQ_S$-algebraic envelope~$\widetilde{H}_S$ of~$\widetilde{U}_S$ is the algebraic group generated by~$U_S$ and~$K\cap Z(\QQ_S)$. The $\QQ_S$-algebraic envelope of~$Z(\QQ_S)\cap K$ is a Zariski open subgroup~$\widetilde{Z}$  of~$Z(\QQ_S)$: one has~$Z^0\leq \widetilde{Z}\leq Z$. Then
\[
\widetilde{H}_S=H_S\cdot \widetilde{Z}.
\]

Let~$H$ be the  $\QQ_S$-algebraic envelope~$\Omega$. One has~$H\leq G_S^{\der}$ and~$H\leq \widetilde{H}_S$ as~$\Omega$ is contained in both groups. 

Let~$\mathfrak{u},\wt{\mathfrak{u}},\mathfrak{\omega},\mathfrak{z},\mathfrak{h}_S,\wt{\mathfrak{h}}_S,\mathfrak{h}$ be the Lie algebras of~$U$, $\wt{U}_S$, $\Omega$, $Z$, $H_S$ and~$H$. All sums and direct sums will be sum of linear spaces, and the resulting sums will be Lie algebras. 
We have~$\wt{\mathfrak{u}}={\mathfrak{u}}+{\mathfrak{z}}$. The decomposition~${\mathfrak{g}}={\mathfrak{g}}^\der\oplus {\mathfrak{z}}$ induces~$\wt{\mathfrak{u}}={\mathfrak{\omega}}\oplus{\mathfrak{z}}$. Taking algebraic envelopes of Lie subalgebras is compatible with sums. We get
\begin{equation}\label{Lie sums}
\wt{\mathfrak{h}}_S=\mathfrak{h}_S+\mathfrak{z}=\mathfrak{h}\oplus \mathfrak{z},
\end{equation}
and deduce~$\mathfrak{h}=\wt{\mathfrak{h}}_S\cap \mathfrak{g}^\der$. This establishes that~$H$ is a Zariski open subgroup of~$\widetilde{H}_S\cap G_S^\der$. 

The product map~$(h,z)\mapsto h\cdot z$ is an homomorphism, as~$H$ and~$Z$ commute.
At the level of the groups, we deduce an isogeny
\[
H^0\times Z^0\to {H_S}^0\cdot Z^0.
\]	
Recall that~$H_S$ and~$Z$ are reductive. Hence~${H_S}^0\times Z^0$ is reductive too, and so is its quotient~${H_S}^0\cdot Z^0$. The finite cover~$H^0\times Z^0$ must then be reductive, and so is its direct factor~$H^0$. This establishes that~$H$ is a reductive group.

The Zariski neutral component~$H^0$ of~$H$ is that~$(\widetilde{H}_S\cap G_S^\der)^0$, and we have~$H=H^0$ if and only if				
\[
\Omega\subseteq H^0(\QQ_S)=(\widetilde{H}_S\cap G_S^\der)^0(\QQ_S).
\]
As~$H^0(\QQ_S)$ is open in~$H(\QQ_S)$, there is a neighbourhood~$V$ of~$H^0(\QQ_S)$ in~$G_S$ that doesnot meet~$H(\QQ_S)\smallsetminus H^0(\QQ_S)$. Note that~$H^0$ does only depend on~$H_S$, not on~$\Omega$ or~$K$. Provided~$K$ and~$U_S$ are sufficiently small, we may ensure that~$\Omega\subseteq U_S\cdot K\subseteq V$. In such a case, the group~$H$ is Zariski connected.


\subsection{The equidistribution property.}
We apply~\cite[Theorem~3]{1-RZ} in the setting recalled in~\S\ref{invocation}. 
Then, possibly passing to a subsequence of~$\left(\wt{\mu}_{n}\right)_{n\geq0}$,
there exist
\begin{itemize}
\item a~$\QQ$-subgroup~$L\leq G^{\text der.}$ of $S$-Ratner type (as in Definition~\ref{defi S Ratner class}),
\item elements~$n_\infty,g_\infty\in G^\text{der.}(\RR\times\QQ_S)$
\end{itemize}
Recall that $L^{\dagger}$ denotes the subgroup of of $L(\RR\times\QQ_S)$ generated by its unipotent one parameter subgroups over $\RR\times\QQ_S$ and  let $\Lpp$ be the closure of $(\Gamma \cap L(\QQ_S))L^{\dagger}$.

Let $\mu_{\Lpp}$ be the canonical $\Lpp$-invariant probability measure. It's support is 
$$
\Gamma\backslash\Gamma\Lpp = \overline{\Gamma\backslash\Gamma L^{\dagger}}. 
$$
We refer to \cite{1-RZ}, Appendix B for details and proofs.
By ~\cite[Theorem~3]{1-RZ} we have:
\[
\lim_{n\to\infty} \wt{\mu}_{n}= \mu_{\Lpp}\!\!\cdot n_\infty \star \nu\cdot g_\infty
\]
This limit measure is the translate by~$g_\infty$ of the convolution of the translated probability measure~$\mu_{\Lpp}\cdot n_\infty$ by the Haar probability~$\nu$ on~$\Omega$.
In terms of Radon measures on~$\Gamma\backslash G^{\text der.}(\RR\times \QQ_S)$,
applied to an arbitrary bounded continuous test function~$f$ in~$C^b(\Gamma\backslash G^{\text der.}(\RR\times \QQ_S))$, this means
\[
	\lim_{n\to\infty}
		\int f\wt{\mu}_{n}	
=
	\int\! f\mu_{\Lpp}  n_\infty\star \nu\cdot g_\infty
=
	\int_{\omega\in\Omega}\!
		\int_{\Gamma l\in\Gamma\backslash\Gamma \Lpp} 
			f (\Gamma l\cdot n_\infty\cdot \omega\cdot g_\infty)
		~\mu_{\Lpp} (\Gamma l)
	\nu(\omega).
\]



We now explain how to derive  the conclusion~\eqref{TheoremeLimite} of Theorem~\ref{Theoreme1}.
We will need the following lemma.
\begin{lem}\label{Lemma 4.6}
Let~$g=(g_\RR,g_S)\in G^\der(\RR\times\QQ_S)$. Then the pushforward of the measure~$\mu_\Lpp\cdot g$ to 
\[
\Gamma\backslash G^\der(\RR\times \QQ_S)/K_S
\] 
is a finite linear combination of Haar measures on right orbits of~${g_\RR}^{-1} L(\RR)^+ g_\RR$.
\end{lem}
\begin{proof}
The image of the support~$\Gamma \backslash \Gamma\Lpp \cdot g$ in the double coset space is a finite union of right~${g_\RR}^{-1} L(\RR)^+ g_\RR$ orbits, by finiteness of the class number of~$L$ (relative to~$\Gamma\cap L(\RR\times\QQ_S)$ and~$K_S\cap L(\QQ_S)$).

Moreover the right action of~${g_\RR}^{-1} L(\RR)^+ g_\RR$ passes to the quotient (note that $K_{S}$ is contained in $L(\QQ_S)$ hence the infinite place is not affected by this right quotient). 
Note that~$\mu_\Lpp\cdot g$ is right~${g_\RR}^{-1} L(\RR)^+ g_\RR$ invariant. Hence its pushforward measure is right invariant too. The pushforward measure is the sum of its restrictions to these finitely many orbits, and each restriction is right invariant, in other words a Haar measure on the orbit.
\end{proof}
\begin{cor}The pushforward of~$\mu_{\Lpp}\!\!\cdot n_\infty \star \nu\cdot g_\infty$ to~$\Sh_K(G,X)$ is a finite linear combination of canonical measures in the sense of the definition~\ref{canonical}.
\end{cor}
\begin{proof}
Let~$R\subseteq\Omega$ be a set of representatives for the quotient~$\Omega g_\infty K_S/K_S$. It is finite because $\Omega$ is a compact subset of~$G^\der(\QQ_S)$ and~$K_S$ is open in~$G^\der(\QQ_S)$. ... . It follows that the pushforward~$\mu$ of~$\mu_{\Lpp}\!\!\cdot n_\infty \star \nu\cdot g_\infty$ in 
\[
\Gamma\backslash G^\der(\RR\times \QQ_S)/K_S
\]
is the same as the pushforward of the sum of~$(\mu_{\Lpp}\!\!\cdot n_\infty \cdot \omega\cdot g_\infty)_{\omega\in R}$.

Applying the lemma~\ref{Lemma 4.6} to each of these~$\mu_{\Lpp}\!\!\cdot n_\infty \cdot \omega\cdot g_\infty$, we learn that~$\mu$ is a sum of Haar measure on right orbits of~${g_\RR}^{-1} L(\RR)^+ g_\RR$. It is enough to prove that the direct image of the latter in~$\Sh_K(G,X)$ is proportional to a canonical measure attached to a $S$-real weakly special submanifold. 
Write~$g=n_\infty \cdot \omega\cdot g_\infty=(g_\RR,g_S)$. By construction this pushforward is of the form~$\mu_{L,(h,t)}$ for~$h$ obtained from~$g_{\RR}$ and~$t$ associated with~$g_S$.
\end{proof}

This proves the equidistribution conclusion~\eqref{TheoremeLimite} of Theorem~\ref{Theoreme1}.

\subsubsection{Getting rid of~$n_\infty$.}\label{n infini}

After replacing by a subsequence, all the conclusions
of~\cite[Theorem~3]{1-RZ} hold. In particular we also know that~$n_\infty$
belongs to the closure of~$(\Gamma\cap N)\cdot (Z_{G_S^\der}(\Omega)\cap N)$.
We claim that 
\begin{equation}\label{eq 56}
(\Gamma\cap N)\cdot (Z_{G_S^\der}(\Omega)\cap N)
\end{equation}
is already a closed subset.
\begin{proof} By the $S$-Shafarevich hypothesis, the quotient
\[
Z_{G_S}(\Omega)/Z_{G_S}(M)
\]
is compact. It follows that for any subgroup~$Z$ of~$Z_{G_S}(\Omega)$
the quotient
\[
Z/(Z\cap Z_{G_S}(M))
\subseteq 
Z_{G_S}(\Omega)/Z_{G_S}(M)
\]
is compact. In particular, for~$Z=Z_N(\Omega)=N\cap Z_{G_S^\der}(\Omega)$, we may write
\begin{equation}\label{eq 56b}
Z_{N}(\Omega)=Z_{N}(M)\cdot C
\end{equation}
for a compact subset~$C$. The closedness of~\eqref{eq 56}, that is of 
\[
(\Gamma\cap N)\cdot (Z_{G_S^\der}(\Omega)\cap N)=(\Gamma\cap N)\cdot (Z_{G_S^\der}(M)\cap N)\cdot C
\]
will follow from that of
\[
(\Gamma\cap N)\cdot (Z_{G_S^\der}(M)\cap N).
\]

Let~$F$ be a finite generating subset of~$M$, as a topological group for the Zariski topology, defined over~$\QQ$. Then the orbit map for the adjoint action
\[
G_S^{\der}\xrightarrow{g \mapsto (gfg^{-1})_{f\in F}}  (G_S^{\der})^F
\]
embeds~$G_S/Z_{G_S^\der}(\Omega)$ as a subvariety of~ the affine variety~$G_S^\der$. (In particular~$G_S/Z_{G_S^\der}(\Omega)$ is quasi-affine). We can linearise this action by choosing a faithful representation~$G_S^\der\to GL(N)$ and embedding~$GL(N)^F$ into~$V={M_N}^F$ a Cartesian power of the corresponding matrix space.

Let~$p$ be a generator of~$\det\mathfrak{l}\subseteq\bigwedge^{\dim L}\mathfrak{g}$, the maximal exterior power of the Lie algebra of~$L$. Then~$N$ is defined as the stabiliser of~$p$ in~$G_S^\der$. The orbit map
\[
G_S^{\der}\xrightarrow{g \mapsto g\cdot p=\bigwedge^{\dim L}({\rm ad}_g)(p)}  \bigwedge^{\dim L}\mathfrak{g}
\]
at~$p$ embeds~$G/N$ into the affine variety~$W\bigwedge^{\dim L}\mathfrak{g}$. 

We deduce, with the product map, an embedding
\begin{equation}\label{eq 57}
G/(Z_{G_S^\der}(\Omega)\cap N)\to V\times W,
\end{equation}
the orbit map at~$((f)_{f\in F},p)$.

By definition~$\Gamma$ stabilises an lattice~$\Lambda$ in the free~$\QQ_S$ module~$V\times W$ defined over~$\QQ$, which is arithmetic, that made of~$\QQ$ rational elements. Up to scaling we may assume~$((f)_{f\in F},p)\in\Lambda$. It follows that
\[
(\Gamma\cap N)\cdot ((f)_{f\in F},p)
\]
which is a subset of~$\Lambda$, is discrete, and in particular closed. It follows that its inverse image in~$G_S^\der$,
which is~\eqref{eq 57} is closed as well.
\end{proof}
We may write~$n_\infty=\gamma\cdot z$ with~$\gamma\in \Gamma\cap N$
and~$z\in Z_{G_S^\der}(\Omega)$. We claim that
\[
\mu_{\Lpp}\cdot \gamma=\mu_{\Lpp},
\]
which is proven below. We view~$\nu$ as a measure on~$G_S^\der$ supported on~$\overline{\Omega}$. (Actually~$\overline{\Omega}=\Omega$). We also have
\[
z\cdot \nu=\nu\cdot z.
\]
So we may rewrite the limit measure
\[
\mu_{\Lpp}\!\!\cdot n_\infty \star \nu\cdot g_\infty
=
\mu_{\Lpp}\!\! \star \nu\cdot z\cdot g_\infty.
\]
Subsituting~$g_\infty=z\cdot g_\infty$ we may assume~$n_\infty=1$.
\begin{proof}[Proof of the claim]
Recall that~$L$ is a normal subgroup of~$N$.
Consequently~$\Gamma\cap L(\RR\times\QQ_S)$ is normal in~$\Gamma\cap N(\RR\times\QQ_S)$.
The subgroup~$L(\RR\times\QQ_S)^+$ is normalised by~$N(\RR\times\QQ_S)$, hence by~$\Gamma\cap N(\RR\times\QQ_S)$. We have seen both factors of
\[
L(\RR\times\QQ_S)^+\cdot\left(\Gamma\cap L(\RR\times\QQ_S)\right)
\]
are normalised by~$\Gamma\cap N(\RR\times\QQ_S)$. As~$\Gamma\cap N(\RR\times\QQ_S)$
acts continuously, it normalises the closure of this product, namely~$\Lpp$.

By definition~$\mu_{\Lpp}$ is the right~$\Lpp$-invariant probability on~$\Gamma\backslash\Gamma\Lpp$. It follows that~$\mu_{\Lpp}\gamma$
is the right~$\gamma^{-1}\Lpp\gamma$-invariant probability on~$\Gamma\backslash\Gamma\Lpp\gamma$. We just have seen~$\gamma^{-1}\Lpp\gamma=\Lpp$. But we also have
\[
\Gamma\backslash\Gamma\Lpp\gamma
=
\Gamma\backslash\Gamma\gamma^{-1}\Lpp\gamma
=
\Gamma\backslash\Gamma\Lpp.
\]
It follows that~$\mu_{\Lpp}\gamma=\mu_{\gamma^{-1}\Lpp\gamma}=\mu_{\Lpp}$.
\end{proof}

\section{Focusing criterion and Internality of the equidistribution.}\label{SectionRZ}

In this section we prove conclusion~\eqref{TheoremeInner} of Theorem~\ref{Theoreme1}, namely the inclusion of  supports of  the measures~$\mu_n$ in the support of the limiting measure. This property is essential in order to derive the conclusions about the topological closure.

\begin{lem}
For all $n$ large enough, we have the inclusion
$$
{\rm Supp}(\mu_n) \subset \pi({\rm Supp}(\mu_{\infty}))
$$
of closed subsets in~$\Sh_K(G,X)$.
\end{lem}
 We will rely on the ``focusing property'', which states that the translating elements~$g_n$ are take a specific form, as a sequence. This inlusion of supports is a phenomenon which occurs at finite level, in the double quotient~$\Sh_K(G,X)$. An important ingredient is that our translating elements are trivial at the archimidean place.


\subsection{A Fundamental case.}
\newcommand{\M}{{\underline{M}}}
\newcommand{\N}{{\underline{N}}}
\newcommand{\Mpp}{{\underline{M}^\ddagger}}
Let us first treat the case where~$g'_n$ belongs to
\[
\Mpp=
\bigcap_{\omega\in \Omega}\omega \Lpp\omega^{-1}
\]
This is the typical situation featuring the dynamics explicited by~\cite[Theorem~3]{1-RZ}. We will see 
in the following sections how to reduce the situation to this case.

We first note that~$\Omega$  is a compact subgroup.  Therefore~$\Omega=\overline{\Omega}$ and
\[{\rm Supp}(\nu)=\overline{\Gamma\backslash\Gamma\cdot\Omega}=\Gamma\backslash\Gamma\cdot\overline{\Omega}\text{ and }\bigcap_{\omega\in \Omega}\omega \Lpp\omega^{-1}=\bigcap_{\omega\in \overline{\Omega}}\omega \Lpp\omega^{-1}.\]

For any~$\omega$ in~$\Omega$ we rewrite
\[
\Gamma\cdot \omega \cdot g'_n=\Gamma\cdot(\omega \cdot g'_n\cdot \omega^{-1})\cdot\omega\in \Gamma\Lpp\omega.
\]
It follows that
\[
{\rm Supp}(\wt{\mu}_n)=\Gamma\backslash\Gamma\Omega\cdot g'_n\subseteq\Gamma\backslash\Gamma\Lpp\Omega={\rm Supp}(\mu_{\Lpp}\star\nu).
\]
It is now enough to observe that
\[
\mu_\infty=\pi_\star(\mu_{\Lpp}\star\nu).\]
Both are combination of canonical measures supported on a union of real weakly $S$-special subvarieties associated with~$L$.

\subsubsection{}

In order to achieve this, we will use a finer form of the main result of~\cite{1-RZ}. Namely that the convergenceof measures happens modulo
$\Gamma\cap N$ instead of merely modulo~$\Gamma$: if~$\widetilde{\mu}_\Omega$, resp.~$\widetilde{\mu}_{\Lpp}$, denotes the direct image of~$\nu$, resp. the quotient measure of the Haar measure on~$\Lpp$, in
\[(\Gamma\cap N)\backslash G^\der(\RR\times\QQ_S),\]
we have
\[
\lim_{n\to\infty} \widetilde{\mu}_\Omega\cdot g'_n=\widetilde{\mu}_{\Lpp}\cdot n_\infty\star\nu\cdot g_\infty.
\]
Let us denote~$\widetilde{\mu}_\infty$ this limit probability measure.
We note that~$\widetilde{\mu}_{\Lpp}$, resp.~$\widetilde{\mu}_{\infty}$ is a probability measure whose direct image in
\[
\Gamma\backslash G^\der(\RR\times\QQ_S)
\]
is~$\mu_\Lpp$. Moreover, by the definition of~$N$, we note that~$\Lpp$ is normalised by~$\Gamma\cap N$ and that~$\widetilde{\mu}_\infty$ is determined by its direct image\footnote{This image exactly describes an ergodic decomposition of~$\widetilde{\mu}_\infty$, and of~${\mu}_\infty$. Here, at level~$K$ it can serve to describe the canonical measures components of~$\mu_\infty$.} in
\[
(\Lpp\cdot (\Gamma\cap N))\backslash G^\der(\RR\times\QQ_S).
\]
(Care that this is not a quotient map by a group action.)

Now we note that the~$\widetilde{\mu}_\Omega\cdot g'_n$ have the same direct image, because of the identity
\[
\int_{\omega\in\Omega}\Lpp(\Gamma\cap N)\cdot \omega g'_n~\nu(\omega)=\int_{\omega\in\Omega}\Lpp(\Gamma\cap N)\cdot \omega g'_n~\nu(\omega).
\]
Passing to the limit, the direct image of~$\widetilde{\mu}_\infty$ is this same measure. By construction, this is the direct image of~$\nu(\omega)$. It follows the identity~$\widetilde{\mu}_\infty=\widetilde{\mu}_{\Lpp}\star\nu$ and, by direct image,~${\mu}_\infty={\mu}_{\Lpp}\star\nu$.


\subsection{The focusing criterion's factorisation}
We recall that we have a limiting distribution
\[
\lim_{n\to\infty} \wt{\mu}_{n}= \mu_{\Lpp} \star \nu\cdot g_\infty
\]
where the probability measures~$\wt{\mu}_{n}=\mu_\Omega\cdot g'_n$ satisfy
\[
	\mu_n=\pi_\star\left(\wt{\mu_n}\right).
\]
We ensured that~$g'_n$ belongs to~$G^\der(\QQ_S)$, which we view as a subgroup of~$G(\RR\times\QQ_S)$.

Our goal is to prove that, after possibly extracting a subsequence, 
\[
{\rm Supp}({\mu}_n)={\rm Supp}(\pi_\star(\wt{\mu}_n))\text{ is contained in }{\rm Supp}({\mu}_\infty)={\rm Supp}(\pi_\star(\wt{\mu}_\infty)).
\]

As we are allowed to extract subsequences we may use the more stringent conclusions of~\cite[Theorem~3]{1-RZ},
the \emph{focusing criterion}, according to which we  may factor
\begin{equation}\label{decompo}
	g'_n=l_n\cdot f_n\cdot b_n
\end{equation}
where~$l_n\in \bigcap_{\omega\in \Omega} \omega L(\RR\times\QQ_S)\omega^{-1}$,
where~$(b_n)_{n\geq 0}$ is a bounded sequence and where~$f_n\in N\cap Z_{G_S^\der}(\Omega)$.

We note that such a decomposition holds if and only if it holds place by place:
the groups involved are place by place product groups (they are even~$\QQ_S$-algebraic); a bounded sequence is a sequence which is bounded  a sequence whose component at each place is bounded.

In particular, as the real component of~$g'_n$ is the neutral element, we may, and we will, substitute the real components of~$l_n$,
of~$f_n$ and of~$b_n$ by the neutral element and still have a decomposition as above, at the real place as well as in~$G^\der(\RR\times\QQ_S)$. To summarize: without loss of generality, we may assume 
\begin{equation}\label{factor decomposition}
l_n, f_n, b_n\in G^\der_S\subseteq G^\der(\RR\times\QQ_S).
\end{equation}

\subsubsection{} Let us comment on the three factors in~\eqref{factor decomposition} above. The bounded factor~$b_n$ does not involve any dynamic, and cannot be discarded in general. In our situation, for Hecke orbits, \emph{at finite level}, we will be able to suppress it. The factor~$l_n$ is the one responsible for equidistribution, happening from the inside of~$(\Gamma\cap L) \backslash L \cdot \Omega$ to the whole of~$(\Gamma\cap L) \backslash L \cdot \Omega$. 

The factor~$f_n$ is trickier, and a toy exemple is when a sequence of distinct points equidistributes to its limit, whereas the limit sequence do not converge in the covering group. For instance in cases where the $S$-Shafarvich hypothesis does not hold. We may build up more intricate cases by considering a cartesian product, of this toy exemple, with a more typical situation of inner equidistribution as above. We can even consider cases of a semi-direct product, or more generally cases where~$L$ is merely normal in the ambiant group (the subgroup~$N$ of~$G$). This gives an outline of why the fator~$f_n$ may happen.

\subsubsection{} We explain how we can pick~$f_n$ in~$Z_{G_S^\der}(M)$, which is defined over~$\QQ$ instead of picking it merely in~$Z_{G_S^\der}(\Omega)$.
As we saw, the $S$-Shafarevich hypothesis implies that the subgroup~$Z_{G_S^\der}(M)$ is cocompact in~$Z_{G_S^\der}(\Omega)$. 
We have morevover the following.
\begin{lem} The subgroup~$Z_{G_S^\der}(M)\cap N$ of~$Z_{G_S^\der}(\Omega)\cap N$ is cocompact.
\end{lem}
\noindent This will allow to restrict to the case where the~$f_n$ are in~$Z_{G_S^\der}(M)\cap N$ rather than merely in~$Z_{G_S^\der}(\Omega)\cap N$, seen in~\eqref{conclusion f'n} below.
\begin{proof}Note that~$Z_{G_S^\der}(\Omega)$, as it is a centraliser, is an algebraic subgroup of~$G_S$. Let~$Z^+$ be the maximal isotropic connected $\QQ_S$-subgroup of~$Z_{G_S^\der}(\Omega)$ (the product of the non compact factors of the factor groups~$Z_{G_S^\der}(\Omega)\cap G(\QQ_p)$). This is a normal subgroup of~$Z_{G_S^\der}(\Omega)$. The isotropic factors are generated by unipotent subgroup and split tori. Hence every regular function on~$Z^+$ which is bounded is actually constant, and more generally every regular map into an affine space, hence in a quasi-affine variety. 

We consider the map
\[
Z_{G_S^\der}(\Omega)\to Z_{G_S^\der}(\Omega)/Z_{G_S^\der}(M).
\]
	We recall that~$Z_{G_S^\der}(M)$ is a reductive group. It follows that the homogeneous space~$Z_{G_S^\der}(\Omega)/Z_{G_S^\der}(M)$ is affine\footnote{See~\cite{1-RichardsonMatsushima} for arbitrary characteristic, \cite[\S B, p. 2 (206) before Th.~3]{1-Matsushima} for the complex field. These provide a reverse statement for quotients of a group which is reductive.}. As we saw, the quotient map must be constant on~$Z^+$ that is we have
\[
Z^+\subseteq Z_{G_S^\der}(M)\subseteq Z_{G_S^\der}(\Omega).
\]

We now consider the map
\[
Z_{G_S^\der}(\Omega)\cap N
\to
\left.(Z_{G_S^\der}(\Omega)\cap N)\middle/(Z^+\cap N) \right.
\]
We may identify the image with the subspace~$NZ^+/Z^+$ of~$Z_{G_S^\der}(\Omega)/Z^+$. As~$Z^+$ is normal in~$Z_{G_S^\der}(\Omega)$, the subspace~$NZ^+/Z^+$ is actually a subgroup of~$Z_{G_S^\der}(\Omega)/Z^+$. It is in particular closed. As it is closed in a compact space, it is itself compact. Its quotient
\[
NZ^+/Z(M)\simeq (Z_{G_S^\der}(\Omega)\cap N)/(Z_{G_S^\der}(M)\cap N)
\]
is then compact.
\end{proof}

We may hence factor~$f_n=f'_n\cdot b'_n$ where~$f'_n$ belongs to~$Z_{G_S^\der}(M)$ and the~$b'_n$ are bounded in~$Z_{G_S^\der}(\Omega)$. Substitutig~$f_n$ with~$f'_n$ and~$b_n$ with~$b_n\cdot b'_n$ 
\begin{equation}\label{conclusion f'n}
\text{we may actually assume that $f_n$ belongs to~$Z_{G_S^\der}(M)$}.
\end{equation}
\subsection{Getting rid of the bounded factor~$b_n$} The bounded factors~$b_n$ are just translating, at the ``infinite level''~$\Gamma\backslash G^\der(\RR\times\QQ_S)$, the situation. They involve no asymptotic dynamical feature, and as they belong to~$G_S$, they will essentially be killed at finite level, modulo~$K$. Here are the details.

Possibly passing to an extracted subsequence, we may assume that the bounded
sequence~$(b_n)_{n\geq0}$ is convergent in~$G^\der(\QQ_S)$, with limit, say,~$b_\infty$. It follows that~$(b_n tK)_{n\geq0}$ is a convergent sequence in the space~$G^\der(\QQ_S)K/K$ (it converges to~$b_\infty tK$). As this quotient is a discrete space the convergent sequence~$(b_n tK)_{n\geq0}$ is actually eventually constant.
Possibly extracting further, we may assume that this is a constant sequence.

Substituting~$t$ with~$b_\infty\cdot t$ we may assume~$b_\infty=1$ in~$G^\der(\QQ_S)$,
and thus~$b_ntK=tK$. For any~$x$ in~$\Gamma\backslash G^\der(\RR\times\QQ_S)$ we have~$\pi(x\cdot b_n)=\pi(x)$. Hence
\[
	\pi_\star(\mu\cdot b_n)
	=
	\pi_\star(\mu)
\]
for every bounded measure~$\mu$ on~$\Gamma\backslash G^\der(\RR\times\QQ_S)$. In particular
\[
	\mu_n
	=
	\pi_\star(\wt{\mu}_n)
	=
	\pi_\star(\mu_\Omega\cdot l_n\cdot f_n\cdot b_n)
	=
	\pi_\star(\mu_\Omega\cdot l_n\cdot f_n).
\]
In order to prove, in~$\Gamma\backslash G^\der(\RR\times\QQ_S)K/K\subseteq \Sh_K(G,X)$
\[
{\rm Supp}(\pi_\star(\wt{\mu}_n))={\rm Supp}(\mu_n)\subseteq{\rm Supp}(\mu_\infty)
\]
we may substitute~$\wt{\mu}_n=\mu_\Omega\cdot l_n\cdot f_n\cdot b_n$ with~$\mu_\Omega\cdot l_n\cdot f_n$. In other terms we may assume~$b_n=1$.

\subsection{Getting rid of the centralising factor~$f_n$} 

We consider the algebraic subgroups of~$G^\der(\RR\times\QQ_S)$ given by 
\begin{subequations}
\begin{eqnarray}
\underline{N}&=N\cap Z_{G}(M)\\
\M&=\bigcap_{\omega\in\Omega}\omega L\omega^{-1}.
\end{eqnarray}
\end{subequations}
The former is defined over~$\QQ$ and the latter over~$\QQ_S$. We write
\[
\Mpp=\bigcap_{\omega\in\Omega}\omega \Lpp\omega^{-1}.
\]
The latter group~$\M$ is denoted~$M$ in~\cite{1-RZ}. Our notation distinguishes it from the Mumford-Tate group that here we denote~$M$.

As~$\N$ is defined over~$\QQ$, the subspace~
\[
\Gamma\backslash \Gamma\cdot \N\simeq(\Gamma\cap \N)\backslash \N
\]
of~$\Gamma\backslash G^\der(\RR\times\QQ_S)$ is closed. We even note the following.
\begin{lem} The space~$(\Gamma\cap \N)\backslash \N$ is compact.
\end{lem}
\begin{proof} We recall that~$Z_{G^\der}(M)$ is~$\RR$-anisotropic: its group of real points is compact. Hence so is the subgroup~$N'$.
It is a fortiori~$\QQ$-anisotropic as a~$\QQ$-algebraic group. By Godement's criterion we may deduce the statement.
\end{proof}
Recall that~$f_n$ belongs to~$\N$. We may hence, in~$\N$, write~$f_n=\gamma_n\cdot \beta_n$ with~$\gamma_n$ in~$\Gamma$ for a bounded sequence~$(\beta_n)_{n\geq0}$ in~$\N$. As~$\N$ commutes with~$\Omega$, we have
\[
\Gamma\cdot\omega\cdot f_n\cdot l_n\cdot b_n
=
\Gamma\cdot\omega\cdot \gamma_n\cdot \beta_n\cdot l_n\cdot b_n
=
\Gamma\cdot \gamma_n\cdot\omega\cdot \beta_n\cdot l_n\cdot b_n
=
\Gamma\cdot \omega\cdot \beta_n\cdot l_n\cdot b_n.
\]
Introducing~$l'_n=\beta_n\cdot l_n\cdot {\beta_n}^{-1}$, we then reshuffle the inner factors
\[
\Gamma\cdot \omega\cdot \beta_n\cdot l_n\cdot b_n
=
\Gamma\cdot \omega\cdot l'_n\cdot \beta_n\cdot b_n.
\]
 As~$\N$ normalises~$\M$ the elements~$l'_n$ belong to~$\M$.
Substituting~$l_n$ with $l'_n$ and~$b_n$ with~$\beta_n\cdot b_n$ we may assume that~$f_n$ is the neutral element.

Doing so, we may lose the property that~$b_n$ is trivial at the real place. Nevertheless, as~$\beta_n$ belongs to~$\N\subseteq Z_{G^\der}(M)$, its real factor~$\beta_{n,\RR}$ belongs to~$K_h$ and is such that
\[
\pi(\mu\cdot \beta_{n,\RR})=\pi(\mu)
\]
for any probability measure on~$\Gamma\backslash G^\der(\RR\times \QQ_S)$. We may hence substitute~$\beta_n$ with its $S$-adic factor.

\section{Zariski closedness: the $S$-Mumford-Tate hypothesis.}\label{SectionMT}

We put ourselves in the situation of Seciton~\ref{SectionRZ}. We will
use the $S$-Mumford-Tate property to reach the stronger conclusion~\eqref{TheoremeMT} of Theorem~\ref{Theoreme1}, namely that we obtain actual weakly special subvarieties.

\subsection{Normalisation by rational monodromy}
\begin{prop} Assume that the algebraic monodromy subgroup~$H_{S}$ is 
definable over~$\QQ$, that is of form~$H_S=H_{\QQ_S}$ for 
some~$\QQ$-subgroup~$H$ of~$G$.

Then the subgroup of Ratner class~$L$ in~$G$ is normalised by~$H$.
\end{prop}
\begin{proof}
We assume that the sequence $(\mu_n)$ converges to a measure $\mu_{\infty}$ associated to a $\QQ$-group $L$. 
We know that 
$$
g_n = L^H \cdot (N \cap Z_G(H)) \cdot O(1)
$$
with 
$$
L^H = \cap_{h} h L h^{-1}
$$
and
$$
N \subset N_G(L).
$$
We want to show that $L=L^H$.

Without loss of generality, we may assume that $O(1)=1$.

We have that $N \cap Z_G(H)$ is defined over $\QQ$ and
$$
N\cap Z_G(H) = N \cap Z_G(H) \cap \Gamma \times \cF.
$$ 
As $Z_G(H)$ is $\QQ$-anisotropic, $N\cap Z_G(H)$ is $\QQ$-anisotropic and therefore $\cF$ is compact.

We write 
$$
g_n = l_n f_n
$$
with $f_n \in (N \cap Z_G(H))(\RR \times \QQ_S)$ and we write
$$
f_n = \gamma_n \cdot \phi_n
$$
with $\gamma_n \in Z_G(H) \cap \Gamma$ and $\phi_n \in \cF$.

As $\cF$ is compact, after extraction, we may assume that $\phi_n$ is convergent and we may assume that the limit is one.

We can therefore assume that 
$$
g_n = l_n \cdot \gamma_n
$$
and $g_n$ stabilises the closed set $\Gamma \backslash \Gamma L^H \cdot U$ (it is closed because $L^H$ is defined over $\QQ$)
and$\Gamma \backslash \Gamma L^H U$ contains the support of $\mu_n$.
Therefore $\Gamma \backslash \Gamma L^H U$ contains $Supp (\mu_{\infty})$, hence 
$$\dim L^H \geq \dim(Supp(\mu_{\infty})) = \dim(L)$$.

With the assumption of the $S$-Mumford-Tate property, we show that $L$ is normalised by $M =H$.
\end{proof}
Under the $S$-Mumford-Tate type hypothesis we immediately deduce the following.
\begin{cor} Assume moreover that~$s$ is of $S$-Mumford-Tate type, that is~$M=H$.

Then~$L(\RR)$ is normalised by~$h$.
\end{cor}

\subsection{A criterion for a weakly $S$-special real submanifold to be a weakly special subvariety.}

In this section we show that stonger conclusions \ref{TheoremeMT} and
\ref{Theoreme2-2} of the main theorems \ref{Theoreme1} abd \ref{Theoreme2} respectively hold under the assumption that the~$S$-Mumford-Tate hypothesis holds.
This follows from the following proposition.


\begin{prop} \label{boutonrouge}
Let~$(G,X)$ be a Shimura datum and~$h$ a point in a connected component~$X^+$ of~$X$.
Let~$L \subset G$ be a subgroup such that~$L_{\RR}$ is normalised by~$h(\SSS)$.

Then the image of~$L(\RR)^+ \cdot x \subseteq X^+$ in~$\Gamma\backslash X$ is a weakly special subvariety,
where~$\Gamma$ is any congruence arithmneitc subgroup of~$G$.
\end{prop}

\begin{proof}
First observe that $L_{\RR}$ is normalised by $h(\sqrt{-1})$ which induces a Cartan involution on $G^{ad}_{\RR}$. 
Therefore $L_{\RR}$ is reductive (see \cite[Th.~4.2]{1-Satake}).

Let $N = N(L)^0$ be the neutral component of the normaliser of $L$ in $G$. 
Because~$L$ is reductive, so are~$N$ and its centraliser~$Z_G(L)$, and we have
$$
N = Z_G(L) \cdot L.
$$
Note that~$h$ factors through $N_{\RR}$.

We have a almost-product decomposition of semisimple groups
$$
N^{der} = Z_G(L)^{der} \cdot L^{der}.
$$
Let~$Z^{c}$ (resp.~$Z^{nc}$) denote the almost product of the almost $\QQ$-factors
of~$Z_G(L)$ which are~$\RR$-compact (resp.\ which are not~$\RR$-compact). Using
the analogous notation for~$L$, we have 
$$
Z_G(L)^{der} = Z^{nc}\cdot Z^c \text{ and } L^{der} = L^{nc}\cdot L^{c}.
$$

By~\cite[Lemme 3.7]{1-U}, $x$ factors through $H = Z(N) Z^{nc} L^{nc}$.

Let $X_H = H(\RR)\cdot h$.
By~\cite[Lemme 3.3]{1-U}, $(H,X_H)$ is a Shimura subdatum of $(G,X)$.

Note that
$$
H^{ad} = Z^{nc,ad} \times L^{nc,ad}
$$
and 
write $h^{ad} = (h_1, h_2)$  in this decomposition.

We have
$$
X_H^{ad} = X_1 \times X_2
$$
where $X_1 =  Z^{nc,ad}(\RR)\cdot h_1$ and $X_2 =  L^{nc,ad}(\RR) \cdot h_2$.

By Lemme 3.3 of \cite{1-U}, both $(Z^{nc,ad},X_1)$ and $(L^{nc,ad},X_2)$ are Shimura data.

The image of $L(\RR)^+ \cdot h$ in $X_{H^{ad}}$ is $\{ h_1 \} \times X_2$,
as in~\cite[Section~2]{1-UY3}. This finishes the proof.
\end{proof}

For the sake of completeness, we give an alternative proof of the statement
\ref{boutonrouge}.

\begin{proof}[Alternative proof] The symmetry~$s_h$ of~$X$ at~$x$ is induced 
by the Cartan involution given by the conjugation action of~$h(i)$. But~$h(i)$
normalises~$L(\RR)$, hence normalises its neutral component~$L(\RR)^+$,
from which we get that~$L(\RR)^+\cdot h$ under~$s_h$: it is symmetric at~$h$.

Moreover~$L(\RR)^+$ is normalised by~$h(U(1))$, whose conjugation induces the
complex structure on the tangent~$T_hX$ space of~$X$ at~$h$. We deduce that the tangent space of~$L(\RR)^+\cdot h$ at~$h$ is complex subspace of~$T_hX$.

Notice that~$L$ is normalised by~$lhl^{-1}$ for every~$l\in L(\RR)^+$. By the
argument above,~$L(\RR)^+\cdot h$ is symmetric at every point and has a complex
tangent space at every point: it is a symmetric quasi-complex subspace.

By \cite{1-Helgason}, this quasi-complex structure is a complex structure:~$L(\RR)^+\cdot h$ is a symmetric holomorphic subvariety.

By a theorem of Baily-Borel~\cite{1-BailyBorel}, the arithmetic quotients
\[
\left(\Gamma\cap L(\RR)^+\right)\backslash L(\RR)^+\cdot h\text{ and }\Gamma\backslash X
\]
are quasi-projective varieties. And by a theorem of Borel~\cite{1-BorelExtension},
the embedding into~$\Gamma\backslash X$ is algebraic.

It implies that it is a totally geodesic subvariety in the sense of~\cite{1-Moonen},
that is a weakly special subvariety in our terminology.
\end{proof}

%% file: 2-Richard-Zamojski/Fusion-Richard-Zamojski.tex
\section{Introduction}
\subsection{Motivations}In several problems in number theory, one is led to the understanding of the limiting distribution of translates of certain orbits, in a space~$G/\Gamma$ of~{$S$-arithmetic} lattices, of some subgroups~$H$ of~$G$.
\begin{enumerate}
\item   For instance, in~\cite{DukRudSar93} is suggested such an approach to the study of the density of integral points on affine homogeneous varieties  under a semisimple Lie group. This approach was further pursued in \cite{EskMcm93,EMSAnn,GorMauOh08,GorOh09}.
\item Another example are some arithmetico-geometric problems involving the stu\-dy of weakly special subvarieties 
of Shimura varieties. A stu\-dy of their ergodic property was conducted in~\cite{Ullmospeciale1,Ullmospeciale2}.
\item \label{AppliGalois}	Yet another example pertains to Galois action on Hecke orbits in Shimura varieties and a conjecture of Pink~\cite{These3}.
\end{enumerate}
A. Yafaev and the first named author recently proved some 
cases of the refined version of the Andr\'{e}-Pink-Zannier conjecture (\cite{0-RY}).
The advances made here are the cornerstone to their proof.

Set aside applications to number theory, the problem of classifying and characterising these limit measures is an interesting problem in homogeneous dynamics in its own right, leading to the development of new results, notably in~\cite{Lemmanew}.

\subsection{$S$-adic setting.}
Our general setting for the problem is as follows. We let $S$ be a finite set of places, and denote~$\Q_S=\prod_{v\in S}\Q_v$ the product of the corresponding completions~$\Q_v$ of~$\Q$. Denote~$G=\G(\Q_S)$ the group of $\Q_S\text{-}$points of a semisimple $\Q$-algebraic group~$\G$, let~$\Gamma\subseteq G$  be an $S$-arithmetic lattice in $G$, and $H$ a connected reductive $\Q_S$-subgroup in $G$ (see~\S\ref{secintro}). 

This~\emph{$S$-adic} setting, rather than merely \emph{archimedean} setting, of real algebraic groups, is fundamental in applications to the study of the the Galois orbits in~\cite{0-RY}. This setting also leads to the counting of~$S$-integer points on homogeneous varieties.

\subsection{Previous work}
In~\cite{EMSAnn}, it is assumed that~$\Q_S=\R$, that the orbit~$H\Gamma/\Gamma$ in~$G/\Gamma$ supports a globally $H\text{-}$invariant probability measure $\mu_H$ and that $H=\H(\R)$ for an algebraic group~$\H$ defined over $\Q$.  One question is then
\begin{equation*}
\text{ to describe weak limits of sequences of translates $(g_n\cdot\mu_H)_{n\geq0}$, with~$g_n\in G$.}
\end{equation*} 

Although~\cite{EMSAnn} obtains an answer for any non-divergent sequences, every application to counting of integral points on affine homogeneous varieties of the form~$G/H$ works under the more stringent assumption that $Z_G(H)$, the centralizer of $H$ in $G$, is defined over~$\Q$ and is $\Q$-anisotropic. This condition is necessary to avoid the trivial issue of pushing all the mass to infinity using the elements of $Z_G(H)$. If we wish to relax the assumption on~$Z_G(H)$, in the above applications, we need to encompass this phenomena in the ergodic method. 
 
\subsection{Filiation}
The purpose of this article is to pursue the study of the limit measures of translates of orbits of reductive subgroups in $G/\Gamma$, removing some of the assumptions of \cite{EMSAnn}. We continue a tradition dating back to the Raghunathan's conjecture, who was motivated by Oppenheim conjecture (now theorems of Ratner~\cite{Rat,Ratnerp} and Margulis~\cite{MargulisOppenheim} respectively); to the linearisation method initiated by Dani and Margulis for unipotent trajectories, pursued in the work of Eskin, Mozes, Shah; together with the more recent notion of $(C,\alpha)$-good functions, and its adaptation to the~$S$-adic setting by Kleinbock and Tomanov. 
\subsection{Pieces of orbits}	
Considering the applications we have mentioned, the following cases arise 
\begin{itemize}

\item where $Z_G(H)$ is not $\Q$-anisotropic and when $H\Gamma/\Gamma$ is not finite volume (e.g. for counting points on more general affine homogeneous spaces~$G/H$)
\item where $H$ is not necessarily defined over~$\Q$ and when $H\Gamma/\Gamma$  is not even closed.\footnote{For instance, let~$A$ be an abelian variety over a field of zero characteristic,~$S$ be made of a $\ell$-adic place and the archimedean place, and~$G$ be the Mumford-Tate group of~$A$. After possibly a finite extension of the base field, the image of the Galois representation on the~$\ell$-adic Tate module sits in~$G(\Q_\ell)$. This is a compact~$\ell$-adic Lie subgroup which is an open in an $\Q_\ell$-algebraic subgroup~$H$ of~$G$. Pursuing~\cite{These3} we wish not to rely on Mumford-Tate conjecture (for a finitely generated base field), which asserts one should have~$H=G$, we are to consider general subgroups~$H$ defined over~$\Q_\ell$ and possibly not over~$\Q$.
}
\end{itemize}
 It is then advisable to focus on a subset of the orbit~$H\Gamma/\Gamma$: we let $\Omega$ be a Zariski dense open bounded subset of~$H$ with zero measure boundary, and let $\mu_\Omega$ be the push forward under $G\to G/\Gamma$ of the restriction to $\Omega$ of an Haar measure on $H$. Under a mild assumption on the sequence $(g_n)_{n\geq0}$, we describe the weak limits of trans\-la\-tes~$g_n\cdot\mu_\Omega$.

We furthermore give a sufficient and necessary condition on the sequence to converge to this limit. Following~\cite{EMSAnn},  we call these conditions ``the~\emph{focusing criterion}''. These conditions allow to describe the equidistibution process, and  are essential to the applications to Andr\'{e}-Pink-Zannier conjecture (in \cite{0-RY}). 



\subsection{Two simplified cases.}
Before making our problem precise in section~\ref{secintro}, we provide two simplified cases of our main result~Theorem~\ref{Theorem}. We hope this will help the reader to understand the general case.

\subsubsection{The first simplified case.}
 This first simplified version is as follows:
 
\addtocounter{theorem}{-1}
\begin{theorem} \label{Theointro} Let~$G$ be a real linear semisimple algebraic group defined over~$\Q$, and let~$\Gamma$ be an arithmetic lattice relative to the~$\Q$-structure on~$G$. Let~$H$ be a real algebraic connected subgroup of~$G$ defined over~$\Q$ which is reductive as an algebraic subgroup\footnote{The centre of~$H$ is semisimple in~$G$.}. Assume the centraliser~$H^\prime$ of~$H$ in~$G$ is $\Q$-anisotropic. 

Let~$\Omega$ be a non-empty open bounded subset of~$H$, let~$\mu|_\Omega$ be a probability on~$\Omega$ which is the restriction of a Haar measure of~$H$, and demote~$\mu_\Omega$ be the image probability measure on~$G/\Gamma$.

Fix any sequence~$\left(g_i\cdot\mu_\Omega\right)_{i\geq0}$ of translated probabilities on~$G/\Gamma$ and let $\mu_\infty$ be any weak limit measure. We have the following.
\begin{enumerate}
\item{\textbf{Tightness.}} The limit $\mu_\infty$ is a probability measure.
\item{\textbf{Limit probabilities.}} There are
\begin{itemize}
\item  a~$g_\infty$ in~$G$ and
\item an algebraic subgroup~$L$ of~$G$ defined over $\Q$, normalised by~$H$ and generated as a~$\Q$-algebraic group by its unipotent elements defined over~$\R$, with analytic connected component of the identity~$\Lpp$,
\end{itemize}  such that
\begin{equation}\label{limitformulasimplified}
\mu_\infty= g_\infty\int_\Omega (\omega\cdot\mu_{\Lpp})~~d\mu(\omega),
\end{equation}
where $\mu_{\Lpp}$ is the~$\Lpp$-invariant probability measure on~$\Lpp\Gamma/\Gamma$.
\item{\textbf{Focusing criterion.}} 
The sequence~$\left(g_i\right)_{i\geq0}$ is of the class
\begin{equation}\label{focusingcriterionsimplified} g_\infty\cdot o(1)\cdot L\left(\Gamma\cap N(L)H^\prime\right),\end{equation}
where~$o(1)$ denotes the class of sequences in~$G$ converging to the identity, $N(L)$ is the normaliser of~$L$ in~$G$, and $H^\prime$ is the centraliser of~$H$.
\end{enumerate}
\end{theorem}

Theorem~\ref{Theointro} is a slight generalisation of~\cite[Statement 1.13]{EMSAnn}. Our main theorem will generalise this statement  further in several directions (see~\S\ref{seccontext}). Most importantly, $H$ will no longer be required to be defined over $\Q$, nor to have a closed orbit $H\Gamma/\Gamma$. We will however need to restrict to sequences having a geometric stability property (see \S\ref{analytic stability}).
Here this hypothesis was avoided thanks to~\cite{Lemma}, which is related to~\cite{EMSGAFA} (see~{\S}\ref{RemarkcasGAFA}).
Nonetheless, this restriction is flexible enough to encompass all applications of~\cite{EMSAnn} to counting integral points on varieties, and even more.

\subsubsection{Another particular case}
As mentioned previously,
our main theorem applies to the study of Galois actions on Hecke orbits in Shimura varieties. Measures on bounded pieces of orbits of irrational groups $H$ arise in this context. To illustrate this, we give the following version of our theorem, requiring much of the power of our main Theorem~\ref{Theorem} (and also Theorem~\ref{thm21bis}). For simplicity we consider here only one ultrametric place. The reader might adapt at wish to finitely many ultrametric places; the arithmetically inclined might consider left cosets of~$\Gamma$, and reverse accordingly the order from left to right actions of~$G$.

\begin{theorem}\label{Theointro2} Let~$G$ be a semi-simple linear algebraic group over~$\Q$. Fix a prime~$p$ and a compact subgroup~$K$ of~$G(\Q_p)$ Let~$\Gamma$ be an~$S$-arithmetic lattice\footnote{With respect to the  couple~$S:=\left\{\Q\to\R;\Q\to\Q_p\right\}$ of places. See~\cite[{\S}(3.1.2)]{Margulis} or~{\S}\ref{secmuomega} for the definition of a $S$-arithmetic lattice we use.} of~$G(\R)\times G(\Q_p)$. Consider the Haar probability measure~$\mu$ on~$K$ and define~$\mu_K$ to be its direct image in the quotient space~$\left.G(\R)\times G(\Q_p)\middle/\Gamma\right.$.

We make the following assumptions on~$K$:
\begin{enumerate}
\item[(H1.)] \label{H1intro} that the centraliser of~$K$ in~$G(\Q_p)$ is the centre of~$G(\Q_p$);
\item[(H2.)] \label{H2intro} that~$K$ is reductive in~$G(\Q_p)$ (equivalently: its Lie algebra is semisimple);
\item[(H3.)] \label{H3intro} that the Zariski closure of~$K$ over~$\Q_p$ is Zariski connected.
\end{enumerate}

Then the following  holds.
\begin{enumerate}
\item The family of translated probabilities
\begin{equation}\left(g\cdot \mu_K\right)_{g\in G(\Q_p)}\end{equation}
 is tight.
\item Each limit point of the family~$\left(g\cdot \mu_K\right)_{g\in G(\Q_p)}$ is a translate of some measure of the form
\begin{equation}\label{theointropadiclimitformula}
\mu_{K{\star} L} = \int_{\kappa\in K} \kappa\cdot \mu_{\Lpp}~d\mu(\kappa).
\end{equation}
\begin{itemize}
\item where~$L$ is a subgroup of~$G$ defined over~$\Q$ without~$\R\times\Q_p$-anisotropic $\Q$-rational factor, generated over~$\Q$ by its unipotent elements defined over~$\R\times\Q_p$,  (of \emph{Ratner class} in~{\S}\ref{secnotations})
\item where~$\Lpp$ is an explicit (defined in~{\S}\ref{secnotations}) open subgroup of finite index in~$L(\R\times\Q_p)$,
\item and where~$\mu_{\Lpp}$ is the~$\Lpp$-invariant probability on~$\Lpp\Gamma/\Gamma$.
\end{itemize}
\item The measure~$\eqref{theointropadiclimitformula}$ can occur as a limit point for a subgroup~$L$ distinct from~$G$ and~$\{e\}$ if and only if there exists an unbounded algebraic subgroup of~$G(\Q_p)$ which is normalised by~$K$ and not Zariski dense over~$\Q$ in~$G$.
\item Assume that~$(g_i\cdot \mu_\rho)$ converges to a translate of~$\mu_{K\convolution G}$ for every sequence~$(g_i)_{i\geq 0}$ in~$G(\Q_p)$ without bounded infinite subsequence. Then the Zariski closure over~$\Q$ of~$K$ in~$G$ contains the~$\Q_p$-isotropic factors of~$G$.
\end{enumerate}
\end{theorem}
The first point is a special case of~\cite{LemmaA}. The second point is a particular case of the limit formula of our result Theorem~\ref{Theorem}.  The third point is derived from the focusing criterion of Theorem~\ref{Theorem}, and the last point is a consequence of the third.
These last two points serve as a criterion for
``indiscriminate'' equidistribution, for all reasonable sequences: there is no focusing phenomenon.

\subsection{Remarks}
We mention an interesting question, which goes reciprocally from the subject of~\cite{These3}.
To what extent, assuming equidistribution properties, namely knowing the limit measures, 
can lead the determination of~$K$, which in the applications from~\cite{These3} would be 
the Mumford-Tate conjecture?

Applications discussed above shall be the subject of future work.
 
The proof follows the strategy of~\cite{EMSAnn}, thus relies on Ratner's measure classification theorem and on Dani and Margulis' linearisation method. However, a significant difference in our proof is that the geometric stability properties from~\cite{Lemmanew} provides a major simplification to the linearisation method (see~{\S}\ref{subsection-proof-structure}). It is responsible for the stronger conclusions obtained in our main Theorem~\ref{Theorem}.



\subsection{Acknowledgements}
This work could not exist without earlier work of (notably) Dani, Eskin, Kleinbock, Margulis, Moz{e}s, Ratner, Shah, Tomanov.

Present work stems from a collaboration of the first named author visiting Pr~Shah at ICTP (Mumbai) in January 2005 (producing \cite{Lemma}), helped by the  \emph{\'{E}cole normale sup\'{e}rieure de Paris}. Tools were developed in the thesis~\cite{Lemma,Lemmap,LemmaA} at IRMAR (Rennes). These were deepened with a visit of Pr Shah at Irchel Universit\"{a}t (Z\"{u}rich) in June 2011 and to him at OSU (Columbus OH) in September 2011. At \'{E}PFL (Lausanne) in the academic year 2011-2012, this gave the preprint~\cite{Lemmanew}, which updates~\cite{Lemma} with the cornerstone result
allowing the present developments, a collaboration with T. Zamojski. Writing occurred at \'{E}PFL, Plouisy (France), ETHZ, and Mahina (Tahiti), 
and were lately supported by Universiteit Leiden and University College London.

The second named author was supported by the European Research Council Grant 228304.

Both authors first met at a summer school in Pisa organised by the Clay institute about homogeneous dynamics.


\section{Setting.} \label{secintro}
Our main result, Theorem~\ref{Theorem} in~\S\ref{sec ennonce thm}, is formulated in the ``$S$-arithmetic setting'', which we introduce in~\ref{Ssetting} below.
This Theorem involves translates of quite general probability measures~\(\mu_\Omega\), the alluded ``pieces of orbits'', which we describe in~\ref{secmuomega}. For full generality, we rely on a technical hypothesis detailed in~\ref{analytic stability}. We then emphasize a particular case
where this hypothesis is void, a particular case which encompass the works of ~\cite{EMSAnn, EMSGAFA}. Lastly we review some of Ratner theory and introduce some notations, especially the measures~$\mu_{\Lpp}$ our limit measures will be made of.
\subsection{The~$S$-arithmetic setting.} \label{Ssetting}
 This setting is analogous to \cite{Borel-Prasad}, \cite{MargulisTomanov}, \cite{TomanovOrbits}, \cite{KT}, \cite{LemmaA}.

\subsubsection{}
Recall the folowing notation. Let~$S$ be a finite set of places~$v\colon\Q\to\Q_v$. We denote:
\begin{equation}\label{notationQS}
\Q_S:=\prod_{v\in S}\Q_v.
\end{equation}
Let~$\G$ be a semisimple algebraic group over~$\Q$. We write~$G$ for the topological group~$\G(\Q_S)$, which we will identify with~$\prod_{v\in S}\G(\Q_v)$. 

\subsubsection{}
The \emph{Zariski topology on~$G$}, resp. on its subsets, will mean the product topology of the Zariski topologies induced on the~$\G(\Q_v)$, resp. the induced topology. The product of Zariski closed subsets at each place is a generating family of closed subsets. These products are also the vanishing locus in~$G$ of sets of regular functions on~$\G$ with~$\Q_S$ coefficients. If~$S$ and~$G$ are not singletons, not every closed subsets must be given by a such set of equations (N.B.:~$\Q_S$ is not an integral domain). 

We will apply the notion of Zariski connected subsets to possibly non algebraic subsets.

\subsubsection{}
We fix, once for all, an \emph{$S$-arithmetic subgroup}~$\Gamma$ in~$\G(\Q_S)$, in the sense, for simplicity, of~\cite[{\S}(3.1.2)]{Margulis} (compare~\cite{TomanovOrbits}, cf.~\cite[\S\,1,\S\,6]{Borelihes}). Namely,~$\Gamma$ is commensurable with, equivalently: the group of~$S$-integer points with respect to a integer model of~$\G$; the stabiliser of a $S$-integral lattice in a faithful representation of~$\G\to\mathbf{GL}(N)$.

As~$G$ is semisimple, the discrete subgroup~$\Gamma$ is automatically a lattice, by the $S$-arithmetic Borel Harish-Chandra criterion, found in~\cite[{\S}(3.2.1)]{Margulis}.

\paragraph{On General lattice subgroups.} There is a more general notion of~$S$-arithmetic lattices (see~\cite{TitsMargulis}), and moreover semi-simple groups can feature interesting non arithmetic lattice.
We limit ourselves to a restricted notion of a lattice for commodity: we rely technically on quite a number of references, of which we do not know counterparts for general lattices. This might hamper applications
to Shimura varieties, which may be associated with general arithmetic lattices (cf.~\cite{These3}.) We expect nevertheless our 
method to adapt for general lattices without serious trouble, and analogous statements for general arithmetic lattices to be deduced form our case\footnote{In general let~$\Gamma$ be commensurable with the image in~$G(\Q_S)$, by  map~$\phi:\G'\to\G$ over~$\Q_S$ with~$\Q_S$-anisotropic kernel, of an arithmetic subgroup~$\Gamma'$ of~$\G'(\Q_S)$. From our results, applied in~$\G'(\Q_S)/\Gamma''$ with~$\phi^{-1}(\Gamma)\cap\Gamma'$, one shall derive, the desired analogous results in~$\G(\Q_S)/\Gamma$: convergence of measures in the former implies convergence in the latter by direct image.}. Our method deals mostly with Dani-Marulis' linearisation, which has already been applied
to more general lattices, and rely on Ratner's theorems and its variants in comparable generalities. Nonetheless, the exigencies of rigour led us
at technical points to require the use of explicit structure properties of the groups~$L$ of Ratner class (see~Appendix~\ref{AppRatner}).

\subsection{The Reductive subgroup~$H$}
We consider a subgroup~$H$ of~$G$ such that
\begin{equation}
\label{H reductif dans}
\text{``$H$ is a Zariski connected~\emph{$\Q_S$-Lie subgroup} which is \emph{reductive in~$G$}'',}
\end{equation}
in the sense of~\cite{Lemma,Lemmap}.
By this we mean~$H$ can be written as~\(\prod_{v\in S}H_v\)
where 
\begin{itemize}
\item if~$\Q_v\ciso\R$ is archimedean,~$H_v$ is a connected real Lie\footnote{This is likely that the whole argumentation in this article still works if more generally~$H_\R$ an ``integral subgroup'' in the sense of~\cite[III \S6.2 Def.\,1]{BBKLie3}, say a pathwise connected subgroup, by a Theorem of Yamabe~\cite{Yamabe}.} subgroup of~$\G(\R)$, the Lie algebra of which is acting, by adjoint representation, completely reducibly on the Lie algebra of~$\G(\R)$;
\item if~$\Q_v$ is ultrametric,~$H_v$ a Zariski connected Lie subgroup, in the $v$-adic sense, of~$\G(\Q_v)$ such that its action on~$\lie{g}(\Q_S)$ is completely reducible.
\end{itemize}
In this definition, the complete reducibility assumption has to do with the notion of \emph{reduced} (sic) subgroup in~\cite[p.\,101]{PR} and of a subgroup \emph{strongly reductive in} $G$ of~\cite[\S16]{Richardsontuple} (defined for Reductive~$G$ only, but in general characteristic). One way to rephrase it here is that the image of~$H_v$ in~$\G^\ad_{\Q_v}$ has a reductive $\Q_v$-Zariski closure. In our situation,~$G$ is semisimple, and this complete reducibility assumption amounts to
\begin{equation}\label{H reductif dans bis}
\text{the $\Q_v$-Zariski closure of~$H_v$ in~$\G_{\Q_v}$, is reductive.}
\end{equation}
\subsection{Sequences of probabilities and Problem}\label{secmuomega} Given this setting, we consider 
\[\text{a non-empty open bounded\footnote{``Bounded'' will throughout be used for~``relatively compact''.} subset~$\Omega$ of~$H$,}\]
which is necessarily~$\Q_S$-Zariski dense in~$H$.

From assumption~\eqref{H reductif dans bis}, we note that~$H$ is a unimodular group (cf.\footnote{For non algebraic~$H$, here are more details: a bilateral Haar measure is obtained from considering, place by place, any non-zero top differential form on~$H_v$, at the neutral element, which, we shall see, are invariant under conjugation. The top differential forms on~$H_v$ at the origin describe a~$\Q_v$ vector line~$\det(\lie{h}_v)$ in~$\bigwedge^{\dim(H_v)}\lie{g}$. This line is stable under the exterior adjoint action of~$H_v$, and thus under its $\Q_v$-Zariski closure, say~$\H^\alg_{\Q_v}$ of~$H_v$ in~$\G_{\Q_v}$ this action is trivial: the semisimple part of~$H_v$ admits no non constant character, and the centre of~$\H^alg_v$ commutes with~$H_v$. The action of~$H_v\leq\H^\alg_v(\Q_v)$ is hence trivial.}~\cite[p.15, Exemple de modules a)]{vigneras1996representations}, \cite[VIII~\S2, Corollary 8.31 (d)]{Knapp}): its Haar measures are both left and right Haar measures. Let~$\mu$ be the Haar measure on~$H$ normalised so that~$\mu(\Omega)=1$. 
Here is our main object of study: we denote
\begin{equation}\label{def mu Omega}
\text{$\mu\restriction_\Omega$ the restriction of $\mu$ to $\Omega$ and$\mu_\Omega$ its image probability measure on~$G/\Gamma$.}
\end{equation}



We are interested in the asymptotic behaviour of a sequence~$\left(g_i\cdot\mu_\Omega\right)_{i\geq0}$ of probabilities on~$G/\Gamma$ made of translates of~$\mu_\Omega$ by elements~$g_i$ in~$G$. We ask:
\begin{enumerate}
\item   what are the limit measures (in particular are they also probabilities);
\item   how to characterise conveniently the sequences converging to a given limit measure.
\end{enumerate}
We will completely answer these questions provided the translating elements~\(g_i\) are constrained to any subset~\(Y\) of~$G$ satisfying the stability property in \S~\ref{analytic stability} below.

\subsubsection{Measures with densities and Full orbits.} Let us mention a possible variation of our setting. As in~\cite{LemmaA}, we could consider more general measures of the form~$\mu_f=f\cdot\mu_\Omega$ where~$f\geq0$ is of class~$L^1(\mu)$ and such that~$\int f\cdot\mu_\Omega=1$. We will not
delve here into such a generalisation. It could be useful nonetheless, for instance in order to rigorously relate, on the one hand our results about the measure~$\mu_\Omega$ on a piece of orbit, and on the other analogous statements about probability measures on a full orbit, rather than a piece (these orbits may not be bounded, and not be of finite volume). We refer to~\cite{LemmaA} for such considerations. (See also~\cite{Lemma} for a variant.) Another possible use of this smoother generality may be to more easily prove some uniformity results in our main statements, with respect to~$\Omega$. We did not include such results here, as it would add quite some lengths, with arguments mostly unrelated to the one involved here.

%


\subsection{\emph{Analytic stability} hypothesis}\label{analytic stability}
Here is our main technical hypothesis, which, though cumbersome is a cornerstone of this whole work.
\subsubsection{}
The \emph{analytic stability} property on a subset~$Y$ of~$G$, studied in~\cite{Lemma,Lemmanew,Lemmap}, is defined as follows:
\begin{equation}\tag{An.S.}\label{AnS}
	\parbox{.88\textwidth}
	{\indent For any $\Q_S$-linear representation~$\rho:G\to GL(V)$, for any norm~$\Nm{-}$ on~$V$,\\
	$\exists c>0,~\forall y\in Y,~\forall v\in V,~\sup_{\omega\in\Omega}\Nm{y \cdot \omega\cdot v }\geq \Nm{v}/c.$}
\end{equation}

 Note that the property can be made independent of~$\Omega$, as long as~$\Omega$ is bounded and Zariski dense in~$H$.
The hypothesis under which our proof works is that
\begin{equation}\tag{H}\label{Hypo}
 \text{ the set~$Y:=\{g_i~|~i\geq0\}$ satisfies~\eqref{AnS}.}
\end{equation}
\subsubsection{}
We explain why such hypothesis is reasonable. On the one hand, it implies (\cite[Theorem~1.3]{LemmaA}) that the sequence~$\left(g_i\cdot \mu_\Omega\right)_{i\geq0}$ is tight (in the sense of~\cite[IX \S5.3 and \S5.5]{BouINT} for instance). One the other hand, this hypothesis is not far from being equivalent to tightness. More precisely, thanks to Mahler's criterion and \cite{KT}, one can closely relate tightness of the sequence~$\left(g_i\cdot \mu_\Omega\right)_{i\geq0}$ to the following~\emph{(uniform) Arithmetic (semi)stability}\footnote{
Came lately to the notice of the authors the work of Grayson~\cite{Grayson1,Grayson2} where he produces related subsets of a symmetric space~$G/K$
defined by a related (more precise) property of \emph{arithmetic stability} (as coined by Stuhler~\cite{Stuhler}), in relation to reduction theory for arithmetic groups.
} statement.
\begin{equation}\tag{Ar.S.}\label{ArS}
\parbox{.85\textwidth}
	{
	For any $\Q$-linear representation~$\rho:G\to GL(V)$, for any norm~$\Nm{-}$
on~$V\tens\Q_S$, for any $\Q$-rational $S$-arithmetic lattice~$\Lambda$ in~$V\tens\Q_S$,\\
$\qquad\exists C>0,~\forall y\in Y,~\forall v\in \Lambda\smallsetminus\{0\}, \max_{\omega\in\Omega}\Nm{y \cdot \omega\cdot v }\geq C$.
	}
\end{equation}
One deduce~\eqref{ArS} from~\eqref{AnS} as follows. We note that~$\Lambda$ is discrete in~$V\tens\Q_S$. Thus its systole~$\inf \{\Nm{\lambda}~|~\lambda\in\Lambda\smallsetminus{0}\}$ is bounded below. Let~$\sigma>0$ be a bound. Then~\eqref{AnS} for~$V\tens\Q_S$ implies~\eqref{ArS} with~$C=\sigma/c$.

Last but not least, references~\cite{Lemma,Lemmap,LemmaA,Lemmanew} show that hypothesis~\eqref{AnS} can be achieved for explicit and quite large subsets~$Y$ (see~Theorem~\ref{thm21bis}) which contain ``suppplements'' in~$G$, in some sense, to the centraliser~$Z_G(H)$ of~$H$: in particular their~$Z_G(H)$-saturated is~$G$.

\subsection{An Important special case}\label{RemarkcasGAFA}
One can circumvent the hypothesis~\eqref{Hypo} in the following special case which is important considering the applications\footnote{Notably~\cite{EMSAnn}, and its wanted generalisations; but also Galois actions on Hecke orbits in Shimura varieties, see~\cite{These3}.}.
\begin{equation}\tag{R.nD.}\label{RnD}
\text{The centraliser~$Z_G(H)$ of~$H$ in~$G$ is defined over~$\Q$ and is~$\Q$-anisotropic.}
\end{equation}
Here, by~``$\Q$-anisotropic" we mean for instance satisfying the Godement compacity criterion~\cite[Th\'{e}or\`{e}me~8.7\footnote{Note that here,~$Z_G(H)$ is necessarily reductive, and has trivial unipotent radical.}]{BorelIntro}. On top of the rationality property for~$Z_G(H)$, the anisotropy property translates into
a non divergence property for any sequence of translates~$(g_i\cdot \mu_\Omega)_{i\geq0}$ (see~\cite{LemmaA} following~\cite{EMSGAFA}), namely that~$(g\cdot \mu_\Omega)_{g\in G}$ is a tight family.

In the case~\eqref{RnD}, the hypothesis~\eqref{Hypo} can always be achieved by modifying the sequence of translators~$\left(g_i\right)_{i\geq0}$ by right multiplying by a sequence of  elements of~$Z_G(H)\cap\Gamma$, such modification leaving the sequence of probabilities~$\left(g_i\cdot\mu_\Omega\right)_{i\geq0}$ unchanged. 

This special case is enough to encompass the setting needed to apply~\cite{EMSGAFA} (and its~$S$-arithmetic generalisation~\cite{LemmaA}.)

The hypothesis~\eqref{RnD} is notably satisfied in the case where
\begin{equation}\label{centraliserless}
\text{$Z_G(H)$ is the centre of~$G$.}
\end{equation}

\subsection{Notations from Ratner theory}\label{secnotations}

A ``one parameter unipotent subgroup'' of~$G$ will mean a subgroup of the form~$U=\U(\Q_S)$ where~$\U$ is an algebraic subgroup of~$\G$ defined over~$\Q_S$ of dimension\footnote{The dimension for~$\Q_S$-algebraic groups is naturally a function on~$\mathbf{Spec}(\Q_S)$. The latter can be identified with~$S$. This is a constant function for an algebraic group which comes from~$\Q$. The group~$H$ is not assumed of constant dimension, as well as groups depending on it. It may furthermore contain an non-algebraic factor at the real place, at which place one considers the real analytic dimension.} at most one at every place in~$S$. 

For a subgroup~$L$ of~$G$ 
\begin{equation}\label{Notation Lp}
\parbox{.9\textwidth}
	{we denote~$L^+$ the subgroup generated by the one parameter unipotent algebraic subgroups of~$G$ contained in~$L$.}
\end{equation}
 (cf.~\cite[{\S}1.5 and Theorem~2.3.1]{Margulis}, \cite[{\S}6]{BorelTits} for interesting properties of such subgroups. See also~\cite{GilleKneserTits} for more actual questions. We recall some properties in Appendix~\ref{AppRatner}.) For a unipotent algebraic group~$\U$, the group~$U=\U(\Q_S)$ is covered by its one parameter subgroups~$\exp(\Q_S\cdot X)$ for~$X$ in the~$\Q_S$ Lie algebra of~$\U$ (cf. Lemma~\ref{Unipotent vs 1param}). We may drop~``one parameter'' in the definition~\eqref{Notation Lp} above.
 
We call\footnote{This terminology of Ratner class and type is non canonical, but idiosyncratic to here.
} the \emph{Ratner class}, the set~$\RatQ$ of 
\begin{equation}\label{defRatclass}
\text{subgroups~$\L$ of~$\G$ over~$\Q$ such that~$L^+$ is $\Q$-Zariski dense in~$L$.}
\end{equation}
This is the class~$\mathscr{F}$ of~\cite{TomanovOrbits}.
 We review this class of groups in the section~\ref{sectionB1} of the appendix. We will allow ourselves to say that~$L$ is in~$\RatQ$ if~$L=\L(\Q_S)$ with~$\L$ in~$\RatQ$.

We define a \emph{Ratner type} as a $\Gamma$-conjugacy class of groups of Ratner class. We denote~$[L]$
the Ratner type to which belongs a group~$L$ of Ratner class. This notion is pertinent in Ratner decomposition theorem, as in~\ref{ss section Ratner decomposition}.

  For any~$L$ in~$\RatQ$, we introduce in this work the notation~$\Lpp$ for the topological closure
\begin{subequations}
\begin{equation}\label{defi Lpp}
\Lpp=\overline{L^+\cdot(\Gamma\cap L(\Q_S))},\end{equation}
which is an open subgroup of finite index in~$L(\Q_S)$ (see~Proposition~\ref{++finiteindex}),
and let~
\begin{equation}\label{mu Lpp}
\text{$\mu_{\Lpp}$ be the~$\Lpp$-invariant probability measure on~$G/\Gamma$ supported on~$\Lpp\Gamma/\Gamma$.}
\end{equation}
\end{subequations}
Note that, though~$L^{+}$ does not depend on~$\Gamma$, the group~$\Lpp$ does. The probability~$\mu_{\Lpp}$ exists and is unique provided~$L$ belongs to~$\RatQ$. The support~$\Lpp\Gamma/\Gamma$ of~$\mu_{\Lpp}$ is also the orbit closure~$\overline{L^{+}\Gamma/\Gamma}$. 

It follows from Ratner's theorem that the unique~$L^+$-invariant probability measure on~$\Lpp\Gamma/\Gamma$ (resp. on~$L\Gamma/\Gamma$) is~$\mu_\Lpp$ (resp. and its finitely many translates~$l\mu_\Lpp$ by some element~$l$ on~$L$). It seems to be the only fact we use from Ratner theory, which rarely appears explicitly in this work.

\subsection{Notations from focusing}One of the important feature of our result, admittedly obscure and easily underestimated, is the focusing criterion, which is formulated in terms of the following groups.
Once a group~$L$ of Ratner class is fixed, we consider its normaliser~$N(L)$ in~$\G$, and the action of~$N(L)$ on the Lie algebra of~$L$.  We write
\begin{equation}\label{def GammaN}
\GammaN=\Gamma\cap N(L)
\end{equation}
We will call \emph{unitary normaliser} of~$L$ in~$G$ the subgroup~$\N^1(L)$ of~$\N(L)$ which is the kernel of the determinant of this action. We will write~$N$ for~$N^1(L)=\N^1(L)(\Q_S)$. We will also be working with the intermediate group~$N\leq N\GammaN\leq N(L)$, which stabilises each Haar measure on~$L$ (see Lemma~\ref{Lemma unimodular}).

We will use the abbreviations
\begin{subequations}\label{defLF}
\begin{align}
\M&:=\bigcap_{h\in H} h\cdot \L_{\Q_S}\cdot h^{-1}
	&M&:=\bigcap_{h\in H} h\cdot L\cdot h^{-1}=\M(\Q_S),
\\
M^\ddagger&:=\bigcap_{x\in HN} x\cdot \Lpp\cdot x^{-1}
&~\text{ and }~F&:=\bigcap_{h\in H} h\cdot N\cdot h^{-1}.
\end{align}
\end{subequations}
The subgroup\footnote{We used the~${}^\dagger$ to evoke that~$\Mpp$  depends on~$\Lpp$; but this is not the notation from~\eqref{defi Lpp} reserved to subgroup of Ratner class, to which may or may not belong~$M$, which may even not be definable over~$\Q$. Still,~$\Mpp$ being open of finite index in~$M$ it does contain~$M^+$. Though the notation~$\Mpp$ not consistent with~\eqref{defi Lpp}, no confusion shall arise as we will only apply~\eqref{defi Lpp} the later to a same group, designated thoughout by~$L$, and the notation~$\Mpp$ will be reserved for the associated group in~\eqref{defLF}.}~$\Mpp$ of~$M$ will turn out to be open of finite index (Lemma~\ref{Lemme M finite index}). 
The following easy observations might be useful while reading our main statement. Their proof is left to the reader.
\begin{lemma}\label{lemdebut}
\begin{enumerate}
\item The groups~$H$, and~$N(L)$, and~$F\leq N\leq N(L)$ normalise~$M$. \label{lem1}
\item We have the identity~$N\cap Z_G(H)=F\cap Z_G(H)$, where~\(Z_G(H)\) denotes the centraliser of~\(H\) in~\(G\). \label{lem2}
\end{enumerate}
\end{lemma}
\noindent We also note that~$F$ is a normal subgroup of~$HF$, and that~$M$ and~$\Mpp$ are a normal
subgroup of~$F$ and of~$HF$. In its latest stages, our proof will introduce the subquotients~$\widehat{G}=HF/\Mpp$ and~$HF/M$ of~$G$.

\subsection{Some notations about asymptotics of sequences in groups}\label{notations O}
For convenience, let us denote~$e$ the neutral element of~$G$, and
 \begin{subequations}
\begin{equation}
\text{let $O(1)$ denote the class of bounded sequences in~$G$,}
\end{equation}
\begin{equation}
\text{and $o(1)$ be the class of sequences converging to~$e$.}
\end{equation}
For a subset~$A$ of~$G$, we will say
\begin{equation}
\text{a sequence~$(g_i)_{i\geq0}$ is \emph{of class~$O(1)A$}, resp.~is \emph{of class~$o(1)A$},}
\end{equation}
\end{subequations}
 as a shorthand to mean
that: there is a sequence~$(b_i)_{i\geq0}$ in~$G$ which is bounded, resp. converging to~$e$, and a sequence~$(a_i)_{i\geq0}$ in~$A$ such that for all~$i\geq0$ one has~$g_i=b_i\cdot a_i$.
If~$A$ is a closed subgroup, this is equivalent, for the sequence of cosets~$(g_iA)_{i\geq0}$ to be bounded, resp. converging to the neutral coset~$A$, in~$G/A$.
\section{Statements}\label{secthm}
\subsection{Main Theorem with comments}\label{sec ennonce thm}
Below is our main theorem. We use the notation and setting of section~{\S}\ref{secintro}, and make some direct reference to these notations inside the statement. Recall that~$o(1)$ stands for the class of sequences in~$G$ converging to the identity~$e$.

We summarise our setting. Let~$\G$ be a semisimple $\Q$-algebraic group, $S$ a finite set of places, $G=\G(\Q_S)$ (see~\eqref{notationQS}), $\Gamma$ an $S$-arithmetic lattice, $H$ a $\Q_S$-Lie subgroup whose Zariski closure in~$G$ is a Zariski connected and reductive in~$G$ (cf.~\eqref{H reductif dans}) and $\Omega$ a Zariski dense open bounded subset of~$H$. Consider the probability on~$\Omega$ which is the restriction of a Haar measure~$\mu$ on~$H$, and let~$\mu_\Omega$ be its direct image on~$G/\Gamma$ (as in~{\S}\ref{secmuomega}). Write~$Z_G(H)$ for the centraliser of~$H$ in~$G$.

\begin{theorem}[Local focusing of translated measures]\label{THEOREM}\label{Theorem} In the setting above, consider a sequence~$\left(g_i\cdot\mu_\Omega\right)_{i\geq0}$ of translates. Assume the sequence~$\left(g_i\right)_{i\geq0}$ in~$G$ satisfies the hypothesis~\eqref{Hypo} p.\pageref{Hypo}.
\begin{enumerate}
\item[]{\textbf{Tightness:}} Then the sequence~$\left(g_i\cdot\mu_\Omega\right)_{i\geq0}$ is \emph{tight}: it is relatively compact in the space of probabilities on~$G$; equivalently any weakly converging subsequence is tightly converging; equivalently any weak limit of a subsequence is a probability.
\end{enumerate}
Assume moreover that~$\left(g_i\cdot\mu_\Omega\right)_{i\geq0}$ converges weakly to non zero limit~$\mu_\infty$.
\begin{subequations}
\begin{enumerate}
\setcounter{enumi}{-1}
\item{\textbf{Ratner decomposition:}}\label{0 Ratner decomposition}
(We use the notations of \S~\ref{ss section Ratner decomposition}). Let~$W=\Stab(\mu_\infty)^+$. Then the $W$-ergodic components of~$\mu_\infty$ arise form a unique Ratner type~$[L_0]$, that is we have
\begin{equation}\label{equation thm Ratner decompo}
\mu_\infty={\mu_\infty}^{[L_0]},
\end{equation}
or equivalently: every $W$-ergodic component of~$\mu_\infty$ is a translate~$c\mu_{{L_0}^\ddagger}$ of~$\mu_{{L_0}^\ddagger}$.
\end{enumerate}
In addition, we can decompose~$\left(g_i\right)_{i\geq0}$ into finitely many subsequences such that, if we substi\-tu\-te~$\left(g_i\right)_{i\geq0}$ with anyone of these subsequences, then there exists a group~$L$ in the Ratner type~$[L_0]$, and, writing~$N,F,{\Lpp},M$ as in~{\S}\ref{secnotations}  (see~\eqref{defLF}) the following holds.
\begin{enumerate} 

\item{\textbf{Limit measure formula:}} \label{limit formula statement}
There exist~$g_\infty$ in~$G$ and~$n_\infty$ in~$N$ such that the limit probability measure~$\mu_\infty$ can be written as the convolution
\begin{equation}\label{limitformula}
\mu_\infty = g_\infty\cdot\int_{\omega\in\Omega} \left(\omega\cdot n_\infty\cdot\mu_{\Lpp}\right)~~\mu(\omega),
\end{equation}
and we may choose~$n_\infty$ is in the topological closure of~${\left(Z_G(H)\cap F\right)\cdot(\Gamma\cap N)}$ in~$N$.
\item{\textbf{Focusing criterion:}} 
\begin{enumerate}
\item \label{thm Fa}
The sequence~$(g_i)_{i\geq 0}$ is of class
\begin{equation}\label{focusingcriterion}
O(1)\cdot (Z_G(H)\cap F)\cdot M.
\end{equation}
\item \label{thm Fa'}
Furthermore we may find a corresponding factorisation~$g_i=b_i\cdot z_i\cdot m_i$, with a bounded sequence~$(b_i)_{i\geq0}$ in~$G$,
and sequences~$(z_i)_{i\geq0}$ and~$(m_i)_{i\geq0}$ in~$Z_G(H)\cap F$ and~$\Mpp$,
such that 
\begin{equation*}
\text{the factor~$(z_i)_{i\geq0}$ is of class~$O(1)\GammaN\Lpp$.}
\end{equation*}
\item \label{thm Fb}
 For such factorisation and for any subsequence such that~$(g_i,z_i\GammaN\Lpp)_{i\geq0}$ converges in~$G\times (N\GammaN/\GammaN\Lpp)$, then, for some~$(g_\infty,n_\infty)$ in~$G\times N$, the limit is of the form~$(g_\infty,n_\infty\GammaN\Lpp)$ and formula~\eqref{limitformula} holds for such~$g_\infty$ and~$n_\infty$. 
\end{enumerate}
\end{enumerate}
\end{subequations}
\end{theorem}
\subsubsection{}
The formula~\eqref{limitformula} is closely related to the desintegration formula of~$\mu_\infty$ into ergodic components (cf.~\cite[\S6.1 Th.~6.2]{EinsiedlerGTM}). The ergodic components of~$\mu_\infty$ will be proved to be the integrand~$g_\infty\omega n_\infty\mu_\Lpp$ (for~$\omega$ in~$\Omega$ outside a negligible subset). To be fair, it is a loose ergodic decomposition as, in the given form, the same ergodic component may possibly occur for several elements~$\omega$ in~$\Omega$. Knowing that formula~\eqref{limitformula} is an ergodic desintegration formula implies retrospectively the statement~\ref{0 Ratner decomposition} about Ratner decomposition, and furthermore that~$[L_0]=[L]$.

\subsubsection{}
We can reformulate~\eqref{limitformula} in terms of test functions as follows. For any continuous function~$f:G/\Gamma\to\R$ with compact support (a posteriori, boundedness of~$f$ will suffice), its integral against~$\mu_\infty$ can be computed via the double integral
$$\int_{x\in G/\Gamma} f(x)~~\mu_\infty(x)=\int_{\omega\in\Omega} \left(\int_{y\in G/\Gamma} f\left({g_\infty}\cdot\omega\cdot n_\infty\cdot y\right)~~\mu_{\Lpp}(y)\right)~~\mu(\omega).
$$
The formula~\eqref{limitformula} can also be abbreviated as
\[
\mu_\infty=g_\infty\cdot\mu\restriction_\Omega\convolution n_\infty\cdot \mu_{\Lpp}.
\] in terms of the convolution product~$\convolution$ which is as follows.
The action of~$G$ on~$G/\Gamma$ induces an convolution action~$\convolution$ of the space~$\Prob(G)$ of probabilities on~$G$ on the one~$\Prob(G/\Gamma)$ of~$G/\Gamma$, which is such that~$\alpha\convolution\beta$ depends continuously and affinely upon~$\alpha$ (and upon~$\beta$) and such that for an element~$g$ of~$G$, one has~$\delta_g\convolution\beta=g\cdot \beta$.

Note that on the right hand side, the integral is against $\mu|_\Omega$ (in~\(G\)) and \emph{not} against~$\mu_\Omega$ (in~\(G/\Gamma\)). The latter would not make sense because~$L$ is not necessarily normalised by~$\Omega$.
 Let us note that in the context of~\cite{EMSAnn}, the fact that~$L$ is normalised by~$\Omega$ is a key technical fact in the proof; this is the subject of \cite{EMSCorr}. Our more general context and our method do not rely on this.

\subsubsection{}
We also note as a consequence of formula~\eqref{limitformula} that~$\mu_\infty$ is a probability: it has mass one; there is no escape of mass at infinity. Thus formula~\eqref{limitformula} implies the ``Tightness'' statement. Actually this follows from hypothesis~\eqref{Hypo} as proved in the generalisation~\cite{LemmaA} of~\cite{EMSGAFA}. The hypothesis that~$\mu_\infty$ is non zero is even superfluous. Thus  our statement  includes the main results of~\cite{LemmaA}, and of~\cite{EMSGAFA} (see~{\S}\ref{RemarkcasGAFA}).


\subsubsection{}
We insist that in~\eqref{limitformula}, neither~$g_\infty$, nor~$n_\infty$, nor even~$L$ are in general determined by~$\mu_\infty$. For example, for any $n\in H'\cap N$, one can replace the pair $(g_\infty,n_\infty)$ by $(g_\infty n,n^{-1}n_\infty)$. Nevertheless, the Ratner type~$[L_0]$ can be defined canonically in terms of~$\mu_\infty$, although implicitly (cf. the begining of the proof, notably footnote~\ref{anticipating} at p.\pageref{anticipating}). 


\subsubsection{} Let us discuss some logical relations between the various assertions contained in the statement. We already discussed that knowing that the formula~\eqref{limitformula} is an $W$-ergodic decomposition encompasses statement~\eqref{0 Ratner decomposition}.
Assertion~\eqref{thm Fb} explicitly states that itself contains the limit formula statement~\eqref{limit formula statement}. The statement~\eqref{thm Fa'} is actually a consequence of~\eqref{thm Fa}, once we know that the limit~$\mu_\infty$ is non zero. If one knows, which is true, that\footnote{Actually, it suffices to know that the limit measure~$\mu_\infty$ in~$G/\GammaN$ is composed of translates of~$\mu_\Lpp$, and that knowledge is already contained in the analogue, which is true too, of statement~\eqref{0 Ratner decomposition} in~$G/\GammaN$.}  the statement~\eqref{limit formula statement} can actually be written in~$G/\GammaN$ rather than merely~$G/\Gamma$, then we can deduce~\eqref{thm Fb} from~\eqref{thm Fa}. That is indeed how our proof proceeds.

\subsubsection{} The technical assumption~\eqref{Hypo} of ``analytic stability'' is crucial to us to single out some (finitely many) subgroups~$L$ in the Ratner type~$[L_0]$, and thus to give sense to the Focusing criterion. 
Without the hypothesis~\eqref{Hypo}, any hypothetical statement would be very messy.



\subsection{Context}\label{seccontext}
We compare our statement to \cite{EMSAnn,EMSGAFA}. 
Theorem~1.7 and Corollary~1.13 in \cite{EMSAnn} concern themselves with limits of translates of a globally $H$-invariant probability measure supported on $H\Gamma/\Gamma$, in the case $\Q_S=\R$, such that~$H$ is defined over~$\Q$ and under the non-divergence hypothesis, i.e. the limit is a probability measure. The latter is intimately related to the arithmetic stability~\eqref{ArS}.

Every application to the counting of integral points on varieties met in~\cite{EMSAnn} and the results in~\cite{EMSGAFA} assumes furthermore that $H'$ is defined over $\Q$ and is $\Q$-anisotropic, named~\eqref{RnD} in \S2.4. As noted there, we can assume the hypothesis~\eqref{Hypo} in that case, and our statement generalises previous results: to the $S$-arithmetic setting; to a piece of measure $\mu_\Omega$ on $H\Gamma/\Gamma$; to a non-necessarily $\Q$-defined $H$ and finally to a non-necessarily closed orbit $H\Gamma/\Gamma$ (in order to recover statements about full closed orbits form statement from a piece of orbit we refer to the technique in~\cite{LemmaA}).

However, without~\eqref{RnD}, the relation is not that simple. The hypothesis~\eqref{Hypo} is not vacuous. It implies non-divergence by \cite{LemmaA}, but not the other way around (consider a divergent sequence~$(g_i)_{i\geq0}$ in~$H^\prime\cap \Gamma$). A side effect of assuming the stronger analytic stability is that our focusing criterion is also stronger than the corresponding~\cite[Corollary~1.13]{EMSAnn}. Nonetheless, there are examples where~\eqref{Hypo} can not be satisfied, yet convergence to a probability measure occurs. We recall that our method does not require the additional assumptions of rationality of~$H^\prime$ and even closedness of the $H$-orbit where~$\mu_\Omega$ is supported.

We note also that~\cite[Theorem 1.7]{EMSAnn} concludes that $\mu_\infty$ is homogeneous, whereas our classification seems more complicated: $\mu_{\infty}$ is described as an integral of homogeneous measures. However, when $H$ is defined over $\Q$, so is the group~$\bigcap_{\omega\in\Omega} \omega\cdot L\cdot \omega^{-1}$, and one should actually replace~$L$ with the $\Q$-Zariski closure of~$\bigcap_{\omega\in\Omega} \omega\cdot L^+\cdot\omega^{-1}$. The latter is normalised by~$\Omega$, hence by $H$, and we retrieve~\cite[Remark~4.1]{EMSAnn}.

As for the proof of Theorem~\ref{Theorem}, it is similar to~\cite{EMSAnn}, with one fundamental difference: the lack of rationality led us to introduce analytic stability to take further advantage of the linearisation techniques.

\subsection{Complements to the Main Theorem.}
\subsubsection{Rational case.} Let us discuss some simplifications which can be made when~$H$ is rational over~$\Q$. This case should be relevant to many applications.

\begin{proposition} In the context of Theorem~\ref{Theorem}, assume that the Zariski closure of~$H$ is definable over~$\Q$.
Then we may assume that~$L$ is normalised by~$H$, and hence so is~$N$: we have~$L=M$ and~$F=N$. The algebraic groups~$L$, and hence~$F$,~$N$ and $F$ are rational over~$\Q$. 
\end{proposition}
To be precise, we derive a similar focusing criterion and a limit formula with respect to a group~$L'$ such as in the proposition.
\begin{proof} Recall that~$L$ is definable over~$\Q$. Hence so are~$N$ and~$N(L)$.
Let~$H^\alg$ be denote the Zariski closure of~$H$ in~$G$ assumed to be definable over~$\Q$. Then~$Z_G(H)=Z_{H^\alg}(G)$ is definable over~$\Q$. Because both~$H$ and~$L$ are definable over~$\Q$, the groups~$M$ and~$F$, by construction, will be definable over~$\Q$.

Let us consider~$L'$ to be the Zariski closure over~$\Q$ of~$M^+$. By construction this subgroup invariant under any $\Q$-rational algebraic automorphism of~$M^+$. In particular under the conjugation action of~$H(\Q)$ and~$F(\Q)$. By weak approximation, the Zariski neutral component of~$H$, which is~$H$ by assumption, and~$F^0$ of~$F$ normalise~$L'$. 

Let~$M'$, $N'$ and~$F'$  denote the groups associated with~$L'$ (and~$H$) in the same manner that~$M$, $N$, $F$ are derived from~$L$.
As~$H$ normalises~$L'$, we have~$M'=L'$ and~$F'=N'$. 
We will prove~$L'=L$, which will achieve our goal.

\paragraph{The focusing criterion} We are in the context of~Theorem~\ref{Theorem}, and use the focusing criterion with respect to~$L$, that is that~$(g_i)_{i\geq0}$ is of class
\[
O(1)(Z_G(H)\cap F)M.
\]
We write~$g_i=b_i z_i m_i$, for~$i\geq0$, a respective factorisation of the~$g_i$, with~$(b_i)_{i\geq 0}$ a bounded sequence in~$G$, with~$z_i$ in~$(Z_G(H)\cap F)$ and~$m_i$ in~$M$.

Our first step will be to deduce the analogue criterion with respect to~$L'$, in the form
\begin{equation}\label{focusing L'}
\text{$(g_i)_{i\geq0}$ is of class~$O(1)(Z_G(H)\cap N(L'))M'$.}
\end{equation}

We noted that~$F^0$ normalises~$L'$, hence~$F^0\leq N(L')$. The group~$((Z_G(H)\cap F)M)^0\leq (Z_G(H)^0\cap F^0)M^0$ is of finite index in~$(Z_G(H)\cap F)M$, hence~$(g_i)_{i\geq0}$ is of class
\[
O(1)(Z_G(H)\cap F^0)M^0.
\]

We work with the sequence of cosets
\[
g_i L=b_i z_i m_iL\in G/L.
\]
The group~$H$ is reductive, definable over~$\Q$, and acts by automorphisms on~$M$, which is definable over~$\Q$. It follows that there is a~$H$-invariant (semi-direct) Levi decomposition~$M=RU$ definable over~$\Q$, with~$R$ an~$H$-invariant maximal reductive subgroup of~$M$ and~$U$ the unipotent radical of~$M$ (\cite[\S7 Cor. to Th.~7.1]{Mostow56}). Again we can find a~$H$ invariant Levi decomposition~$L'=R'U$ of~$L'$. Possibly replacing~$R$ by~$R'K$, we may assume~$R'\leq R$. The group~$R$ is reductive and has a decomposition, definable over~$\Q$, as~$CKI$ where~$C$ is its centre, where~$KI$ is the semisimple part, with~$K$ (resp.~$I$) generated by the~$\Q_S$-anisotropic (resp. $\Q_S$-isotropic)~$\Q$-quasi factors of~$KI$. This decomposition is invariant under~$\Q_S$-algebraic automorphisms, and is in particular~$H$-invariant. As~$IU$ is generated, over~$\Q$, by~$\Q_S$-rational unipotents, we have~$IU\leq L'$. (This is actually an equality.)

Let us decompose
\[
m_iL=c_i k_i L. 
\]
We have
\[
g_i L=b_i z_i m_iL= b_i z_i c_i k_i L
\]
Let us prove the claim that we have an inclusion
\begin{equation}
C\leq Z_G(H)\cap F^0
\end{equation}
\begin{proof} We have~$C\leq M\leq L$, and as~$L$ is unimodular,~$L\leq N$. Finally~$C\leq N$.

As the Zariski connected group~$H$ normalises~$C$ and the~$\Q_S$-algebraic automorphism group of the torus~$C$ is disconnected, we get that~$H$ centralises~$C$, that is~$C\leq Z_G(H)$.

We proved
\(
C\leq N\cap Z_G(H)=F\cap Z_G(H).
\)
\end{proof}
Substituting~$z_i$ with~$z_ic_i$ we may assume the~$c_i$ are all the neutral element. 

The subgroup~$KL\leq M$ is invariant under algebraic automorphisms: it is the intersection of the kernels of the characters definable over~$\overline{\Q}$. We recall that~$Z_G(H)\cap F\leq F$ normalises~$M$. It follows, in which we write~$[g]=gL'$ the image in the group~$(Z_G(H)\cap F)L'/L'$ of an element~$g$ in~$(Z_G(H)\cap F)L'$,
\[
z_i k_iL'=[z_i][k_i]=[z_i] [k_i] [z_i^{-1}][z_i]= [k'_i][z_i]=(k'_i L')(z_iL')=k'_i z_iL'
\]
with~$k'_i$ in~$K$. The sequence of the~$k'_i$ is hence bounded.
In
\[
g_i L'= b_i k'_i z_i L'
\]
substituting~$b_i$ with~$b_i k'_i$ we may assume that the~$k'_i$ are the neutral element. Finally
\[
g_i L'=b_i z_i L'.
\]
which implies the focusing criterion~\eqref{focusing L'} with respect to~$L'$.
\paragraph{The limit measure} We now jump to the latter stages on the main proof and indicate some adaptations to our situation. We can assume the factor~$b_i$ of~$g_i$ is trivial. Then~$\Omega$ and the~$g_i$, by the focusing criterion for~$L'$,  belongs to~$HFL'$, which is a group as~$H$ normalises~$F$ and both normalise~$L'=M'$; this is even a subgroup of the~$\Q$-group~$N(L')$.

The proof of the main theorem establishes a limit
\begin{equation}\label{cvg twist}
\lim\left(g_i\widetilde{\mu}_\Omega\right)_{i\geq0}=\widetilde{\mu}_\infty
\end{equation}
in~$G/\GammaN$. The convergence~\eqref{cvg twist} actually occurs in the closed subset~$N(L')/(N(L')\cap\GammaN)$. 
\newcommand{\GammaNp}{\Gamma_{N(L')}}
Let~$\Lambda=\GammaN\cap N(L')$, which normalises~${L'}^{\ddagger}$. We thus have convergence of the image measures, say
\[
\lim\left(g_i\widetilde{\mu}'_\Omega\right)_{i\geq0}=\widetilde{\mu}'_\infty
\]
in~$N(L')/\Lambda$.

We know that~$\mu_\infty$ is the direct image of~$\widehat{\mu}_\infty$, which is a quotient measure of a right~$\Lpp$-invariant measure~${\widetilde{\mu}_\infty}^\sharp$ on~$G$. This~${\widetilde{\mu}_\infty}^\sharp$ is a fortiori a right~${L'}^{\ddagger}$ invariant measure. As~$N(L')$ contains~$L'$, the restriction of~${\widetilde{\mu}_\infty}^\sharp$ to~$N(L')$ is also right~${L'}^{\ddagger}$ invariant, whose quotient by a suitably normalised Haar measure on~$\Lambda$ gives~$\widetilde{\mu}_\infty$. It follows that~$\widetilde{\mu}_\infty$ is determined by its image in~$G/{\Lambda\L'}^{\ddagger}$, which is contained in~$N(L')/{\Lambda\L'}^{\ddagger}$. Let
\begin{equation}\label{cvg twist}
\lim\left(g_i\widehat{\mu}'_\Omega\right)_{i\geq0}=\widehat{\mu}'_\infty
\end{equation}
the convergence of images measures in the quotient group~$N(L')/{\Lambda\L'}^{\ddagger}$. The element~$g_i$ acts through the coset~$g_i{\L'}^{\ddagger}=z_i\cdot{\L'}^{\ddagger}$, which reduces to the action of elements~$z_i$ centralising~$\Omega$.
The situation is dynamically trivial and we can finish the proof as for the main Theorem to deduce~$\widehat{\mu}'_\infty$ and then~$\widetilde{\mu}_\infty$ and~$\mu_\infty$. We obtain a formula as in~\eqref{limitformula} but with respect to~$L'$.
\end{proof}

\subsubsection{Addenda from the proof: the lifted convergence.} 
The way our proof proceeds (see {\S}\ref{subsection-proof-structure}) allows us to state finer properties on the way~$(g_i\mu_\Omega)_{i\geq0}$
does converge to~$\mu_\infty$. Actually, writing
\[
\GammaN=\Gamma\cap N(L),
\]
all convergence (and the ergodic decomposition) occur at the higher~``level''~$G/\GammaN$.

Moreover in this statement we give a precise meaning to the property that equidistribution of
homogeneous sets is ``inner'': from the inside to the outside. In the typical case, one expects the 
support of measures in an equidistributing sequence to be eventually contained in the support
of the limit probability.  In the general case, one has to consider we may also be translating
by a convergent sequence: the support of the measures in an equidistributing sequence are
contained in the support of a closer and closer translate of the limit probability. Here, roughly speaking, 
the equidistribution phenomenon actually occurs inside right~$\Lpp$-orbits. Passing to 
the right quotient by~$\Lpp$, we are left with little dynamic: that of a bounded sequence of
translators.

\begin{proposition}[Addenda to~Theorem~\ref{Theorem}] Let us consider the situation of~\eqref{limit formula statement} in the Theorem~\ref{Theorem}, once~$(g_i)_{i\geq0}$ is  substituted with one of the finitely many considered subsequences. We can conclude moreover, 
denoting~$\widetilde{\mu}_\Omega$ the direct image of~$\mu|_\Omega$ into~$G/\GammaN$, the sequence of translated measures~$g_i\cdot\widetilde{\mu}_\Omega$ has limit
\begin{equation}\label{limitformulatilde1}
\widetilde{\mu}_\infty := g_\infty\cdot\int_{\omega\in\Omega} \left(\omega\cdot n_\infty\cdot\widetilde{\mu}_{\Lpp}\right)~~\mu(\omega),
\end{equation}
in the space of probabilities on~$G/\GammaN$, where~$\widetilde{\mu}_{\Lpp}$ is the left (and right)~$\Lpp$-invariant probability on~$G/\Gamma$ supported on~$\Lpp\GammaN/\GammaN$.


\end{proposition}
We won't prove it here. It will be a by-product of our proof of Theorem~\ref{Theorem}. Let us just note that, by direct image along~$\varphi:G/\GammaN\to G/\Gamma$, this implies the limit formula~\ref{limitformula} in statement~\eqref{limit formula statement} of Theorem~\ref{Theorem}.

\section{Proof of Theorem~\ref{Theorem} -- Introduction}\label{secproof}
We will now turn to the proof of~Theorem~\ref{Theorem}, after introducing some notations, mostly from the Dani-Margulis linearisation method of~\cite[\S\,{\S}2-3]{DM} and Ratner theory.

\subsection{Singular sets}\label{RatnerNotations}
Recall from~\ref{secnotations} that a \emph{Ratner type} was defined to be the~$\Gamma$-conju\-ga\-cy class~$[L]$ of a subgroup~$L$ in the Ratner class~$\RatQ$. Let~$W$ be a subgroup of~$G$ such that~$W^+=W$. Dani and Margulis define the subset
\begin{subequations}
\begin{equation}\label{X1}
X(L,W) = \Cj(W,L)^{-1} = \left\{g\in G\middle| g^{-1}Wg \subseteq L \right\}\subseteq G
\end{equation}
of inverse of elements in~$G$ that conjugate~$W$ into~$L$.
These~$X(L,W)$ are closed subsets, actually real algebraic subvarieties, of~$G$. We have the right invariance property
\begin{equation}\label{right invariance N(L)}
X(L,W)=X(L,W)N(L)
\end{equation}
under the normaliser~$N(L)$ of~$L$. For a fixed~$W$, this kind of subset is stable by intersection; more precisely we have
\begin{equation}\label{intersection X}
X(L_1,W)\cap X(L_2,W)= X(L_3,W)
\end{equation}
where~$L_3$ is the Zariski closure over~$\Q$ of~$(L_1\cap L_2)^+$.
\end{subequations}
\begin{subequations}
If~$[L]$ is the Ratner type of~$L$, we denote
\begin{equation}\label{X2}
X([L],W) = X(L,W)\Gamma/\Gamma \subseteq G/\Gamma
\end{equation}
the \emph{singular locus of type~$[L]$ for the action of~$W$ on~$G/\Gamma$}. This locus depends indeed on~$L$ only via~$[L]$, because, for~$\gamma$ in~$\Gamma$
\begin{equation}\label{gammaX} X(L^\gamma ,W)=X(L ,W)\cdot\gamma.\end{equation}
 Denote
\end{subequations}
\begin{subequations}
\begin{equation}\label{X3}
X^*(L,W) = X(L,W) \smallsetminus \bigcup_{K\in \RatQ,\, \dim(K)<\dim(L)} X\left(K,W\right).
\end{equation}
Note that the latter union is right-invariant under~$\Gamma$\footnote{\label{rightgammainv}	The set~$\RatQ$  is invariant under the action of~$\Gamma$ by conjugation. Conjugation preserve the dimension. We conclude by~\eqref{gammaX}. This statement and in~\eqref{gammaX} still hold with~$\G(\Q)$ instead of~$\Gamma$.}. Together with~\ref{right invariance N(L)} we get, with~$\GammaN=\Gamma\cap N(L)$, the right-invariance
\begin{equation}\label{right invariance GammaN}
X^*(L,W) =X^*(L,W) \cdot \GammaN
\end{equation}
\end{subequations}
We then denote
\begin{equation}\label{X4}
X^*([L],W)=X^*(L,W)\Gamma/\Gamma=X([L],W) \smallsetminus \bigcup_{K\in \RatQ,\,\dim(K)<\dim(L)} X\left(\left[K\right],W\right)
\end{equation} 
the \emph{exclusive} singular locus of type~$[L]$. Using~\eqref{intersection X}, we could as well have indexed the union over the~$K$ in~$\RatQ$ that do not contain~$L$, or again over the~$K$ which are strictly contained in~$L$.

\subsubsection{}
Here are two facts we will use elsewhere.
\begin{subequations}
\begin{equation}\label{Borel}
\text{the~$X^*(L,W)$ and the~$X^*([L],W)$ are Borel subsets of~$G$ and~$G/\Gamma$ resp.}
\end{equation}
\begin{proof} The subsets~$X(L,W)$ are closed subsets by construction. As~$\RatQ$ is countable,~$X^*(L,W)$ is the complement of countably many closed subsets in~$X(L,W)$, hence is a Borel subset of~$G$. As~$\Gamma$ is countable,~$X^*(L,W)\Gamma$ is a Borel subset of~$G$. At this point, as~$X^*(L,W)\Gamma$ is the inverse image of~$X^*([L],W)$
we may invoke~\cite[IX \S6.6 Cor.\,2]{TOP}. Alternatively we may proceed as follows. The map~$G/\Gamma$ is a local homeomorphism, and~$X^*([L],W)$ is locally, in~$G/\Gamma$ homeomorphic to~$X^*(L,W)$. As~$G/\Gamma$ is locally compact and countable at infinity, we may get a countable cover of~$X^*([L],W)$ by Borel subsets.
\end{proof}
\begin{proof} Here is an alternate proof. We first note that~$G$ can be embedded as a closed subset of some~$\R^N$. Hence is complete for the induced metric. Similarly~$G/\Gamma$ is also separable (consider for instance~$G(\Q)$ and weak approximation). Hence~$G$ and~$G/\Gamma$ are a Polish spaces (see also~\cite[IX\,\S6.1 Cor. to Prop.~2]{TOP}). The closed subset~$X(L,W)$ and is Polish too. As~$X^*(L,W)$ is a countable union of open subsets in~$X(L,W)$, it is Polish (\cite[IX\,\S6.1 Th.~1]{TOP}).
Its continuous image~$X^*([L],W)$ is then a Suslin space. 
\end{proof}
It follows, by~\cite[IX\S3.3 Cor. to Prop.\,2]{BouINT} (we refer to~\cite[IX\S3.3, in part. D\'{e}f.\,2]{BouINT} about Radon spaces) that
\begin{equation}\label{Radon}
\text{$X(L,W)$ and~$X^*([L],W)$  are Radon subspaces.}
\end{equation}
Note that~$X^*(L,W)$ need not be locally compact.
\end{subequations}
One can show that, for any~$g$ in~$G$, one has~$g\in X^*(L,W)$ if and only if~$L$ is the~$\Q$-Zariski closure of~$g^{-1}Wg$ in~$G$.
For a fixed~$W$, the~$X^*(L,W)$ induce a countable partition of~$G$, globally stable under the action of~$\Gamma$, and the~$X^*([L],W)$ induce a countable partition of~$G/\Gamma$ (confer~\cite[Proposition~2.1]{DM}). 

\subsubsection{Ratner decomposition}\label{ss section Ratner decomposition}
 For any measure~$\mu$ on~$G/\Gamma$,
\begin{subequations}
\begin{equation}\label{notationrestriction}
\text{we denote }{\mu}^{[L]}\text{ the restriction~}\mu|_{X^*([L],W)},
\end{equation}
so that we have the countable decomposition, indexed by Ratner types,
\begin{equation}\label{Ratnerdecomposition}
\mu=\sum_{[L]}{\mu}^{[L]}.
\end{equation}
This is a particularly meaningful decomposition if~$\mu$ is~$W$-invariant. In such a case,
by Ratner classification theorem (and some precision from~\cite{TomanovOrbits} in the~$S$-a\-rith\-me\-tic case), the measure~${\mu_\infty}^{[L]}$ is made of the $W$-ergodic components of~$\mu_\infty$ which are translates of~$\mu_{\Lpp}$.
\end{subequations}

\subsubsection{}\label{ss section 412}
Here we give a dynamical interpretation of these sets. We refer to Appendix~\ref{sectionB1} for the definition of~$\Lpp$ and the associated measure~$\mu_\Lpp$.

\begin{itemize}
\item By construction, the set~$X(L,W)$ is the set of elements~$g$ in~$G$ for which the translated measure~$g\cdot\mu_{\Lpp}$ is~$W$-invariant.

\item Then~$X([L],W)$ is the set of points in~$G/\Gamma$ belonging to the support of some~$W$-invariant translate of~$\mu_{\Lpp}$.  
\item
The no self-intersection statement~\cite[Proposition~3.3]{DM}, in the form of Proposition~\ref{No intersection}, says the elements of~$X^*([L],W)$ are the points in~$G/\Gamma$ belonging to the support of \emph{exactly one}~$W$-invariant translate of~$\mu_{\Lpp}$. 
\item
Finally, one knows that~$X^*([L],W)$ is also the set of points~$x$ in~$G/\Gamma$ such that the topological closure~$\overline{W\cdot x}$ of the~$W$-orbit at~$x$ is of the form~$g\Lpp\Gamma/\Gamma$ for some~$g$ in~$G$.
\end{itemize}
This last point relies on Ratner's theorem on orbit's closure (\cite{Ratnerp}, used in a context to which applies~\cite[Theorem~2]{MargulisTomanov}),
and also uses~\cite[Theorems~1 and~2, {cf.} comments between Th.~2 and~3]{TomanovOrbits}). We refer to Appendix~\ref{sectionB1}.

\subsection{Structure of the proof}\label{subsection-proof-structure} The main strategy is classical in the context of linearisation methods.
The key new inputs are the two stability results in linearised dynamics from~\cite{Lemmanew}, at the expense of the hypothesis~(\ref{Hypo}).
\footnote{The first named author advocates a rapprochement between Homogenous dynamics, especially linearisation, and Arithmetic stability, in particular in Arakelov geometry context (\cite{Seshadri,Burnol,ACLT,Zhang,Bost}).}
This new input allows us to complete the linearisation approach, concluding with the detailed conclusion from our main Theorem.

The proof will begin, until~\eqref{eq0}, with introducing a suitable~$W$ and an application of Dani-Margulis linearisation method (Appendix~\ref{AppA}), as used in~\cite{EMSGAFA} Prop.~3.13. Then a simplification will occur as, from~\eqref{eq1} to~\eqref{eq5}, we will apply Theorem~1 of~\cite{Lemma,Lemmanew}. It will allow us to apply Theorem~2 of~\cite{Lemmanew}, and when interpreted in the linearisation setting, in~\eqref{keyvariant},~\eqref{keypoint}, and~\eqref{purity}, this shows that the limiting measure~$\mu_\infty$ is made of a unique Ratner type of~$W$-ergodic components.

This work also allows the next step, which is to lift the measures in~\eqref{lift} to~$G/\Gamma_N$,
based on ``no self intersection'' properties of singular sets from the work of Dani and Margulis.

We finally initiate an induction process in~\eqref{induction}. We prove some \emph{right} $\Lpp$-invariance, which allows us to work on a quotient. By maximality of~$W$, we will be able to conclude.

As a guide through our induction, we assemble in a diagram the arrows through which our proof will proceed.
\[ 
\begin{tikzcd}[row sep=large, column sep=small]
G\arrow{dr}\arrow{rr}	&   				&	G/\Gamma_N\arrow{dl}\arrow{dr}	&\\
								& G/\Gamma	&											  		& 
								G/(\Lpp\Gamma_N) & X\arrow[hookrightarrow,swap]{l}\arrow[loop right]{}
															&\widehat{G}
\end{tikzcd}
\]
We will end up working inside~$X=\overline{HN/(\Lpp\Gamma_N)}$ on which~$\widehat{G}=HF/M$ acts 
(see~\eqref{Ghat} and~\eqref{defX}). At this step, we will have killed all the dynamical features of the sequence
of measures we have been working with, and we identify its limit.

A treatment of standard linearisation statements, adapted to our setting has been gathered in Appendix~\ref{AppA}.
We encourage the reader to focus on the main part of the proof, in~{\S}\ref{mainproof} below, and first consider a simpler setting, as for instance the one in~\ref{Theointro}, where most of the postponed technicalities can be ignored (and for example where Appendix~\ref{AppA} reduces to results well known to experts for quite some time now).

\section{Proof of Theorem~\ref{Theorem} -- Uniqueness of Ratner type}\label{mainproof}
The setting is that of~Theorem~\ref{Theorem}.
Note that the tightness statement is settled\footnote{This relies on~\cite{KT} as well as~\cite{Lemma,Lemmap}.} by~\cite{LemmaA}. In this section we will start proving~Theorem~\ref{Theorem} by establishing the statement about the uniqueness of Ratner type in the Ratner decomposition.
\begin{proof}[Proof of Theorem~\ref{Theorem}]
 
We write~$\Stab_G(\mu_\infty)$ for the stabiliser of~$\mu_\infty$ in~$G$, and, using notation~\eqref{Notation Lp} define
\begin{equation}\label{def W}
W=\left(\Stab_G(\mu_\infty)\right)^+.
\end{equation}
The group~$W$ is normal in~$\Stab_G(\mu_\infty)$.

\subsection{Ratner decomposition.}
By assumption, the weak limit~$\mu_\infty$ is non zero. From Ratner decomposition~\eqref{Ratnerdecomposition}, there must be at least one group~$L_0$ in the Ratner class~$\RatQ$ such that
\begin{equation*}{\mu_\infty}^{[L_0]}\neq 0.\label{Rdec}\end{equation*}
Let us choose
\footnote{ \label{anticipating}Anticipating on the proof, we will see that the Ratner decomposition will actually involve a unique non zero term (a unique Ratner type). 
Anticipating furthermore, we will actually find the group~$L$ fulfilling the statement among the conjugates of~$L_0$ by an element of~$\Gamma$.}
a group~$L_0$ in~$\RatQ$ of minimal dimension such that~${\mu_\infty}^{[L_0]}\neq 0$.
The minimality of~$\dim(L_0)$ gives the identity (see notation~\eqref{X4})
\begin{equation}\label{eqmini}
{\mu_\infty}|_{X([L_0],W)}={\mu_\infty}|_{X^*([L_0],W)}.
\end{equation}

\subsection{Linearisation method and focusing}\label{subsection linearisation}
We now apply the linearisation method, and more precisely \cite[Prop.~3.13]{EMSAnn} in the form of Proposition~\ref{AProp313}. 

We first introduce the setting of~\cite[Section~3, paragraph~3]{DM}, adapted to our $S$-arithmetic setting.
We write~$\lie{g}$ and~$\lie{l}_0$ the Lie algebras over~$\Q_S$ of~$G$ and~$L_0$ respectively, and~$\lie{l}_{0,\Q}$ the Lie algebra over~$\Q$ of~$L_0$.
\begin{itemize}
\item  Let~$V=\bigwedge^{\dim(L_0)}\g$, as a $\Q_S$-linear representation of~$G$, an exterior power of the adjoint representation.

\item We pick a generator~$p_{L_0}$ of the~$\Q$-line~$\det(\l_{0,\Q}):=\bigwedge^{\dim(L_0)}\l_{0,\Q}$ viewed as an element of the free $\Q_S$-submodule~$\det(\l_0):=\bigwedge^{\dim(L_0)}\l_0$ of rank one in~$V$.
\item The $\Q_S$-sub\-mo\-du\-le generated by~$X(L_0,W)\cdot p_{L_0}$ is denoted
\begin{equation}\label{def A}
A_{L_0}:=\left< X(L_0,W)\cdot p_{L_0}\right>.
\end{equation}

\end{itemize} 
Let~$\w$ denote the Lie algebra of~$W$ and define the sub\-module
\[ V_W := \left\{v\in V~\middle|~\forall\, w\in\w,\ v\wedge w={0}\text{ inside of }{\bigwedge}^\bullet\g \right\}.\]
We observe (cf.~\cite[Prop.~3.2]{DM}) that~$X(L_0,W)\cdot p_{L_0}$ is the intersection of~$V_W$ with the orbit~$G\cdot p_{L_0}$. Note that~$G\cdot p_{L_0}$ depends on~$L_0$ only up to conjugation. 
Consequently, both~$G\cdot p_{L_0}\cap V_W$ and the linear space~$A_{L_0}$ it generates depend on~$L_0$ only up to conjugation.

By~Proposition~\ref{AProp313} (or~\cite[Prop. 3.13]{EMSAnn} in the archimedean case), there exists a compact subset~$D$ of~$A_{L_0}$, and a sequence~$(\gamma_i)_{i\geq 0}$ in~$\Gamma$, such that, for any neighbourhood~$\Phi$ of~$D$ in~$V$, one has\footnote{The quantifier~$\forall i \gg 0$ being the usual substitute for the pair of quantifiers~$\exists i_0,\forall i > i_0$.}
\begin{equation} \label{eq0}
\forall i \gg 0,~g_i\cdot\Omega\cdot\gamma_i\cdot p_{L_0}\subset \Phi.
\end{equation}

In other words, the sequence~$\left( g_i\cdot \Omega\cdot \gamma_i\cdot p_{L_0}\right)_{i\geq 0}$ of subsets of~$V$ is uniformly boun\-ded, and any limit subset, for the Hausdorff topology, is contained in~$D$.

By Lemma~\ref{Gorbit discrete} the orbit~$\Gamma\cdot p_{L_0}$ is discrete inside~$V$. Choose, for each place~$s$ in~$S$, a~$s$-adic norm~$\Nm{-}_s:V\tens_{\Q_S}\Q_s\to\R_{\geq0}$, and introduce the product norm~$\Nm{-}:V\to\R_{\geq0}, (x_s)\mapsto\max_{s\in S} \Nm{x_s}_s$ 
(confer~\cite[(2.6), {\S}8.1]{KT} for such a setting.) We note that this norm is a proper map. Hence the previous discreteness property translates into
\begin{equation}\label{eq1}
\lim_{x\in\Gamma\cdot p_{L_0}} \Nm{x} =+\infty,
\end{equation}

\subsubsection{Linearised nondivergence and consequences} For any subset~$A$ of~$V$, we abbreviate~$\Nm{A}=\sup_{a\in A}\Nm{a}$.
We now use hypothesis~\eqref{Hypo} (see~\cite[Theorem~2.1 under the form~Theorem~\ref{thm21bis}]{LemmaA}) about the~$g_i$.  In particular
\begin{equation}\label{eq2}
\exists c>0, \forall \gamma \in G, \forall i\geq 0,~\Nm{g_i\Omega\gamma p_{L_0}}\geq c\cdot \Nm{\gamma\cdot p_{L_0}}.
\end{equation}
Combining~\eqref{eq1} with~\eqref{eq2} results in~
$$\lim_{x\in\Gamma\cdot p_{L_0}}\inf_i \Nm{g_i\Omega\gamma x} \geq c\cdot \lim_{x\in\Gamma\cdot p_{L_0}} \Nm{x} =+\infty$$
As a consequence, for any bound~$\Lambda<+\infty$, the set 
\begin{equation}\label{eq3}
\left\{\gamma\cdot p_{L_0}\in\Gamma p_{L_0}
~\middle|~\exists i, \Nm{g_i\cdot \Omega\gamma\cdot p_{L_0}}\leq\Lambda\right\}
\end{equation}
is finite. If~$\Lambda>\Nm{D}$, there exists a neighbourhood~$\Phi$ of~$D$ such that~$\Nm{\Phi}\leq\Lambda$. Fix such a~$\Phi$. Applying~\eqref{eq0} to this~$\Phi$, we get that, for~$i\gg0$, the element~$\gamma_i p_{L_0}$ belongs to the finite set~\eqref{eq3}. Consequently, the set~$\left\{\gamma_ip_{L_0}~\middle|~i\geq0\right\}$ described by the sequence~$(\gamma_ip_{L_0})_{i\geq0}$ implied in~\eqref{eq0} is finite.


%


Let us decompose~$(g_i)_{i\geq0}$ into finitely many subsequences according to the value of~$\gamma_i\cdot p_{L_0}$. Passing to anyone of these subsequences, we can assume~$(\gamma_i\cdot p_{L_0})_{i\geq0}$ is constant. We set~$L={\gamma_0}L_0{\gamma_0}^{-1}$. It belongs to the Ratner type of~$L_0$, that is~$[L]=[L_0]$, and is such that~$p_L=\gamma_0\cdot p_{L_0}=\gamma_i\cdot p_{L_0}$ for all~$i\geq 0$. We will show that this group~$L$ is the one involved in the conclusions of the theorem.

We have deduced the following stronger form of~\eqref{eq0}. For any neighbourhood~$\Phi$  of~$D$,
\begin{equation}\label{eq5}
 \forall i\gg0, g_i\cdot\Omega\cdot p_L \subseteq \Phi.
\end{equation}
We can choose such a~$\Phi$ to be bounded. The sequence~$\left( \Nm{g_i\cdot \Omega\cdot p_L}\right)_{i\geq0}$ is then uniformly bounded (and eventually bounded by~$\Nm{\Phi}$).

\subsection{Linearised focusing and consequences}\label{section-lin-focusing} Recall the definitions~\eqref{defLF}.
The stabiliser in~$G$ of~$p_L$ is~$N$, and the fixator\footnote{By~\emph{fixator} we understand the point-wise stabiliser, in general distinct from the stabiliser (as a whole).} of~$\Omega\cdot p_L$ is~$F$.
Applying \cite[Theorem~2]{Lemmanew}, we deduce the following important ingredient of our proof
$$\text{the sequence~}\left( g_i\right)_{i\geq 0}\text{ is of class }O(1)\cdot F.$$

Without loss of generality, translating by a bounded sequence, we may assume that~$g_i$ belongs to~$F$.

By definition~$g_i\in F$ implies
\begin{equation}\label{stabilitybygi}
\omega N=g_i\cdot \omega\cdot N=\omega\cdot(\omega^{-1} {g_i}\omega)\cdot N.
\end{equation}
Hence that each conjugate~$\omega^{-1}{g_i}\omega$ belongs to~$N$.
This holds for any~$\omega$ in~$\Omega$; hence for any~$\omega$ in the $\Q_S$-Zariski closure of~$\Omega$; in particular for any~$\omega$ in the topological closure~$\overline{\Omega}$.
It follows~$g_i\cdot \overline{\Omega}\cdot N=\overline{\Omega}\cdot N$, and~$g_i\overline{\Omega}\cdot p_L=\overline{\Omega}\cdot p_L$.

Consequently, in~\eqref{eq5}, we have actually~$g_i\cdot\overline{\Omega}\cdot p_L=\overline{\Omega}\cdot p_L\subseteq \overline{\Phi}$. As we can vary~$\Phi$ through arbitrary small neighbourhoods of~$D$, we conclude
\begin{equation}\label{keyvariant}
\overline{\Omega}\cdot p_L\subseteq D.
\end{equation}
 Hence~$\overline{\Omega}\cdot p_L$ is in~$\left(G\cdot p_L\right)\cap V_W$. In other words,
\begin{equation}\label{keypoint}
\text{$\overline{\Omega}\cdot N$ is contained in~$X(L,W)$.}
\end{equation}
It follows that~$\overline{\Omega}\cdot N\Gamma/\Gamma$ is contained in~$X([L],W)$. 

By Corollary~\ref{coro loc compact closed}, the set~$\overline{\Omega}N\Gamma/\Gamma$ is closed inside~$G/\Gamma$.  But~$\overline{\Omega}N\Gamma/\Gamma$ contains the support of~$\mu_\Omega$. From~\eqref{stabilitybygi}, it is also stable under the~$g_i$, and will contain the support of each of the~$g_i\mu_\Omega$. Finally,
\begin{equation}\label{contains support}
 \text{$\overline{\Omega}N\Gamma/\Gamma$~contains the support of~$\mu_\infty$,}
\end{equation}
and the support of~${\mu_\infty}^{[L]}$.
 
Hence the support of~$\mu_\infty$ is contained in~$X([L],W)$. Consequently (recall~\eqref{notationrestriction} and~\eqref{eqmini}),
\begin{equation}\label{purity}
\mu_\infty=\mu_\infty|_{X([L],W)}={\mu_\infty}^{[L]}|_{X^*([L],W)}={\mu_\infty}^{[L]}.
\end{equation}
\textbf{We have proved that the Ratner decomposition~\eqref{Ratnerdecomposition} of~$\mu_\infty$ has a unique Ratner type.}

\section{Proof of Theorem~\ref{Theorem} -- Lifting ergodic decomposition}
\emph{We will henceforth drop the exponent~${}^{[L]}$ and write~$\mu_\infty$ instead of~${\mu_\infty}^{[L]}$.}

We know that the~$W$-ergodic components of~${\mu_\infty}$ are all translates of~$\mu_{\Lpp}$ (see~\ref{ss section Ratner decomposition} and~\ref{ss section 412}).
If~$c\cdot\mu_{\Lpp}$ is one of the~$W$-ergodic components of~$\mu_\infty$, we necessarily have
\begin{itemize}
\item firstly~$c\Gamma\in\overline{\Omega}N\Gamma/\Gamma$, as, by~\eqref{contains support}, the support of~$\mu_\infty$ is contained in~$\overline{\Omega}N\Gamma/\Gamma$;
\item secondly~$c\in X^*(L,W)$, because~$W$ acts ergodically on~$c\cdot\mu_{\Lpp}$.
\end{itemize}
In other words
\begin{subequations}
\begin{equation}\label{cbelongs}
\text{
the element~$c$ belongs to the intersection~$\left(\overline{\Omega}N\Gamma\right)^*:=\left(\overline{\Omega}N\Gamma\right)\cap X^*(L,W).$
}
\end{equation}
One can consider~$\left(\overline{\Omega}N\Gamma\right)^*$ as a to~$G$ of the following subset of~$G/\Gamma$
\begin{equation}\label{eqa}\left.\left(\overline{\Omega}N\Gamma\right)^*\Gamma\middle/\Gamma\right.=\left(\left.\overline{\Omega}N\Gamma\middle/\Gamma\right.\right)\cap X^*([L],W).
\end{equation}
On the right hand side of equality~\eqref{eqa}, the measure~${\mu_\infty}$ is concentrated on the first factor, by~\eqref{purity}, and the second factor contains the support of~${\mu_\infty}$, by~\eqref{contains support}. As a consequence,
\begin{equation}\label{Full measure}
\text{the measure~${\mu_\infty}$ is concentrated on~$\left.\left(\overline{\Omega}N\Gamma\right)^*\Gamma\middle/\Gamma\right.$.}
\end{equation}	
\end{subequations}
\subsection{Linearisation and self-intersection}\label{sec lin self}
We now use~\cite[Prop.~3.3, Cor.~3.5]{DM}\footnote{The role played by~$H$ in~\cite[pp.99-100]{DM} is played by~$L$ here. The Ratner class~$\RatQ$ here is a substitute for~$\mathscr{H}$ in~\cite[pp.99-100]{DM}. The~$\GammaN$ we are about to define is denoted~$\Gamma_L$ in~\cite[pp.99-100]{DM}.} under the form of Appendix~\ref{sectionA2}.
Let~$\GammaN=\Gamma\cap N(L)$.
After~\cite[Prop.~3.3]{DM}, the set~$X^*(L,W)$ ``does not have~$(L,\GammaN)$-self-intersection" in the sense of~\cite[{\S}3, p.100]{DM}. Namely\footnote{We note that the implication is actually an equivalence, by~\eqref{right invariance GammaN}.}
\begin{subequations}
\begin{equation}\label{self1}
\forall x\in X^*(L,W), \forall \gamma\in\Gamma, x\gamma\in X^*(L,W)\Rightarrow \gamma\in\GammaN.
\end{equation}
Equivalently the quotient map~$\varphi:G/\GammaN\to G/\Gamma$ induces a bijection\footnote{A similar geometric interpretation of~\cite[p.100]{DM} ``self-intersetion" property can be found at the end of \cite[Corollary~3.5]{DM} "The quotient map~[\ldots] is injective. that is~[\ldots]".}
\begin{equation}\label{self2}
\left.X^*(L,W)\,\GammaN\middle/\GammaN \right.
\xrightarrow{}
 \left.X^*(L,W)\,\Gamma\middle/\Gamma\right. .
\end{equation}
In yet other terms, for any subset~$A$ of~$X^*(L,W)$, we have then
\begin{equation}\label{self3}
\left(A\Gamma\right)\cap X^*(L,W)=\left(A\GammaN\right)\cap X^*(L,W).
\end{equation}
\end{subequations}

\begin{subequations}
Recall that~$\overline{\Omega}N$ is contained in~$X(L,W)$, by~\eqref{keypoint}. As~$X(L,W)$ is right~$N(L)$-invariant, see~\ref{right invariance N(L)}, we have furthermore~$\overline{\Omega}N\GammaN\subseteq\overline{\Omega}N(L)\subseteq X(L,W)$. We apply~\eqref{self3} taking~$A$ to be the intersection
\begin{equation}\label{omegaself1}
\left(\overline{\Omega}N\GammaN\right)^*:=\left(\overline{\Omega}N\GammaN\right)\cap X^*(L,W).
\end{equation}
We get
\begin{equation}\label{omegaself2}
\left(\overline{\Omega}N\GammaN\right)^*\Gamma\cap X^*(L,W)=\left(\overline{\Omega}N\GammaN\right)^*\GammaN\cap X^*(L,W).
\end{equation}
\end{subequations}
From~\eqref{right invariance GammaN}, namely that~$X^*(L,W)$ is right saturated by~$\GammaN$, we get
\begin{equation}
\left(\overline{\Omega}N\GammaN\right)^*=\left(\overline{\Omega}N\GammaN\right)^*\GammaN=\left(\overline{\Omega}N\right)^*\GammaN.
\end{equation}
We will now prove that
\begin{equation}\label{gammaoops}
\left(\overline{\Omega}N\GammaN\right)^*=\left(\overline{\Omega}N\Gamma\right)^*.
\end{equation}
\begin{proof} We prove the identity by double inclusion. The inclusion~$$\left(\overline{\Omega}N\GammaN\right)^*\subseteq\left(\overline{\Omega}N\Gamma\right)^*$$ is immediate. Let us prove the reverse inclusion.

Let~$x$ be in~$\left(\overline{\Omega}N\Gamma\right)^*$. We can write~$x=y\gamma$ with~$y\in\overline{\Omega}N$ and~$\gamma\in\Gamma$. We already know that~$x$ belongs to~$X^*(L,W)$, on of the factor defining~$\left(\overline{\Omega}N\Gamma\right)^*$ in~\eqref{cbelongs}. This is also the second factor defining~$x\in\overline{\Omega}N\GammaN$, in~\eqref{omegaself1}. It remains to prove that~$x$ belongs to the other factor~$\overline{\Omega}N\GammaN$ in~\eqref{omegaself1}.

Assume~$y$ belongs to~$\left(\overline{\Omega}N\right)^*$. Then~\eqref{omegaself2} implies~$\gamma\in\GammaN$, and hence~$x$ belongs to~$y\GammaN\subseteq \overline{\Omega}N \GammaN$.  We are done.

Assume~$y$ belongs to~$\overline{\Omega}N$ but not to~$X^*(L,W)$. It then belongs to~$X(L,W)\smallsetminus X^*(L,W)$, by~\eqref{keypoint}. But the right hand side of~\eqref{X3} is right-invariant under~$\Gamma$, by footnote~\ref{rightgammainv} on page~\ref{rightgammainv}. So~$x=y\gamma$ can't belong to~$X^*(L,W)$. This case doesn't occur.
\end{proof}
We recall that the map~$\phi$ is injective on~$\left(\overline{\Omega}N\GammaN\right)^*$ by~\ref{self2}, and
that the image of~$\left(\overline{\Omega}N\Gamma\right)^*$ is given by a subset~\eqref{eqa} on which~$\mu_\infty$ is concentrated, by~\eqref{Full measure}. To conclude we have a continuous bijection
\begin{equation}\label{finalbij}
\left.\left(\overline{\Omega}N\right)^*\GammaN\middle/\GammaN\right.
=
\left.\left(\overline{\Omega}N\Gamma\right)^*\middle/\GammaN\right.
\xrightarrow{\psi}
\left.\left(\overline{\Omega}N\Gamma\right)^*\Gamma\middle/\Gamma\right.
\end{equation}
which we'll write for short
\[
Y\xrightarrow{\psi} Z.
\]

\subsection{Lifting along~$G/\GammaN\xrightarrow{\varphi}G/\Gamma$ of the measures}
Our next objectives are
\begin{itemize}
\item to lift the limit measure~$\mu_\infty$ back along~$\psi$, which is done with \S~\eqref{lift};
\item to lift the ergodic decomposition of~$\mu_\infty$, done with \S~\ref{section lift ergo};
\item to lift the sequence~$(g_i\cdot\mu_\Omega)_{i\geq0}$, explained with~\eqref{naivelift};
\item to  show that we have still convergence of measures, achieved in \S~\eqref{liftlimit};.
\end{itemize}
We start with a remark.
\begin{quote}
{\hspace{\parindent}\emph{We are about to use several operations on measures: restriction, lifting, direct image, desintegration, convergence and, later on, quotient by a Haar measure. The spaces we will concerned with, the one studied in~\S~\ref{sec lin self} above, are complement in a closed subset of~$G$, of countably many real subvarieties, and the image thereof in~$G/\Gamma$ or~$G/\GammaN$. A prototype example is~$\R\smallsetminus\Q$, the space of irrational real numbers. It may happen that these spaces are not locally compact (this prototype isn't); rigour will require some extra care. We will rely on the theory of measure along the lines of \cite[IX]{BouINT}, that we will call Radon measures. We will see that the topological spaces we will be dealing with are \emph{Radon spaces}: every bounded Borel measure --- as a function on the Borel~$\sigma$-algebra --- is inner regular, corresponds to a Radon measure. These will also be \emph{completely regular spaces} and \emph{Lusin spaces} (and probably Polish spaces).}}
\end{quote}
\subsubsection{}\label{subsec lift mu infty} We first want to restrict the probability measure~$\mu_\infty$ on~$G/\Gamma$, to the image~$Z$ of~$\psi$. In order to do this restriction operation, let us prove the  claim
\begin{subequations}
\begin{equation}\label{Borel claim 1}
\text{the image~$Z=\left.\left(\overline{\Omega}N\Gamma\right)^*\Gamma\middle/\Gamma\right.$ of~$\psi$ is a Borel subset of~$G/\Gamma$}.
\end{equation}
\begin{proof} We know from Corollary~\ref{coro loc compact} that~$\overline{\Omega}N\Gamma$ is closed in~$G$,
hence~$\overline{\Omega}N\Gamma/\Gamma$ is closed in~$G/\Gamma$ (the closedness is a local property on~$G/\Gamma$.)
Since
\[\left.\left(\overline{\Omega}N\Gamma\right)^*\Gamma\middle/\Gamma\right.
=
\left(\overline{\Omega}N\Gamma/\Gamma\right)\cap \left( X^*(L,W)\Gamma/\Gamma\right)
=\left(\overline{\Omega}N\Gamma/\Gamma\right)\cap X^*([L],W)
\]
it suffices to prove that~$X^*([L],W)$ is a Borel subset of~$G/\Gamma$. We refer to~\eqref{Borel}.
\end{proof}
We may hence consider the restriction~${\mu_\infty}\restriction_Z$, as a bounded positive map on the Borel $\sigma$-algebra of~$Z$. This is actually a Radon measure, by the following claim
\begin{equation}\label{Borel claim 2}
\text{The topological space~$Z$ is a Radon\footnote{A Radon space is a topological space on which bounded Borel measures are inner regular. They can be described by measures, as defined in~\cite[Chap. IX]{BouINT}, by its restrictions to compact subsets, till inner regular, which corresponds to linear functionals by Riesz representation theorem.} space.}
\end{equation}
\begin{proof} As~$Z$ was proved above to be a Borel subset of~$G/\Gamma$, it actually suffices (cf. \cite[IX \S3.3 Cor. to Prop.\,2]{BouINT}) that~$G/\Gamma$ be a Radon space, which is standard\footnote{We left to the reader to check that~$G$ and~$G/\Gamma$ are separable (weak approximation) and metrisable (cf.~\cite[IX \S3.1]{TOP}), that is ``Polish''.} (cf.~\cite[IX \S3.3 Prop.\,3]{BouINT}).
\end{proof}
(The restriction~${\mu_\infty}\restriction_Z$ can hence be understood in the sense of~\cite[IX \S2.1]{BouINT}, the sense of Radon measures). Recall~\eqref{Full measure} that the measure~$\mu_\infty$ is concentrated on the target~$Z$. We claim
\begin{equation}
\text{That~$Y$ is a Borel subset of~$G/\GammaN$.}
\end{equation}
\begin{proof} We could argue similarly than to the case of~$Z$, \emph{mutatis mutandis}. We will deduce it from the latter case. Let us note that~$\left(\overline{\Omega}N\Gamma\right)^*\Gamma$
is the inverse image of~$Z$ in~$G$ hence a Borel subset. This is also the inverse image of~$Y$, this time by the map~$G\to G/\GammaN$. We are done as, by~\cite[IX \S6.6 Cor.\,2]{TOP}, a subset of~$G/\GammaN$ is a Borel subset if and only if its inverse image in~$G$ is a Borel subset. It is also a consequence of the existence of a Borel section~$G/\GammaN\to G/\Gamma$ (cf.~\cite[IX \S6.9]{TOP}).
\end{proof}
It follows (\cite[IX \S6.7 Lem.\,6]{TOP}) that
\begin{equation*}
\text{$Y$ is a Lusin space,}
\end{equation*}
and that we may apply~\cite[IX\,\S3.3 Rem. on p.\,IX.49]{BouINT}
\begin{equation}\label{BBK Rem}
\text{$\psi$ is a Borel isomorphism that identifies Radon measures on~$Z$ and on~$Y$.}
\end{equation}
\subsubsection{Lifting of~$\mu_\infty$}\label{lift infty}
As an instance of~\eqref{BBK Rem}, there is a unique lift~$\mu_\infty\restriction_Z$ to a Radon measure
\begin{equation}\label{lift}
{\widetilde{\mu}_\infty}\restriction_Y:=
{\psi^{-1}}_\star\left({\mu_\infty}|_Z\right)
\end{equation} 
on~$Y$. Recall that~$\mu_\infty$ is concentrated on~$Z$, by~\eqref{Full measure}; its restriction~${\mu_\infty}|_Z$ is then a probability measure; hence the measure in~\eqref{lift} is a probability measure on~$Y$. The latter has a direct image by~$Y\hookrightarrow G/\GammaN$  (cf. \cite[IX \S2.3]{BouINT}), which is a probability Radon measure, and whose direct image satisfies
\begin{equation}\label{Radon lifted}
{\varphi}_\star\left({{\widetilde{\mu}}_\infty}\right)={\mu_\infty}.
\end{equation}
\end{subequations}
Its support is contained in any closed set containing~$Y$, for instance the closed subset~$\left.\overline{\Omega}N\GammaN\middle/\GammaN\right.$ (this is closed by~Corollary~\ref{coro loc compact closed}).

We stress two consequences of~\eqref{Radon lifted}: firstly
\begin{equation}\label{remarque determination}
\text{$\mu_\infty$ is determined by this lifting~${\widetilde{\mu}_\infty}$;}
\end{equation}
secondly, any left translation by an element~$g$ of~$G$ that leaves~${\widetilde{\mu}_\infty}$ invariant also leaves~$\mu_\infty$ invariant. 

\subsubsection{Lifting of the ergodic decomposition}\label{section lift ergo}
We use another instance of~\eqref{BBK Rem}. As~$\mu_\infty$ is concentrated on~$Z$, so is almost all~$W$-ergodic component~$c\mu_\Lpp$ of~$\mu_\infty$. For such a component, we may likewise restrict it to~$Z$ and obtain its unique lift it to~$Y$, which will be a probability measure, say
\[
\widetilde{c\mu_\Lpp}\restriction_Y:=
{\psi^{-1}}_\star\left(\left( c\mu_\Lpp\right)|_Z\right).
\]
Let~$(T,\tau)$ be a probability space, and~$t\mapsto\mu_t$ a measurable map from~$T$ into the space of probability measures on~$G/\Gamma$ such that we have the ergodic desintegration formula
\[
\mu_\infty=\int_{t\in T} \mu_t\hspace{1.5em}\tau(t).
\]
(cf.~\cite[\S6.1 Th.~6.2]{EinsiedlerGTM} in the case of dynamics of a continuous transformation.)
The corresponding lifting~\(\int_{t\in T} \widetilde{\mu_t}\hspace{1em}\tau(t)\) satisfies
\[
\phi_\star\left(\int_{t\in T} \widetilde{\mu_t}\hspace{1.5em}\tau(t)\right)=\int_{t\in T} \phi_\star\left(\widetilde{\mu_t}\right)\hspace{1.5em}\tau(t)
=\int_{t\in T} {\mu_t}\hspace{1.5em}\tau(t)=\mu_\infty.
\]
and is concentrated on~$Y$. This must be the unique lifting~$\widetilde{\mu_\infty}$ of~$\mu_\infty$.

Denote~$\widetilde{\mu}_{\Lpp}$  the left~$\Lpp$-invariant probability with support~$\Lpp\GammaN/\GammaN$. We have an explicit lift~$c\cdot\widetilde{\mu}_{\Lpp}$ of~$c\cdot\mu_{\Lpp}$ from~$G/\Gamma$ to~$G/\GammaN$. We will prove the identity
\[
	c\cdot\widetilde{\mu}_{\Lpp}=\widetilde{c\mu_{\Lpp}}.
\]
\begin{proof} We first recall that~$c$ belongs to~$\left(\overline{\Omega}N\Gamma\right)^*$ by~\ref{cbelongs}, which is also~$\left(\overline{\Omega}N\right)^*\GammaN$, by~\ref{gammaoops}. We note that~$\left(\overline{\Omega}N\GammaN\right)$ is right~$\Lpp$-invariant (as~$N\GammaN$ contains~$L$.) Thus the support of~$c\cdot\widetilde{\mu}_{\Lpp}$ is contained in~$\left(\overline{\Omega}N\right)\GammaN/\GammaN$, which is contained in~$X(L,W)\GammaN/\GammaN$, which is in turn equal to~$X(L,W)/\GammaN$. But~$Z$ is contained in~$X([L],W)^*$ the inverse image thereof in~$X(L,W)$ is~$X(L,W)^*$. As~$c\mu_\Lpp=\phi_\star(c\widetilde{\mu}_\Lpp)$ is concentrated on~$Z$, hence on~$X([L],W)^*$ the measure~$c\widetilde{\mu}_\Lpp$ is concentrated on~$X(L,W)^*$. Finally~$c\widetilde{\mu}_\Lpp$ is concentrated on~$X(L,W)^*$ and~$\left(\overline{\Omega}N\right)\GammaN/\GammaN$, hence on~$Y$.
\end{proof}
By Proposition~\ref{propo right inv},
\begin{equation}\label{claimLinvariant}\text{${\widetilde{\mu}_\infty}$ is the quotient of a \emph{right} $\Lpp$-invariant measure on~$G$.}\end{equation}
(Lemma~\ref{Lemma unimodular} applies to the situation.) This will be used in order to invoke Proposition~\ref{prop quotient measures}.

We will use another consequencce of the decomposition of~$\widetilde{\mu}_\infty$. 
\begin{lemma}\label{Mpp invariance de mu tilde}
 The measure~$\widetilde{\mu}_\infty$ is left~$\Mpp$-invariant.
\end{lemma}
\begin{proof} It suffices to prove that every component~$c\mu_\Lpp$ is. The invariance group of~$c\mu_\Lpp$ is~$c\Lpp c^{-1}$. But~$c$ belongs to~$HN\GammaN$. Recall that~$\GammaN$ normalises~$\Lpp$. Hence~$c\Lpp c^{-1}$ is of the form~$x\Lpp x^{-1}$ with~$x$ in~$HN$. But such~$x\Lpp x^-1$ contains~$\Mpp$, by the very definition~\ref{defLF} of~$\Mpp$.
\end{proof}
 
\subsubsection{Lifting of the converging sequence}
We will now lift the sequence~$(g_i\cdot\mu_\Omega)_{i\geq0}$. This time we will be explicit. Mimicking the definition~\eqref{def mu Omega} of~$\mu_\Omega$, we define~$\widetilde{\mu}_\Omega$ to be the direct image in~$G/\GammaN$ (instead of~$G/\Gamma$) of the restriction~$\mu\restriction_\Omega$ to~$\Omega$ of the Haar measure~$\mu$ on~$H$ normalised so that~$\mu(\Omega)$. We can then consider the sequence
\begin{equation}\label{naivelift}
\left( g_i\cdot\widetilde{\mu}_\Omega\right)_{i\geq 0}
\end{equation}
of probabilities on~$G/\GammaN$, whose direct image in~$G/\Gamma$ gives back~$\left( g_i\cdot\mu_\Omega\right)_{i\geq 0}$:
\[
\forall i\geq 0,~\phi_\star( g_i\cdot\widetilde{\mu}_\Omega)=_i\cdot\mu_\Omega.
\]
 Note that the~$g_i\cdot\widetilde{\mu}_\Omega$ are supported inside the closed subset~$\overline{\Omega}N\GammaN/\GammaN$.
 (cf. the argument between~\eqref{keypoint} and~\eqref{purity}.)

\subsubsection{Convergence of the lifted sequence}\label{section 633}
In order to apply Proposition~\ref{Propoliftcvg} to the map
\[\overline{\Omega}N\GammaN/\GammaN\xrightarrow{\pi}\overline{\Omega}N\Gamma/\Gamma,\] we need to check a few things. As~$\overline{\Omega}N\GammaN/\GammaN$ (resp.~$\overline{\Omega}N\Gamma/\Gamma$) is locally isomorphic to~$\overline{\Omega}N\GammaN$ (resp.~$\overline{\Omega}N\Gamma$) these are locally compact spaces by Corollary~\ref{coro loc compact}.
\begin{lemma}
\begin{enumerate}
\item  The locus 
\[
S=\left\{x\in\overline{\Omega}N\Gamma/\Gamma\middle|~\#\stackrel{-1}{\pi}(\{x\})=1\right\},
\]
of points~$x$ in~$\left.\overline{\Omega}N\Gamma\middle/\Gamma\right.$, for which the fiber along~$\pi$ is a singleton  is an open subset of~$\overline{\Omega}N\Gamma/\Gamma$.
\item The projection~$\pi:\left.\overline{\Omega}N\GammaN\middle/\GammaN\right.\to\left.\overline{\Omega}N\Gamma\middle/\Gamma\right.$
is proper, and the induced res\-tric\-ted map~$\stackrel{-1}{\pi}(S)\to S$ above~$S$ is an homeomorphism.
\end{enumerate}
\end{lemma}
\begin{proof}[Proof of the openness of~$S$ in statement~(1)]
The subset~$E=\overline{\Omega}N\cdot p_L=\overline{\Omega}p_L$ of~$V$ is compact. We
consider the associated function~$\chi_E$ as in Lemma~\ref{upper semicontinuity}. According to Corollary~\ref{coro upper semicontinuity} to this Lemma~\ref{upper semicontinuity},
the locus
\[
\chi_E< 2
\]
defines an open subset~$U$ of~$G/\Gamma$. We necessarily have~$\chi_E\geq 1$ on~$\overline{\Omega}N\Gamma/\Gamma$.
By construction the locus~$S$ is the locus on~$\overline{\Omega}N\Gamma/\Gamma$ where~$\chi_E=1$. This~$S$ is then the intersection of the open subset~$U$ with our~$\overline{\Omega}N\Gamma/\Gamma$, which open is an open subset of the latter.

%
%
%
%
\end{proof}
\begin{proof}[Proof of the properness of~$\pi$ in statement~(2)] Let~$\Gamma_N=\Gamma\cap N$. It will suffice to prove instead that the map~$\overline{\Omega}N/\Gamma_N\to \overline{\Omega}N\Gamma/\Gamma$ is proper. Indeed, for a compact subset~$K$ of~$\overline{\Omega}N\Gamma/\Gamma$ its inverse image in~$\overline{\Omega}N/\Gamma_N$, proved to be compact, will maps surjectively to the inverse image of~$K$ in~$\overline{\Omega}N\GammaN/\GammaN$, as
\[
\overline{\Omega}N/\Gamma_N\to\overline{\Omega}N\GammaN/\GammaN
\]
is surjective. The latter inverse image, being the image of a compact, is compact.

By Corollary~\ref{coro loc compact closed}, we note that~$\overline{\Omega}N\Gamma/\Gamma$ is closed of~$G/\Gamma$, and that~$\overline{\Omega}N/\Gamma_N$ is closed in~$G/\Gamma_N$. It will suffice to prove the properness of~$\overline{\Omega}N/\Gamma_N\to G/\Gamma$. 

Let~$K$ be a compact of~$G/\Gamma$. This~$K$ is contained in~$C\Gamma/\Gamma$ for a sufficiently large compact subset~$C$ of~$G$.  We want to show that its inverse image
\[
\stackrel{-1}{\pi}(C\Gamma/\Gamma)= C\Gamma/\Gamma_N\cap\overline{\Omega}N/\Gamma_N
\]
is relatively compact in the closed subset~$\overline{\Omega}N/\Gamma_N$ or, equivalently, in~$G/\GammaN$, as~$\overline{\Omega}N/\Gamma_N$ is closed in~$G/\Gamma_N$. We have
\[
C\Gamma/\Gamma_N\cap\overline{\Omega}N/\Gamma_N
\subseteq
C\cdot
\left[
\Gamma/\Gamma_N\cap C^{-1}\overline{\Omega}N/\Gamma_N
\right].
\]
It suffices to see that the square bracket is bounded. As~$\Gamma/\Gamma_N\xrightarrow{\sim}\Gamma N/N$,
we have
\[
\Gamma/\Gamma_N\cap C^{-1}\overline{\Omega}N/\Gamma_N
\simeq
\Gamma N/N\cap C^{-1}\overline{\Omega}N/N\simeq \Gamma\cdot p\cap C^{-1}\overline{\Omega}\cdot p.
\]
But~$\Gamma\cdot p$ is discrete in~$V$ and~$C^{-1}\overline{\Omega}\cdot p$ is compact. These intersection
are finite, and in particular bounded. We have proved the properness.

\end{proof}
\begin{proof}[Proof of statement~(2)] 
%
%
By lemma~\ref{lemma restriction proper}, the restriction~$S\to \pi(S)$ of~$\pi$ is also proper. This is a bijective continuous proper map between locally compact spaces. By Lemma~\ref{Lemma homeo} this an homeomorphism.
This concludes the proof of the Lemma~\ref{naivelift}.
\end{proof}
We may now apply Proposition~\ref{Propoliftcvg}, with~$X\xrightarrow{\pi}Y=\overline{\Omega}N/\Gamma_N\xrightarrow{\pi}\overline{\Omega}N\Gamma/\Gamma$, and for~$U_X\xrightarrow{\sim}U_Y=\stackrel{-1}{\pi}(S)\to S$. We note that~$
\left.\left(\overline{\Omega}N\Gamma\right)^*\Gamma\middle/\Gamma\right.$, on which~$\mu_\infty$ is concentrated, is contained in~$S$ (cf. Proposition~\ref{No intersection}). The sequence~\eqref{naivelift} has a tight limit, and this limit is given by~\eqref{lift}.
\begin{equation}\label{liftlimit}
\lim_{i\geq 0} g_i\cdot \widetilde\mu_\Omega=\widetilde\mu_\infty.
\end{equation}

\section{Proof of Theorem~\ref{Theorem} -- Reduction, induction and conclusion}
\emph{We begin our induction process. To facilitate reading, we briefly explain the strategy. We will eventually work with a closed subset~$X$ of a quotient~$G/(\GammaN\Lpp)$ of our homogeneous space~$G/\GammaN$, on which a subquotient group~$\widehat{G}$ of~$G$ acts. The measure~$\widetilde{\mu}_\infty$ on~$G/\GammaN$ will correspond to a measure~$\widehat{\mu}_\infty$ concentrated on~$X$, and the ergodic components of~$\widetilde{\mu}_\infty$ will be corresponding to Dirac masses: passing to~$X$ we have killed the dynamics of~$W$. Actually~$W$ maps in~$\widehat{G}$ onto the neutral element. The measure~$\widehat{\mu}_\infty$ will be likewise a limit of translated measures~$g_i\cdot\widehat{\mu}_\infty$. We will then prove by contradiction that, on~$X$ the acting elements~$g_i$ on will involve no more dynamic than the case a bounded sequence, which is trivial: the limit~$\widehat{\mu}_\infty$ is then obvious, it will just be some translate~$g_\infty\widehat{\mu}_\infty$.}

\subsection{Reduction to a quotient space}\label{induction} 
By Lemma~\ref{Lemma unimodular} the subgroup~$\GammaN\cdot\Lpp$ of~$G$ is a closed (and unimodular) subgroup.
Let us define the measure~${\widehat{\mu}_\infty}$ as the direct image of~${\widetilde{\mu}_\infty}$ along the quotient map\footnote{As a piece of warning, though map~$G/\Lpp\to G/(\GammaN\Lpp)$ is a quotient map by the right action of~$\GammaN$, a group which normalises~$\Lpp$, in contrast our map~\eqref{induction quotient space} is not a necessarily quotient by the right action of~$\Lpp$, which may not be defined. Nevertheless, the notion of right~$\Lpp$ ``orbits'' in~$G/\GammaN$ does make sense, and~\eqref{induction quotient space} is a quotient map by the associated equivalence relation. This cumbersomeness is the reason we will rely, with~Proposition~\ref{prop quotient measures}, on the theory of quotient measures, on~$G$.}
\begin{equation}\label{induction quotient space}
G/\GammaN\xrightarrow{}\left.G\middle/\left(\GammaN\cdot \Lpp\right)\right..
\end{equation} 
Likewise, we define~$\widehat{\mu}_\Omega$ to be the direct image of~$\widetilde{\mu}_\Omega$, and hence, of~$\mu\restriction_\Omega$. Thanks to the structural property~\eqref{claimLinvariant} of~$\widetilde{\mu}_\infty$, we may apply Proposition~\ref{prop quotient measures}, from which we learn that~${\widehat{\mu}_\infty}$ and~${\widetilde{\mu}_\infty}$ determine each other completely.
\subsubsection{}
We stress that, in order to establish formula~\eqref{limitformula} identifying~$\mu_\infty$ in the Theorem~\ref{Theorem}, it will suffice to identify~$\widehat{\mu}_\infty$ (cf. the observation~\eqref{remarque determination}). More precisely, in order to prove
\begin{equation}
\text{$\widetilde{\mu}_\infty$ is of the form~$
g_\infty\cdot\int_{\omega\in\Omega} \left(\omega\cdot n_\infty\cdot\widetilde{\mu}_{\Lpp}\right)~~\mu(\omega)$}
\end{equation}
which is clearly the quotient of the right invariant measure, with~$\beta_{\Lpp}$ a Haar measure of~$\Lpp$, of the right~$\Lpp$-invariant measure
\[
g_\infty\cdot\int_{\omega\in\Omega} \left(\omega\cdot n_\infty\cdot\beta_{\Lpp}\right)~~\mu(\omega)
\]
it is enough to prove
\begin{equation}
\widehat{\mu}_\infty=g_\infty\cdot\int_{\omega\in\Omega} \omega\cdot n_\infty\cdot\delta_{\Lpp\GammaN}~~\mu(\omega).
\end{equation}

\subsubsection{}
Similarly, again by Proposition~\ref{prop quotient measures}, for each~$g$ in~$G$ the probability measure~$g\widetilde{\mu}_\infty$ is uniquely determined by~$g\widehat{\mu}_\infty$ as its unique lift which is a quotient of a right $\Lpp$-invariant measure on~$G$. In particular, direct image map induces 
\begin{equation}\label{G-bij}
\text{a~$G$-equivariant bijection~}
G\cdot\widetilde{\mu}_\infty\xrightarrow{\sim} G\cdot\widetilde{\mu}_\infty.
\end{equation}

\subsection{About~$W$ and left-invariance of~$\mu_\infty$}\label{subsubsection-W} The definitions of~$W$ and of~$M=\M(\Q_S)$ are found at~\eqref{def W} and in~\ref{defLF} respectively.

These~$M$ and~$W$ are related by the following.
\begin{lemma}\label{Lemme W M} We have
\begin{equation}\label{WM}
W=(\Mpp)^+=M^+.
\end{equation}
\end{lemma}
\begin{proof}
By Lemma~\ref{Lemme M finite index}, the group~$\Mpp$ is open of finite index in~$M$, the latter being the Zariski closure of the former. We may then apply Corollary~\ref{coro finite index +}, from which we obtain~\eqref{M+ et Mpp+} whose translation in our setting readily is the second identity~$(\Mpp)^+=M^+$ in~\eqref{WM}

We argue by double inclusion, establishing~\eqref{WM eq1} and~\eqref{WM eq2} below. 

We rely on~\eqref{cbelongs}, namely that each of the~$W$-ergodic components of~$\mu_\infty$ are of the form~$c\cdot\mu_{\Lpp}$ where~$c$ belongs to~$\left(\overline{\Omega}N\Gamma\right)^*$ (cf.~{\S}\ref{RatnerNotations}). By~\eqref{gammaoops}, one gets
\[
c\in \left(\overline{\Omega}N\GammaN\right)^*\subseteq\overline{\Omega}N\GammaN.
\]
Then~$c\cdot\mu_{\Lpp}$ is left invariant under~$c\Lpp n^{-1}c^{-1}$. Hence
\begin{equation}\label{M invariance}
\text{$\mu_\infty$ is left invariant under~$ \Mpp:= \bigcap_{c\in HN\GammaN} c{\Lpp}c^{-1}$.}
\end{equation}
(N.B.: we may drop the~$\GammaN$ factor by Lemma~\ref{gamma normalise Lpp}.)
Thus, we have~$\Mpp\subseteq \Stab(\mu_\infty)$, whence
\begin{equation}\label{WM eq1}
(\Mpp)^+\leq \Stab(\mu_\infty)^+=W.
\end{equation}

For the converse inclusion we invoke~\eqref{keypoint}, that~$\overline{\Omega}N$ is contained in~$X(L,W)$. Namely,~$W\subseteq cLc^{-1}$ for any $c$ in~$\overline{\Omega}N$. Hence~$W=W^+\subseteq \left(cLc^{-1}\right)^+=cL^+c^{-1}$. It follows 
$$W\subseteq 
\bigcap_{c \in \overline{\Omega}N}c{L^{+}}c^{-1}
= \bigcap_{\omega \in \overline{\Omega}} \omega{L^{+}}\omega^{-1}
= \bigcap_{\omega\in H} h{L}h^{-1} = M.$$
(The equalities arise from~$L^+$ being characteristic subgroup of~$L$ and~$\Omega$ being Zariski dense in~$H$.)
We conclude with the reverse inclusion
\begin{equation}\label{WM eq2}
W=W^+\subseteq M^+.
\end{equation}
\end{proof}
\begin{lemma}\label{Lemma Mpp trivial sur X}
 In~$\left.G\middle/\left(\GammaN\cdot \Lpp\right)\right.$ we consider the closed subset
\begin{equation}\label{def X}
X=\overline{\left.HN\GammaN\middle/\left(\GammaN\cdot \Lpp\right)\right.}.
\end{equation}

Then the left action of~$HF$ stabilises~$X$ and the subgroup~$\Mpp$ of~$HF$ fixes every point of~$X$.
\end{lemma}
\begin{proof} The proof is immediate and straightforward. The only thing worth mentioning is that the action of~$\Mpp$ and~$HF$ being continuous, the closure of an~$HF$ invariant subset is invariant, and of a $\Mpp$-fixed subset is closed $\Mpp$-fixed.
\end{proof}
A rephrasing gives the following.
\begin{corollary} \label{coro X}
The space~$X$ is a locally compact space acted upon on the left by~$\widehat{G}=HF/\Mpp$.
\end{corollary}
Combining Lemma~\ref{Lemma Mpp trivial sur X} with Lemma~\ref{Lemme W M} gives the following.
\begin{corollary}The left action of~$W$ stabilises~$X$ and fixes every one of its point.
\end{corollary}

\subsection{Stabilisers, subquotient, and unipotent elements}
As a consequence of~\eqref{G-bij}, the stabiliser of~${\widehat{\mu}_\infty}$ in~$G$ equals the stabiliser of~${\widetilde{\mu}_\infty}$ in~$G$. We write it
\begin{equation}\label{stab identity}
\Stab_G(\widetilde{\mu}_\infty)=\Stab_G(\widehat{\mu}_\infty),
\end{equation}
with an index denoting the acting group.
Another direct image of~$\widetilde{\mu}_\infty$, the one along~$\varphi:G/\GammaN\to G/\Gamma$,  is~$\mu_\infty$. The~$G$-equivariance of~$\varphi_\star$ implies
\[
\Stab_G(\widetilde{\mu}_\infty)\subseteq\Stab_G(\mu_\infty).
\]
A particular consequence is
\[
\Stab_G(\widetilde{\mu}_\infty)^+\subseteq\Stab_G(\mu_\infty)^+=W.
\]
Recall from~Lemma~\ref{Lemme W M} that~$W=M^+\subseteq \Mpp$. By Lemma~\ref{Mpp invariance de mu tilde}, the group~$\Mpp$ is a subgroup of~$\Stab_G(\widetilde{\mu}_\infty)$. This gives a the reverse of the preceding inclusion, whence the identity
\[
\Stab_G(\widetilde{\mu}_\infty)^+=\Stab_G(\mu_\infty)^+=W.
\]

\subsubsection{} We will consider another acting group than~$G$, which is its following subquotient.
By~Lemma~\ref{lemdebut}, the subgroups~$H$ and~$F$ of~$G$ both normalise~$M$. As~$F$ is contained~$N$ they both normalise~$\Mpp$. We can define a subquotient group of~$G$ by
\begin{equation}\label{Ghat}
\widehat{G}:=HF/\Mpp.
\end{equation}
As~$\Mpp$ stabilises~$\widetilde{\mu}_\infty$, hence also stabilises~${\widehat{\mu}}_\infty$. We can hence consider
\[
\Stab_{HF}(\widetilde{\mu}_\infty)/\Mpp=\Stab_{\widehat{G}}(\widetilde{\mu}_\infty)=\Stab_{\widehat{G}}(\widehat{\mu}_\infty)=\Stab_{HF}(\widehat{\mu}_\infty)/\Mpp.
\]
We want to study~$\widehat{G}$ in terms of an algebraic group, the algebraic subquotient
\begin{equation}\label{Gcheck}
\widecheck{\G}:=\H\F/\M\text{ and denote }
\widecheck{G}=\widecheck{\G}(\Q_S).
\end{equation}
The groups~$\widehat{G}$ and~$\widecheck{G}$ are closely related by the following.
\begin{lemma} The map
\[
\widehat{G}\xrightarrow{}\widecheck{G}
\]
given by~$g\Mpp\mapsto gM$ 
is a a open proper map, with finite kernel and finite cokernel.
\end{lemma}
\begin{proof}
The image of~$\widehat{G}$ is~$HFM/M$ which is open of finite index by Lemma~\ref{Lemme epi ouvert cofini} applied to the epimorphism~$\H\times\F\times\M\to\widecheck{\G}$. It implies that the image is a closed subgroup.
The kernel of the map~$HF\Mpp/\Mpp\to HFM/M$ is the group~$M/\Mpp$, which is finite according to Lem\-ma~\ref{Lemme M finite index}. To conclude, the considered map is a finite cover of a closed subset, hence proper.
\end{proof}
\subsubsection{}
We can use the algebraically defined notion of unipotent element in~$\widecheck{G}$. The following will be used in a forthcoming proof by contradiction.
\begin{lemma}\label{Lemme unipotent contradiction}
The image~$(\Stab_{HF}(\widetilde{\mu}_\infty))M/M$ of~$\Stab_{HF}(\widetilde{\mu}_\infty)=HF\cap\Stab_{G}(\widetilde{\mu}_\infty)$ in~$\widecheck{G}$ does not contain a non trivial one parameter unipotent subgroup: we have
\[
\left(\Stab_{HF}(\mu_\infty)M/M\right)^+=\{e\}.
\]
Moreover the image of~$\Stab_{\widehat{G}}(\widehat{\mu}_\infty)$ in~$\widecheck{G}$ is also~$(\Stab_{HF}(\widetilde{\mu}_\infty))M/M$.
\end{lemma}
\begin{proof} The first assertion is an instance of~Corollary~\ref{corliftuni}. We use~$\H\F$ as~$\H$, we use~$\M$ as~$\N$, we use~$\widecheck{G}$ as~$\Q$ and~$\Stab_{HF}(\widetilde{\mu}_\infty)$ as~$W$. The hypothesis of this corollary is satisfied, as we have~\[\Stab_{HF}(\widetilde{\mu}_\infty)^+\subseteq\Stab_{G}(\widetilde{\mu}_\infty)^+=W\subseteq M\]
by the definition of~$W$ and Lemma~\ref{WM}.

The second assertion follows from~\eqref{stab identity}, after intersecting with~$HF$, and taking images by the quotient by~$M$.
\end{proof}

N.B.: This Lemma stems from the choice, at the very beginning of our proof, of~$W$ as the biggest group generated by one parameter unipotent subgroups which stabilises~$\mu_\infty$. So far, most of the argumentation could go through with any group~$W$ generated by unipotent subgroups which stabilises~$\mu_\infty$.

%

\subsection{Proving boundedness by contradiction}\label{provingbdd}
For~$i\geq0$, the translating element~$g_i$, which belongs to~$F$, has an imege in~$\widehat{G}$, say~$\widehat{g}_i$ . We defined~$\widehat{\mu}_\Omega$ and~$\widehat{\mu}_\infty$ in~\S\ref{induction}. Then the image of~$g_i\widetilde{\mu}_\Omega$ is~$\widehat{g}_i\widehat{\mu}_\Omega$. As~$\widehat{g}_i\widehat{\mu}_\Omega$ converges to~$\widetilde{\mu}_\infty$, we deduce, taking pushforwards, the corresponding convergence
\[
\widehat{\mu}_\infty= \lim\left(\widehat{g_i}\widehat{\mu}_\Omega\right)_{i\geq0}.
\]

Let~$\widehat{H}$ be the image of~$H$ in~$\widehat{G}$, 
and~$\widehat{\lie{h}}$ the corresponding Lie subalgebra of~$\widehat{\lie{g}}$.
We denote~$Z_{\widehat{G}}(\widehat{\lie{h}})$ for the centraliser of~$\widehat{\lie{h}}$ in~$\widehat{G}$. 

Our goal here is too prove the following claim
\begin{equation}\label{claim1}
\text{``the sequence~$\left(\widehat{g_i}\right)_{i\geq 0}$ is of class~$O(1)Z_{\widehat{G}}(\widehat{\lie{h}})$''.}
\end{equation}

By virtue of Proposition~\ref{Propobddcrit} (applied to~$\widecheck{G}$), the claim~\eqref{claim1} is actually equivalent to
\begin{equation}\label{claim2}
\text{``the sequence~$\left(\Ad_{\widehat{g_i}}\right)_{i\geq 0}$ is uniformly bounded, as a linear map, on~$\widehat{\lie{h}}$''.}
\end{equation}
\emph{We will prove~\eqref{claim2} by contradiction. We first show that the failure of~\eqref{claim2} implies that~$\widehat{\mu}_i$ is invariant under non trivial elements of~$\widehat{G}$ which are unipotent (viewed in~$\widecheck{G}$). This part of the argument is well known. We stress this argument is the main reason 
why Ratner theory of unipotent flows is relevant to our subject, and is responsible for the strategy of proof we have been following since the start, that is studying $W$-ergodic decomposition.}

\noindent{N.B.: Here we rely on the hypotheses that the Zariski closure of~$H$ is reductive, and that~$\Omega$ is open in~$H$.}
\begin{proof}[Proof of~\eqref{claim2} by contradiction]
Let us recall the closed subset~$X$ of~$G/(\Gamma_N\cdot \Lpp)$ defined as the closure of~$HN/(\Gamma_N\cdot \Lpp)$. 
\begin{equation}\label{defX}
X=\overline{HN/(\Lpp\Gamma_N)}\subseteq G/(\Gamma_N\cdot \Lpp)
\end{equation}
We noted in Corollary~\ref{coro X} that~$X$ is a locally compact space under which~$\widehat{G}$ acts continuously. Let us note again that
\begin{itemize}
\item that~$\widetilde{\mu}_\Omega$ is supported in~$X$;
\item that the action~$\widehat{g}_i$ stabilises~$X$;
\item that~$X$ is closed in~$G/(\Lpp\GammaN)$.
\end{itemize}
It follows notably that
\begin{equation}
\text{the support of~$\widehat{\mu}_\infty$ is contained in~$X$.}
\end{equation}

Assume by contradiction that,  that property~\eqref{P3} of Proposition~\eqref{Propobddcrit} fails. Namely, possibly passing to a subsequence, there is a sequence  in~$\widehat{\lie{h}}$, say~$(X_i)_{i\geq0}$, converging to~$0$ in~$\widehat{\lie{h}}$ such that the sequence~$\left(\Ad_{\widehat{g_i}}(X_i)\right)_{i\geq 0}$ has a nonzero limit~$X_\infty$ in~$\widehat{\g}$. We will prove that~$\lambda\mapsto\exp_{\widecheck{G}}(\lambda X_\infty):\Q_S\to\widecheck{G}$ is a nontrivial one parameter unipotent subgroup contained in the image of~$\Stab_{\widehat{G}}(\widehat{\mu}_\infty)$, thus contradicting Lemma~\ref{Lemme unipotent contradiction}.

We note that for any non zero scalar~$\lambda$ in~$\Q_S$, the sequence~$(\lambda\cdot X_i)_{i\geq0}$ converges to~$0$ in~$\widehat{\lie{h}}$ and the sequence~$\left(\Ad_{\widehat{g_i}}(X_i)\right)_{i\geq 0}$ has a, possibly zero, limit~$\lambda\cdot X_\infty$ in~$\widehat{\g}$.

Associated to the Lie group~$\widehat{G}$ (resp. algebraic Lie group~$\widecheck{G}$) there is an exponential map~$\exp_{\widehat{G}}$ (resp.~$\exp_{\widecheck{G}}$) converging in a neighbourhood of the origin of~$\widehat{\lie{g}}=\widecheck{\lie{g}}$. We may uniquely extend the definition of the exponential map~$\exp_{\widecheck{G}}$ by to the conjugacy classes meeting this neighbourhood. We may choose~$\exp_{\widecheck{G}}$ to be the composite of~$\exp_{\widecheck{G}}$ with the quotient map~$\widehat{G}\to\widecheck{G}$. Let us prove the claim
\begin{equation}
\text{For~$g=\exp_{\widehat{G}}(X)$ with~$X\in\widehat{\lie{h}}$ sufficiently close to~$0$, we have~$g\in\widehat{H}$}.
\end{equation}
\begin{proof}We have, whenever defined, that~$\exp_{\widecheck{G}}(X)$ belongs to~$\H\M/\M(\Q_S)$. The map
\[HM/M\to HM/M\to \H\M/\M(\Q_S)\] is open, hence so is~$HM/\Mpp\to \H\M/\M(\Q_S)$. Furthermore,~$H\Mpp/\Mpp$ is closed of finite index in~$HM/\Mpp$. For~$X$ sufficiently close to~$0$, we have a well defined equality~$\exp_{\widecheck{G}}(X)=\exp_{\widehat{G}}(X)M/M$. The element~$\exp_{\widehat{G}}(X)$ is in the inverse image, say~$I$, of~$\H\M/\M(\Q_S)$, of which the $M$-saturated~$HM/\Mpp$ is an open subgroup, and so is~$H$. This~$H$ contains a neighbourhood of the neutral element in~$I$, hence contains~$\exp_{\widehat{G}}(X)$ for~$X$ sufficiently close to~$0$, element
\end{proof}
We assume that~$\exp_{\widehat{G}}$ is defined on a sufficiently small neighbourhood of~$0$, so that we may write its restriction~$\exp_{\widehat{H}}=\exp_{\widehat{G}}\restriction_{\widehat{H}}:\widehat{\lie{h}}\to\widehat{H}$.

We note the following. 
\begin{itemize}
\item For large enough~$i$, the element~$\lambda X_i$ will be close enough to~$0$ so that the exponential map~$\exp_{\widehat{H}}:\widehat{\lie{h}}\to\widehat{H}$ converges at~$\lambda X_i$, to some element, say~$\widehat{h}_i(\lambda)$, which belongs to~$\widehat{H}$. We write~$\widecheck{h}_i(\lambda)$ its image in~$\widecheck{G}$.
\item For such~$i$, one has~$\exp_{\widecheck{G}}\left(\Ad_{\widehat{g_i}}(\lambda X_i)\right)=\widehat{g_i}\exp_{\widecheck{G}}(\lambda X_i)\widehat{g_i}^{-1}=\widehat{g_i}\widecheck{h}_i(\lambda)\widehat{g_i}^{-1}$;
\item The measure~$\widehat{\mu}_\infty$ is supported inside~$X$, this is the limit of~$(\widehat{g}_i\cdot\widehat{\mu}_\Omega)_{i\geq0}$, and each~$\widehat{g}_i\cdot\widehat{\mu}_\Omega$ is also supported inside~$X$.
\item the sequence~$(\widehat{h_i}(\lambda))_{i\geq0}$ converges, in~$\widehat{H}$, to the neutral element, and, by Lemma~\ref{Lemma strong convergence}, the sequence~$(\widehat{h}_i(\lambda)\cdot\widehat{\mu}_\Omega)_{i\geq0}$ converges strongly to~$\widehat{\mu}_\Omega$;
\item the sequence~$\left(\widehat{g_i}\widecheck{h}_i(\lambda)\widehat{g_i}^{-1}\right)_{i\geq0}$ converges the element we denote~$\widecheck{g}_\infty(\lambda):=\exp_{\widecheck{G}}(\lambda X_\infty)$;
\item as~$\widehat{G}\to\widecheck{G}$ is proper, the lifted sequence~$\left(\widehat{g_i}\widehat{h}_i(\lambda)\widehat{g_i}^{-1}\right)_{i\geq0}$ admits a subsequence converging to some element, say~$\widehat{g}(\infty)$, possibly depending on the subsequence, but such that its image in~$\widecheck{G}$ is~$\widecheck{g}(\infty)$.
\end{itemize}
Thus the hypotheses of Proposition~\ref{unipotescence} are satisfied. From this Proposition we learn that~$\widehat{g}_\infty(\lambda)$ fixes~$\widehat{\mu}_\infty$. We consider the image
\[
\widecheck{g}_\infty(\lambda)=\exp_{\widecheck{G}}(\lambda X_\infty)
\]
of~$\widehat{g}_\infty(\lambda)$ in~$\widecheck{G}$. We note that~$X_\infty$ is nilpotent by Lemma~\ref{Lemme nilpotent stabilite}. Hence~$\lambda\mapsto \widecheck{g}_\infty(\lambda)$ is a one parameter unipotent subgroup, which is non trivial as~$X_\infty\neq0$, and is contained in the image of~$\Stab_{\widehat{G}}(\widehat{\mu}_\infty)$ in~$\widecheck{G}$.
This contradicts Lemma~\ref{Lemme unipotent contradiction}.
%
%
%
%
\end{proof}

We now have proved~\eqref{claim1}. This actually is a form of the focusing criterion, from which 
we will derive~\eqref{focusingcriterion}.

\subsection{Conclusion: focusing criterion and limit distribution formula.}

\emph{ We will now collect the information we have obtained so far.
In~\S\ref{proof first focusing} we prove the first part of the focusing criterion. We then  prove (in \ref{conclusion number}) a finer form of focusing, and finally determine the limit distributions~$\widehat{\mu}_\infty$, $\widetilde{\mu}_\infty$ and~$\mu_\infty$.}

\subsubsection{Focusing criterion}\label{proof first focusing}
 Let~$\widecheck{H}$ be the image of~$H$ in~$\widecheck{G}$. It is Zariski connected
with Lie algebra~$\widehat{\lie{h}}$. The centraliser~$\widecheck{Z}$ of~$\widecheck{H}$ in~$\widecheck{G}$ is that of~$\widehat{\lie{h}}$.
Let~$Z$ be its inverse image in~$HF$; this is also the inverse image of~$\widehat{Z}$. It contains~$(Z_G(H)\cap HF)\cdot M$. By~\eqref{centralisermodulo} from
Lemma~\ref{Lemmacentmodulo2}, 
the Zariski closure of~$(Z_G(H)\cap HF)\cdot M$ is open of finite index in that of~$Z$. As~$(Z_G(H)\cap HF)\cdot M$ of open of finite
index in its Zariski closure, it is open of finite index in~$Z$. 
It follows that the following are equivalent
\begin{itemize}
\item the sequence~$(g_i)_{i\geq0}$ is of class~$O(1)\cdot Z$ in~$HF$;
\item the sequence~$(g_i)_{i\geq0}$ is of class~$O(1)\cdot (Z_G(H)\cap HF) M$ in~$HF$;
\item the sequence~$(\widehat{g}_i)_{i\geq0}$ is of class~$O(1)\cdot \widehat{Z}$ in~$\widehat{G}$;
\item the image sequence in~$\widecheck{G}$ is of class~$O(1)\cdot \widecheck{Z}$.
\end{itemize}
But we proved the latter with~\eqref{claim1}. Hence we have that
\begin{subequations}
\begin{equation}
\text{``the sequence~$\left(g_i\right)_{i\geq 0}$ is of class~$O(1)\cdot (Z_G(H)\cap HF)\cdot M$".}
\end{equation}
We recall from~{\S}\ref{section-lin-focusing} that the sequence~$\left(g_i\right)_{i\geq 0}$ also evolves in~$F$. 
Let us write
\[g_i=b_i\cdot z_i\cdot m_i\]
with a bounded sequence~$(b_i)_{i\geq0}\in O(1)$ in~$G$, a sequence~$(z_i)_{i\geq0}$ in~$(Z_G(H)\cap HF)$
and~$(m_i)_{i\geq0}$ in~$M$. Reducing in~$HF/F$, the~$m_i$ maps to the neutral element. It follows 
that the sequence of cosets~$(z_iF)_{i\geq0}$ is bounded in
\[
(Z_G(H)\cap HF)F/F\simeq(Z_G(H)\cap HF)/(Z_G(H)\cap F),
\]
as it is inverse
to the sequence~$(b_iF)_{i\geq0}$ (these quotients are quotient groups.)
As a consequence we have actually
\begin{equation}
\text{``the sequence~$\left(g_i\right)_{i\geq 0}$ is of class~$O(1)\cdot (Z_G(H)\cap F)\cdot M$".}
\end{equation}
\begin{equation*}
\textbf{We obtained the first part~\eqref{focusingcriterion} of the focusing criterion from Theorem~\ref{Theorem}.}
\end{equation*}


As~$\Mpp$ is normal in~$HF$ we have an epimorphism
\[
HF/\Mpp\to HF/M
\]
As~$\Mpp$ is open of finite index in~$M$, this is a finite cover, and a proper map.
Hence
\begin{equation}
\text{``the sequence~$\left(g_i\right)_{i\geq 0}$ is of class~$O(1)\cdot (Z_G(H)\cap F)\cdot \Mpp$".}
\end{equation}
\end{subequations}
\subsubsection{Limit formula and fine focusing}\label{conclusion number}
We now now prove a finer version of focusing. We can write
\[g_i=b_i\cdot z_i\cdot m_i\]
with a bounded sequence~$(b_i)_{i\geq0}\in O(1)$ in~$F$, a sequence~$(z_i)_{i\geq0}$ in~$(Z_G(H)\cap F)$
and~$(m_i)_{i\geq0}$ in~$\Mpp$. We write~$\widehat{b}_i$ and~$\widehat{z}_i$ the images in~$\widehat{G}$ of~$b_i$ and~$z_i$ respectively; these give the corresponding factorisation 
\[\widehat{g}_i=\widehat{b}_i\cdot \widehat{z}_i\]
in~$\widehat{G}$.
From the convergence~\eqref{liftlimit} in~$G/\GammaN$ we deduce, in~$X\subseteq G/(\GammaN\Lpp)$,
\begin{equation}
\lim_{i\to\infty} \widehat{g_i}\cdot \widehat{\mu}_\Omega=\widehat{\mu}_\infty.
\end{equation}

\emph{Contrasting with the lifted situation, in~$G/\GammaN$, the translating elements~$\widehat{g}_i$ are made of a bound\-ed factor~$\widehat{b}_i$ and a factor~$\widehat{z}_i$ commuting with the image of~$\Omega$ in~$\widehat{G}$. The dynamics is mostly trivial in this situation, and fits in the context of Lemma~\ref{lemme trivial dyn}.}


By~\eqref{trivial dyn 1} of Lemma~\ref{lemme trivial dyn}, if the sequence of cosets~$\left(\widehat{z}_i\cdot (\GammaN \Lpp)\right)_{i\geq0}$ is not bounded, then, possibly extracting a subsequence, we have the~$\lim_{i\to\infty} \widehat{g_i}\cdot \widehat{\mu}_\Omega=0$ for vague convergence, which contradicts that the limit, which is~$\widehat{\mu}_\infty$, is a probability measure. 
Hence the sequence~$(\widehat{z}_i)_{i\geq0}$, which is in~$Z_G(H)\cap F$, is also of class~$O(1)\cdot (\GammaN \Lpp)$.
\begin{equation*}
\textbf{This proves the assertion~\eqref{thm Fa'} of Theorem~\ref{Theorem}.}
\end{equation*}
Possibly extracting a subsequence we may assume that the bounded sequence~$z_i\cdot (\GammaN \Lpp)$ of cosets converges to some~$x_\infty$ in~$G/(\GammaN \Lpp)$. We stress that the orbit of~$Z_G(H)\cap F$, in~$X=\overline{HF/(\GammaN\Lpp})$, trough the neutral coset need not be closed in general. But~$x_\infty$ belongs to the 
closed orbit~$N/(\GammaN\Lpp)$. It can be lifted to an element~$n_\infty$ of~$N$ in the topological closure of~$(Z_G(H)\cap F)\GammaN$. 

Possibly extracting a subsequence we may assume moreover that, in~$F$, the se\-quen\-ce~$(b_i)_{i\geq0}$ is convergent to, say, some~$g_\infty$. Then, in~$\widehat{G}$, the image sequence~$(\widehat{b}_i)_{\i\geq 0}$, converges to its image~$\widehat{g}_\infty$ of~$g_\infty$. 

We use the convolution notation~$\mu\convolution\nu=\int_Gg\cdot\nu~\mu(g)$.
Applying Lemma~\ref{lemme trivial dyn} we conclude that
\[\widehat{\mu}_\infty=\lim_{i\to\infty} \widehat{b}_i\widehat{z}_i\cdot\widehat{\mu}_\Omega={(g_\infty\cdot\mu_{\Omega})}\convolution{(\delta_{n_\infty\GammaN\Lpp})}\]
where the right hand side is the image measure of~$\mu|_\Omega$ via the map
\[\Omega\to {G/(\Gamma_N\Lpp)}:\omega\mapsto g_\infty\cdot\omega \cdot n_\infty(\GammaN\Lpp).\]
In~$G/\GammaN$, the measure~\(
(g_\infty\cdot\mu_{\Omega})\convolution (n_\infty\widetilde{\mu}_\Lpp)\)
is the quotient of a right~$\Lpp$-invariant measure on~$G$, and has direct image~$\widehat{\mu}_\infty$. These properties characterise~$\widetilde{\mu}_\infty$. We have
\[
\widetilde{\mu}_\infty={(g_\infty\cdot\mu_{\Omega})}\convolution{(n_\infty\widetilde{\mu}_\Lpp)}.
\]
By direct image, in~$G/\Gamma$, we deduce
\[
{\mu}_\infty={(g_\infty\cdot\mu_{\Omega})}\convolution{(n_\infty{\mu}_\Lpp)}.
\]
which is a form of formula~\eqref{limitformula}. 
\begin{equation*}
\textbf{We proved formula~\eqref{limitformula}.}
\end{equation*}

%
%

\end{proof}
This concludes the proof of Theorem~\ref{Theorem} modulo the statements we postponed to the forthcoming appendices.
\appendix
\section{Proof of Theorem~\ref{Theorem} -- Postponed statements.}
In this section we collected various statements which reflects arguments of the main proof of Theorem~\ref{Theorem}.
We give them a separate treatment to alleviate the main proof and to get room to give a proper and clean treatment of the arguments. 
\subsection{Discreteness of~$\Gamma$-orbits of $\Q$-rational vectors, closed subgroups.}

For reference, we make the following remark. We refer to~\cite[Theorem~3.4.]{DM} in the archimedean case, which also covers non arithmetic lattices.
\begin{lemma}\label{Gorbit discrete} The notations are as in §\ref{subsection linearisation}.
The orbit~$\Gamma\cdot p_{L_0}$ is discrete in~$V$.
\end{lemma}
This is a particular case of the following. We refer to~\cite[Argument in~{\S}8.1]{BorelIntro}.
\begin{lemma}\label{Lemma Q discrete}
Consider a linear representation, defined over~$\Q$, of~$\G$ on a vector space~$V_\Q$.
Write~$V=V\tens\Q_S$. For any~$p$ in~$V_\Q$, and any $S$-arithmetic lattice~$\Gamma$
in~$G$, the orbit~$\Gamma\cdot p$ is discrete in~$V$.
\end{lemma}
\begin{proof} Let~$M$ be the $\Z[1/S]$-submodule generated by a basis of~$V_\Q$. Because~$\Z[1/S]$ is a lattice in~$\Q_S$, passing to the cartesian product of~$\dim(V)$ copies, we deduce that~$M$ is a lattice in~$V$. It is in particular discrete in~$V$.

Write~$G(\Z[1/S])$ for the stabiliser in~$G(\Q_S)$ of this~$M$. By definition of an~$S$-arithmetic lattice,~$\Gamma$
is commensurable with~$G(\Z[1/S])$. 

As a finite union of discrete sets is discrete, we may replace~$\Gamma$ by a finite index subgroup, for instance the stabiliser~$\Gamma\cap G(\Z[1/S])$ of~$M$ in~$\Gamma$. 

As~$p$ is defined over~$\Q$ --- that is belongs in~$V_\Q$ --- it belongs to a rational multiple of~$\lambda\cdot M$ of~$M$ (pick~$\lambda$ to be the inverse of the lowest common multiple of the denominators of the coordinates of~$p$).
We may assume~$p$ is non zero. Then this multiplicator is non zero. It is then invertible in~$\Q_S$, and acts by homeomorphism. It follows that~$\Gamma\cdot p_L$ is included in~$\lambda\cdot M$, which is discrete.
\end{proof}
We provide to topological consequences of Lemma~\ref{Gorbit discrete}.
\begin{corollary}\label{coro closed}
Let~$N$ be the stabiliser of~$p_{L_0}$. Then for any subgroup~$\Lambda$ of~$\Gamma$, the subset~$N\Lambda$ is closed in~$G$.
\end{corollary}
\begin{proof} This is the inverse image of the discrete, hence closed, subset~$\Lambda\cdot p_L\subseteq\Gamma\cdot p_L$
along the continuous orbit map~$g\mapsto g\cdot p_L$.
\end{proof}
\begin{corollary}\label{coro loc compact}
 For any compact subset~$\overline{\Omega}$ of~$G$, the subset~$\overline{\Omega}\cdot N\cdot \Lambda$ is closed in~$G$, and is locally compact.
\end{corollary}
\begin{proof} The closedness is  local statement on~$G$. Pick~$g$ in~$G$ and a compact neighbourhood~$K$ of~$g$ in~$G$. We know~$C=N\cdot \Lambda$ is closed. Let
$x=\omega\cdot c$
be an element of~$K$ with~$\omega\in\overline{\Omega}$ and~$c$ in~$C$. Then~$c$ belongs to~$\overline{\Omega}^{-1}\cdot K$, which is a continuous image of the compact~$\overline{\Omega}^{-1}\times K$, hence compact. It follows
\[
\overline{\Omega}\cdot C\cap K=\overline{\Omega}\cdot (C\cap\overline{\Omega}^{-1}\cdot K) \cap K.
\]
We note that~$C\cap\overline{\Omega}^{-1}\cdot K$ is compact, as the intersection of a compact and a closed subset. Hence the product~$\overline{\Omega}\cdot (C\cap\overline{\Omega}^{-1}\cdot K)$ is compact (the same argument as above). Finally it intersection with~$K$ is compact, and in particular closed.

The local compactness follows form the closedness, as~$G$ is locally compact.
\end{proof}
\begin{corollary}\label{coro loc compact closed}
The subset~$\overline{\Omega}\cdot N\cdot \Lambda/\Lambda$ of~$G/\Lambda$ is closed.
\end{corollary}
\begin{proof} Closedness is a local property on~$G/\Gamma$. But~$G\to G/\Gamma$ is a local homeomorphism, as~$\Lambda$ is discrete, and is such that the subset~$\overline{\Omega}\cdot N\cdot \Lambda/\Lambda$ of~$G/\Lambda$ corresponds to~$\overline{\Omega}\cdot N\cdot \Lambda$. As the latter is closed, it follows that locally \emph{on $G/\Gamma$}, the subset~$\overline{\Omega}\cdot N\cdot \Lambda/\Lambda$ is closed in~$G/\Lambda$ which concludes.
\end{proof}

\subsection{A Variant of~\cite[Theorem~2.1]{LemmaA}}
Here we provide a variant of~\cite[Theorem~2.1]{LemmaA} which differs in two ways. Firstly we refer to the (product) norm
$$\left(x_v\right)_{v\in S}\mapsto\max_{v\in S}\Nm{x_v}$$
rather that the product function
$$\left(x_v\right)_{v\in S}\mapsto\prod_{v\in S}\Nm{x_v}.$$
The latter may be not proper. Secondly we will state it for any bounded open subset~$\Omega$, instead of a product~$\prod_{v\in S}\Omega_{v}$.

This Theorem is not used in the proof of Theorem~\ref{Theorem}. It is used to remove hypothesis~\eqref{Hypo} in the Theorems~\ref{Theointro} and~\ref{Theointro2} from the Introduction (see~{\S}\ref{analytic stability}).
\begin{theorem}\label{thm21bis} The Theorem~2.1 from~\cite{LemmaA} holds with the following modifications\footnote{Other than the syntactically erroneous~"given".} applied together. We use the set~$Y=Y_S$ (minding the typo)
\begin{itemize}
\item Let~$\Omega$ be any bounded open subset of~$H$;
\item the concluding formula~\cite[(9)]{LemmaA} is replaced by
\begin{equation}
\forall y\in Y,~\forall\left(x_v\right)_{v\in S}\in V,
\sup_{\omega=\left(\omega_v\right)_{v\in S}\in\Omega}
\max_{v\in S}\Nm{\rho_v(y\cdot\omega_v)(x_v)}_v
\geq
\max_{v\in S}\left.\Nm{x_v}_v\right/c.
\end{equation}
\end{itemize}
\end{theorem}
\begin{proof}[Proof of Theorem~\ref{thm21bis}]
As~$\Omega$ is nonempty open in~$H$, it contains a basic non-empty open subset of the form~$\prod_{v\in S}\Omega_v$, as these form a basis of the product topology. As each~$\Omega_v$ is non-empty and open, it is Zariski dense in~$H_v$.

The proof then goes forward analogously as in~\cite[p.VI-6/124]{LemmaA}. We have at each place
\[\forall y\in Y_v,~\forall x_v\in V_v, \sup_{\omega\in\Omega_v}\Nm{y_v \omega_y x_v}_v\leq\Nm{x_v}_v/c_v.\] 
We deduce the inequality
\[
\max_{v\in S}\sup_{\omega\in\Omega_v}\Nm{y_v \omega_y x_v}_v\geq \max_{v\in S}\left(\Nm{x_v}_v\middle/c_v\right)
\]
and finally
\[
\sup_{\omega\in\Omega_v}\max_{v\in S}\Nm{y_v \omega_y x_v}_v\geq \left.\left(\max_{v\in S}\Nm{x_v}_v\right)\middle/\left(\max_{v\in S}c_v\right)\right.	
\]
which gives the conclusion with~$c=\max_{v\in S}c_v$.
\end{proof}

\subsection{Lifting weak convergence of probabilities outside a negligible closed set.}
This property is used in the main proof, at section~\ref{section 633}, to lift convergence~$g_i\mu_\Omega\to\mu_\infty$ 
inside~$G/\Gamma$ to a convergence~$g_i\widetilde{\mu}_\Omega\to\widetilde{\mu}_\infty$ in~$G/\Gamma_N$,
using the map~\eqref{finalbij}.

\begin{proposition}\label{Propoliftcvg}
Let~$X\xrightarrow{\pi} Y$ be a continuous map between locally compact spaces inducing an homeomorphism from an open subset~$U_X$ of~$X$ to an open subset~$U_Y$ of~$Y$.

Let~$(\mu_i)_{i\geq0}$ be a sequence of probabilities on~$Y$ with weak limit a probability~$\mu_\infty$, such that~$\mu_\infty$ is concentrated on~$U_Y$.
Then for any sequence~$(\widetilde{\mu}_i)_{i\geq0}$ of probabilities on~$X$ such that
$$\forall i\geq 0, \pi_\star\widetilde{\mu}_i={\mu}_i,$$
the weak limit~$\widetilde{\mu}_\infty:=\lim (\widetilde{\mu}_i)_{i\geq0}$ exists, is a tight limit, and satisfies
\begin{equation}\label{eq Prop 11}
	\pi_\star\widetilde{\mu}_\infty={\mu}_\infty.
\end{equation}

\end{proposition}
\begin{proof} By~\eqref{vague implique tight} Lemma~\ref{lemme recap mesures}, 
\begin{equation}\label{proof Prop11}
\text{the sequence~$(\mu_i)_{i\geq0}$ converges tightly to~$\mu_\infty$.}
\end{equation}

Our next objective is to show that 
\begin{equation}\label{claim tightness}
\text{the sequence~$(\widetilde{\mu}_i)_{i\geq0}$ is tight as well.}
\end{equation}
\begin{proof}
By inner regularity of~$\mu_\infty$, for any~$\eps>0$ there is a compact subset~$K$ of~$U_Y$ such that~$\mu_\infty(K)>\mu_\infty(U_Y)-\eps=1-\eps$.

We now prove there is a compact neighbourood~$C$ of~$K$ in~$U_X$.
 As~$Y$ is locally compact, so is~$U_Y$, hence every point~$x$ of~$K$ has a compact neighbourhood~$C_x$. As~$K$ is compact, we may extract a finite cover~$K\subseteq C=C_{x_1}\cup\ldots\cup C_{x_n}$. 
 
Then~$C$ is a compact subset of~$U_Y$ whose interior~$\mathring{C}$ contains~$K$. We have
\[
\mu_\infty(\mathring{C})\geq\mu_\infty(K)\geq 1-\eps.
\]
By~\eqref{proof Prop11} we may apply~\eqref{semicontinu ouvert} of Lemma~\ref{lemme recap mesures}, whence follows
\[
\liminf (\mu_i(\mathring{C}))_{i\geq 0}\geq 1-\eps.
\]
A fortiori we have the lower bound
\[
\liminf (\mu_i(C))_{i\geq 0}\geq\liminf (\mu_i(\mathring{C}))_{i\geq 0}\geq 1-\eps.
\]
As~$\pi_\star(\widetilde{\mu}_i)=\mu_i$ we will be able to substitute
\[
\widetilde{\mu}_i(\stackrel{-1}{\pi}(C))=\mu_i(C).
\]
As~$\pi$ is an homeomorphism from~$U_X$ to~$U_Y$, we get that~$\stackrel{-1}{\pi}(C)$
is a compact subset of~$U_X$. We also have
\[
	\liminf_{i\geq0}\widetilde{\mu}_i(\stackrel{-1}{\pi}(C))
		=
	\liminf (\mu_i(\mathring{C}))_{i\geq 0}
		\geq
	1-\eps.
\]
%
As~$\epsilon$ is arbitrarily small, this proves our claim~\eqref{claim tightness} of tightness for~$(\widetilde{\mu}_i)_{i\geq0}$.
\end{proof}

By tightness, the set~$\{\widetilde{\mu}_i|i\geq 0\}$ is relatively compact, for tight topology, in the space of probabilities on~$X$.
Let~$\mu$ be any weak limit of the sequence~$(\widetilde{\mu}_i)_{i\geq0}$. This is necessarily a probability. Then, by continuity of~$\pi_\star$ from the space of probabilities of~$X$ towards the space of probabilities of~$Y$, we have
\begin{equation}\label{eq proof Prop 11}
\pi_\star(\mu)=\mu_\infty.
\end{equation}

Consequently~$\mu$ is concentrated on~$U_X=\stackrel{-1}{\pi}(U_Y)$, and~$\mu$ is characterised by its restriction to~$U_X$, which is necessarily~$\stackrel{-1}{\pi}_\star({\mu_\infty}|_{U_Y})$.

The sequence~$\widetilde{\mu}_\infty$ is tight, and has a unique limit point. This is hence a limit~$\widetilde{\mu}_\infty$ of the sequence~$(\widetilde{\mu}_i)_{i\geq0}$. The relation~\eqref{eq Prop 11} is~\eqref{eq proof Prop 11} above. This concludes.
\end{proof}

We regroup in following lemma some standard facts we use regularly.
\begin{lemma}\label{lemme recap mesures}
 Let~$(\mu_i)_{i\geq 0}$ be a weakly convergent sequence of probability measures on a locally compact space~$X$, with vague (resp. weak) limit a probability measure~$\mu_\infty$.
\begin{enumerate}
\item \label{vague implique tight} Then~$(\mu_i)_{i\geq 0}$ is tightly converging to~$\mu_\infty$. 
\item \label{semicontinu ouvert}	For any open subset~$U$ of~$X$ one has
\[
\liminf \mu_i(U)\geq\mu_\infty(U).
\]
\item \label{semicontinu ferme} For any closed subset~$V$ of~$X$ one has
\[
\limsup \mu_i(V)\leq\mu_\infty(V).
\]
\end{enumerate}
\end{lemma}
\begin{proof}
As~$\mu_\infty$ is a probability (there is no escape of mass at infinity),
it follows by~\cite[\S5.3, Prop.~9]{BouINT} that vague convergence implies tight convergence.

It suffices to prove one conclusion. The other follows as~$\mu(V)=1-\mu(U)$ for any probability measure.

Let~$f$ be the indicator function of~$U$. As~$U$ is open,~$f$ is lower semi-continuous. By~\cite[\S5.3]{BouINT}, the evaluation~$\mu\mapsto\mu(f)=\mu(U)$ against the test function is lower semicontinuous.

We may argue similarly with a closed subset, or deduce it using the symme\-try relation~$\mu(V)=1-\mu(U)$ for any probability measure~$\mu$, with the complementary open subset~$U=X\smallsetminus V$.
\end{proof}

We shall have a use for the two following formal facts
\begin{lemma}\label{lemma restriction proper} Let~$f:X\to Y$ be a proper map. Then for any subset~$\Sigma\subseteq Y$, with preimage~$\Pi={\stackrel{-1}{f}(\Sigma)}$, the restriction~$f|_\Pi:\stackrel{-1}{f}(\Sigma)\to \Sigma$ is proper.
\end{lemma}
\begin{proof} Let~$K$ be a compact subset of~$\Sigma$. Then its inverse image~$\stackrel{-1}{f}(K)$ in~$X$ is compact, by properness. It is contained in~$\Pi$ by definition. We conclude
\begin{itemize}
\item that this is also the inverse image by~$f|_\Pi$;
\item this is a compact subset of~$\Pi$.
\end{itemize}
\end{proof}
\begin{lemma}\label{Lemma homeo}
Let~$f:X\to Y$ be a bijective continuous proper map between locally compact spaces. Then this is an homeomorphism.
\end{lemma}
\begin{proof} We have to show that~$f$ is an open map, or equivalently, as it is bijective, a closed map.

Let~$C$ be a closed subset of~$X$. We want to prove that~$f(C)$ is closed. We can work locally on~$Y$. 
As~$Y$ is locally compact, it suffices to show that~$f(C)\cap K$ is closed for any compact subset~$K$ of~$Y$.
We use the properness~$\stackrel{-1}{f}(K)$ is compact. Hence~$C\cap\stackrel{-1}{f}(K)$ is compact.
We see with the projection relation
\[f(C)\cap K=f(C\cap\stackrel{-1}{f}(K))\]
that~$f(C)\cap K$ is the image of a compact, hence compact, and indeed closed.
\end{proof}
\subsection{Unimodularity and quotient measures}
\begin{lemma}\label{Lemma unimodular} Let~$L$ be a subgroup of~$G$ defined over~$\Q$, let $\Lpp$ be an open subgroup of~$L$. Let~$\Gamma$ be an~$S$-arithmetic subgroup of~$G$ and~$\Lambda$ be subgroup of~$\Gamma$ normalising~$\Lpp$.
\begin{itemize}
\item Then~$\Lambda \Lpp$ is a closed subgroup of~$G$.
\item Assume that~$L$ is unimodular. Then~$\Lambda L$ is unimodular.
\end{itemize}
\end{lemma}
We first prove the closedness.
\begin{proof}We note that~$\Lambda$ normalises the Zariski closure of~$\Lpp$, which is~$L$, that is~$\Lambda$ is contained in the normaliser~$N(L)$ of~$L$ in~$G$. This~$N(L)$ is algebraic and defined over~$\Q$, hence closed.
 It suffices to prove  that~$\Lambda\Lpp$ is a closed subgroup in its normaliser~$N(L)$. 
We will prove that~$\Lambda\Lpp$ is locally equal to~$L$ at the neutral element. It will in particular be locally closed (Lemma~\ref{loc closed 1}), hence closed~(Lemma~\ref{loc closed 2}). We consider the quotient map
\[N(L)\to N(L)/L,\] which is defined over~$\Q$. We denote~$\GammaN=\Gamma\cap N(L)$; it contains~$\Lambda$.
The image of~$\GammaN$ is then an arithmetic subgroup of the affine algebraic group~$N(L)/L$, and is in particular discrete. Hence the image of~$\Lambda$ is discrete too. It follows that~$L$ is open and closed in~$\Lambda\cdot L$. Then~$\Lambda\cdot \Lpp$ is locally isomorphic to~$(\Lambda\cdot \Lpp)\cap L$, which contains~$\Lpp$, which is an open subgroup of~$L$. So~$\GammaN\cdot \Lpp$ is the same as~$L$ in a neighbourhood of the identity. But~$L$ is a closed subgroup. We conclude by applying Lemma~\ref{loc closed 1} and then Lemma~\ref{loc closed 2}.
\end{proof}
We now prove the unimodularity.
\begin{proof}
In~$\bigwedge^{\dim(L)}\g$, we consider a generator~$p_L$, defined over~$\Q$, of the free~$\Q_S$-submodule~$\det(\l)$ of rank one. This~$p_L$ is the value at the origin of a non zero left-(resp. right-)invariant maximal differential form on~$L$, to which is associated a left-(resp. right-)invariant Haar measures~$\nu_{p_L}$ on~$L$. The action of~$\Lambda$ on~$p_L$ stabilises the associated rational line~$\Q\cdot p_L$, and acts through a character defined~$N(L)\to GL(1)$ over~$\Q$ of the normaliser~$N(L)$ of~$L$. The image of~$\Lambda$ is an arithmetic subgroup of~${\Q_S}^\times$, hence contains with finite index a subgroup of~$\Z[1/S]^\times$. We consider the map~$\abs{-}_S:{\Q_S}^\times\to{\R}^\times$ given by the product of the normalised absolute values. this map is constant over~$\Z[1/S]^\times$, as can be checked on a set of generators:~$-1$ and the primes in~$S$. As~$\R^\times$ is without torsion, this map is constant on the image of~$\Lambda$. By the Jacobian formula, for~$\lambda$ in~$\Q_S$ we have that~$\nu_{\lambda p_L}=\abs{p_L}_S\cdot \nu_{p_L}$, which can be checked place by place. It follows that the conjugation action of~$\Lambda$ on~$\nu_{p_L}$ (hence right, resp. left, action on left, resp right, invariant~$\nu_{p_L}$) is trivial.

By hypothesis the (left, right, conjugation) action of~$L$ is trivial on~$\nu_{p_L}$ too. We can pick as a Haar measure on~$\Lpp$ the restriction~$\nu_{p_L}\restriction_\Lpp$ of~$\nu_{p_L}$, which is also invariant under~$\Lpp\Lambda$ and~$\Lambda$. 
As we saw, in the previous proof,~$\Lambda\Lpp$ is locally equal to~$\Lpp$. We can right (resp. left) transport~$\nu_{p_L}\restriction_\Lpp$ to other~$\Lpp$ cosets to build up a left (resp. rigt) invariant Haar measure on~$\Lambda\Lpp$. By the~$\Lambda$ invariance by conjugation this is the same measure, which is both right and left invariant.
\end{proof}
\begin{lemma}\label{loc closed 1} Let~$H$ be a subgroup of topological group~$G$, and~$U$ be neighbourhood in~$G$ of an element~$h$ in~$H$ such that~$Z=H\cap U$ is closed in~$U$. Then~$H$ is locally closed.
\end{lemma}
\begin{proof} It is essentially an argument by homogeneity. 

Every element~$h'$ of~$H$ can be put under the form~$h''\cdot h$ by taking~$h''=h'h^{-1}$ in~$H$.
Then~$V=h''U$ is a neighbourhood of~$h'$ in~$G$ and is such that~$V\cap H=h''U\cap H=h''(U\cap H)=h'Z$. The map~$g\mapsto h''g$ is an homeomorphism from~$U$ to~$V$ sending~$Z$ to~$V\cap H$. Hence~$V\cap H$ is closed in~$V$. That is~$H$ is locally closed at~$h'$. We remind being locally closed is a local property on~$H$. As~$h'$ was arbitrary,~$H$ is locally closed in~$G$. 
\end{proof}
\begin{lemma}\label{loc closed 2} A locally closed subgroup~$H$ of a topological group~$G$ is a closed subgroup.
\end{lemma}
This is a standard fact. Here is a proof borrowed from~\cite[1 \S3.2 p.\,20]{Robert}.
\begin{proof} The closure~$Z$ of~$H$ in~$G$ is a closed subgroup of~$G$. A closed subset of a closed subset being closed, it suffices to show that~$H$ is closed in~$Z$. By hypothesis~$H$ is locally closed; equivalently~$H$ is open subgroup in its closure~$Z$. The complement of~$H$ in~$Z$ is a union of translates of~$H$, which are open as well in~$Z$. This complement is then open. Equivalently~$H$ is closed in~$Z$. 
\end{proof}
The following applies notably in the situation of Lemma~\ref{Lemma unimodular}. It is used to interpret our measure~$\tilde{\mu}_\infty$ composed of the~$c\mu_\Lpp$ in terms of a measure~$\widehat{\mu}_\infty$ on~$G/\GammaN\Lpp$.
\begin{proposition}\label{prop quotient measures} Assume~$L$ and~$\Lambda$ are closed unimodular subgroups of a unimodular group~$G$. Assume that~$\Lambda$ normalises~$L$ and~$L\Lambda$ is a closed unimodular subgroup. Assume~$\Lambda$ is discrete and intersects~$L$ in a lattice.

Let~$\mu$ be a probability measure on~$G/\Lambda$ which is quotient~$\mu^\sharp/\beta'$, by a Haar measure~$\beta'$ on~$\Lambda$, of a right $L$-invariant measure~$\mu^\sharp$ on~$G$. We consider its direct image~$\pi_\star(\mu)$ along the quotient map
\[
\pi:
G/\Lambda\to G/(L\Lambda).
\]
Then we recover~$\mu$ from~$\pi_\star(\mu)$ by the desintegration formula
\begin{equation}\label{desintegration quotient}
\pi_\star(\mu)\mapsto\mu=\int_{cL\Lambda\in G/(\Lambda L)} c\cdot\mu_L~\pi_\star(\mu).
\end{equation}
\end{proposition}
Some of the hypotheses of this Proposition are superfluous. Being discrete,~$\Lambda$ is necessarily unimodular (the counting measure is both left and right Haar measure); admitting a lattice,~$L$ is automatically unimodular (\cite[Prop. 2.4.2]{BasicLieTheory}).
\begin{proof} 
This is essentially a particular case of the situation studied in~\cite[VII\S2.8]{Integration}.

By definition of quotient measures,~$\mu$ is the quotient~$\mu^\sharp/\beta'$ of a unique measure~$\mu^\sharp$ on~$G$, and this~$\mu^\sharp$ is right~$\Lambda$-invariant. By assumption~$\mu^\sharp$ is also right~$L$-invariant; it is then right~$L\Lambda$-invariant. For every the Haar measure~$\beta$ on~$L\Lambda$, the quotient measure~$\mu_L=\beta/\beta'$ is a~$L$-invariant measure on~$L\Lambda/\Lambda$. We assume~$\beta'$ is normalised so that~$\mu_L$ is a probability measure. For~$c$ in~$G$, the translated measure~$c\mu_L$ are the one denoted~$(\beta'/\beta)_u$ in \emph{loc. cit.} with~$u=cL\Lambda/L\Lambda$.

We have the decomposition, by \cite[VII\S2.8 Prop.~12]{Integration},
\[
\mu^\sharp/\beta'=\int_{u\in G/(\Lambda L)} (\beta/\beta')_u\qquad\mu^\sharp/\beta(u),
\]
that is, with our notations,
\[\mu=\int_{u=cL\Lambda\in G/(\Lambda L)} c\cdot\mu_L\qquad\mu^\sharp/\beta(u).
\]
All is left to prove is that~$\mu^\sharp/\beta(u)=\pi_\star(\mu)$. But~$c\mu_L$ is a probability measure, and its support is~$\stackrel{-1}{\pi}(cL\Lambda/L\Lambda)$, hence~$\pi_\star(c\mu_L)$ is a probability measure supported on the singleton~$\{cL\Lambda/L\Lambda\}$, which is necessarily
\[\pi_\star(c\mu_L)=\delta_{cL\Lambda/L\Lambda}.\]
It follows
\[
\pi_\star(\mu^\sharp/\beta')
=\int_{u\in G/(\Lambda L)} \pi_\star((\beta/\beta')_u)\quad\mu^\sharp/\beta(u)
=\int_{u\in G/(\Lambda L)} \delta_u\quad\mu^\sharp/\beta(u)=\mu^\sharp/\beta.
\]
\end{proof}
We will see that the measures~$\mu$ satisfying the hypothesis of last Proposition are actually the one which can be put in the form of formula~\eqref{desintegration quotient}.
\begin{lemma} Every translate probability measure~$c\mu_L$ on~$G/\Lambda$ is the quotient of a right $L$-invariant measure on~$G$.
\end{lemma}
\begin{proof} By definition~$\mu_L$ is the quotient of a Haar measure~$\beta$ on~$L\Lambda$ (viewed as a measure on~$G$).
The translated probability measure~$c\mu_L$ is the quotient of~$c\beta$.
\end{proof}
\begin{proposition}\label{propo right inv}
 Let~$t\mapsto\mu_t$ be a measurable function defined on a probability space~$(T,\tau)$ into the space of translates~$c\mu_L$ of~$\mu_L$. Then
\[
\mu=\int_T \mu_t\quad\tau(t)
\]
is the quotient measure of a right~$L$-invariant measure on~$G$.
\end{proposition}
\begin{proof} For every~$t$ in~$T$, the probability measure~$\mu_t$ is the quotient of a unique measure~$\mu_t^\sharp$ on~$G$. We consider the $\tau$-averaged measure~$\mu^\sharp=\int_T \mu^\sharp_t\quad\tau(t)$ on~$G$. Its quotient is~$\mu$. (We used that the quotient operation is compatible with barycentres and is continuous.) Each~$\mu_t^\sharp$ is right $L$-invariant by the previous lemma, and so is~$\mu^\sharp$.
\end{proof}
We remark that the space of translates of~$\mu_L$ can be identified with~$G/\Stab(\mu_L)$ which is actually the quotient~$G/L\Lambda$. The identification can be achieved via the explicit correspondence
\[G\cdot\mu_L\ni c\mu_L\mapsto\pi_\star(c\mu_L)=\delta_{cL\Lambda}\mapsfrom cL\Lambda \in G/L\Lambda.\]

\subsection{Unipotent generated subgroups, liftings, finite generation}\label{subsubsection-liftuni} 
\subsubsection{Unipotent elements and unipotent subgroups}We refer to~{\S}\ref{secnotations} for the~$\text{◌}^+$ notation. In this section~\ref{subsubsection-liftuni} we will also make use the following notation.
\begin{definition}\label{defi Mu} Let~$M$ be a subgroup of the group~$\H(k)$ of points of an algebraic group~$\H$ over a field~$k$. We denote
\begin{equation}
M^u
\end{equation}
the subgroup generated by elements of~$M$ which are unipotent in~$\H(k)$.
\end{definition}
Obviously, we always have~$M^+\leq M^u$. The converse inclusion is not always true for~$M=SL(2,\Z)\leq GL(2,\R)$ (with~$k=\R$) and for~$M=SL(2,\Z_p)\leq  GL(2,\Q_p)$ (for~$k=\Q_p$) one has~$M^+=\{e\}\neq M^u=M$. Nevertheless we have the following.
\begin{lemma}\label{Lemma u+} Within the setting of preceding Definition~\ref{defi Mu}, assume~$M$ is Zariski closed in~$\H(k)$. Assume~$k$ is of characteristic~$0$. Then we have~$\H(k)^u=\H(k)^+$.
\end{lemma}
\begin{proof} We argue by double inclusion. Only one the inclusion~$M^u\leq M^+$ is left to prove. By Lemma~\ref{Unipotent vs 1param} below, it is enough to prove that every unipotent element~$u$ of~$M$ is contained in an algebraic connected unipotent subgroup of~$M$. We use the Jordan decomposition (as in Proposition~\ref{JCD} below) in the algebraic subgroup generated by~$u$, which is commutative: this commutative subgroup is actually unipotent. It is also connected as, in characteristic~$0$ unipotent subgroups are connected. 

\noindent N.B.: we could have applied Lemma~\ref{Unipotent vs 1param} only to commutative unipotent algebraic subgroups.
\end{proof}
\begin{corollary}\label{coro finite index +}
 In Lemma~\ref{Lemma u+}, the conclusion still hold if we assume instead merely that~$M$ is of finite index in its Zariski closure in~$\H(k)$. If~$\M$ denotes the Zariski closure of such an~$M$ in~$\H$, we have
\begin{equation}\label{M+ et Mpp+}
\M(k)^+=M^+.
\end{equation}
\end{corollary}
\begin{proof} Here again we want to prove~$M^u\leq M^+$. Let~$\M(k)$ be the Zariski closure of~$M$ in~$\H(k)$, where~$<M$ is the Zariski closure in~$\H$. 

As~$k$ is of characteristic~$0$, the group~$(k,+)$ is a~$\Q$-module, that is a divisible group. Then so are the homomorphic images of~$(k,+)$. Then these images cannot be a non trivial finite abelian group: by a theorem of Lagrange, multiplication by the cardinal of the group is the zero map, hence is surjective only for a trivial group. We conclude that~$(k,+)$ has no non trivial finite index subgroup.

For any one-parameter subgroup~$\U$ of~$\M$, there is no non trivial finite index subgroup in~$\U(k)$: indeed one has~$\U(k)$ is isomorphic to~$\{0\}$ or to~$\U(k)\simeq (k,+)$. It follows that~$\U(k)$ is contained in~$M$, and hence~$\M(k)^+\leq M$, which the non trivial inclusion in the identity~$\M(k)^+=M^+$. 

From Lemma~\ref{Lemma u+} we have~$\M(k)^+=\M(k)^u$. We conclude with
\[
M^u\leq \M(k)^u=\M(k)^+=M^+.
\]
\end{proof}
\begin{lemma}\label{Unipotent vs 1param}
For every connected algebraic unipotent group~$\U$ over a field~$k$ of characteristic~$0$, every element~$u$ of the group~$\U(k)$ is contained in a one parameter unipotent subgroup.
\end{lemma}
\begin{proof} We use that the exponential map~$\exp_\U:\lie{u}\to\U$ is a polynomial isomorphism of affine algebraic varieties. An element~$u$ of~$\U(k)$ can be put under the form~$\exp(X)$. We recall that the exponential of commuting elements is the product of the exponential of these elements. It follows that the map~$\G_a\to\G_a\cdot X\to \exp(\G_a\cdot X)$ is an homomorphism; that is hence a one parameter subgroup of~$\U$ passing through~$u$. 

\noindent N.B.: if~$u$ is non trivial, one can show along these lines that there is a unique such one parameter subgroup, up to change of parametrisation, through~$u$.
\end{proof}
This last lemma ensures that our notation~$\text{◌}^+$ is compatible with~\cite[\S6]{BorelTits}. We can also use it in the following.
\begin{lemma}\label{Lemme Unip prof}
For every connected algebraic unipotent group~$\U$ over a field~$k$ of characteristic~$0$, the group~$\U(k)$ does not contain a non trivial finite index subgroup.
\end{lemma}
\begin{proof} In other words we want to prove that any finite index subgroup~$H$ contains~$\U(k)$. As the characteristic is~$0$, we may apply Lemma~\ref{Unipotent vs 1param}, and we are reduced to the case of a one parameter unipotent subgroup, that to the case of~$(k,+)$. Again, thanks to the characteristic hypothesis, we know that~$(k,+)$ is a~$\Q$-module, in other words divisible. As~$(k,+)$ is abelian, the finite index subgroup~$H\cap(k,+)$ is the kernel of a map into a finite abelian group~$A$. The image of~$(k,+)$ is divisible, as~$(k,+)$ is, and torsion, as~$A$ is, hence trivial. Hence~$H$ contains~$(k,+)$.

N.B.: this proves that~$\U(k)$ is divisible. Actually it is uniquely divisible, and the corresponding maps~$g\mapsto\exp(\log(g)/n)$  are polynomial. It is possible to argue directly that a divisible group has no non trivial subgroup of finite index.
\end{proof}

For the following we refer to~\cite[I.\S4]{BorelLAG}, especially \S4.4 Theorem (1) and (4), and Theorem~4.7. (Care the remark following Theorem~4.7 too.)
\begin{proposition}[Jordan decomposition]\label{JCD} \begin{enumerate}
\item \label{Jordan1} In a linear algebraic group~$\G$ over a perfect field~$k$ every commutative algebraic sub\-group~$\mathbf{C}$ can be written uniquely as a direct product~$\mathbf{S}\cdot\mathbf{U}$ of an algebraic subgroup~$\mathbf{S}$ of semisimple elements of~$\mathbf{C}$ and a unipotent subgroup~$\U$ of~$C$.
\item \label{Jordan2} The decomposition
\[
g=g_s\cdot g_u
\]
of an element~$g$ of~$\mathbf{C}(k)$, with~$g_s$ in~$\mathbf{S}(k)$ and~$g_u$ in~$\U(k)$ is called the \emph{(Che\-\mbox{valley-)}Jordan decomposition} of~$g$, and does not depend on the group~$\mathbf{C}$.
\item \label{Jordan3} The Jordan decomposition is functorial in the sense that for any algebraic homomorphism~$\varphi:\G\to \G'$, and every element~$g$ in~$\G(k)$, one has~$\varphi(g)_s=\varphi(g_s)$ and~$\varphi(g)_u=\varphi(g_u)$.
\end{enumerate}
\end{proposition}
We note the following corollary of point~\ref{Jordan3}.
\begin{corollary}\label{coro unipotent lift} Let~$u$ be an element in~$\varphi(\G(k))$ which is unipotent element in~$\G'$. Then~$u$ is the image of a unipotent element of~$\G(k)$.
\end{corollary}
\begin{proof} We may write~$u=\varphi(g)$ for some~$g$ in~$\G(k)$. We are done as~$u=\varphi(g)_u=\varphi(g_u)$ and~$g_u$ is unipotent in~$\G(k)$.
\end{proof}
Noting that, by~\ref{Jordan3}, the image of unipotent element is unipotent, we may immediately deduce another corollary, which we will combine with Lem\-ma~\ref{Lemma u+} in characteristic~$0$.
\begin{corollary}We have
\[
\varphi(\G(k)^u)=\varphi(\G(k))^u\leq\G'(k)^u.
\]
\end{corollary}
\subsubsection{}This statement explains, for algebraic linear groups over a field of characteristic~$0$, the behaviour of unipotent elements when passing to a quotient group. This amounts to the Jordan decomposition in algebraic groups. 
We then give a consequence which is used in our main proof to relate the maximality property satisfied by~$W$ with the induction step~{\S}{\S}\ref{induction}--\ref{provingbdd}.
\begin{proposition}\label{propoliftuni}
 Let~$\G$ be an linear algebraic group over a local field~$k$ of characteristic~$0$, with a subgroup~$\H$, and a normal subgroup~$\N$ of~$\H$, with subquotient group~$\Q=\H/\N$. 

Then 
\begin{enumerate}
\item \label{propplus1} we have~${\H(k)}^+\subseteq{\G(k)}^{+}\cap \H(k)$;
\item \label{propplus2} write~$\varphi:\H\to \Q$ the quotient map, we have~${\Q(k)}^+=\varphi\left({\H(k)}^{+}\right)$.
\end{enumerate}
\end{proposition}
\noindent(cf.~\cite[{\S}III 4.3]{SerreLNM5} and~\cite[{\S}3.18-19]{BorelTits} in the case of local fields of  non zero characteristic.)

We will use the following complement.
\begin{lemma}\label{Lemme epi ouvert cofini}
 For any epimorphism~$\phi:\H\to\Q$ of algebraic groups over~$\Q_S$, the subgroup~$\varphi(\H(\Q_S))$ of~$\Q(\Q_S)$ is open of finite index.
\end{lemma}
\begin{proof} We may argue place by place. The openness use that we have a submersion; the finiteness uses Galois cohomology bounds. For this we refer to~\cite[{\S}6.4: Proposition~6.13 p.\,316; Corollary~2 {\S}6.4 p.\,319]{PR}.
\end{proof}
\begin{proof}
Let~$u$ be unipotent element of~$\H(k)$. It is in particular a unipotent element of~$\G(k)$ which belongs to~$\H(k)$. Thus~${\H(k)}^{u}\cap \G(k)$ contains a generating set of~${\H(k)}^u$. It finally contains~${H(k)}^u$. We have proved
\[{\H(k)}^u\subseteq{\G(k)}^{u}\cap \H(k)\]
which, given Lemma~\ref{Lemma u+}, is our first point. (N.B.: This inclusion may be strict.)

We now prove the second point. We use Corollary~\ref{coro unipotent lift} for~$\H=\G$ and~$\Q=\G'$ together with  Lem\-ma~\ref{Lemma u+} for~$\H$ and for~$\Q$. This yields
\[
\varphi(\H(k)^+)=\varphi(\H(k)^u)=\varphi(\H(k))^u\leq\Q(k)^u=\Q(k)^+.
\]
We want to prove the converse inclusion. As~$k$ is a local field of characteristic~$0$, we have that~$\varphi\left({\H}(k)\right)$ is a finite index subgroup of~$\Q(k)$ by Lemma~\ref{Lemme epi ouvert cofini}.  By~Lemma~\ref{Lemme Unip prof},~$\varphi(\H(k))$ contains~$\U(k)$ for every algebraic unipotent subgroup~$\U$ of~$\Q$. That is~$\Q(k)^+\leq\varphi(\H(k))$. It follows
\[\Q(k)^+\leq\varphi(\H(k))^+\leq\varphi(\H(k))^u.\]
\end{proof}
We deduce immediately the following.
\begin{corollary}\label{corliftuni}
Let~$W$ be a subgroup of~$\H(k)$ containing~$\N(k)$ such that~$W^+\subseteq \N(k)$. Then~$\varphi(W)^+$ is reduced to the neutral element~$e$ of~$\Q(k)$.
\end{corollary}
\begin{proof}Assume by contradiction that~$\varphi(W)$ contains~$U=\U(k)$ for a non trivial algebraic (one parameter) unipotent subgroup~$\U$ of~$\Q$.
Then
\[
\stackrel{-1}{\varphi}(\U(k))\leq\stackrel{-1}{\varphi}(\varphi(W)).
\]
By hypothesis the right hand side is~$W$. The left-hand side is the group~$L=\L(k)$ of rational points of the algebraic subgroup~$\stackrel{-1}{\varphi}(\L)$ of~$\H$. Reformulating, we have~$L\leq W$, hence
\[
L^+\leq W^+\leq \N(k).
\]
and thus~$\varphi(L^+)=\{e\}$. We also have, using Proposition~\ref{propoliftuni} \eqref{propplus2}, 
\[
\{e\}=\varphi(L^+)=U^+=U.
\]
Thus~$U$, and hence~$U$, is trivial, and this is our contradiction.
\end{proof}

\subsubsection{Finite generation and Finite index subgroups}
We start with a standard fact.
\begin{lemma} \label{Lemma(F)}
Let~$M=\mathbf{M}(\Q_S)$ the topological group of rational points of an algebraic group~$\mathbf{M}$ over~$\Q_S$. For every integer~$n$, there are only finitely many open subgroups of index~$n$ in~$M$.
\end{lemma}
\begin{proof} We can argue for each place separately. 

For an archimedean place, we know that~$\M(\R)$ has finitely many connected components (by Whitney Theorem, cf.~\cite[{\S}3.2 Cor.~1 p.\,120]{PR}), from which follows: the neutral component of~$\mathbf{M}(\R)$ is open, hence is a minimal open subgroup, and is of finite index.

We turn to ultrametric places for the rest of the proof. We may assume~$M$ is (Zariski) connected. We can decompose~$\M(\Q_p)$ as a product of its unipotent radical~$\U(\Q_p)$ with a Levi factor~$\mathbf{S}(\Q_p)$. This Levi factor~$\mathbf{S}$ is algebraically the almost direct product three type of components: the isotropic quasi-simple quasi-factors, the maximal split central torus, and anisotropic groups (the anisotropic quasi-factors and the maximal anisotropic central torus). The image in~$\mathbf{S}(\Q_p)$ of the group of rational points of these components is of finite index (use Lemma~\ref{Lemme epi ouvert cofini} to the epimorphism from the direct product of these components to their product in~$\M$).

Let start with the case where~$\M$ is anisotropic. Then~$\M(\Q_p)$ is compact (\cite[3.18-19]{BorelTits},\cite[{\S}3.1 Theorem~3.1]{PR}). We notice that some neighbourhood of the origin is (topologically) finitely generated, using for instance the Lie group exponential which has a positive radius of convergence. Hence~$\M(\Q_p)$ is itself topologically finitely generated. By applying~\cite[III-\S4.1 Proposition9]{SerreLNM5} this implies that there are finitely many open subgroups of finite index. 

The case of a split torus is easily directly checked, it amounts to the finiteness of finite index subgroups of~${\Q_p}^\times$ of given index. (Actually the profinite completion of~${\Q_p}^\times$, isomorphic to~$\Z_p^\times\times\widehat{\Z}$ is topologically of finite type.)

By Lemma~\ref{Lemme Unip prof}, the unipotent group~$\U(\Q_p)$ has no non trivial finite index subgroup, even if we replace~$\Q_p$ by any characteristic~$0$ field. (N.B.: The profinite completion of~$\U(\Q_p)$ is thus trivial.)

We are left with the case of an isotropic quasi-simple group~$\M$. By~\cite[Corollaire~6.7]{BorelTits}, every subgroup of finite index of~$M$ will contain~$M^+$. (In characteristic~$0$, relying on the argument above for~$\U$ suffice.) We are done using the Proof of~\cite[Proposition~6.14]{BorelTits}, where in characteristic~$0$ the inseparable degree~$q$ is~$1$, starting from~``Il reste à faire voir[...]''.

N.B.: Along these lines, one actually proves that the profinite completion of~$M$ is topologically of finite type, which is a stronger property by~\cite[III-\S4.1 Proposition9]{SerreLNM5}. More explicitly, at a $p$-adic place, if~$r$ is the maximal rank of a central split torus (that is the rank of the group~$X_{\Q_p}(\M)$ of $\Q_p$-rational characters), then the profinite completion of~$M$ is an extension of a finite group by the product of~$\widehat{\Z}^r$ by a topologically $p$-nilpotent algebraic $p$-adic Lie subgroup of~$M$.
\end{proof}
\begin{lemma}\label{Lemme M finite index} Let~$\L$ be an algebraic subgroup of an algebraic group~$\G$ over~$\Q_S$. 
Let~$\Lpp$ be a finite index subgroup of~$L=\L(\Q_S)$, and~$H$ be a subset of~$G=\G(\Q_S)$.

Then~$\Mpp=\bigcap_{h\in H} h^{-1}\Lpp h$ is a finite index subgroup of~$L^H=\bigcap_{h\in H} h^{-1}L h$.
 
\end{lemma}
\begin{proof} Let us first observe The group~$L^H$ is algebraic over~$\Q_S$.

As~$\Lpp$ is open and of finite index, say~$n=[L:\Lpp]\in\Z_{>0}$, in~$L$, the subgroup~\[M_0=\Lpp\cap L^H\] of~$L^H$ is open and of finite index at most~$n\in\Z_{>0}$. The group~$L^H$ is normalised by~$H$, the conjugate each conjugate~$h^{-1} M_0 h$, with~$h$ in~$H$, is also open subgroup of~$L^H$ of index~$n$.
But by previous Lemma~\ref{Lemma(F)}, applied for~$\mathbf{M}=\bigcap_{h\in H} h^{-1}\L h$, the subgroups~$h^{-1} M_0 h$ range through at most finitely many subgroups. Their intersection, which is actually~$M$ is hence of finite index, at most~$n^m$ where~$m=\#\{h^{-1} M_0 h| h\in H\}$.
\end{proof}

\subsection{Creating unipotent invariance} Here is a version of a standard argument to obtain unipotent groups stabilising a limit of translates of homogeneous measures.

We say that a sequence~$(\mu_i)_{i\geq 0}$ of bounded measures \emph{converges strongly} to a limit~$\mu_\infty$ if there is a sequence~$(\eps_i)_{i\geq 0}$ of class~$o(1)$ in~$\R$ such that, for any bounded function~$\varphi$
\[
\abs{(\mu_i-\mu_\infty)(\varphi)}\leq\eps_i\cdot \Nm{\varphi}.
\]

\begin{proposition}\label{unipotescence} Let a locally compact group~$G$ acts continuously on a locally compact topological space~$X$. Let~$(g_i)_{i\geq0}$ be a sequence in~$G$, let~$\mu$ be a probability on~$X$, and assume the sequence of translated measures~$(g_i\cdot\mu)$ is weakly converging toward some limit~$\mu_\infty$. Assume there is a sequence~$h_i$ converging to the neutral element such that
\begin{itemize}
\item the sequence~$h_i\cdot \mu$ converges strongly to~$\mu$;
\item the sequence~$g_i h_i {g_i}^{-1}$ has a limit point~$g_\infty$.
\end{itemize}

Then~$\mu_\infty$ is fixed under translation by~$g_\infty$.
\end{proposition}
\begin{proof} Let~$\varphi$ be a test function: continuous with compact support. We write
$$g_i h_i {g_i}^{-1} g_i\mu (\varphi)=g_i h_i \mu (\varphi) = g_i \mu (\varphi) +o(1)\Nm{\varphi}$$
by strong convergence. But~$g_i \mu (\varphi)=\mu_\infty(\varphi)+o(1)$, by weak convergence. We get
\begin{equation}\label{combi1} g_i h_i {g_i}^{-1} g_i\mu (\varphi)=\mu_\infty(\varphi)+o(1).\end{equation}
On the other hand, we can write~$g_i h_i {g_i}^{-1} =o(1)g_\infty$. Note that, by uniform continuity of~$\varphi$, one has 
$\Nm{o(1)\cdot\varphi-\varphi}=o(1)$. Hence, for any probability~$\nu$, we have
$$o(1)\nu(\varphi)=\nu(o(1)\varphi)=\nu(\varphi)+o(1),$$
with error term independant of~$\nu$. Applying to~$\nu=g_i\mu$, we get
$$g_i h_i {g_i}^{-1} g_i\mu (\varphi)=g_\infty o(1) g_i\mu (\varphi)=g_\infty g_i\mu (\varphi) + o(1).$$
But~$g_i\mu$ converges weakly to~$\mu_\infty$, and
\begin{equation}\label{combi2}
g_\infty g_i\mu (\varphi) =g_i\mu ({g_\infty }^{-1}\varphi)= \mu_\infty ({g_\infty }^{-1}\varphi)+o(1)=g_\infty\mu_\infty (\varphi) +o(1).
\end{equation}
Combining~\eqref{combi1} and~\eqref{combi2} yields
$$\mu_\infty(\varphi)+o(1)=g_\infty\mu_\infty (\varphi) +o(1).$$
This concludes.
\end{proof}

\begin{lemma}\label{Lemma strong convergence}
 Let~$G$ be a locally compact group with countable basis of neighbourhood. Let~$\mu$ be a bounded measure on~$G$ and~$\Omega$ an open subset. 

Then~$g\mu\restriction_\Omega$ converges strongly to~$\mu\restriction_\Omega$ as~$g$ converges to the neutral element~$e$ of~$G$.
\end{lemma}
\begin{proof} Let~$(V_n)_{n\geq 0}$ be a decreasing countable basis of neighbourhood of~$e$ in~$G$. Define
\[
\Omega_n=\{\omega\in\Omega|\omega\cdot V_n\subseteq \Omega\}.
\]
As~$\Omega$ is open, and a neighbourhood of each of its points, for each~$\omega$ in~$\Omega$ there is some~$n$ for which~$\omega\cdot V_n\subseteq\Omega$. In short, we have the increasing union
\[
\Omega=\bigcup_{n\geq 0}\Omega_n.
\]
By countable additivity~$\lim_{n\to\infty}\mu(\Omega_n)\to\mu(\Omega)$. Or equivalently (as~$\mu$ is bounded)
\[
\mu(\Omega\smallsetminus\Omega_n)\to 0.
\]
As~$g$ converges to~$e$, it eventually belongs to~$V_n$. Hence, eventually,
\[\Omega\smallsetminus g\Omega\subseteq \Omega\smallsetminus\Omega_n.\]
It follows~$\lim_{g\to e}\Omega\smallsetminus g\Omega\leq \mu(\Omega\smallsetminus\Omega_n)$. As~$n$ is arbitrary,
\[\lim_{g\to e}\Omega\smallsetminus g\Omega\leq \lim_{n\to\infty}\mu(\Omega\smallsetminus\Omega_n)=0.\]
We now write, in terms of the symmetric difference~$\Omega\Delta g\Omega$
\[
\abs{g\mu\restriction\Omega-\mu\restriction_\Omega}
=
\mu\restriction_{\Omega\Delta g\Omega}
=
\mu\restriction_{\Omega\smallsetminus g\Omega}+g\mu\restriction_{g\Omega\smallsetminus\Omega}
=
\mu\restriction_{\Omega\smallsetminus g\Omega}+g\mu\restriction_{\Omega\smallsetminus g^{-1}\Omega}.
\]
For any bounded test function~$\varphi$ on~$G$, using~$\mu_X(\abs{\varphi})\leq\mu(X)\Nm{\varphi}$, we have
\begin{align*}
		\abs{(g\mu\restriction\Omega-\mu\restriction_\Omega)(\varphi)}
	\leq
		\abs{g\mu\restriction\Omega-\mu\restriction_\Omega}(\abs{\varphi})
	&=
		\mu\restriction_{\Omega\smallsetminus g\Omega}(\abs{\varphi})
		+
		g\mu\restriction_{\Omega\smallsetminus g^{-1}\Omega}(\abs{\varphi})
	\\&\leq
		\mu({\Omega\smallsetminus g\Omega})\cdot\Nm{\varphi}
		+
		\mu(\Omega\smallsetminus g^{-1}\Omega)\cdot\Nm{g^{-1}\varphi}
	\\&=
		o(1)\cdot\Nm{\varphi}.
\end{align*}
This is what we had to prove.
\end{proof}

\begin{lemma} Let~$G=\G(\Q_S)$ be the group of points of a linear algebraic group over~$\Q_S$ with Lie algebra~$\lie{g}$. Let~$(X_i)_{i\geq0}$ resp.~$(g_i)_{i\geq0}$ be a sequence in~$\lie{g}$ resp.~$G$ such that~$\lim_{i\geq0} X_i$ is nilpotent and that~$(g_iX_i{g_i}^{-1})_{i\geq0}$ has a limit~$X_\infty$ in~$\lie{g}$. 

Then~$X_\infty$ is nilpotent.
\end{lemma}\label{Lemme nilpotent stabilite}
\begin{proof} We may embed~$G$ in a general linear group~$GL(V)$ and assume~$G=GL(V)$. We view~$\lie{gl}(V)$ as a matrix space. We consider the characteristic polynomial, as a continuous function~$\chi$ on~$\lie{gl}(V)$. 
Then, on the one hand
\[
\lim_i \chi(X_i)=\chi(0)
\]
and on the other hand
\[
\lim_i \chi(g_iX_i{g_i}^{-1})=\chi(X_\infty).
\]
But~$\chi$ is a conjugation invariant: for all~$i\geq0$, we have~$ \chi(g_iX_i{g_i}^{-1})= \chi(X_i)$. Finally~$\chi(X)=\chi(0)$, which is a criterion of nilpotence for~$X_\infty$.
\end{proof}

\subsection{Boundedness criterion and Geometric stability} This kind of results is related to the work of Richardson:~\cite{RichardsonConjugacy,Richardsontuple}. It is an instance of a phenomenon coined as ``stability''
by Mumford (see~\cite{Richardsontuple} for discussion of Mumford stability in this context; see also Corollary~\ref{corostabilityrichardson} below).
We give provide a method which relies on~\cite{Lemmanew}, thus emphasising the meaning of~\cite{Lemmanew} in the context of stability properties.

\begin{proposition}\label{Propobddcrit} Let~$ \G $ be a Zariski connected linear~$\Q_S$-group, and let~$H$ be a~$\Q_S$ Lie subgroup whose Zariski closure~$\H^\alg$ in~$G$ is a Zariski connected~$\Q_S$-subgroup which is reductive in~$ \G $,
following~\eqref{H reductif dans}. For any sequence~$\left( g _i\right)_{i\geq0}$ in the group~$G=\G(\Q_S) $, the following are equivalent.
\begin{enumerate}
\begin{subequations}
\item \label{P1} If~$Z=\Z_\G(H)(\Q_S)$ denotes the centraliser of~$H$ in~$G$, then
\begin{equation}\label{eqp1}
\left( g _i\right)_{i\geq0}\text{ is of class }O(1)  Z.
\end{equation}
\item \label{P2} If~$\lie{h}$ (resp.~$\lie{g}$, resp.~$\Ad$) denotes the Lie algebra of~$ H $ (resp. of~$ G $, resp the adjoint representation of~$ G $ on~$\lie{g}$), then for any element~$X$ in~$\lie{h}$, the sequence
\begin{equation}\label{eqp2}
\left(\Ad_{ g _i}(X)\right)_{i\geq0}
\end{equation}
is bounded in~$\lie{g}$.
\item \label{P3} (with the same notations) For any bounded sequence~$\left( X_i\right)_{i\geq0}$ in~$\lie{h}$, the sequence
\begin{equation}\label{eqp3}
\left(\Ad_{ g _i}(X_i)\right)_{i\geq0}
\end{equation}
is bounded in~$\lie{g}$.
\item \label{P4} There does not exist a sequence~$\left( X_i\right)_{i\geq0}$ in~$\lie{h}$, converging to~$0$, and such that
\begin{equation}\label{eqp4}
\left(\Ad_{ g _i}(X_i)\right)_{i\geq0}
\end{equation}
has a non zero limit point in~$\lie{g}$.
\end{subequations}
\end{enumerate}
\end{proposition}
\begin{proof} 
The implication~$\eqref{P1}\Rightarrow\eqref{P2}$ is immediate. Assume~\eqref{P1}. As~$ \Z_\G(\H)$ fixes~$\lie{h}$, and in particular, fixes~$X$, we may replace~$\left( g _i\right)_{i\geq0}$ by a bounded sequence. Thus the sequence~\eqref{eqp2} wil be bounded.

The equivalences~$\eqref{P2}\Leftrightarrow\eqref{P3}\Leftrightarrow\eqref{P4}$ are merely reformulations. Here are some details.

Firstly,~$\eqref{P2}$ is a specialisation of~$\eqref{P3}$ to constant sequences~$(X_i)_{i\geq0}=(X)_{i\geq0}$.

Conversely,~$\eqref{P3}$ reduces to~$\eqref{P2}$, by decomposing~$X_i$ into a fixed base:~$X_i=\sum_j a_{i,j} e_j$, using finitely many bounded sequences~ $(a_{i,j})_{i\geq0}$. We then apply~$\eqref{P2}$ to each of the case~$X=e_j$, then deduce that each sequence~$\left(\Ad_{ g _i}(e_j)\right)_{i\geq0}$ is bounded. Consequently, each sequence~$\left(\Ad_{ g _i}(a_{i,j}e_j)\right)_{i\geq0}$ is bounded. Finally, taking a finite sum~$\left(\Ad_{ g _i}(\sum_j a_{i,j} e_j)\right)_{i\geq0}$ of bounded sequences yields a bounded sequence.

\paragraph{NB:}
 The equivalence~$\eqref{P2}\Leftrightarrow\eqref{P3}$ falls merely under an (easy) instance of Banach-Steinhaus uniform boundedness principle.
The former~$\eqref{P2}$ states pointwise boundedness, the latter~$\eqref{P3}$ states a sequential form of boundedness.

The negation of~$\eqref{P3}$ claim there is a sequence~$(X_i)_{i\geq0}$ such that the sequence~\eqref{eqp3} is unbounded. Choose a $\Q_S$-norm on~$\lie{g}$. We can find a sequence of scalars~$(\lambda_i)_{i\geq0}$ in~$\Q_S$ such that~$0<\limsup\Nm{\lambda_i\cdot\Ad_{ g _i}(X_i)}<\infty$ .
Consequently, the sequence
\[\left(\Ad_{ g _i}(\lambda_i\cdot X_i)\right)_{i\geq0}\] takes infinitely many value in the compact set
$$\left\{X\in\lie{g}~\middle|~1/2\leq\frac{\Nm{X}}{\limsup\Nm{\lambda_i\cdot\Ad_{ g _i}(X_i)}} \leq 2\right\}.$$
The sequence~$\left(\Ad_{ g _i}(\lambda_i\cdot X_i)\right)_{i\geq0}$ has a limit point in this compact. This limit point may not be~$0$. This negates~$\eqref{P4}$ for the sequence~$\left(\lambda_i\cdot X_i)\right)_{i\geq0}$.

Assume the negation of~$\eqref{P4}$. We can find a sequence of non invertible scalars such that~$(\lambda_i)_{i\geq0}$ with limit~$0$ and such that~$\Nm{X_i}=O(\Nm{\lambda_i})$. One can deduce the negation of~$\eqref{P3}$ for the sequence~$\left( X_i/\lambda_i)\right)_{i\geq0}$. 

The remaining implication~$\eqref{P1}\Leftarrow\eqref{P2}$ is proved below.
\end{proof} 

\begin{lemma}\label{Lemma centralisateur fixtateur} In the setting of Proposition~\ref{Propobddcrit}, the centraliser of~$H$ in~$\G$ is the the fixator of~$\lie{h}$ in the adjoint representation action~$\Ad$ of~$\G$.
\end{lemma}
\begin{proof} Let~$U$ be a sufficiently small neighbourhood of~$0$ in~$\lie{h}$ such that the exponential map~$\exp_H:U\to H$ is convergent on~$U$. The exponential map intertwins the adjoint representation~$\Ad$ of~$G$ and the adjoint action of~$G$ on itself: for~$g$ in~$G$ and~$X$ in~$U$, we have
\[
g\exp(X)g^{-1}=\exp(\Ad_g(X)).
\]
Hence the centraliser of~$\exp(U)$ is the fixator of~$\lie{h}$. The former is the centraliser of~$H$, as~$U$ is Zariski dense in~$H$ (recall that~$H$ is Zariski connected); the latter is the fixator of~$\lie{h}$, by linearity.
\end{proof}

We first start with a special case since that is the setting of~\cite{Lemmanew}. (See also~{\cite[Theorem 2.16]{PR}}.)
\begin{proof}[Proof of the implication~$\eqref{P1}\Leftarrow\eqref{P2}$, for semisimple~$ G $. ]
~
Assume that~$ G $ is semisimple.

Applying \cite[Theorem~2.1]{LemmaA}, we get be a subset~$Y$ such that~$ G =Y\cdot  Z$. Consequently any class~${g_i}  Z$ can be written~$ y_i  Z$ with~$y_i$ in~$Y$. We can replace~$( g _i)_{i\geq0}$ by~$(y_i)_{i\geq0}$ without changing the validity of~\eqref{P1}.

As~$Z$ acts via identity on~$\lie{h}$, we have
$$ \forall X \in\lie{h}, \Ad_{ g _i}(X)=\Ad_{{y}_i}(X).$$
Thus we can replace~$( g _i)_{i\geq0}$ by~$(y_i)_{i\geq0}$ without changing the validity of~\eqref{P2}.

From now on, we assume  the identity~$( g _i)_{i\geq0}=(y_i)_{i\geq0}$. We consider the space
\[{V}:=\Hom_{\Q_S}\left(\lie{h},\lie{g}\right)\]
of~$\Q_S$-linear maps~$\lie{h}\to\lie{g}$, with adjoint action of~$ G $ on the values on a $\Q_S$-linear homomorphisms
\[( g \phi)(X):=\Ad_{ g }(\phi(X)).\]

Note that the stabiliser of the identity embedding~$\iota:\lie{h}\xrightarrow{X\mapsto X}\lie{g}$ is the centraliser of its image~$\lie{h}$. This centraliser is~$ Z$, as~$H$ is Zariski connected and characteristics are~$0$. For any~$\omega$ in~$ H $, the map~$\omega\cdot\iota$ has same image~$\lie{h}$, and same stabiliser~$ Z$.

Consider a Zariski dense bounded subset~$\Omega$ in~$H$. 
Remark that, for~$\omega\in\Omega$ and~$g$ in~$G$, the map~$g\cdot\omega\cdot \iota$ is the composite~$\left.\Ad_{g}\right|_\h \circ\Ad_\omega:\h\to\h\to\g$. Fix some norm on~$\h$ and on~$\g$, and consider the corresponding operator norm on~$V$. For the operator norm, 
$$\Nm{g_i\cdot\omega\cdot \iota}=\Nm{\left.\Ad_{g}\right|_\h \circ\Ad_\omega}\leq\Nm{\left.\Ad_{g_i}\right|_\h}\cdot\Nm{\Ad_\omega}= \Nm{g_i\cdot\iota}\cdot\Nm{\Ad_\omega}.$$
From hypothesis~\eqref{P2}, we know that~$\Ad_{y_i}(\iota)$ is bounded in~${V}$. As~$\Omega$ is also bounded, so is~$\Ad_\omega$ as~$\omega$ ranges through~$\Omega$.
Consequently,~$\Nm{g_i\cdot\Omega\cdot \iota}$ is uniformly bounded.
We may apply~\cite[Theorem~2]{Lemmanew} to the representation~$V$. Its conclusion uses the (pointwise) fixator of~$\Omega\cdot \iota$, which is~$Z$ by Lemma~\ref{Lemma centralisateur fixtateur}, and is normalised by~$\Omega$. This conclusion asserts that the sequence~$(y_i)_{i\geq0}$ is of class~$O(1) Z$.
\end{proof}
We will need a slight extension of this special case.
\begin{proof}[Proof of the implication~$\eqref{P2}\Leftarrow\eqref{P1}$, assuming~$ G $ is reductive.] Assume that~$ G $ is reductive.

Let us write~$ G =C\cdot D$ as an almost direct product of its center~$C$ by its derived group~$D$. Decompose correspondingly~$ g _i=c_i\cdot d_i=d_i\cdot c_i$.

Note that~\eqref{eqp1} does not depend on~$c_i$ as~$C$ is contained in~$Z$. Neither do the sequence~\eqref{eqp2}, has~$C$ is contained  in~$Z$ which fixes~$\lie{h}$, to which~$X$ belongs.

We can replace~$ g _i$ by~$d_i$. The result can be deduced from the preceding case, applied to the semisimple group~$D$, its subgroup~$( H  C)\cap D$ and the sequence~$(d_i)_{i\geq0}$.
\end{proof}

We will now reduce the general case to the preceding special case. We need this general case as we will apply it to a subquotient~$\widecheck{G}$ of our reductive group~$G$; the former may not be reductive even though the latter is.
\begin{proof}[Proof of the implication~$\eqref{P1}\Leftarrow\eqref{P2}$, in general]
Let~$U$ be the unipotent Radical of~$ G $. By~\cite[Theorem~7.1]{Mostow56}, we can decompose~$ G =M\cdot U$ as a semi product of a maximal reductive subgroup~$M$, which we choose to contain~$ H $, with~$U$.


Let~$\overline{G}=G/U$ be the quotient of~$ G $ by its unipotent radical.  We denote~$\overline{\lie{g}}$ the corresponding Lie algebra. Let us denote the image of the sequence~$( g _i)_{i\geq0}$ in~$\overline{G}$ with~$(\overline{g}_i)_{i\geq0}$.

Note that the quotient map from~$M$ to~$\overline{G}$ is an isogeny, and thus its differential at the neutral element identifies the Lie algebra of~$M$ with that, say~$\overline{\lie{g}}$, of~$\overline{G}$. We will denote~$\lie{h}$ the Lie algebra of~$H$ and~$\overline{\lie{h}}$ its image in~$\overline{\lie{g}}$. This is the Lie algebra of the image group~$\overline{H}$ of~$H$ in~$\overline{G}$.

As above, we consider the space~${V}:=\Hom_{\Q_S}(\lie{h},\lie{g})$, with its element~$\iota$ and we consider moreover its variant~$\overline{V}:=\Hom_{\Q_S}(\overline{\lie{h}},\overline{\lie{g}})$, both as representations of~$ G $. We denote with~$\overline{\iota}$ the identity map~$\overline{\lie{h}}\to\overline{\lie{g}}$. The representation~$\overline{V}$ is a quotient of the representation~${V}$, via a quotient map which sends~$\iota$ to~$\overline{\iota}$.
 
By hypothesis~\eqref{P2}, the sequence~$({g_i}\cdot{\iota})_{i\geq0}$ is bounded in~${V}$. It follows that~$(\overline{g_i}\cdot\bar\iota)_{i\geq0}$ is bounded in~$\overline{V}$. We apply the previous case to~$\overline{G}$, its subgroup~$ \overline{H }$, and the sequence~$(\overline{g_i}\cdot\bar\iota)_{i\geq0}$, and deduce that the sequence
\[
(\overline{g_i})_{i\geq0}\text{ is of class~}O(1)Z_{\overline{G}}(H),
\] 
where~$Z_{\overline{G}}(H)$ is the centraliser of~$H$ in~$\overline{G}$. 

By~Lemma~\ref{Lemmacentmodulo2} we deduce that, for some finite set~$A$, the sequence
\[
({g_i})_{i\geq0}\text{ is of class~}O(1)U\cdot A\cdot Z,
\] 
that is: we can write~$g_i=b_i u_i a_i z_i$ where~$(b_i)_{i\geq0}$  is a bounded sequence, the~$u_i$ belong to~$U$, the~$a_i$ to~$A$ and the~$z_i$ to~$H^\prime$.
We get
\[g_i\cdot \iota=b_i\cdot u_i\cdot a_i\cdot  z_i \cdot \iota= b_i\cdot u_i\cdot a_i\cdot  \iota.\]
Thus the boundedness of~$(g_i\cdot \iota)_{i\geq 0}$ in~$V$ reduces to
that of the~$(u_i\cdot a\cdot \iota)_{i\geq 0}$ for~$a$ in~$A$. But the orbit of unipotent group in a linear representation is closed, by the Kostant-Rosenlicht theorem \cite[2.4.14]{SpringerLAG}. For each~$a$ in~$A$, the orbit map~$U/(U\cap Z) \to U\cdot a\cdot\iota$ is proper. Thus the sequence~$(u_i)_{i\geq 0}$ is actually bounded in~$U/(U\cap Z)$.
The sequence~$a_i$ is obviously bounded. Finally~$g_i=b_i u_i a_i z_i$ becomes
\[
({g_i})_{i\geq0}\text{ is of class~}O(1)(O(1)O(1)(U\cap Z))\cdot Z=O(1)Z.
\]


%
%
%
This conclude the proof of the implication~$\eqref{P1}\Leftarrow\eqref{P2}$ in the general case, and conclude the proof of~Proposition~\ref{Propobddcrit}.

\end{proof}

Let us note the following corollary, which, arguably, summarise the part of information inside Proposition~\ref{Propobddcrit} with the deepest meaning.
\begin{corollary}\label{corostabilityrichardson} The orbit
$$ G \cdot \iota $$
is closed in~${V}$. Equivalently, the map
$$ \left. G \middle/{Z}\right.\xrightarrow{g\mapsto g\cdot \iota}{V}$$
is proper.
\end{corollary}
Our proof, relying on~\cite{Lemmanew}, actually uses convexity properties for symmetric spaces (and Bruhat-Tits buildings), similar to the
work of Kempf-Ness \cite{KempfNess}, with a generalised
Cartan decomposition of Mostow (and an analogue for buildings). The ultrametric part of~\cite{Lemmanew}, which is the suject of~\cite{Lemmap}, does rely on Rosenlicht results.

\subsection{Centralizer and cosets} We give some lemmas about centraliser and quotient groups. This is an adaptation of arguments which can be found
for instance in~\cite[{\S}5]{EMSAnn}.

\begin{lemma}\label{Lemmacentmodulo}Let~$G$ be a group of rational points of an algebraic group~$\G$ over a field~$k$ of\footnote{It may be that the algebraic groups are smooth suffices in general characteristic.} characteristic~$0$, and let~$\H$ and~$\L$ be two algebraic subgroups of~$\G$.

 Write, after~\cite[{\S}5]{EMSAnn}
$$\Z(\H,\L)=\left\{g\in \G~\bigl|\forall h\in \H,~[g,h]\in \L\right\},$$
the ``$\pmod \L$-centraliser'' of~$\H$ in~$\G$.

\begin{enumerate}
\item We remark that~$\Z(\H,\L)$ is invariant  under the centra\-li\-ser~$\Z(\H)$ acting by right translations.
\item Assume~$\L$ is normalised by~$\H$. Then~$\Z(\H,\L)$ is invariant  under~$\L$ acting by left translations.
\item Under the latter assumption,~$\Z(\H,\L)$ is moreover a finite union of double co\-sets~$\L g\Z(\H)$. \label{centralisermodulo}
\end{enumerate}
\end{lemma}

\begin{proof}With start with the two first points, for which this is straightforward.
 
 The condition~$[h^{-1},g]=-[g,h^{-1}]\in \L$ can be rewritten as~
\begin{equation}\label{eq1 Z}
ghg^{-1}\L=h\L.
\end{equation}
 The element~$h$ is ``centralised by~$g$ mod~$\L$'' (if~$\L$
is normal in~$\G$, the cosets~$g\L$ and~$h\L$ commute in~$\G/\L$).

If~$z$ belongs to~$\Z(\H)$, then, for any~$h$ in~$\H$
\begin{equation}\label{eq2 Z}
zhz^{-1}=h.
\end{equation}
Invoking~\eqref{eq2 Z} and then~\eqref{eq1 Z}, we get
$${gz}h({gz})^{-1}\L={g}(zhz^{-1})g^{-1}\L=g·h·g^{-1}\L=h\L$$
for any~$g$ in~$\Z(\H,\L)$. Hence the first remark in the conclusion.

If~$l$ belongs to~$\L$, then obviously~$l^{-1}\L=\L$. If~$h$ is in~$\H$, which is assumed to normalise~$\L$, we have~$h^{-1}lh\in\L$. For any~$g$ in~$\Z(\H,\L)$
$${lg}h(lg)^{-1}\L={l}(ghg^{-1})(l^{-1}\L)=l({g}hg^{-1}\L)=l(h\L)=h\cdot (h^{-1}lh)\L=h\L.$$
Hence~$l\cdot g$ belongs to~$\Z(\H,\L)$ as soon as each of the~$h^{-1}lh$ belong to~$\L$. This establishes the second statement of the conclusion.

We now prove the third point. The algebraic variety~$\Z(\H,\L)$ is made of finitely many irreducible components. Consider an irreducible component~$V$ of~$\Z(\H,\L)$. We will show that for any point~$g$ in~$V$, the map
$$\L\times \Z(\H)\to \G:(l,z)\mapsto lgz$$
is submersive at the origin. By the ``closed orbit lemma'', this orbit, with reduced structure is smooth and locally closed. By submersivity its tangent space is that of~$V$. Hence it is an open subvariety: all the orbits of~$\L\times \Z(\H)$ in~$V$ have non empty interior. As~$V$ is irreductible, two such orbits meets at some point: there can be only one orbit.

We prove the submersivity. We follow an argumentation from~\cite[{\S}5]{EMSAnn}. Without loss of generality, possibly substituting~$\H$ with~$g^{-1}\H g$, we may assume~$g$ is the neutral element. We work inside the tangent space~$\lie{g}$ at the neutral element. We let~$\lie{l}$ and~$\lie{h}$ denote the Lie algebras of~$L$ and~$H$ respectively. As~$H$ normalises~$L$, we have
\[
[\lie{h},\lie{l}]\subseteq\lie{l}.
\]
As~$H$ is reductive, the adjoint action of~$\lie{h}$ is completely reducible. We use that there is a~$\lie{h}$ stable supplementary space~$\lie{l}'$: we have
\[
[\lie{h},\lie{l}']\subseteq\lie{l}'\text{ and }\lie{g}=\lie{l}\oplus\lie{l}'.
\]
Let~$X$ be a tangent vector to~$Z(\H,\L)$ at the neutral element. We decompose accordingly
\[
X=X_\ell+X_{\lie{l}'}\text{ with }X_{\lie{l}}\in\lie{l}\text{ and }X_{\lie{l}'}\in\lie{l}'.
\]
We know use that~$X$ belongs to~$Z(\H,\L)$ (viewed as a point defined over the dual numbers~$k[\eps]/(\eps)^2$).
For a tangent vector~$Y\in\lie{h}$ of~$H$, we have that~$YXY^{-1}=\Ad_Y(X)$ belongs to~$L$. As this is a tangent vector at the origin (specialising to~$\eps=0$), we have~$\Ad_Y(X)\in\lie{l}$. The adjoint action~$\Ad$ of~$H$, at the level of Lie algebras, is given by the adjoint representation~$\mathbf{ad}$ of~$\lie{h}$, which gives that~$\Ad_Y(X)$ is represented by~$\mathbf{ad}_Y(X)=[Y,X]$. We have, for arbitrary~$Y$ in~$\lie{h}$,
\[
[Y,X]\in\lie{l}.
\]
As~$[Y,X_{\lie{l}}]\in\lie{l}$ already, we have
\[
[Y,X_{\lie{l}'}]=[Y,X]-[Y,X_{\lie{l}}]\in\lie{l}.
\]
But we have~$[Y,X_{\lie{l}}']\in\lie{l}'$ too. Hence
\[
[Y,X_{\lie{l}'}]\in\lie{l}\cap\lie{l}'=\{0\}.
\]
That is~$X_{\lie{l}'}$ commutes with~$\lie{h}$. The commutant of~$\lie{h}$ in~$\lie{g}$ is
the Lie algebra~$\lie{z}$ of~$\Z(\H)$. We have
\[
X_{\lie{l}'}\in\lie{z}.
\]
Finally
\[
X=X_{\lie{l}}+X_{\lie{l}'}\in\lie{l}+\lie{z}.
\]
On the other hand the differential of the action of~$\L\times\Z(\H)$ is given by the sum map
\[
\lie{l}\times\lie{z}\to\lie{l}+\lie{z}\subseteq\lie{g}.
\]
It can be checked on~$\lie{l}$ and~$\lie{z}$ individually, that is for the action of~$\L$ and~$\Z(\H)$ individually,
in which case this orbit map is the embedding of the respective group. The image of this differential is the tangent space of~$V$
at the neutral element. In other words we have a submersion.
\end{proof}
\begin{lemma}\label{Lemmacentmodulo2}
In Lemma~\ref{Lemmacentmodulo} assume that~$k$ is a local field of characteristic~$0$. Then, in the last point
there is a finite set~$A$ such that~$\Z(\H,\L)(k)=Z_{\G(k)}(\H(k))A\L(k)$.
\end{lemma}
\begin{proof} It suffices to work with each irreducible component of~$\Z(\H,\L)$ individually, which have been shown to be an~$\L\times\Z(\H)$ orbit. To this orbit, an homogeneous space of~$\L\times\Z(\H)$, we can apply~\cite[{\S}6.4: Proposition~6.13 p.\,316; Corollary~2 {\S}6.4 p.\,319]{PR}.
\end{proof}

\subsection{Trivial dynamics of translates}
Let~$\mu$ be a probability measure on a locally compact group~$G$ with support~$\Omega$, and let~$G$ act continuously on a locally compact Radon space~$X$, with distinguished point~$x$. We denote~$\mu_\Omega=\mu\convolution\delta_x$ the image probability measure by the orbit map at~$x$. We denote~$Z$ the centraliser of~$\Omega$ in~$G$, and consider the family of translates
\[(g\mu_\Omega)_{g\in G}\]
in the space of probabilities on~$X$.
\begin{lemma}\label{lemme trivial dyn} Let~$b_i$ be a bounded sequence in~$G$. A sequence
\begin{equation}\label{suite lemme trivial}
(\mu_i)_{i\geq 0}=(b_iz_i\mu_\Omega)_{i\geq 0},
\end{equation}
with~$z_i$ in~$Z$,
\begin{enumerate}
\item \label{trivial dyn 1}	is vaguely converging to~$0$ if and only if the sequence~$x_i=z_i\cdot x$ is diverging to infinity.
\item \label{trivial dyn 2} is a tight family if and only if the sequence~$x_i=z_i\cdot x$ is relatively compact.
\end{enumerate}
If~$(b_i)_{i\geq0}$ converges to~$g_\infty$ in~$G$ and~$(x_i)_{i\geq0}$ to~$x_\infty$ then~$(z_i\mu_\Omega)_{i\geq 0}$ 
is tightly convergent, with limit
\[
g_\infty\cdot \mu\convolution \delta_{x_\infty}: f\mapsto \int_{x\in X} f(x)\quad g_\infty\cdot\mu\convolution \delta_{x_\infty}=\int_{\omega \in G} f(g_\infty\omega\cdot x)\quad\mu,
\]
the direct image of~$\mu$ by the orbit map at~$x_\infty$, with support~$\overline{\Omega\cdot x_\infty}$.
\end{lemma}
We start by proving the ``if'' part of the two numbered statements. We start with first statement.
\begin{proof} We first note that if~$\mu$ has compact support and~$z_i\cdot x$ diverges to the infinity, then, for any compact subset~$K$ of~$X$, 
\[
\Supp(b_iz_i\mu_\Omega)=b_iz_i\Omega\cdot x=b_i\Omega\cdot z_ix
\]
will, for~$i$ big enough, be disjoint from~$K$.
\begin{proof}Let~$U$ be a compact set containing~$\{b_i|i\geq 0\}$.
 We consider the compact
\[C=\Omega^{-1}U^{-1}\cdot K.\]
By hypothesis, for~$i$ big enough, $z_ix$ is not in~$C$. Let us write~$b_i\omega z_ix$ an arbitrary element of~$\Supp(b_iz_i\mu_\Omega)$. We have
\[
\omega^{-1}{b_i}^{-1}K\subseteq C\not\ni z_ix
\]
from which we deduce~$b_i\omega z_ix\notin K$.
\end{proof}
It follows that~$z_i\mu_\Omega$ converges vaguely to~$0$, hence does not converge to a probability measure.

Let us deduce the general case, in which the support~$\Omega$ is not assumed to be compact. (This case is not needed for our main result)
For general~$\Omega$, let~$C$ be a big enough compact subset of~$G$ so that~$\Omega\cap C$ is Zariski dense in~$\Omega$. Then the centraliser of~$\Omega$ and of~$\Omega\cap C$ --- which is the support of~$\mu\restriction_C$ --- are the same. From previous argumentation, we get that~$\mu_{\Omega\cap C}:=z_i\mu\restriction_{C}\convolution\delta_x$ converges vaguely to~$0$. But, as~$C$ gets bigger,~$\mu\restriction_{C}$ converges strongly to~$\mu$. We have, for the strong norm
\[
\Nm{z_i\mu_\Omega-z_i\mu_{\Omega\cap C}}=\Nm{\mu_\Omega-\mu_{\Omega\cap C}}=\mu(\Omega\smallsetminus C)=o_{C\to G}(1).
\]
Let~$\mu_\infty$ be a limit measure of the sequence~$(z_i\mu_\Omega)_{i\geq 0}$. 
We have, for a test function~$f$, along the corresponding extracted subsequence,
\[
\mu_\infty(f)=z_i\mu_{\Omega} + o_{i\to\infty}(1)=z_i\mu_{\Omega\cap C}(f)+o_{C\to G}(1)\Nm{f}_\infty + o_{i\to\infty}(1).
\]
But we proved
\[
z_i\mu_{\Omega\cap C}(f)=0+o_{i\to\infty}(1).
\]
Finally~$\Nm{\mu_\infty}\leq o_{C\to G}(1)$, hence~$\mu_\infty=0$.

\end{proof}
We know pass to the ``if'' part of statement~\eqref{trivial dyn 2}.
\begin{proof} As~$(b_i)_{i\geq0}$ is bounded in~$G$ and~$(x_i)_{i\geq0}$ in~$X$, these sequences converge to some~$g_\infty$ in~$G$ and~$x_\infty$ in~$X$ along any ultrafilter of the integers. By the concluding assertion of the Lemma, the sequence~\eqref{suite lemme trivial} of measures converges to a probability measure along this ultrafilter. We proved that the sequence of measures converges in the space of probabilities along any ultrafilter of the integers. This is one characterisation that the sequence~\ref{suite lemme trivial} of measures is relatively compact in the space of probabilities. This is the definition of a tight family.
\end{proof}
We turn to the ``only if'' statements. They are implied by the ``if'' part of the other numbered assertion. We start with assertion~\eqref{trivial dyn 1}.
\begin{proof} If the sequence~$(x_i)_{i\geq0}$ is not divergent we may extract a bounded subsequence. As we have proved, that implies that the corresponding subsequence of~$(\mu_i)_{i\geq0}$ is relatively compact in the space of probabilities, hence there is a convergent subsequence to a probability, that is to a non zero measure. The original sequence~$(\mu_i)_{i\geq0}$ cannot converge to~$0$.
\end{proof}
We now finish the converse implication in assertion~\eqref{trivial dyn 2}.
\begin{proof}If the sequence~$(x_i)_{i\geq0}$ is not bounded then we may extract a divergent subsequence. As we have proved, that implies that the corresponding subsequence of~$(\mu_i)_{i\geq0}$ converges to~$0$. It cannot be bounded in the space of probabilities, hence neither can the original sequence.
\end{proof}
We still have to prove the concluding assertion. We will only cover the compact support case.
\begin{proof}What is to prove is that for any (bounded continuous) test function~$f:X\to\R$, we have the convergence
\[
\int_\Omega f(b_i\omega z_ix)\xrightarrow[i\to\infty]{} \int_\Omega f(g_\infty\omega x_\infty).
\]
We define the following functions~$\Omega\to\R$
\[
f_i:\omega\mapsto f(b_i\omega z_ix)\quad f_\infty:\omega\mapsto f(g_\infty\omega x_\infty).
\]
We note that we have simple convergence of~$(f_i)_{i\geq0}$ to~$f_\infty$, hence uniform convergence over compacts.
For~$\eps>0$ let~$K$ be a big enough compact of~$G$ so that~$\mu(\Omega\smallsetminus K)\leq\eps$.
We decompose
\[
\mu(f_i-f_\infty)=\mu\restriction_K(f_i-f_\infty)+\mu\restriction_{\Omega\smallsetminus K}(f_i-f_\infty)
\]
where the latter term is absolutely bounded by
\[
\abs{\mu\restriction_{\Omega\smallsetminus K}(f)}\leq\Nm{f}\cdot\Nm{\mu\restriction_{\Omega\smallsetminus K}}\leq\Nm{f}\cdot\eps.
\]
and the first term by
\[
\abs{\mu\restriction_K(f_i-f_\infty)}\leq\mu(K)\Nm{(f_i-f_\infty)\restriction_K)}=o_{i\to\infty}(1).
\]
Combining the bounds, we get
\[
\abs{\mu(f_i-f_\infty)}\leq\eps+o_{i\to\infty}(1)\quad\text{ hence }\limsup\abs{\mu(f_i-f_\infty)}\leq \eps
\]
and as~$\eps>0$ was arbitrary, we get~$\limsup\abs{\mu(f_i-f_\infty)}\leq0$, and, as absolute values are positive,
\[\lim\abs{\mu(f_i-f_\infty)}=0.\]
This translates into the desired convergence.
\end{proof}

\section{Linearisation. A Review in the $S$-arithmetic setting}\label{AppA}
\emph{We use the notations from section~\ref{RatnerNotations} (singular sets) and section~\ref{subsection linearisation} (linearisation).}
\vspace{.5em}

The linearisation method of Dani and Margulis is the second essential ingredient to our proof. Specifically, our proof (we refer to equation~\ref{eq0}) relies on an $S$-arithmetic version of \cite[Proposition~3.13]{EMSAnn}. In this appendix, we prove this analogue, namely Proposition~\ref{AProp313}.

The proof we present follows the original proof of \cite{EMSAnn}, but makes use of simplifications introduced in \cite{KT} using the Besicovich covering property and $(C,\alpha)$-good functions. With these simplifications, the multidimensional case is treated e\-qual\-ly to the one-dimensional case.

\subsection{Double fibration of $G/\Gamma_N$} We defined the vector~$p_L$ of~$V$ in section~\ref{subsection linearisation}.
The fixator of~$p_L$ is~$N$ and we denote~$\Gamma_N=\Gamma\cap N$.

We have a double fibration of $G/\Gamma_N$ arising on the one hand from the natural projection
\(
\phi: G/\Gamma_N\to G/\Gamma,
\) and on the other hand from
\begin{equation}\label{orbit map eta}
\text{the orbit map $\eta_L:G\to V$ given by $g\mapsto g\cdot p_L$}
\end{equation}
(see~\ref{subsection linearisation}), which factors through a map $\overline{\eta}_L$ from $G/\Gamma_N$. 

\emph{These two maps induce a correspondence~$G/\Gamma\xleftarrow{\phi} G/\Gamma_N\to V$ from~$G/\Gamma$ to~$V$ which allows to translate dynamical questions on~$G/\Gamma$ into problems in the linear space~$V$.}

The orbit $\Gamma\cdot p_L$ is discrete by Lemma~\ref{Gorbit discrete}.
There are two important consequences for us. Firstly, the product map
\[\phi\times\overline{\eta}_L:G/\Gamma_N\to G/\Gamma\times V\] 
is proper. Secondly, for any compact set $E$ of $V$, the formula
\[g\Gamma\mapsto\#\left(\stackrel{-1}{\phi}(g\Gamma)\cap\stackrel{-1}{\overline{\eta}_L}(E)\right)\] 
takes finite values and defines a  counting func\-tion~$\chi_E:G/\Gamma\to\Z_{\geq 0}$. As the stabiliser of~$p_L$ in~$\Gamma$ is~$\Gamma_N=N\cap \Gamma$, we have the alternative formula
\[
\chi_E(g\Gamma)=\#\left(g\Gamma\cdot p_L\cap E\right).
\]

We shall make use of the following.
\begin{lemma}\label{upper semicontinuity} The function~$\chi_E$
is upper semi-continuous.
\end{lemma}
More specifically we will use its immediate corollary.
\begin{corollary}\label{coro upper semicontinuity}
The locus where~$\chi_E<2$ is an open subset.
\end{corollary}
We prove the Lemma.
\begin{proof} This is a local property. We can work locally in a neighbourhood of~$g\Gamma$. 
As~$G\to G/\Gamma$ is a local homeomorphism, we may pick this neighbourhood of the form~$U\Gamma/\Gamma$
for a neighbourhood~$U$ of~$g$ in~$G$. As~$G$ is locally compact, we may even assume~$U$ is compact. 
Changing~$E$ for~$g^{-1}E$, we may assume~$g$ is the neutral element. As~$U^{-1}\cdot E$ is compact, and~$\Gamma\cdot p$ is discrete, the  set
\[F=U^{-1}\cdot E\cap \Gamma\cdot p\]
 is finite. 
As~$u$ ranges through~$U$, we must have, in terms of characteristic functions,
\[{\chi_E}(u\Gamma)=\# \left[E\cap \left( u\Gamma\cdot p\right)\right]=\#\left[\left( u^{-1} E\right)\cap\left(\Gamma\cdot p\right)\right]=\#\left[\left( u^{-1} E\right)\cap F\right]=\sum_{f\in F}1_{E}(u\cdot f).\] 
As~$E$ is closed,~$1_E$ is an upper semicontinuous function on~$V$. As~$u\cdot f$ depends continuously on~$u$, 
each~$u\mapsto 1_{E}(u\cdot f)$ is upper semicontinuous as well. Hence their finite sum~$\chi_E$ is upper continuous too.
\end{proof}

\subsection{Linearisation of singular sets}\label{sectionA2} Following the terminology of~\cite[\S 3]{DM}, the \emph{subset of $(L,\Gamma_N)$-self-intersection of a set $X$ in $G$} is the set where the map~$\phi:G/\Gamma_N\to G/\Gamma$ restricted to $X\Gamma_N/\Gamma_N$ fails to be bijective.


\begin{proposition}[\emph{c.f.}~\cite{DM} Prop.~3.3]\label{No intersection} For any~$L$ in~$\RatQ$, the set of points of~$(L,\Gamma_N)$-self-intersection of~$X(L,W)$ is contained in\[X^*(L,W)=\bigcup_{K\in \RatQ, \dim(K)<\dim(L)} X(K,W).\]
\end{proposition}
This implies in particular the following
\begin{equation}
\text{$X^*(L,W)$ has no point of~$(L,\Gamma_N)$-self-intersection.}
\end{equation}
\begin{proof} We follow the proof of~\cite{DM} Prop.~3.3. 


If $g\in X(L,W)$ is a point of~$(L,\Gamma_N)$-intersection, then there exists $\gamma\in \Gamma\setminus \Gamma_N$ such that $g^{-1}Wg\subset L\cap \leftexp{\gamma}{L}$. 
We recall~$W=W^+$.  The $\Q$-Zariski closure~$\K$ of~$ (L\cap \leftexp{\gamma}{L})^+$ is a group of class~$\RatQ$ by criterion~\eqref{CriterionB2}. Let~$K=\K(\Q_S)$

We recall~$W=W^+$. We have
\[
g^{-1}Wg=g^{-1}W^+g=(g^{-1}Wg)^+\subseteq (L\cap \leftexp{\gamma}{L})^+\subseteq K.
\]
Therefore, $g\in X(K,W)$.

Finally, since $\gamma$ does not normalise $\L$,
 the group $\L\cap\leftexp{\gamma}{\L}$, and a fortiori~$\K$, is a proper subgroup of $L$. As $\L$ is $\Q$-connected (and hence actually irreducible), we have~$\dim(\K)<\dim(\L)$.
\end{proof}

Recall that, in~\eqref{def A}, we defined~$A_L$ as the $\Q_S$-submodule generated by $X(L,W)\cdot p_L$.
\begin{proposition}[\emph{c.f.}~\cite{DM} Cor.~3.5]\label{linearisation of neighbourhood}
Let~$D$ be a compact subset of~$A_L$. Let~$Y$ be the subset of $(L,\Gamma_N)$-self-intersection points of~$\eta_L^{-1}(D)$. Let~$K$ be a compact subset of~$G\smallsetminus Y_H\Gamma.$ Then there exists a neighbourhood~$\Phi$ of~$D$ in~$V$ such that~$\eta_L^{-1}(\Phi)\cap(K\Gamma)$ has no point of~$(L,\Gamma_N)$-self-intersection.
\end{proposition}
\begin{proof} Let $\mathfrak{N}$ be the collection of compact neighbourhoods of $D$ in $V$. For any~$\Psi\in\mathfrak{N}$, applying Corollary~\ref{coro upper semicontinuity} to~$E=\Psi$, we get that the locus~$U_\Psi$ where $\chi_\Psi< 2$ is open.

By hypothesis, $\eta_L^{-1}(D)\cap(K\Gamma)$ has no point of $(L,\Gamma_N)$-self-intersection. In other words, for any $x\in K\Gamma/\Gamma$, we have~$\chi_D(x)< 2$. Together with the fact that $\chi_D(x)=\inf\{\chi_\Psi(x)|\Psi\in\mathfrak{N}\}$, it implies that the collection $\{ U_\Psi|\Psi\in\mathfrak{N}\}$ covers the compact set $K\Gamma/\Gamma$. 

By compacity of~$K$, there is a finite subset~$\mathfrak{N}_f\subseteq\mathfrak{N}$ so that we have a finite subcovering~$\left\{U_\Psi\middle| \Psi\in\mathfrak{N}_f\right\}$. Let~$\Phi=\cap_{\mathfrak{N}_f}\Psi$. As $\chi_\Phi\leq\chi_\Psi$ for any $\Psi\in\mathfrak{N}_f$, it follows that for any $x\in K\Gamma/\Gamma$, we have~$\chi_\Phi(x)<2$. In other words, $\eta_L^{-1}(\Phi)\cap (K\Gamma)$ has no $(L,\Gamma_N)$-self-intersection.
\end{proof}

The proof above is essentially the same as the original proof of~\cite{DM}. However, our argument proves a little more, as the number $2$ could be any arbitrary number.

\subsection{Besicovich Property and Good Functions}

This is preparatory work for our main statements to come in forthcoming section~\ref{section A4}.

Following~\cite{KT}, the Besicovich covering lemma and good functions are used in conjunction to cleanly generalize the one-dimensional real case of~\cite[Prop.3.13]{EMSAnn} to the multi-dimensional $S$-arithmetic setting.

In this appendix, for any $\Q_S$-module $\prod_{v\in S}{\Q_v}^{d_v }$, with non-negative integers $d_v$, we consider the metric induced from the max-norm of each component in $\Q_v$ and consider a measure~$\lambda$ which is a product of Haar measures on the $\Q_v$. As a metric and measure space, it is \emph{doubling}, by which we mean that there exists a constant $c_d$ such that for any ball~$B$, 
\begin{equation}\label{doubling}
\lambda(3B)\leq c_d\cdot\lambda(B),
\end{equation}
where $3B$ denotes the ball with same centre as $B$ and thrice its radius .

\subsubsection{Good Functions} Let $X$ be a metric space endowed with a nowhere zero locally finite positive Borel measure~$\nu$.  For a function $f:X\to \R$ and a non empty subset~$B$ in~$X$, we denote $\Nm{f}_B$ the supremum of~$\abs{f}$ on~$B$. 
\begin{definition} For two constants $C>0$ and $\alpha>0$.
A real continuous function~$f$ on~$X$ is~\emph{$(C,\alpha)$-good} if for any ball~$B$ in~$X$ and for any $\varepsilon>0$,
\begin{equation}\label{def good} 
	\nu\left(\left\{ x\in B\,\middle|\, \abs{f(x)}<\varepsilon \Nm{f}_B^{\vphantom{l}}\right\}^{\vphantom{l}}\right) \leq C\varepsilon^\alpha\nu(B).
\end{equation}
\end{definition}	
We remark that due to the strict inequality on the left hand side, the zero function is $(C,\alpha)$-good. 
In some sense, the~$\alpha$ constant rules out functions to have zeroes of arbitrarily larger order (the function~$\exp(-1/x):(0;1]\to\R$ does not qualify); the~$C$ constant provide a sort of uniformity. The next proposition shows that one can find uniform constants ~$(C,\alpha)$ on some small classes
of analytic functions (in a $S$-adic sense).

%
%

\begin{proposition}\label{goodness} 

 Let $U\subset\prod_{v\in S}{\Q_v}^{d_v }$ be a bounded neighbourhood of zero and let $\phi:U\to \prod_{v\in S}{\Q_v}^{e_v }$ be analytic. Then there exist $C$, $\alpha$ and a neighbourhood $V\subset U$ of zero such that for any $\Q_S$-linear map $\Lambda:\prod_{v\in S}{\Q_v}^{e_v }\to\prod_{v\in S}{\Q_v}^{f_v }$, the function $\Nm{\Lambda\circ\phi}$ is $(C,\alpha)$-good on $V$.
\end{proposition}

\begin{proof} Since the max-norm and the supremum of $(C,\alpha)$-good functions is itself $(C,\alpha)$-good, it is enough to consider a single place $S=\{v\}$ and~$\Lambda:\Q_v^{e}\to\Q_v$ a single linear functional.

Write~$\phi=(\phi_1,...,\phi_{e})$ into coordinates. If $\phi\equiv 0$, then we are done since the zero function is $(C,\alpha)$-good for any constants $C$ and $\alpha$. Otherwise, restricting the maps to $\operatorname{span}\phi(U)$, we can assume that $\phi(U)$ spans $\Q_v^{e }$. Since $\phi$ is analytic, it implies that it is nondegenerate at zero. By~\cite[Thm.4.3]{KT},  there are $C$, $\alpha$ and a neighbourhood $V\subset U$ of zero such that for any constants $c_i\in\Q_v$, $\abs{c_1\phi_1+...+c_e\phi_e}$ is $(C,\alpha)$-good on $V$. In other words, for any linear functional $\Lambda:\Q_v^{e}\to\Q_v$, $\abs{\Lambda\circ\phi}$ is $(C,\alpha)$-good on $V$, thus concluding the proof.
%
%
\end{proof}

\subsubsection{Good Parametrisation}\label{parametrisation} In the presence of non-archimedean places, a good pa\-ra\-me\-tri\-sa\-tion of $\Omega$ would exist only locally, due to the lack of convergence of the exponential map. However, since our~$\Omega$ are bounded, a given local pa\-ra\-me\-tri\-sa\-tion has positive measure, and this will be enough for our purposes.

  Let~$H$ be a Lie subgroup of~$G=\G(\Q_S)$, by which we mean~$H$ is locally a product of Lie subgroups of $\G(\Q_v)$ at each places of $S$. In particular, it is locally analytic, by virtue of the locally convergent exponential map.
  
  Let~$\Omega$ be a bounded open subset of~$H$, and let us pick any coordinate chart $\Theta:U\to H$ around a point of $h\in\Omega$, where $U$ is a neighbourhood of $0$ in $\prod_{v\in S}{\Q_v}^{d_v }$ for nonnegative integers $d_v$ for each $v\in S$.
  
  Also, for any representation $\rho:G\to GL(V)(\Q_S)$ of Lie groups, possibly shrinking $U$, the map $\rho\circ\Theta$ is analytic on~$U$.

We fix such a representation $\rho$ of $G$ and a neighbourhood $U$ on which $\rho\circ\Theta$ is analytic. There is a smaller neighbourhood~$V\subseteq U$ of zero on which Proposition~\ref{goodness} holds. Let $B$ be a ball centred at zero and of radius $r$ small enough such that:
\begin{enumerate}
\item the ball~$3B$, with centre zero and radius, $3r$, is contained in $V$,
\item the projection $G\to G/\Gamma$ is injective on $\Theta(B)$.
\item $\Theta(B)$ is contained in $\Omega$.
\end{enumerate}
Then $\Theta$ provides our desired parametrisation by~$B$ of the bounded open subset~$\Omega_0=\Theta(B)$ of $H$.

Let us prove the claim
\begin{lemma} Let~$\lambda$ (resp.~$\mu$) denote a Haar measure on~$\prod_{v\in S}{\Q_v}^{d_v }$ (resp. on~$H$).
The direct image~$\Theta_*\lambda|_B$ of the restriction of~$\lambda$ to~$B$ is comparable with the restriction of~$\mu$ to~$\Theta(B)$.
\end{lemma}
\begin{proof}
The Haar measure $\lambda$ on $\prod_{v\in S}{\Q_v}^{d_v }$ is associated to an invariant top differential form $\omega$. Then $\Theta_*\omega|_{B}$ is a continuous top differential form on $\Omega_0\subset H$. Therefore, the associated measure $\Theta_*\lambda|_{B}$ is absolutely continuous with respect to any Haar measure on $H$ restricted to $\Omega_0$, hence in particular to the probability measure $\mu|_{\Omega_0}$. Moreover, since $\Omega_0$ is bounded, both measures are comparable: 
\begin{equation}\label{comp meas}
\exists\, c_m\geq 1\text{ such that } c_m^{-1}\mu|_{\Omega_0}\leq \Theta_*\lambda|_{B}\leq c_m\,\mu|_{\Omega_0}.
\end{equation}
\end{proof}

%

\subsubsection{Besicovich Covering Property} It is proved in~\cite[\S~\,1.1]{KT} that the metric space~$X=\prod_{v\in S}{\Q_v}^{d_v }$ (with the max-norm metric) satisfies the Besicovich covering property: for any boun\-ded subset~$A$ of~$X$, and for any family~$\mathcal{B}_0$ of nonempty open balls such that any point in~$A$ is the centre of some ball of~$\mathcal{B}_0$, there is an at most countable subfamily~$\mathcal{B}\subseteq\mathcal{B}_0$ which covers~$A$ and has multiplicity at most~$N_X$: in terms of characteristic functions, 
\begin{equation}\label{besicovich}
 1_A \leq \sum_{B\in\mathcal{B}}1_B\leq N_X\cdot1_X.
 \end{equation}

%

\subsection{Linearisation of focusing, after A.~Eskin, S.~Mozes and N.~Shah}\label{section A4}
We use the notations of section~\ref{parametrisation} about~$\Omega_0$.
\subsubsection{} We provide a piece of introduction. The following statement studies ``good maps'' in the neighbourhood of the singular locus~$X([L],W)$ in term of the linearisation~$V$. One feature is that one passes from a measure theoretic property (spending some~$\eps$ proportion of time in the neighbourhood corresponding to~$\Psi$) to an everywhere property (being contained in the neighbourhood associated with~$\Phi$). Such a feature fits in the realm of the $(C,\alpha)$-good properties recalled before. Another feature of this kind of results, \emph{a feat actually}, is that the element~$\gamma$ in case~\eqref{alternative 1} is \emph{independent from} the element~$\omega$ in~$\Omega_0$. To these effect one uses Besicovich covering property. To conclude, we remark that the analysis stability hypothesis, our new input in this article, allows this~$\gamma$ to be furthermore independent of~$g$.

\begin{proposition}[\emph{c.f.}~{\cite[Prop.~3.8, Prop.~3.12]{EMSAnn}}]\label{EMS312}
We consider a $(C,\alpha)$-good pa\-ra\-metrisation~$\Theta:3B\to\Omega$, and~$\Omega_0=\Theta(B)$, as in section~\ref{parametrisation}, with respect to the linear representation~$\rho$ on~$V$.

Let $\eps>0$. For any compact $D_0\subset A_L$, there exists a compact $D\subset A_L$ (which is explicit) such that for any neighbourhood $\Phi$ of $D$ in $V$, there exists a neighbourhood $\Psi$ (which is explicit) of $D_0$ in $V$ such that for any $g\in G$, at least one of the following holds.
\begin{enumerate}
\item \label{alternative 1}	There exists $x=\gamma\cdot p_L\in\Gamma\cdot p_L$ such that $g\Omega_0\cdot x\subset \Phi$.
\item \label{alternative 2}	One has the upper bound~$\mu\left( \left\{\omega \in\Omega_0 \bigl| 
g\cdot\omega\cdot \Gamma\cdot p_L\cap\Psi\neq\emptyset \right\}\right)<\eps.$
\end{enumerate}
\end{proposition}

\begin{proof} 
 Choose $M\geq 1$ a large enough real number, to be fixed later (cf.~\eqref{eq M}), and write
 \[\Nm{D_0}=\max_{x\in D_0}\Nm{x}.\]
\paragraph{First step}
Depending on~$M$ and~$R$, our first task is to make explicit the compact~$D$, and the neighbourhoods~$\Phi$ and~$\Psi$ we will be working with.

 We choose for~$D$ the closed ball in~$A_L$ of radius~$MR$:
\[ D=\left\{v\in A_L^{\vphantom{M}}\ \middle|\ \Nm{v}\leq M\Nm{D_0}\right\}.\]
Let $\Phi$ be a given neighbourhood of $D$ in~$V$. If the conclusion of the proposition holds for a neighbourhood~$\Phi$ of $D$ in $V$, then it necessarily holds for any larger neighbourhood.

Let $\Lambda$ be a $\Q_S$-linear endomorphism of~$V\to V$ with kernel exactly $A_L$. Since $\Lambda$ is linear, it is an open map, and thus there is a constant $b>0$ and~$R>\Nm{D_0}$ such that
\begin{equation}\label{Choice Phi}
	\left\{
		v\in V^{\vphantom{M}}\ 
	\middle|\ \Nm{v}< MR\text{ and }\Nm{\Lambda(v)}<b
	\right\}
		\subseteq
	\Phi.
\end{equation}
As we may shrink~$\Phi$, we will assume this inclusion is an equality. We note~$\Phi$ is then an open neighbourhood.
Associated with~$\Phi$, and these~$b$ and~$R$, we define the neighbourhood~$\Psi$ of~$D_0$ to be
\begin{equation}\label{Choice Psi}
\Psi = \left\{v\in V\ \middle|\ \Nm{v}< R\text{ and }\Nm{\Lambda(v)}<\frac{b}{M}\right\}.
\end{equation}
As we ensured~$M\geq1$, we have
\begin{equation}\label{A8}
\Psi\subseteq \Phi.
\end{equation}

We fix~$g\in G$.
We will prove that for these choices of $D$, of~$\Phi$ and $\Psi$, the alternative between~\ref{alternative 1} and~\ref{alternative 2} is exhaustive, for a suitable~$M$, yet to be fixed, but fixed independently from~$g$. 
\paragraph{Second step} We introduce some notations to interpret~\ref{alternative 1} and~\ref{alternative 2}.
Let us denote, for any $v\in V$, the following the two maps
\begin{align*}
N_{v}&:t\mapsto g\Theta(t)\gamma\cdot p_L \text{ and}\\
T_{v}=\Lambda\circ N_{v}&:t\mapsto\Lambda(g\Theta(t)\gamma\cdot p_L)
\end{align*}
from~$3B$ to~$V$. We define the subset
\[ 
	E=
		\left 
			\{t\in B^{\vphantom{M}}\ 
		\middle|
			\ \exists x\in\Gamma\cdot p_L,
			N_x(t)\in\Psi
		\right\}\subseteq B.
\]
The case~\eqref{alternative 2} readily means
\begin{equation}\label{A11}
\mu(\Theta(E))<\eps.
\end{equation}

For~$s$ in~$E$ and~$x=\gamma\cdot p_L$ in~$\Gamma\cdot p_L$, we denote~$B_{s,x}=B(s,r_{s,x})$ the open ball centred at~$s$ with largest radius~$r_{s,x}$ such that
\begin{enumerate}[label=(\roman*)]
\item \label{As2} the ball~$B_{s,x}$ is contained in~$3B$,
\item \label{As3} its image by~$N_x$ satisfies
\begin{equation}\label{As eq}
N_{x}(B_{s,x})\subseteq\Phi.
\end{equation}
\end{enumerate}
This~$E$ and the~$B_{s,x}$ are so defined that the case~\ref{alternative 1}, for some~$x$ in~$\Gamma\cdot p_L$, would be implied (actually equivalent in a pure ultrametric setting) by
\begin{equation}\label{A10}
B\subseteq B_{s,x}
\end{equation}
for some~$s$.

Let~$s$ be in~$E$. By definition of~$E$ there is~$x$ in~$\Gamma\cdot p_L$ such that~$N_{x}(s)\in\Psi$. By~\eqref{A8}, we have a fortiori~$N_x(s)\in\Phi$. Since~$\Phi$ is is open and~$N_{x}$ is continuous, the radius~$r_{s,x}$ is non zero, and~$B_{s,x}$ is non empty. Note that~$E$ is contained in~$B$ hence bounded. As a consequence
\begin{equation}\label{induce covering}
\text{ the balls~$B_{s,x}$ induce a covering of the bounded subset~$E$.}
\end{equation}
This will be used to invoke Besicovich covering property.

\paragraph{Third step}\label{step 3} In step we will put aside the case~\ref{alternative 1} of the Proposition.
Let~$r$ be the radius of~$B$. By maximality of the radius~$r_{s,x}$ there is a point~$t$ at distance~$r_{s,x}$ form~$s$ for which one of the two properties~\ref{As2} or~\ref{As3} fails. 
Assume~\ref{As2} fails: $t$ does not belong to~$3B$. We have~$\Nm{s}<r$ as~$s$ belong to~$b$, and we have~$\Nm{t}\geq 3$ as~$t$ does not belong to~$3B$. By triangular inequality we get
\[
r_{s,x}=\Nm{t-s}\geq\Nm{t}-\Nm{s}> 3r-r=2r.
\] 
and, for~$b$ in~$B$,
\[
\Nm{s-b}\leq\Nm{s}+\Nm{b}<r+r=2r<r_{s,x}.
\]
As this holds for every~$b$ in~$B$, we have~$B\subseteq A_s$: by~\eqref{A10} we are in the case~\eqref{alternative 1} of the Proposition.
\paragraph{Fourth step} We now assume that the case~\eqref{alternative 1} of the Proposition does not occur. Our objective is to prove the case~\eqref{alternative 2} occurs, by proving~\eqref{A11}.
By contraposition of the argument in the previous step~\ref{step 3}, it follows that, for any~$s$ in~$E$, and~$x$ in~$\Gamma\cdot p_L$, we have the following.
\begin{itemize}
\item Firstly~$r_{s,x}\leq 2r$, which implies that~$r_{s,x}$ is not the maximal radius for which~\ref{As2} holds: the closed ball, say~$\overline{B}_{s,x}$, of radius~$r_{s,x}$ centred at~$s$ (not necessarily the closure of~$B_{s,x}$) satisfies
\begin{equation}\label{A contained B}
\overline{B}_{s,x}\subseteq 3B.
\end{equation}
As~$s$ and~$x$ are arbitrary it follows 
\begin{equation}\label{Union contained 3B}
\bigcup_{s,x}B_{s,x}\subseteq 3B.
\end{equation}
\item and thus that~$r_{s,x}$ is a maximal radius for which~\ref{As3} holds: there is a point~$t$ at distance~$r_{s,x}$ from~$s$ for which
\begin{equation}\label{not in Phi}
N_x(t)\notin\Phi.
\end{equation}
\end{itemize}
By our choice~\eqref{Choice Phi} of~$\Phi$, this means that, at least one of the following occurs
\[
	\Nm{N_x}_{\overline{B}_{s,x}}
		=
	\sup_{\overline{B}_{s,x}}\Nm{N_{x}}
		\geq
	\Nm{N_{x}(t)}
		\geq 
	MR
	~\text{ or }~
	\Nm{N_x}_{\overline{B}_{s,x}}
		=
	\sup_{\overline{B}_{s,x}}\Nm{T_{x}}
		\geq
	\Nm{T_{x}(t)}
		\geq 
	b.
\]
Let~$f_{s,x}=\Nm{N_{x}}$ if the left inequalities occur, or~$f_{s,x}=\Nm{T_{x}(t)}$ otherwise. 

We shift our interest to~$E$. 
If an element~$t\in \overline{B}_{s,x}$ also belong to~$E$, it satisfies the property~$N_x(t)\in\Psi$. This implies similarly, by our choice~\eqref{Choice Psi} of~$\Psi$, that both~$\Nm{N_x(t)}\leq R$ and~$\Nm{T_x(t)}\leq b/M$ occur. Hence
\[
	\Nm{N_x}_{\overline{B}_{s,x}\cap E}
		=
	\sup_{\overline{B}_{s,x}\cap E}\Nm{N_{x}}
		\leq 
	R 
		\leq \Nm{N_x}_{\overline{B}_{s,x}}
	~\text{ and }~
	\Nm{T_x}_{\overline{B}_{s,x}\cap E}
		=
	\sup_{B_{s,x}\cap E}\Nm{T_x} 
		\leq 
	\frac{b}{M}
		\leq 
	\Nm{T_x}_{\overline{B}_{s,x}}.
\]

\paragraph{Fifth step} We invoke the~$(C,\alpha)$ good property.

 Note that the maps~$N_x$ and~$T_x$ are obtained from $\rho\circ\Theta:3B\to\mathrm{End}(V)$ by post-composing by linear maps into~$V$, namely the evaluation map~$\rho(g)\mapsto \rho(g)\cdot x$ and our~$\Lambda:V\to V$. 

From our choice of $\Theta$ and from Proposition~\ref{goodness}, there exist constants $(C,\alpha)$ such that~$\Nm{T_\gamma}$ and $\Nm{N_\gamma}$ are all $(C,\alpha)$-good functions on~$3B$. The ball~$A_{s,x}$ is contained in~$3B$ by~\eqref{A contained B}.

Applying the $(C,\alpha)$-good property apply to~$f=f_{s,x}=\Nm{N_x}$ or~$\Nm{T_x}$ as above, with~$\eps=R$ or~$b/M$ respectively, we have
\begin{equation}\label{eq good} 
	\text{for any }s\in E,
	\quad
	\lambda(B_{s,x}\cap E)
		\leq
	\lambda(\overline{B}_{s,x}\cap E)
		\leq 
	CM^{-\alpha}\lambda(\overline{B}_{s,x}).
\end{equation}
\paragraph{Sixth step} We use the Besicovitch Property~\eqref{besicovich} to the covering of~$E$ by the~$B_{s,x}$: there exists a constant $N$ (not depending on~$M$) and a finite subset~$F\subseteq E\times\Gamma\cdot p_L$ such that $1_E\leq \sum_{(s,x)\in F} 1_{B_{s,x}}\leq N$. Applying~$\lambda$ gives
\begin{equation}\label{eq besicovich1} \lambda(E)\leq \sum_{(s,x)\in F} \lambda(B_{s,x}\cap E)
\end{equation}
and, using~\eqref{Union contained 3B},
\begin{equation}\label{eq besicovich2}
\quad \sum_{(s,x)\in F} \lambda(B_{s,x})\leq N\lambda\left(\bigcup_{(s,x)\in F} B_{s,x}\right)\leq N\lambda(3B).
\end{equation}
\paragraph{Final step}
We combine equations~\eqref{comp meas}, \eqref{eq besicovich1}, \eqref{eq good},  \eqref{eq besicovich2} and \eqref{doubling} to obtain
\begin{align*}
\mu(\Theta(E))
	&\leq c_m\lambda(E)\\
	&\leq c_m\sum_{(s,x)\in F} \lambda(B_{s,x}\cap E)\\ 
	&\leq c_mCM^{-\alpha}\sum_{(s,x)\in F}\lambda(B_{s,x})\\
	&\leq c_mNCM^{-\alpha}\lambda(3B)\\ 
	&\leq c_dc_mNCM^{-\alpha}\lambda(B).
\end{align*}
Choosing $M$ greater than 
\begin{equation}\label{eq M}
(\varepsilon^{-1}c_dc_mNC\lambda(B))^{1/\alpha}
\end{equation} yields
\[  \mu(\Theta(E))= \mu\left( \left\{\omega \in\Omega_0\, \bigl|\, g\cdot\omega\Gamma \in K,\, g\cdot\omega\Gamma\cap\Psi\neq\emptyset \right\}\right)<\eps.\]
We have proven~\eqref{A11} as we announced. This completes the proof.
\end{proof}
\subsubsection{} At last we conclude with the objective of this whole appendix~\ref{AppA}.

A general setting can be the following. We fix~$L$ in~$\RatQ$ and denote~$\rho$ the linear representation of~$G$ on~$V=\bigwedge^{\dim(L)}\lie{g}_{\Q_S}$.
We consider a probability measure~$\mu$ on~$G$ of the following kind. There is a measurable bounded subset~$\Omega$ of~$G$ of full measure, meaning~$\mu(\Omega)=1$, and such that for every~$\omega$ in~$\Omega$ there is neighbourhood~$\Omega_0$ of~$\omega$ in~$\Omega$ and a good parametrisation~$\Theta:B\to \Omega_0$ with respect to~$\rho$ such that:~$\Omega_0$ is Zariski dense in~$\Omega$, and~$\Theta(\lambda\restriction_B)$ is comparable with~$\mu\restriction_{\Omega_0}$. We let~$\mu_\Omega$ denote the direct image of~$\mu$ in~$G/\Gamma$.
\begin{proposition}[\emph{c.f.}~{\cite[Prop.~3.13]{EMSAnn}}]\label{AProp313}The setting is as above.
We consider a sequence of translated measures~$\mu_i=g_i \cdot \mu_\Omega$ of probabilities on~$G/\Gamma$. Suppose that it converges to a limit~$\mu_\infty$ such that~$\mu_\infty(X([L],W))>0$.

Then there is a compact subset~$D$ in~$A_L$, and a sequence~$\left(\gamma_i\cdot p_L\right)_{i\geq0}$ in~$\Gamma\cdot p_L$, such that, for any neighbourhood~$\Phi$ of~$D$ in~$V$, 
\[\forall i\gg 0,\ g_i\cdot\Omega\cdot\gamma_i\cdot p_L\subseteq \Phi.\]
\end{proposition}

\begin{proof} We note that~$\mu_\infty$ is a bounded measure and~$G/\Gamma$ is a Radon space, hence is inner regular~(\cite[IX\S3.3 Prop.\,2\,a)]{BouINT}). As~$X([L],W)$ is a Borel subset (see~\eqref{Borel}),
for any~$\delta>0$, there exists a compact set $K\subset X([L],W)$ such that 
\[\mu_\infty(K)> \eps'=\mu_\infty(X([L],W)-\delta.\]
We choose~$\delta$ small enough so that~$\eps'>0$, for instance~$\delta=\mu_\infty(X([L],W)/2$.
We can lift $K$ to a compact set $K'$ in $X(L,W)$. !!!

Similarly, as~$\mu$ is bounded,~$G$ is Radon, and~$\Omega$ is measurable, taking~$0<\eps''<\eps'$, for instance~$\eps''=\eps'/2$, there is a compact subset~$C$ of~$\Omega$ such that
\[\mu(C)\geq\mu(\Omega)-\eps''.\]

By hypothesis, For every~$\omega$ in~$C$ there is a neighbourhood~$\Omega_0$ of~$\omega$ in~$\Omega$ that admits a good parametrisation~$\Theta:B\to\Omega_0$. As~$C$ is compact, we may extract a finite cover~$\Omega_1,\ldots ,\Omega_n$ of~$C$
by subsets of~$\Omega$, each admitting a good parametrisation, say~$\Theta_i$ resp., with respect to the representation~$\rho$ as in \S\ref{parametrisation}.

We choose~$0<\eps<(\eps'-\eps')/n$ (say~$\eps=\mu_\infty(X([L],W)/8n$ with above values.) 
We take~$D_0=K'\cdot p_L$, which is a compact subset of $A_L\subset V$. For every~$1\leq j\leq n$, we apply Proposition~\ref{EMS312} for~$\Omega_0:=\Omega_j$, for this~$\eps$ and for this~$D_0$. There is a compact subset~$D_j=D\subset A_L$ such that Proposition~\ref{EMS312} holds for $\Omega_0=\Omega_j$, and $\eps$. Note that each~$D_j$ is arbitrarily large. Writing~$D'=\bigcup_{1\leq j\leq n} D_j$, which is a compact subset of~$A_L$, we may assume~$D_j=D'$ for every~$1\leq j\leq n$.

We let~$M_j$ be as in Lemma~\ref{lemma bound Zariski} applied to~$\Omega_0=\Omega_j$. Denote~$M=\max_{1\leq j\leq n} M_j$.
As each~$M_j$ is arbitrarily large, we may assume~$M=M_j$ for every~$1\leq j\leq n$. We finally define
\begin{equation}
D=\left\{v\in A_L\,\middle|\,\Nm{v}\leq M\max_{d\in D'}\Nm{d} \right\}
\end{equation}
We will prove Proposition~\ref{AProp313} for this~$D$. To this effect let us now argue with an arbitrary neighbourhood~$\Phi$ of~$D$. As~$\Phi$ can be chosen arbitrary small, we may, and will, take it of the form~\eqref{Choice Phi} with~$R>\Nm{D'}$. Let~$\Phi'$ be correspondingly given by~\eqref{Choice Psi}. This is a neighbourhood of~$D'$.

We ensured that, for every~$1\leq j\leq n$, Proposition~\ref{EMS312} holds for~$\Omega_0=\Omega_j$, for~$D_0$, for~$\eps$ and for~$D$ as~$D'$. To the neighbourhood~$\Phi'$ of~$D'$ as~$\Phi$, there is a corresponding arbitrarily small neighbourhood of~$D_0$, say~$\Psi_j$, as~$\Psi$. We set~$\Psi=\bigcap_{1\leq j\leq n}\Psi_j$, as~$\Psi_j$ is arbitrary small, and we may assume that, for every~$1\leq j\leq n$, we have~$\Psi=\Psi_j$.

We now move back along the correspondence from~$G/\Gamma$ to~$V$. Firstly~$\stackrel{-1}{\eta_L}(\Psi)$ is a neighbourhood of~$K'$ in~$G$, and, as~$\phi$ is an open map,~$\Psi'=\phi\circ\stackrel{-1}{\eta_L}(\Psi)$ is a neighbourhood of~$K=\phi(K')$.

We defined~$\eps$ so as to ensure that we may apply Lemma~\ref{LemmaCover} below. We deduce that for each~$i\gg 0$, there is~$1\leq j_i\leq n$ such that we have
\[g_i\mu_{\Omega_{j_i}}(\Psi')>\eps.\] 
We claim  
\begin{equation*}
\text{that case~\eqref{alternative 2} of Proposition~\ref{EMS312} does not occur.}
\end{equation*}
\begin{proof}
We now go from~$G/\Gamma$ to~$V$. We have~$\widetilde{\Psi}\cdot\Gamma=\stackrel{-1}{\phi}(\Psi)$, hence
\[
g_i\mu\restriction_\Omega(\widetilde{\Psi}\cdot\Gamma)=g_i\mu_{\Omega_{j_i}}(\Psi')>\eps.
\]
We have then
\[
g_i\mu\restriction_\Omega(\widetilde{\Psi}\cdot\Gamma)=\mu(\{\omega\in\Omega_{j_i}~|~g_i\cdot \omega\in\widetilde{\Psi}\cap\Gamma\}.
\]
We note finally that
\begin{align*}
g_i\cdot \omega\in\widetilde{\Psi}\cap\Gamma
&\Leftrightarrow
g_i\omega\Gamma\cap\widetilde{\Psi}\Gamma\neq \emptyset
\Leftrightarrow
g_i\omega\Gamma\cap\widetilde{\Psi}\neq \emptyset
\Leftrightarrow
g_i\omega\Gamma\cdot p_L\cap\widetilde{\Psi}\cdot p_L\neq \emptyset
\\&\Leftrightarrow
g_i\omega\Gamma\cdot p_L\cap\Psi\neq\emptyset.
\end{align*}
Hence~$\mu(\{\omega\in\Omega|g_i\omega\Gamma\cdot p_L\cap\Psi\neq\emptyset\})=g_i\mu_{\Omega_{j_i}}(\Psi')>\eps$. This proves our claim
\end{proof}
Necessarily this is case~\eqref{alternative 1} of Proposition~\ref{EMS312} that occurs. Namely there is~$\gamma=\gamma_i$ in~$\Gamma$ such that
\[
g_i\cdot \Omega_{j_i}\cdot \gamma_i\cdot p_L\subseteq\Phi'.
\]
By Lemma~\ref{lemma bound Zariski} (for~$\Lambda=\mathrm{Id}_V$ and~$\Lambda$ as in the definition of~$\Phi$), our choice of~$M$, of~$\Phi$ and of~$\Phi'$, this implies
\[
g_i\cdot \Omega\cdot \gamma_i\cdot p_L\subseteq\Phi.
\]

\end{proof}
\begin{lemma}\label{LemmaCover}
 Let~$\Omega_j\subseteq \Omega$, where~$1\leq j\leq n$ be finitely many subsets of~$\Omega$ such that~
\[
\mu\left(\Omega\smallsetminus\bigcup_{1\leq j\leq n}\Omega_j\right)<\eps''.
\]
Assume that for a sequence of translates, we have a limit measure~$\mu_\infty=\lim_{i\geq0}g_i\cdot\mu_\Omega$ and let~$K$ be such that~$\mu_\infty(K)>\eps’$. 

Then for every neighbourhood~$\Phi$ of~$K$, there is a sequence~$(j_i)_{i\geq0}$ in~$\{1;\ldots;n\}$ such that
\[
\forall i\gg 0,~g_i\cdot\mu_{\Omega_{j_i}}(\Phi)\geq\frac{\eps''-\eps'}{n}.
\]
\end{lemma}
\begin{proof} Assume by contradiction that we have a neighbourhood~$\Phi$ of~$K$ such that for~$i\gg 0$, there does not exist~$j$ such that
\[g_i\cdot\mu_{\Omega_{j}}(\Phi)\geq\frac{\eps''-\eps'}{n}.\]
Let~$A=\bigcup_{1\leq j\leq n}\Omega_j$ and~$B=\Omega\smallsetminus\bigcup_{1\leq j\leq n}\Omega_j$. We decompose
\[
g_i\cdot\mu_\Omega=g_i\cdot\left(\mu_\Omega\restriction_A\right)+g_i\cdot\left(\mu_\Omega\restriction_B\right).
\]
We bound
\[
	\left(g_i\cdot\left(\mu_\Omega\restriction_B\right)\right)(\Phi)
		\leq
	\left(g_i\cdot\left(\mu_\Omega\restriction_B\right)\right)(G/\Gamma)
		\leq
	\mu_\Omega\restriction_B(G/\Gamma)
		=
	\mu_\Omega(B)	
		=
	\eps'',
\]
and for~$i\gg0$,
\[
	\left(g_i\cdot\mu_\Omega\restriction_A\right)(\Phi)
		\leq
	\sum_{j=1}^n g_i\cdot\left(\mu_\Omega\restriction_{\Omega_j}\right)(\Phi)
		\leq
	\sum_{j=1}^n \frac{\eps'-\eps''}{n}
		=
	\eps'-\eps''.	
\]
Adding these bounds gives, for~$i\gg 0$
\[
	g_i\cdot\mu_\Omega(\Phi)
		=
	g_i\cdot\left(\mu_\Omega\restriction_A\right)(\Phi)
	+
	g_i\cdot\left(\mu_\Omega\restriction_B\right)(\Phi)
		\leq
	(\eps'-\eps'')
	+
	\eps''
		=
	\eps'.
\]
On the other hand, for the open neighbourhood~$\mathring{\Phi}$ of~$K$, we get~$\mu_\infty(\mathring{\Phi})\geq\mu_\infty(K)>\eps'$, whence
\[
	\lim g_i\cdot\mu_\Omega(\mathring{\Phi})
		=
	\mu_\infty(\mathring{\Phi})>\eps'.
\]
Combining these inequalities yields the contradiction
\[
	\eps'
		\geq
	\varlimsup g_i\cdot\mu_\Omega(\Phi)
		\geq 
	\varlimsup g_i\cdot\mu_\Omega(\mathring{\Phi})
		=
	\lim g_i\cdot\mu_\Omega(\mathring{\Phi})
		=
	\mu_\infty(\mathring{\Phi})
		>
	\eps'.
\]
\end{proof}
\begin{lemma}\label{lemma bound Zariski}
 Let~$\Omega_0$ be Zariski dense subset of a bounded subset~$\Omega$ of~$G$ and a linear representation~$\rho$ of~$G$ on~$ V$.

There is a constant~$M$ such that for every endomorphism~$\Lambda$ of~$V$, and every~$v$ in~$V$
\begin{equation}\label{eq lemma bound Zariski}
	\sup_{\omega\in \Omega}\Nm{\Lambda(\omega\cdot v)}
		<
	M\cdot\sup_{\omega\in \Omega_0}\Nm{\Lambda(\omega\cdot v)}.
\end{equation}
\end{lemma}
\begin{proof} The subspace~$W$ generated by~$\rho(\Omega_0)$ in~$\mathrm{End}(V)$ contains~$\rho(\Omega)$: the condition~$\rho(\omega)\in W$ defined a Zariski closed subset which contains~$\Omega_0$. Pick a basis, say~$\omega_1,\ldots,\omega_{\dim(W)}$, of~$W$ made with elements of~$\Omega_0$. As~$\Omega$ is bounded in~$G$ so is~$\rho(\Omega)$ in~$W$. We may write~$\rho(\omega)=c_1(\omega)\omega_1+\ldots+c_{\dim(W)}\omega_{\dim(W)}$. Each coefficient~$c_j(\omega)$ is continuous in~$\rho(\omega)$ hence describe a bounded range of values as~$\omega$ ranges through~$\Omega$. Let~$M'$ be a common upper bound to the value of the~$c_j$ on~$\Omega$ and~$M>\dim(W)\cdot M'$.
Let~$X=\sup_{\omega\in \Omega_0}\Nm{\Lambda(\omega\cdot v)}$.
We have
\begin{align*}
	 \Nm{\Lambda(\omega\cdot v)}
 	&=
 	\Nm{\Lambda(c_1(\omega)\omega_1+\ldots+c_{\dim(W)}\omega_{\dim(W)})}
 	\\&=
 	\Nm{	c_1(\omega)\Lambda(\omega_1 v)+\ldots+c_{\dim(W)}\Lambda(\omega_{\dim(W)}v))}
 	\\&\leq
 	\Nm{c_1(\omega)}\cdot\Nm{\Lambda(\omega_1 v)}+\ldots+ \Nm{c_{\dim(W)}(\omega)}\cdot\Nm{\Lambda(\omega_{\dim(W)} v)}
 	\\&\leq
 	M' \cdot X+\ldots+M'\cdot X
 	\\&=
 	M' \cdot\dim(W)\cdot X
 	\\&<M\cdot X.
\end{align*}
Applying~$\sup_{\omega\in \Omega_0}$ we obtain~\eqref{eq lemma bound Zariski}
\end{proof}

\section{Review on Ratner's theorems and some variants}\label{AppRatner}

Confusingly, ``Ratner's theorem'' can refer to various theorems, depending on the author. These theorems have variants, of different degrees of generality. In addition, a same statement may be known under different names. 

Fortunately, the review article~\cite{RatnerICM} makes precise many such statements; with comments on the variants found in the literature, on the history of their proof, their evolution from weaker theorems, the logical links between the statements, etc.

Generally speaking, \emph{Ratner's theory} can be understood as 
$$\text{``rigidity'' properties of ``unipotent'' ``flows'' on ``homogeneous spaces''.}$$
Let us explain these terms. \emph{Rigidity} is an informal term (cf.~\cite{WhatisMeasureRig}) which means here that objects of a (``richer'') algebraic nature (the homogeneous subsets) are actually ubiquitous at the a priori ``weaker'' topological level (questions about orbit closure), or the ``even weaker'' measure-theoretic level (ergodicity of finite Borel measures).\footnote{Cartan's theorem about analyticity in Lie group theory, or Margulis arithmeticity theorem, Mostow-Prasad rigidity theorem may bring a grasp on the rigidity phenomenon.} 
 \emph{Unipotent group} can correspond to various kind of groups: one parameter, monogeneous, higher dimensional algebraic unipotent, $\Ad$-unipotent, connected, disconnected, discontinuous, generated by unipotents or~$\Ad$-unipotent. \emph{Flows} is a dynamical image to indicate the underlying action (of the unipotent group), which may be more suitable for a one parameter group (the parameter being seen as a time parameter). An \emph{Homogeneous space}~$G/\Gamma$ is understood in the theory of topological groups~$G$ and~$\Gamma$, but belonging to some quite general class, which encompass at least the connected real Lie groups, but with more general groups involved as the $S$-arithmetic regular Lie groups of~\cite{RatnerICM}; the stabiliser~$\Gamma$ is often chosen to be a lattice (sometimes assumed to be arithmetic, \cite{TomanovOrbits}), but not always (see~\cite{MT3}.)

Mostly three kind of theorems are pinpointed.
\begin{enumerate}
\item \emph{Ratner's orbit closure}. A topological kind of theorem, about the closure of orbits of unipotent groups, which is the original Raghunathan's conjecture. (\cite[Conjecture~1, Theorem~3, {\S}4 Theorem~S2]{RatnerICM})
\item \emph{Ratner's classification}. A measure-theoretic or ergodic-theoretic kind, about the algebricity of the measures with unipotent invariance. This is sometimes known as Raghunathan's measure conjecture, as Ratner's (strict) measure rigidity, or as Ratner's classification theorem.(\cite[Conjectures~2,3, Theorems~1,2, {\S}4 Theorem~S1]{RatnerICM} see also~\cite[Theorems~10,13,14, Remark after Theorems~14]{RatnerICM})
\item \emph{Ratner's equidistribution}. A dynamic kind, about the asymptotic distribution of a unipotent trajectory in the closure of their orbit. This is sometimes known as Ratner's uniform distribution theorem, or as distribution rigidity. One also refers to uniform distribution using the term ``equidistribution''. (\cite[Theorems~6,8, {\S}4 Theorem~S3]{RatnerICM} see also~\cite[Theorem~10]{RatnerICM}))
\end{enumerate} 
 
In classification statements, algebraicity of a probability measure~$\mu$ has a group theoretic meaning, namely, in an homogeneous space~$G/\Gamma$:
\begin{equation}\label{algebraicity def}
\begin{array}{c}
 \text{the support of~$\mu$ is a closed orbit of a closed group~$L$ of~$G$}\\
 \text{and~$\mu$ under which~$\mu$ is $L$-invariant.}
\end{array}
\end{equation}
This ``algebraicity'' is sometimes referred as ``homogeneity'' by recent authors. The definition~\eqref{algebraicity def} is implicit. It is sometimes useful to have a more explicit descriptions of the algebraic measures, namely of the groups~$L$ which can be involved in the statement above. In particular in the case in which~$\Gamma$ is an arithmetic or $S$-arithmetic lattice (when this definition makes sense for the considered group~$G$), we expect these groups to be algebraic and defined over~$\Q$. This can be found in~\cite[Proposition~3.2]{Shah91} and~\cite[Theorems~1,~2]{TomanovOrbits}.

Here, we will mostly rely on the measure-theoretic classification theorem, in the~$S$-arithmetic setting for algebraic groups, as can be found in~\cite{MargulisTomanov}. More particularly for groups generated by unipotents as in~\cite[Theorem~2]{MargulisTomanov}.

In the remaining of this appendix~\ref{AppRatner}, we will add some complement to~\cite[Theorem~2]{MargulisTomanov}. We need to explicit the class of algebraic measures involved in~\cite{MargulisTomanov}, in terms of algebraic group; we were unable to find such precisions in the available literature, when the context of~\cite[Theorem~2]{MargulisTomanov} is concerned. In the archimedean case~$\Q_S\ciso\R$, these complements are more simply stated, and well known. The knowledgeable reader interested in the archimedean case only can skip what remains of this appendix.

\subsection{On groups of Ratner class in the $S$-arithmetic case} \label{sectionB1}
\subsubsection{}
 When applying the classification theorem, one faces the following class of subgroups. For a subgroup~$L$ of~$G$, we denote by~$L^+$ the group generated by the ``unipotent subgroups'' (the definition of which depends on the context).

\begin{definition}\label{defiRG}	 Let~$\Lscr$ be the class of closed subgroups~$L$ of~$G$ fulfilling the following conditions.
\begin{enumerate}
\item[1a)] \label{RG1a}	The subset~$L\cdot\Gamma$ is closed. Equivalently, the orbit~$L\cdot\Gamma/\Gamma$ is closed in~$G/\Gamma$.
\item[1b)] \label{RG1b} (stronger than 1a) The intersection~$\Gamma_L:=\Gamma\cap L$ is a lattice in~$L$. Namely: the orbit~$L\cdot\Gamma/\Gamma$ is the support of a~$L$-invariant probability measure. We will denote~$\mu_L$ the mentioned probability measure (though it depends on~$\Gamma$ as well). 
\item[2)] \label{RG2} (assuming 1b.) The probability~$\mu_L$ is ergodic under the action of~$L^+$.
\end{enumerate}
Note that the condition~2) is actually equivalent to the seemingly weaker one: there exists a \emph{unipotent} subgroup in~$L$ which acts ergodically on~$\mu_L$. Another variant of condition~2) is
\begin{enumerate}
\item[2$^\prime$)] The orbit~$L^+\Gamma/\Gamma$ is dense in~$L\cdot\Gamma/\Gamma$.
\end{enumerate}

\end{definition}
\subsubsection{}

From now on, we assume that~$G$ is the topological group~$\G(\Q_S)$ of~$\Q_S$-points of a semi-simple linear algebraic group~$\G$. We assume that~$\Gamma$ is an $S$-arithmetic lattices.
The ``unipotent subgroups'' of~$G$ are the subgroups of the form~$\U(\Q_S)$ where~$\U$ is an algebraic unipotent subgroup of~$\G$ defined over~$\Q_S$. For a subgroup~$L$ in~$G$, we denote~$L^+$ the subgroup generated the subgroups in~$L$ of the form~$\U(\Q_S)$ as before.

 In this context, the previous class~$\Lscr$ is closely related with the following more explicit, and algebraically defined, one. This adds a slight precision to the work of \cite{TomanovOrbits}. We first settle some terminology.
\begin{definition}
\begin{enumerate}
\item We say that an linear algebraic group~$\H$ over~$\Q$ is~\emph{$\Q_S$-a\-ni\-so\-tro\-pic} if, for \emph{every} place~$\Q_v$ in~$S$, the algebraic group~$\H_{\Q_v}$ is anisotropic. Equivalently, the topological group~$\H(\Q_S)$ is compact.

\item We say that~$\H$ is \emph{of non compact type}
if every quasi-factor of a Levi factor is \emph{not} $\Q_S$-anisotropic.
\end{enumerate}
\end{definition}
Here is what we call the ``Ratner class'' of algebraic subgroups of~$\G$. This is also the class denoted~$\mathscr{F}$ in \cite{TomanovOrbits}.
\begin{definition}Let~$\RatQ$ be the class of algebraic subgroups~$\L$ of~$G$ defined over~$\Q$ which satisfy the following conditions.
\begin{enumerate}
\item The algebraic group~$\L$ is Zariski connected over~$\Q$.
\item The radical of~$\L$ is \emph{unipotent}. Equivalently the Levi factors of~$\L$ are \emph{semi-simple} (rather than merely reductive).
\item The algebraic group~$\L$ is of non compact type.
\end{enumerate}
Abusing notations, we will often~$L\in\RatQ$ instead of~$\L\in\RatQ$ when~$L=\L(\Q_S)$.
\end{definition}
The first two conditions together means that~$\L_{\overline{\Q}}$ admits no non constant character.

We claim that the conjunction of these three conditions is actually equivalent to any the following.
\begin{equation}\label{CriterionB2}
\text{The subgroup~$\L(\Q_S)^+$ is $\Q$-Zariski dense in~$\L$.}
\end{equation}
\begin{equation}\label{CriterionB3}
\text{The $S$-arithmetic lattice~$\Gamma_L=\Gamma\cap \L(\Q_S)$ is $\Q$-Zariski dense in~$\L$.}
\end{equation}
\begin{proof}
The condition~\eqref{CriterionB2} is proved firstly by decomposing a Levi factor into quasi-factors, in order to be reduced to the case of a connected almost simple group. We then check, in the latter case, the subgroup~$\L(\Q_S)^+$ is either trivial (anisotropic case), or otherwise is infinite and normal in the $\Q$-Zariski dense subgroup~$\L(\Q_S)$ of~$\L$. So is its Zariski closure, which must hence be~$\L$. For condition~\eqref{CriterionB3}, one implication is given by the following, the converse of which is easy.
\end{proof}
\begin{proposition}[Borel Zariski density~\cite{BorelDensity}~{\cite[Theorem~4.10 p.\,205]{PR}}, in its $S$-arithmetic variant of Wang {\cite{Wang}, cf. \cite[{\S}I.3.2 3.2.10]{Margulis}}
under the form of]\label{BorelWang} Let~$\H$ be a connected algebraic semi-simple group over~$\Q$ of non compact type.

Then any $S$-arithmetic subgroup is a $\Q_S$-Zariski dense lattice.
\end{proposition}
\begin{proof} Actually Borel-Wang density theorems concludes that the $\Q_S$-Zariski closure of a lattice~$\Gamma$ contains~$\H(\Q_S)^+$, without the non compact type assumption.

If~$\H$ is of non compact type, then~$\H(\Q_S)$ is~$\Q$-Zariski dense.

\end{proof}

We first state some easy consequences of criterion~\eqref{CriterionB2}.
\begin{lemma} Let~$W$ be a subgroup of~$G$ such that~$W^+=W$. Then the Zariski closure of~$W$ over~$\Q$ is of class~$\RatQ$.

Every group of class~$\RatQ$ can be achieved in such a way.
\end{lemma}
\begin{proof} As~$W^+$ is $\Q$-Zariski dense in its $\Q$-Zariski closure~$\L$, then~$W^+\subseteq\L(\Q_S)^+$ is a fortiori Zariski dense. Hence the first claim.

Given~$\L$ in~$\RatQ$ we may pick~$W=\L(\Q_S)^+$. This concludes.
\end{proof}
\subsubsection{}
We now turn to the relations between the classes~$\Lscr$ and~$\RatQ$.
\begin{itemize}
\item One passes from~$\Lscr$ to~$\RatQ$ by associating to~$L$ the Zariski closure (over~$\Q$) of a small enough neighbourhood in~$L$ of the neutral element. This is also the Zariski closure~(over~$\Q$) of~$L^+$.
\item Conversely, one passes from~$\RatQ$ to~$\Lscr$ by the following construction.
\end{itemize}
In the converse direction we introduce this definition, which makes explicit the results of~\cite{TomanovOrbits}.
\begin{definition}\label{defiLpp}
 For an algebraic group~$\L$ in~$\RatQ$, writing~$L=\L(\Q_S)$, we denote
\begin{equation}\label{defiLppeq}
\Lpp=\overline{(\Gamma\cap L)\cdot L^+}.
\end{equation}
\end{definition}
\begin{lemma}\label{LemmeLppclosed}
 This~$\Lpp$ is a closed subgroup of~$L$.
\end{lemma}
\begin{proof} Both~$\Gamma\cap L$ and~$L^+$ are subgroups of~$L$. The latter,~$L^+$, is invariant under all algebraic automorphisms of~$\L$ over~$\Q_S$. It is in particular normal in~$L$. As a consequence~$(\Gamma\cap L)\cdot L^+$ is a subgroup of~$L$. As~$L$ is closed, the closure is a closed subgroup of~$L$.
\end{proof}
\begin{lemma}  This~$\Lpp$ is of class~$\Lscr$.
\end{lemma}
\begin{proof} The group~$L$ satisfies Borel and Harish-Chandra criterion. Hence conditions~1a) and~1b) of Definition~\ref{defiRG} are fulfilled. We will see below that~$\Lpp$ is open of finite index in~$L$. Consequently\footnote{Actually~$\mu_{\Lpp}$ can be constructed as  the restriction of~$[L:\Lpp]\cdot \mu_L$ to the open subset~$\Lpp\Gamma/\Gamma$ of~$\Supp(\mu_L)=L\Gamma/\Gamma$.}~$\Lpp\cdot\Gamma/\Gamma$ 
will also satisfy conditions~1a) and~1b). By construction~$\Lpp\cdot \Gamma/\Gamma=\overline{L^+\cdot\Gamma/\Gamma}$,
which is condition~2$^\prime$).
\end{proof}

The following is essentially from~\cite{TomanovOrbits}, to which we refer.
\begin{proposition}
Let~$L$ be a group of class~$\RatG$ (for an $S$-arithmetic lattice~$\Gamma$).
Let~$\L$ be the corresponding group of class~$\RatQ$ and let~$\Lpp$ be as above.

Then~$\Lpp$ is of class~$\RatG$ and~$\mu_L=\mu_\Lpp$.
\end{proposition}

We end with some properties of groups of the form~$\Lpp$, which were not explicitely 
constructed in~\cite{TomanovOrbits}.

\begin{proposition}\label{++finiteindex} Let~$\L$ be in~$\RatQ$. We write~$L=\L(\Q_S)$, and, for an arbitrary~$S$-a\-rith\-me\-tic lattice~$\Gamma_L$ in~$L$, then
\[
\Lpp=\overline{\Gamma_L\cdot L^+}.
\]
is an open subgroup of finite index in~$L$.
\end{proposition}
\begin{proof} According to Lemma~\ref{LemmeLppclosed},~$\Lpp$ is a closed subgroup of~$L$. It suffices to prove the finiteness of the index. The openness will follow.

We first reduce to the case of a semisimple~$\L$. consider the radical~$\R$ of~$\L$ and the projection~$\pi:\L\to \L/\R$. 
The radical~$\R$ of~$\L$ is unipotent, hence~$R=\R(\Q_S)$ is
contained in~$L^+$. It suffices to prove~$\overline{\pi(\Gamma_L)\pi( L^+)}$ is of finite index in~$L/R$. 
The map~$L\to (\L/\R)(\Q_S)$ is surjective, by additive Hilbert~90 for a perfect field~\cite[II~Prop.1, III~Prop.6]{SerreLNM5}.
Hence~$L/R=\L/\R(\Q_S)$. The algebraic group~$\L/\R$ is semisimple by hypothesis~$\L\in\RatQ$. The image of~$\Gamma_L$ is a~$S$-arithmetic subgroup of~$L/R$. Finally we prove~$(L/R)^+=L^+/R$ by double inclusion: on the one hand the image of an unipotent element is unipotent; on the other hand,~$\L/\R$ is isomorphic to a Levi factor of~$\L$, and unipotent elements lift to unipotent element in this Levi factor.

Without loss of generality, we can replace~$\Gamma_L$ by a subgroup of finite index, and assume~$\Gamma_L$ is the product of its intersections with the almost $\Q$-simple $\Q$-quasi-factors of~$\L$. We may reduce to the case of a almost~$\Q$-simple~$\L$.

By~\cite[Proof of Proposition~6.14]{BorelTits},~$L^+$ is open and of finite index in the
isotropic $\Q_S$-factors of~$L$ (see Proposition~\ref{KneserTits} below). It suffices to prove that the projection of~$\Gamma$ on any anisotropic $\Q_S$-factors~$F$ has a closure~$K$ which is a finite index subgroup of~$F$. This~$K$ is a closed subgroup, as~$F$ is compact, ~$K$ is compact. The stabiliser of the Lie algebra~$\lie{k}$ of~$K$, for the adjoint action, is~$\Q_S$-algebraic, and contains~$\Gamma$. Recall that the $\Q_S$-Zariski closure of~$\Gamma$  is~$\L$ by Borel-Wang Zariski density Proposition~\ref{BorelWang}. Consequently the image of~$\Gamma$ in~$F$ is infinite, and~$K$ is not discrete: the Lie algebra~$\lie{k}$ is non zero. But it is normalised by~$F$, and~$F$ is almost simple. The Lie algebra~$\lie{k}$ is the same as that of~$F$. In other words~$K$ is open in~$F$, hence of finite index.


\end{proof}
As a remark we have the following.
\begin{corollary} 
If~$S$ contain the real place, then~$\Lpp$ contains the neutral component of~$L(\R)$.
\end{corollary}

The following was used  in the proof of \ref{++finiteindex}.
\begin{proposition}[Kneser-Tits for characteristic~$0$ local fields, \cite{TitsBBK}{\cite[2.3.1 (d)]{Margulis}}\cite{GilleKneserTits}]\label{KneserTits} Let~$\H$ be an algebraic simply connected semi-simple group over~$\Q_S$ without~$\Q_S$-anisotropic quasi-factor.

 Then~$\H(\Q_S)=\H(\Q_S)^+$.
\end{proposition}

\subsubsection*{Note} We add a few comments about the definition of $\Lpp$..
Suppose~$S$ consists of the real place only. Then we could have taken the neutral component of~$L=\L(\R)$ instead of ~$L^\ddagger$ . In our definition here,~$\Lpp$ is the union of the connected components of~$\L(\R)$ that meet~$\Gamma$. For~$\Gamma$ small enough these two constructions coincide. On the other hand, in an ultrametric setting, considering the connected component is no longer the right object to consider: we require~$\Lpp$ as constructed above. 
Note that, as~$\Gamma$ gets smaller,~$\Lpp$ does not generally stabilise: it will get closer to~$\L(\Q_S)^+$ which, in the case where there is an anisotropic factor  \emph{defined over~$\Q_v$}, will be a lower dimensional subgroup.

Lastly we prove the lemma that had been used previously in the paper:

\begin{lemma}\label{gamma normalise Lpp}
 Let~$N$ be a subgroup of~$G$ normalising~$L$.
Then~$\Gamma\cap N$ normalises~$\Lpp$.
\end{lemma}
\begin{proof}
We note that~$\Gamma$ normalises itself and~$N$ normalises~$L$. Hence~$\Gamma\cap N$ normalises~$L\cap\Gamma$.

We know that~$L^+$ is invariant under~$\Q_S$-algebraic automorphisms, and in particular under conjugation by~$N$.
The subgroup~$\Gamma\cap N$ of normalises a fortiori~$L^+$.

Finally~$\Gamma\cap N$ normalises~$\Gamma\cap L$ and~$L^+$, hence normalises the product~$(\Gamma\cap L)L^+$ and its closure~$\Lpp$.
\end{proof}

%% file: 3-Resultat-ultrametrique-these-V/RichardUltrametrique.tex

\section{Introduction}
Cet article adapte les résultats de \cite{RS09}\footnote{Dont l'article publié~\cite{RichardShah} est une version augmentée.} au cas d'un groupe algébrique semi\-sim\-ple~$G$ sur un corps local~$\kk$ muni d'une valeur absolue ultramétrique~$\abs{-}$.

Nous suivons, dans les grandes lignes, la méthode développée dans~\cite{RS09} pour le contexte archimédien (dont nos Propositions~\ref{Prop31} et~\ref{Prop34} reprennent certains arguments). Pour pallier l'absence du théorème de décomposition de Mostow (\cite{Mos55}) pour~$G(\kk)$ ainsi que la propriété de convexité de l'application exponentielle, qui n'est plus partout définie, nous considérons le plongement~$\Theta:\Ik(G)\to G^\an$ de l'\emph{immeuble de Bruhat-Tits}~$\Ik(G)$ de~$G$ sur~$\kk$ dans l'\emph{espace analytique}~$G^\an$, au sens de Berkovich, associé à~$G$.

Notre démonstration repose en effet sur le Théorème~\ref{TheoF3} de l'annexe, qui se base sur les travaux~\cite{RTW09} de Bertrand Rémy, Amaury Thuillier et Anette Werner. Ces auteurs généralisent une construction du chapitre 5 de~\cite{Ber90}, où V.~Berkovich se restreint aux groupes de Chevalley (semi-simples déployés). Le Théorème~\ref{TheoF3} met à profit les propriétés de convexité dans~$\Ik(G)$, et remplace la propriété de convexité de l'application exponentielle de~\cite{RS09} par la convexité de fonctions de la forme~$x\mapsto \abs{f(x)}$ (Proposition~\ref{Prop36}) lorsque~$f$ est dans~$\kk[G]$, une fonction régulière.

La décomposition de Mostow est remplacée par la décomposition moins
précise du Théorème~\ref{TheoF1} et sa conséquence en la Proposition~\ref{Prop21}. Pour obtenir l'énoncé~\ref{Prop21}, notre démonstration utilise l'existence de points fixes pour l'action
de groupes d'isométries compacts sur les immeubles de Bruhat-Tits. Ces propriétés
découlent de l'existence de métriques hyperboliques, et justifient le choix
de la géométrie ultramétrique au sens de Berkovich, plutôt que rigide, qui permet de considérer des
espaces métriques complets.

Les résultats de~\cite{RichardShah} dans leur intégralité (et pas seulement de~\cite{RS09}), qui sont archimédiens, et leur analogue ultramétrique, qui repose sur le présent article, sont à la base du travail de~\cite{RichardZamojski} en dynamique homogène.

Les travaux de~\cite{RichardZamojski} sont à leur tour à la base du travail~\cite{RichardYaffaev} établissant la conjecture d'André-Pink-Zannier, dans certains cas seulement, mais, en revanche, sous une forme améliorée: on y détaille l'équidistribution et l'adhérence topologique en plus de l'adhérence de Zariski.

Pour résumer, le texte ici présent est l'un des socles sur lequel sont bâtis les travaux~\cite{RichardShah,RichardZamojski,RichardYaffaev} 

\section*{Remerciements}
Que Bertrand Remy, Amaury Thuillier et Georges Tomanov reçoivent ici mes
remerciements pour leur accueil chaleureux et leur conversation enrichissante à
l'occasion de mon déplacement à l'institut Camille Jordan de l'université Lyon 1
Claude Bernard.

C'est en côtoyant, à l'IRMAR, Antoine Chambert-Loir, Antoine Ducros et Jérôme
Poineau que j'ai pu me familiariser avec les espaces de Berkovich. Que cet article
leur témoigne de ma reconnaissance.
Emmanuel Breuillard qui a trouve une erreur dans une version ant\'erieure de ce texte.

\addtocontents{toc}{\protect\setcounter{tocdepth}{1}}
\section{Hypothèses et Énoncé}

Nous convenons qu'un \emph{groupe algébrique linéaire}, et plus généralement un sous-groupe algébrique linéaire, est supposé \emph{affine de type fini}, \emph{réduit} et \emph{connexe}.
La nécessité de ces hypothèse n'a pas été vérifiée : notons que, 
comme notre résultat principal ne concerne que les groupes de points rationnels, il peut s'appliquer aux groupes non réduits, quitte à passer au sous-groupe réduit associé. Remarquons aussi que, dans un groupe algébrique linéaire, le centralisateur d'un sous-groupe algébrique linéaire n'est pas toujours réduit, ni toujours connexe. Par exemple, en caractéristique non nulle~$p$, le centre de~$\SL(p)$ est connexe, de dimension~$0$ mais a une algèbre de Lie
non nulle. En caractéristique~$0$ le centre de~$\SL(2)$ n’est pas connexe.

Soit~$G$ un groupe algébrique linéaire sur un corps~$\kk$. Étant donnée une représentation linéaire de degré fini~$\rho : G \rightarrow \GL(V)$, on notera~$\Ad_\rho : G \rightarrow GL (\gl(V ))$ l'action par conjugaison~$G$ sur~$\gl(V)$.
\begin{equation}\label{eq1}
\text{Pour~$g$ dans~$G$, et un endomorphisme~$e$ de~$V$, on a~$\Ad_\rho(g)=\rho(g)e\rho(g)^{-1}$.}
\end{equation}
Pour tout sous-groupe algébrique linéaire~$H$ de~$G$, on notera~$C_H(\Ad_\rho)$ l'espace vectoriel sur~$\kk$ engendré par les coefficients matriciels de l'action de~$H$ sur~$\gl(V)$. L'espace~$C_H(\Ad_\rho)$ est formé de fonctions régulières sur~$H$. Étant donné un ensemble~$\Omega$ de points de~$H$, nous considérons la propriété suivante.
\begin{equation}\tag{\text{$\ast$}}\label{*}
\text{Tout coefficient matriciel de~$\Ad_\rho$ qui s'annule sur~$\Omega$ s'annule en fait sur~$H$.} 
\end{equation}
Cette propriété est notamment vérifiée si~$\Omega$ est Zariski dense dans~$H$. Mettons en exergue deux autres conditions sur le $H$-module~$V$.
\begin{multline}\tag{\text{$\ast\ast$}}\label{**}
	\text{L'action~$\rho$ de~$H$ sur~$V$ est telle que~$V^H$, le plus grand sous-module de}\\
	\shoveleft{\text{points fixes, a un unique supplémentaire $H$-stable.}}
\end{multline}
\begin{multline}\tag{\text{$\ast\ast^\prime$}}\label{**'}
	\text{En surcroît de~\eqref{**}, ce supplémentaire de~$V^H$, comme représentation }\\ \shoveleft{\text{de~$H$, n'a pas de quotient isomorphe à la représentation triviale.}}
\end{multline}
La condition \eqref{**} revient à supposer que~$V$ ne contient pas d'extension non triviale \emph{de} la représentation triviale, et la condition~\eqref{**'} que~$V$ ne contient pas non
plus d'extension non triviale \emph{par} la représentation triviale. Les conditions~\eqref{**}~et~\eqref{**'} sont automatiquement vérifiées si le $H$-module~$V$ est semi-simple. C'est le cas si~$H$ est réductif et le corps~$\kk$ de caractéristique nulle.

\subsubsection*{Remarqes:}{~}
\newline\noindent
\textit{i)}\label{Remarquei}
 Lorsque la condition~\eqref{**} est vérifiée, le supplémentaire~$H$-stable de~$V^H$ est le noyau d'un \emph{unique} projecteur~$H$-équivariant de~$V$ sur~$V^H$ (le \emph{projecteur de
Reynolds}, \emph{confer} \cite{Dem76}). Étant donné un morphisme $H$-équivariant~$\Phi : V \rightarrow W$ entre deux $H$-mo\-du\-les~$V$ et~$W$ satisfaisant~\eqref{**}, ce morphisme commute aux projecteurs sur le lieu fixe si~$V$ satisfait la condition~\eqref{**'}.
\newline
\textit{ii)}\label{Remarqueii}
En revanche, en caractéristique non nulle~$p$, l'action adjointe de ~$\GL(p)$ sur~$\gl(p)$ ne vérifie pas la condition~\eqref{**}. Le sous-espace fixe est la droite~$\kk\cdot\Id$ formée des homothéties. Or cette droite n'a pas de supplémentaire stable: c'est déjà le cas 
pour l'action des matrices de permutation sur le sous-espace diagonal. Lorsque~$p$ est impair, l'action induite de~$\GL(p)$ sur~$\frac{\lie{gl}(p)}{\kk\cdot\Id}$ satisfait la condition~\eqref{**} car il n'y a pas d'élément invariant non nul (cf.~\cite{Bou60},~\S6, Exercice~24). Elle ne satisfait pas la condition~\eqref{**'} car l'action quotient~$\left.\frac{\lie{gl}(p)}{\kk\cdot\Id}\middle/\frac{\lie{sl}(p)}{\kk\cdot\Id}\right.$ est triviale.

Il nous faut encore considérer la propriété suivante, de Richardson.
\begin{equation}\tag{\text{$\ast\ast\ast$}}\label{***}
\text{Le sous-groupe~$H$ de~$G$ est «~fortement réductif dans~»~$G$.}
\end{equation}
En caractéristique nulle elle est  satisfaite si et seulement si~$H$ est réductif. Un propriété aisément vérifiable qui implique~\eqref{***} est la suivante, dite~$H$ est «~réductif dans~» $G$ selon les termes de~\cite{RichardShah}.
\begin{equation}\tag{\text{$\ast\ast\ast^\prime$}}\label{***'}
\text{L'action de~$H$ action sur~$\lie{g}$ est semi-simple.}
\end{equation}

Notre résultat principal est le suivant. Pour motiver cet énoncé nous ren\-voy\-ons à~\cite{RS09} et~\cite{Ric09}.
\begin{theoreme}\label{Theo11}

Soit~$(\kk,\abs{-})$ un corps local normé ultramétrique, soit~$G$ un groupe linéaire algébrique semisimple connexe sur~$\kk$, et soit~$H$ un sous-groupe algébrique linéaire fortement réductif dans~$G$. Notons~$Z$ le centralisateur de~$H$ dans~$G$. Alors il existe une partie~$Y$ de~$G(\kk)$, fermée pour la topologie ultramétrique, et telle que
\begin{enumerate}[label=\textrm{\arabic*}.]
\item[\textup{1.}] d'une part on ait~$G(\kk)=Y\cdot Z(\kk)$,
\item[\textup{2.}] d'autre part, étant donnés
\begin{itemize}
\item une représentation linéaire~$\rho:G\rightarrow \GL(V)$ de degré fini et définie sur~$\kk$, telle que les~$H$-modules~$\gl(V)$ et~$C_H(\Ad_\rho)$ satisfassent~\eqref{**}, et que~$\lie{gl}(V)$ satisfasse~\eqref{**'},
\item une partie non vide~$\Omega$ de~$H(\kk)$ ayant la propriété~\eqref{*},
\item une norme~$\Nm{-}$ sur~$V$, supposée homogène relativement à~$\abs{-}$,
\end{itemize}
il existe une constante~$c>0$ telle que
\begin{equation}\label{eq2}
\forall y\in Y,~\forall v\in V,~\sup_{\omega\in\Omega}\Nm{\rho(y\cdot \omega)(v)}\geq\Nm{v}/c.
\end{equation}
\end{enumerate}
\end{theoreme}

Fixons~$\rho$. L'inégalité~\eqref{eq2} est vérifiée pour toute constante~$c$ lorsque le vecteur~$v$
est nul. Le théorème est donc vérifié, avec~$Y = G(\kk)$, si~$V$ est de dimension nulle.
\emph{Dorénavant nous supposerons que la représentation~$\rho$ a un degré non nul}. \label{constantes coeffs}	En particulier les fonctions constantes sont des coefficients matriciels. En effet, tout coefficient diagonal de~$g\mapsto \Ad_\rho(g)\Id_V$ vaut la constante~$1$. Ainsi~$C_G(\Ad_\rho)$ et~$C_H(\Ad_\rho)$ seront non nuls. Dans ce cas, la non vacuité de~$\Omega$ découle de la
condition~\eqref{*}.

Remarquons que pour établir la formule~\eqref{eq2}, on peut remplacer~$\Omega$ par un sous-ensemble~$\Omega_b$ de~$\Omega$, car cela a pour effet de diminuer le membre de gauche de l'inégalité, sans modifier le membre de droite. Montrons que, comme~$\Omega$ satisfait~\eqref{*}, il
existe un sous-ensemble fini~$\Omega_b$ de~$\Omega$ satisfaisant la propriété~\eqref{*}.

\begin{proof}
La condition~\eqref{*} signifie que lorsque~$\omega$ décrit~$\Omega$, les morphismes d'é\-va\-lu\-a\-tion~$f \mapsto f (\omega)$, définis sur~$C_H(\Ad_\rho)$, engendrent~$C_H(\Ad_\rho)^\vee$, le dual algébrique de l’epace~$C_H(\Ad_\rho)$. 
Comme V est de dimension finie,~$C_H(\Ad_\rho)$, qui
est un quotient de~$\gl(V ) \otimes \gl(V )^\vee$, est de dimension finie.
Il suffit donc d'extraire
de la famille génératrice précédente une base, nécessairement finie, et de choisir pour~$\Omega_b$ un sous-ensemble fini de~$\Omega$ paramétrant cette base.
\end{proof}

\emph{Dorénavant~$\Omega$ sera supposé borné dans~$H(\kk)$.}

Notre démonstration utilise les énoncés~\ref{TheoF3} et~\ref{TheoF1} de l'annexe, à laquelle nous renvoyons pour les définitions et conventions utilisées, en particulier concernant la notion de convexité telle que définie dans la section~\ref{sectionF}. Pour plus d'approfondissement,
on pourra également consulter~\cite{RTW09}, ainsi que~\cite[chapitre 5]{Ber90} pour le cas des groupes semisimples déployés (« de Chevalley »).

Dans la section suivante, nous rappelons la situation et fixons les notations
utilisées jusque la fin de la démonstration, soit les sections~\ref{section2}, \ref{section3} et~\ref{section4}. Nous y explicitons en particulier la
partie~$Y $. La première conclusion du Théorème~\ref{Theo11} résulte de la Proposition~\ref{Prop21},
que nous déduisons de l'énoncé~\ref{TheoF1}. Nous énonçons également, avec la Proposition~\ref{Prop22}, une variante effective de la seconde conclusion du Théorème~\ref{Theo11} pour la
partie~$Y$ construite. 

Dans la section~\ref{section3}, nous réunissons quelques énoncés indépendants qui seront
utilisés dans la démonstration de la Proposition~\ref{Prop22}. Les énoncés~\ref{Prop31} à~\ref{Prop34} reprennent des arguments de~\cite{RS09}. La démonstration de la Proposition~\ref{Prop36}
repose sur l'énoncé~\ref{TheoF3}. Le cœur de la démonstration de la
Proposition~\ref{Prop22} occupe la section~\ref{section4}.

\section{Notations}\label{section2}

Rappelons la situation. Nous désignons par~$\kk$ un corps local muni d'une valeur
absolue ultramétrique~$\abs{-}$, par~$G$ un groupe algébrique linéaire semi-simple sur~$\kk$,
par~$H$ un sous-groupe algébrique linéaire fortement réductif, et notons~$Z_G(H)$ le centralisateur de~$H$ dans~$G$, vu
comme groupe algébrique affine non nécessairement réduit. Notons que~$Z_G (H)$ n'interviendra toutefois que via son groupe~$Z_G (H)(\kk)$ des points rationnels.

Nous nous sommes fixés~$\rho : G \rightarrow \GL(V )$, une représentation linéaire de~$G$ de degré fini non nul et définie sur~$\kk$ et~$\Nm{-} : V \rightarrow \R$ une norme $(\kk,\abs{-})$-homogène sur~$V$. Nous notons~$\gl(V )$ l'algèbre de Lie des endomorphismes de~$V$, et~$\Ad_\rho$ la
représentation adjointe~\eqref{eq1} de~$G$ sur~$\gl(V )$. Le sous-$G$-module de la représentation régulière~$\kk[G]$ engendré par les coefficients matriciels de~$\Ad_\rho$ est noté $C_G(\Ad_\rho)$, et le~$H$-module formé de
la restriction à~$H$ de ces fonctions régulières est noté~$C_H(\Ad_\rho)$.

Nous désignons par~$\Omega$ une partie bornée non vide de~$H(\kk)$ sur laquelle aucune
fonction régulière non nulle sur~$H$ issue de~$C_H(\Ad_\rho)$ ne s'annule identiquement.

Notons~$\z$ le centralisateur de~$H$ dans~$\gl(V )$. D'après l'hypothèse~\eqref{**} pour~$\Ad_\rho$, il existe un unique projecteur $H$-équivariant de~$\gl(V )$ sur~$\z$ (cf. Remarque~\emph{i)} en page~\pageref{Remarquei}).
Comme~$V$ est de dimension non nulle, les coefficients diagonaux de~$h\mapsto \Ad_\rho(h)(\Id_V )$ forment un coefficient matriciel constant non nul de~$\Ad_\rho$. En outre le sous-module fixe de~$C_H(\Ad_\rho)$ est contenu dans celui de~$\kk[H]$, qui est aussi donné par les fonctions constantes. Donc le sous-module fixe de~$C_H(\Ad_\rho)$ est le sous-module formé des fonctions
constantes, et s'identifie à~$\kk$, muni de la représentation triviale. Utilisant l'hypothèse~\eqref{**} pour~$C_H(\Ad_\rho)$ nous obtenons un projecteur~$H$-équivariant de~$C_H(\Ad_\rho)$ sur~$\kk$. En vertu de l'hypothèse~\eqref{**'} pour~$\Ad_\rho$, tout morphisme~$H$-équivariant~$\gl(V)\rightarrow C_H(\Ad_\rho)$ commute aux projecteurs~$\pi_\kk$ et~$\pi_\z$ (cf. Remarque~\emph{i)} en page~\pageref{Remarquei}).

Nous nous fixons un tore déployé maximal~$T$ de~$G$ sur~$\kk$, notons
\begin{equation*}
Y(T)=\Hom(T,\GL(1))
\end{equation*}
le groupe des cocaractères et~$\Lambda=Y(T)\tens\R$ l'espace vectoriel associé. L'\emph{immeuble de Bruhat-Tits de~$G$ sur~$\kk$} est noté~$\Ik(G)$.
C'est le quotient de~$G(\kk) \times \Lambda$ par la relation d'équivalence considérée dans \cite[(1.3.2)]{RTW09} (cf. \cite[2.1]{Tit79}).
Rappelons que~$G^\an$ désigne l'espace analytique associé à~$G$, vu comme espace topologique des semi-normes multiplicatives bornées sur l'algèbre~$\kk[G]$ des fonctions régulières sur~$G$, pour la topologie de la convergence simple. Notons~$\theta:\Ik(G)\rightarrow G^\an$ une application telle que dans l'énoncé~\ref{TheoF3}, et notons~$G_\theta$ son stabilisateur \emph{à droite} dans~$G(\kk)$ qui est compact et ouvert.

Nous fixons un point~$o$ de~$\Ik(G)$, et notons~$G_o$ son stabilisateur dans~$G_\theta$. Le grou\-pe~$G_o$ est compact dans~$G(\kk)$ (\cite[3.2]{Tit79}) et Zariski dense dans~$G$~,\cite[Lemma~1.4]{RTW09}. Nous notons~$H_o$ le groupe compact~$G_o\cap H(\kk)$, et~$\Ik(G)^{H_o}$ le lieu fixe de l'action de~$H_o$ sur~$\Ik(G)$. D'après le Théorème~\ref{TheoF1}, nous pouvons choisir un compact~$C$ de~$\Ik(G)$ tel que~$\Ik(G)^{H_o}=Z_G(H)(\kk)\cdot C$.

La notion de convexité utilisée est celle introduite dans la section~\ref{sectionF}.

\subsubsection*{Définition}\label{Defi Y}
Soit alors~$Y$ le lieu des points~$y$ de~$G(\kk)$ tels que dans l'enveloppe
convexe de~$H_o \cdot y^{-1}\cdot o$, il se trouve un point de~$C$.

On notera que la construction de~$Y$ ne dépend que (de~$G$, de~$H$ et) du choix
de~$\theta$, de~$o$ et de~$C$. La partie~$Y$ ne dépend donc ni de~$\rho$, ni de~$\Omega$. L'énoncé suivant
démontre que la partie~$Y$ est fermée dans~$G(\kk)$ et satisfait la première condition du Théorème~\ref{Theo11}.

\begin{proposition}\label{Prop21}
La partie~$Y$ de~$G(\kk)$ est fermée ; l'intersection~$Y \cap Z_G (H)(\kk)$ est compacte; on a~$G(\kk) = Y \cdot Z_G (H)(\kk)$.
\end{proposition}
\begin{proof}
Montrons que la partie~$Y$ est fermée dans~$G(\kk)$. Par construction,
la partie~$Y$ est invariante à gauche sous le stabilisateur de~$o$. C'est donc une partie stable sous~$G_o$, qui est ouvert. Le complémentaire de~$Y$ est donc ouvert, car
stable sous~$G_o$.

Montrons que le saturé~$Y \cdot Z_G (H)(\kk)$ de~$Y$ par~$Z_G (H)(\kk)$ vaut~$G(\kk)$. Soit~$g$
dans~$G(\kk)$, et formons l'enveloppe convexe~$\langle H_o g^{-1}o\rangle$ de~$H_o g^{
-1}o$ dans~$\Ik(G)$.
L'action de~$H_o$ sur le convexe~$\langle H_o g^{-1}o\rangle$ a au moins un point fixe, d'après~\cite[2.3.1]{Tit79}. Choisissons-en un, disons~$p$. Comme~$p$ est fixe sous~$H_o$, il s'écrit, d'après~\ref{TheoF1}, sous
la forme~$z^{-1}\gamma$ avec~$\gamma$ dans un compact~$C$ et~$z$ dans~$Z_G (H)(\kk)$.
Comme l'action de~$z$ commute à celle de~$H_o$, nous avons~$zH_o g^{
-1}o = H_o zg^{-1}o$. Comme l'action de~$z$ sur~$\Ik(G)$
échange les appartements, et est affine sur chaque appartement, nous avons
\begin{equation*}
\langle H_o zg^{-1}o \rangle=\langle z H_o g^{-1}o \rangle=z\langle H_o g^{-1}o \rangle.
\end{equation*}
Par conséquent~$\langle H_o zg^{-1}o \rangle$ contient le point~$zp = z(z^{-1}\gamma) = \gamma$, qui appartient à~$C$.
Autrement dit~$g z^{-1}$ appartient à~$Y$, ce qu'il fallait démontrer.

Montrons que l'intersection~$Y \cap Z_G (H)(\kk)$ est compacte. Tout d'abord c'est l'intersection de deux fermés, donc c'est un fermé. Comme le stabilisateur~$G_o$ de~$o$ est compact, et que l'action de~$G(\kk)$ sur~$\Ik(G)$ est propre, il suffit de montrer que l'intersection~$(Y^{-1}\cdot o)\cap(Z_G (H)(\kk)\cdot o)$ est relativement compacte. Or~$o$ étant fixe sous~$G_o$, donc sous~$H_o$, l'en\-sem\-ble~$Z_G (H)(\kk) \cdot o$ est contenu dans le lieu fixe~$\Ik(G)^{H_o}$. Mais, par définition même de~$Y$, l'intersection de~$Y^{-1}\cdot o$ avec~$\Ik(G)^{H_o}$ est contenue dans le compact~$C$.
\end{proof}

Ceci étant, il nous reste à démontrer la seconde condition du Théorème~\ref{Theo11}, autrement dit à établir la formule~\eqref{eq2}. Quitte à changer la constante~$c$, la validité
de la formule~\eqref{eq2} ne dépend de la norme~$\Nm{-}$ qu'à équivalence près. Or,~$V$ étant de
dimension finie, toutes les normes (homogènes) sont équivalentes. 
Soit~$B$ une
boule du dual de~$V$. Quitte à appliquer la section C.1, nous pouvons supposer que les hypothèses de la Proposition~\ref{Prop22} concernant la norme~$\Nm{-}$ sont satisfaites.\footnote{\label{pied B}N.d.É.: Cette hypothèse est vraisemblablement superfétatoire, si on considère la boule unité~$B$ du dual après être passé à une extension ultramétrique~$\kk'$ de~$\kk$ à groupe de valuation non discret (dense dans~$\R_>0$). On tâchera de modifier en conséquence la définition de~$c_4$.} 
\begin{proposition}\label{Prop22}
La situation est celle du Théorème~\ref{Theo11}. Nous utilisons les notations
précédentes. En particulier~$\rho$ et~$V$ sont fixés, et nous avons choisi~$\theta$,~$o$ et~$C$ (de sorte, la partie~$Y$ est bien définie).

Supposons que~$\Omega$ soit bornée, et que~$\Nm{-}$ soit~$G_o$-invariante, ultramétrique et\footnote{Voir la note de pied de page précédente.} ne prenne que des valeurs prises par~$\abs{-} : \kk \rightarrow \R$. Alors la formule (2) est satisfaite avec la constante~$c_4/(c_1c_2c_3)$, où
\[
c_1=1+\sup_{f\in C_G(\Ad_\rho)}
	\frac{\pi_\kk(f)}
	{
		\displaystyle{\sup_{\omega\in\Omega}\abs{f}(\omega)}
	},
c_2=\min_{\omega\in\Omega}\NM{\rho(\omega)}^{-1},
c_3=\sup_{f\in C_G(\Ad_\rho)}
	\frac{\abs{f}(\theta(o))}
	{
		\displaystyle{\sup_{k\in K}\abs{f}(k)}
	}
\]
et~$c_4$ sont obtenues en appliquant les Propositions~\ref{Prop32},~\ref{Prop31},~\ref{Prop35}, et~\ref{Prop39} respectivement.
\end{proposition}

\section{Propositions}\label{section3}
Dans cette section nous réunissons quelques arguments généraux qui serviront à la démonstration de la Proposition~\ref{Prop22}. On pourra passer directement à la section suivante et se reporter aux énoncés ci-dessous au besoin. Les notations sont celles introduites dans la section précédente.

\begin{proposition}\label{prop31}\label{Prop31}
Il existe une constante positive inversible~$c_2$ telle que pour tout élément~$g$ de~$G(\kk)$ et tout vecteur~$v$ de~$V$, on ait
\begin{equation}\label{eq3}
\sup_{\omega\in\Omega}\Nm{\rho(g\cdot \omega)(v)}
	\geq
c_2\cdot\sup_{\omega\in\Omega}\Nm{\rho(\omega^{-1}\cdot g\cdot \omega)(v)}.
\end{equation}
\end{proposition}
\begin{proof} Par définition de la norme d'opérateur~$\NM{\rho(\omega)}$, pour tous~$g$,~$\omega$ et~$v$ comme dans l'énoncé, nous avons l'inégalité
\begin{equation}\label{eq4}
\Nm{\rho(\omega^{-1}\cdot g\cdot\omega)}
\geq
\NM{\rho(\omega)}^{-1}\cdot\Nm{\rho(g\cdot\omega)(v)}.
\end{equation}
Posons~$c_2=\inf_{\omega\in\Omega}\Nm{\rho(\omega)}^{-1}$. Comme~$\Omega$ est une partie bornée et non vide, son image par~$\omega\mapsto\NM{\rho(\omega)}^{-1}$ est une partie bornée et non vide de~$\R_{<0}$. La constante~$c_2$ est donc positive et inversible, et répond à l'énoncé.
\end{proof}
\begin{proposition}\label{prop32}\label{Prop32}
Il existe une constante positive inversible~$c_1$ telle que pour tout coefficient matriciel~$f$ dans~$C_H(\Ad_\rho)$, on ait
\begin{equation}\label{eq5}
\sup_{\omega\in\Omega}\abs{f(\omega)}\geq\frac{1}{c_1}\abs{\pi_\lie{\kk}(f)}.
\end{equation}
\end{proposition}
\begin{proof} Comme~$\Omega$ est borné, l'application~$\Nm{-}_\Omega:f\mapsto \sup_{\omega\in\Omega}\abs{f}(\omega)$ est bien définie sur~$C_H(\Ad_\rho)$. C'est manifestement une semi-norme. D'eprès la condition~\eqref{*}, elle ne s'annule pas: c'est une norme. Comme~$C_H(\Ad_\rho)$ est de dimension finie, l'application linéaire~$\pi_\kk$, de~$C_H(\Ad_\rho)$ sur~$\kk$, est un opérateur borné, relativement à~$\Nm{-}_\Omega$ et~$\abs{-}$. Si~$\NM{\pi_\kk}$ désigne sa norme en tant qu’opérateur~$(C_H(\Ad_\rho),\Nm{-}_\Omega)\to (\kk,\abs{-})$, alors
\begin{equation}\label{eq6}
\NM{\pi_\kk} \cdot \sup_{\omega\in\Omega}\abs{f(\omega)} \geq \abs{\pi_\kk(f)}.
\end{equation}
Par conséquent~$c_1=\NM{\pi_\kk}$ convient si~$\NM{\pi_\kk}\neq0$. Si\footnote{En définitive le cas~$\pi_\kk=0$ ne se produit pas si~$V$ est non nulle. Voir~p.\,\pageref{constantes coeffs} pourquoi les constantes non nulles sont coefficients matriciels dans~$C_H(\Ad_\rho)$.}, en revanche~$\NM{\pi_\kk}=0$, alors~$c_1$ convient. Quoiqu'il en soit,~$c_1=1+\NM{\pi_\kk}$ convient toujours.
\end{proof}
\begin{proposition}\label{prop33}\label{Prop33} Pour tout vecteur~$v$ de~$V$, pour toute forme linéaire~$\phi$ dans~$V^\vee$ et tout élément~$g$ de~$G(\kk)$,
\begin{equation}\label{eq7}
\omega\mapsto(\rho(\omega^{-1}\cdot y\cdot \omega)(v)|\phi)
\end{equation}
définit une fonction sur~$H$ (resp.~$G$) appartenant à~$C_H (\Ad_\rho )$ (resp.~$C_G (\Ad_\rho )$).
\end{proposition}
\begin{proof} Comme~$e\mapsto(e(v)|\phi)$ est une forme linéaire sur~$\gl(V)$, l'application~\eqref{eq7} est un coefficient matriciel de l’action~$\Ad_\rho$.
\end{proof}
\begin{proposition}\label{propr34}\label{Prop34}
Nous utilisons les hypothèses du Théorème~\ref{Theo11} concernant la représentation~$\rho$. Pour toute forme~$\kk$-linéaire~$\phi$ dans~$V^\vee$, la fonction constante sur~$H$
\begin{equation}\label{eq8}
\pi_\kk\left(\omega\mapsto\left(\rho\left(\omega^{-1}\cdot y\cdot\omega^{\vphantom{l}}\right)^{\vphantom{l}}(v)\middle|\phi\right)\right)
\end{equation}
vaut
\begin{equation}\label{eq9}
\left(\pi_{\z}\left(\rho(y)^{\vphantom{l}}\right)^{\vphantom{i}}(v)\middle| \phi \right)
\end{equation}
et, lorsque~$y$ varie dans~$G$, définit une fonction régulière appartenant à~$C_G(\Ad_\rho)$ et invariante sous l'action par conjugaison de~$H$ sur~$G$.
\end{proposition}
\begin{proof}
La fonction~$y\mapsto\left(\pi_{\z}(\rho(y))(v)\middle|\pi\right)$ est manifestement régulière, et est invariante pour l'action par conjugaison de~$H$ sur~$G$, vu que, pour~$\omega$ dans~$H$,
\begin{equation}\label{eq10}
	\left(\pi_{\z}(\rho(\omega y \omega^{-1})\right)
		=
	\rho(\omega)\pi_{\z}(\rho(y))\rho(\omega^{-1})
		=
	\pi_{\z}(\rho(y))
\end{equation}
car~$\pi_\z$ est~$H$-équivariant et d'image dans~$\z$.

Pour établir l'égalité de~\eqref{eq8} et~\eqref{eq9}, considérons l'application~$\Phi:\gl(V)\to C_H(\Ad_\rho)$  qui envoie~$e$ vers le coefficient matriciel~$\omega\mapsto\left(\rho(\omega)e(\rho(y))\rho(\omega^{-1})\middle|\phi\right)$.
C'est une application~$H$-équivariante, et, d'après la Remarque~\emph{i)} en page~\pageref{Remarquei}, elle commute aux projecteurs~$\pi_\z$ et~$\pi_\kk$. Autrement dit
\[
\pi_\kk(\omega\mapsto(\rho(\omega)e\rho(\omega^{-1})(v)|\phi)=\omega\mapsto(\rho(\omega)\pi_\z(e)\rho(\omega^{-1})(v)|\phi).
\]
Prenons~$e=\rho(y)$. Alors le membre de gauche s'identifie à~\eqref{eq8}, et, d'après~\eqref{eq10}, le membre de droite s'identifie à~\eqref{eq9}.
\end{proof}
\begin{proposition}\label{prop35}\label{Prop35}
Pour tout point~$o$ de~$\Ik(G)$, l'application~$f\mapsto \sup_{k\in G_o}\abs{f}(k)$ et la restriction de la semi-norme~$\theta(o)$ définissent, sur~$C_G(\Ad_\rho)$, deux normes comparables.

En particulier, il existe une constante positive et inversible~$c_3$ telle que, pour toute fonction~$f$ dans~$C_G(\Ad_\rho)$, on a~$\sup_{k\in G_o}\abs{f}(k)\geq\frac{1}{c_3}\abs{f}(\theta(o))$.
\end{proposition}
\begin{proof}
L’existence de~$c_3$ découle de la définition de la comparabilité des nor\-mes de la section~\ref{sectionC1}. D’après cette section, il suffit de vérifier que l’on a
bien deux normes~$(k,\abs{-})$-homogènes.

L’application~$f \mapsto \sup_{k\in G_o}\abs{f}(k)$ est bien définie car~$G_o$ est compact. C’est manifestement
une semi-norme~$(\kk,\abs{-})$-homogène. Comme~$G_o$ est ouvert, donc Zariski dense dans~$G$,
cette semi-norme ne s’annule en aucune fonction régulière. C’est donc une norme.

Quant à~$\theta(o)$, comme il s’agit par définition d’une semi-norme (non nulle,
\cite[1.1]{Ber90}) multiplicative~$(\kk,\abs{-})$-homogène sur~$\kk[G]$, il suffit de vérifier que
c’est en fait une norme. Comme~$\theta(o)$ est multiplicative, son noyau définit un idéal
de~$\kk[G]$. Comme~$\theta(o)$ vaut~$1$ en~$1$, c’est idéal est strict. Or le stabilisateur de~$p$
dans~$G(\kk)$ est Zariski dense. Cet idéal définit une sous-variété~$G(\kk)$-invariante
de~$G$: cette sous-variété ou bien est vide ou bien vaut~$G$ lui-même. Or~$G$ est réduit,
et l’idéal considéré ne contient pas l’unité. Par conséquent cet idéal est nul:~$\theta(p)$
ne s’annule pas sur~$\kk[G]$ et \emph{a fortiori} sur~$C_G(\Ad_\rho)$.
\end{proof}
L'énoncé suivant est un corollaire à l’énoncé~\ref{TheoF3}. La notion de convexité est précisée dans la section~\ref{sectionF} correspondante, à laquelle nous renvoyons. Mentionnons juste que cette notion de convexité est naturellement induite par la structure affine par morceaux standard sur l'immeuble~$\Ik(G)$.
\begin{proposition}[Convexité]\label{prop36}\label{Prop36}
Pour toute fonction régulière~$f$ sur~$G$,
l'application~$p \mapsto \abs{f}(\theta(p))$ est convexe sur $\Ik (G)$.
\end{proposition}
Il s'agit du Corollaire~\ref{coroF4} au Théorème~\ref{TheoF3}.
\begin{proposition}[Hyperbolicité]\label{prop37}\label{Prop37}
Pour tout point~$p$ de~$\Ik (G)$, l'enveloppe
convexe de~$H_o \cdot p$ contient un point fixe de~$H_o$.
\end{proposition}
\begin{proof} Notons que comme~$\Ik(G)$ est localement réunion finie d'appartements, et que~$H_o\cdot p$ est compact (C'est l'image du groupe compact~$H_o$ par une application continue vers un espace séparé), l'enveloppe convexe de~$H_o \cdot p$ est compacte. Il suffit alors d'appliquer~\cite[2.3.1]{Tit79} et~\cite[3.2.3]{BT72}.
\end{proof}
\begin{proposition}\label{Prop38}Soient~$Y$ et~$C$ comme en page~\pageref{Defi Y}.
Soit~$f$~une fonction convexe sur~$\Ik (G)$ et~$H_o$-invariante à gauche sur~$Y \cdot o$, et un point~$p$ appartenant à~$Y \cdot o$ . Alors
\begin{equation}\label{eq11}
f(p)\geq\inf_{\gamma\in C }f(\gamma).
\end{equation}
\end{proposition}
\begin{proof}Comme~$f$ est~$H_o$-invariante sur~$Y\cdot o$, on a~$f(p)=\sup_{x\in H_o\cdot p} f(x)$. Notons~$\langle H_o\cdot p\rangle$ l'enveloppe convexe de~$H_o\cdot p$. Comme~$f$ est convexe,~$\sup_{x\in H_o\cdot p} f(x)=\sup_{x\in \langle H_o\cdot p\rangle} f(x)$. Lorsque~$p$ appartient à~$Y\cdot o$, l'intersection~$C\cap\langle H_o\cdot p\rangle$ est non vide. D'où
\[
f(p)=\sup_{x\in\langle H_o\cdot p\rangle} f(x)\geq \sup_{x\in \langle H_o\cdot p\rangle\cap C} f(x)\geq\inf_{x\in \langle H_o\cdot p\rangle\cap C} f(x)\geq \inf_{\gamma\in C} f(\gamma).
\]
\end{proof}
Dans la proposition suivante,~$C$ désigne le compact de~$\Ik(G)$ défini dans la section~\ref{section2} précédente, et~$B$ la boule unité du dual de~$V$ (cf. la note de l'éd. au pied de la page~\pageref{pied B}).

Dans cette proposition, on étudie les coefficients matriciels qui sont de la forme~$g\mapsto\phi(\pi_\z(\rho(g))(v)$ comme fonction sur~$G^\an$, et en particulier sur l'image de~$C$ dans~$G^\an$ par l'application~$\theta:\Ik(G)\to G^\an$. On notera donc
\[
\abs{\phi}(\pi_\z(\rho(\theta(\gamma)))(v))
\]
a valeur obtenue en appliquant, pour~$\gamma$ dans~$C$ la norme~$\theta(\gamma)$ au coefficient matriciel~$g\mapsto\phi(\pi_\z(\rho(g))(v))$. Pour alléger les notations, on pourra omettre~$\theta$ et~$\rho$ dans les notations, soit
\[
\abs{\phi}(\pi_\z(\gamma)(v))
=
\abs{\phi}(\pi_\z(\rho(\theta(\gamma)))(v)).
\]
\begin{proposition}\label{Prop39}
Il existe une constante~$c_4$ telle que pour tout~$\gamma$ de~$C$, et tout vecteur~$v$ dans~$V$, on a
\begin{equation}\label{eq12}
\sup_{\phi\in B}\pi_{\z}(\rho(\theta(\gamma)))(v)\geq\Nm{v}/c_4.
\end{equation}
\end{proposition}
\begin{proof} Remarquons tout d'abord que, d'après la Proposition~\ref{prop35}, la semi-nor\-me~$\theta(\gamma)$ est une norme, pour tout~$\gamma$ dans~$C$.

La restriction de la norme~$\theta(\gamma)$ sur~$\kk[G]$ à~$C_G(\Ad_\rho)$ dépend continûment de~$\gamma$, pour la
topologie faible. Comme~$C_G(\Ad_\rho)$ est de dimension finie et~$C$ est compact, ces nor\-mes sont « uniformément équivalentes » pour~$\gamma$ dans~$C$:  il existe une constante~$c'$ telle que pour tous~$\gamma$ et~$\gamma'$ dans~$C$ nous ayons
\[\forall f\in C_G(\Ad_\rho),\abs{f}(\theta(\gamma))\geq c'\abs{f}(\theta(\gamma')).\] Il suffira de vérifier la formule~\eqref{eq12} pour une constante~$c_4'$ et \emph{pour un seul~$\gamma$ de~$C$}: la constante~$c_4=c'c_4'$ conviendra alors pour tout~$\gamma$ dans~$C$.

Fixons ~$\gamma$ dans~$C$. La formule est évidente pour~$v = 0$. Nous pouvons donc supposer que~$v \neq 0$, et même, par homogénéité, que, pour une certaine constante~$c''$ ne dépendant que de~$\Nm{-}$, le vecteur~$v$ appartienne au compact~$V_{c''}$ où l'inégalité~$1/c'' \leq \Nm{v} \leq c''$ est satisfaite.

Tout revient ainsi à montrer que
\[
\inf_{v\in V_{c''}}\sup_{\phi\in B}(\pi_\z(\gamma)(v))>0.
\]
Soit, par l'absurde, un suite~$(v_n)_{n\in \Z_{\geq0}}$ de~$V_{c''}$ telle que
\[
\lim_{n\in \Z_{\geq0}}\sup_{\phi\in B}(\pi_\z(\gamma)(v_n))>0.
\]
Comme~$V_{c''}$ est compact, et n'adhère pas à~$0$, la suite~$v_n$ a une valeur d'adhérence non nulle~$v_\infty$. Nous allons montrer, ce qui sera une contradiction, que~$v_\infty$ est nécessairement nul.

Pour tout~$\phi_0$ dans~$B$, on conclut de l'encadrement
\[
0\leq\abs{\phi_0}\left(\pi_\z(\gamma)(v_n)\right)\leq\sup_{\phi\in B}\abs{\phi}\left(\pi_\z(\gamma)(v_n)\right)\to 0
\]
que~$\lim_{n\in \Z_{\geq0}}\abs{\phi_0}\left(\pi_\z(\gamma)(v_n)\right)=0$. Par continuité de~$v\mapsto\abs{\phi_0}\left(\pi_\z(\gamma)(v)\right)$, il s'ensuit que l’on a~$\abs{\phi_0}(\pi_\z(\gamma)(v_\infty)=0$.

Ainsi, pour tout~$\phi$ de~$B$, on a~$\abs{\phi}(\pi_\z(\gamma)(v_\infty))=0$. Autrement dit, comme~$\theta(\gamma)$ est une
norme, chaque coefficient matriciel~$g\mapsto\pi(\pi_\z(\rho(g))(v_\infty)$ 
est identiquement nul. Par conséquent, pour tout~$g$ dans~$G(\kk)$, 
le vecteur~$\pi_\z (\rho(g))(v_\infty)$ est nul. Mais,
lorsque~$g$ vaut l'élément neutre,
\[
\pi_\z(\rho(g))(v_\infty)=\pi_\z(\Id_V)(v_\infty)=\Id_V(v_\infty)=v_\infty\neq0.
\]

\end{proof}
\section{Démonstration}\label{section4}
Démontrons la Proposition~\ref{Prop22}. Nous utilisons les notations de la section~\ref{section2}, et
les arguments de la section~\ref{section3}.
\begin{proof}
Comme la norme~$\Nm{-}$ est supposée~$G_o$-invariante, nous pouvons sub\-sti\-tuer~$\sup_{k\in G_o}\Nm{\rho(k\cdot y\cdot \omega)}$ à~$\rho(y\cdot \omega)$ dans la formule~\eqref{eq2}, ce qui donne
\begin{equation}\label{eq13}
\forall y\in Y, \forall v\in V, \sup_{\omega\in\Omega}\sup_{k\in G_o}\Nm{\rho(k\cdot y\cdot \omega)}\geq \Nm{v}/c.
\end{equation}
D'après la Proposition~\ref{Prop31}, il suffit d'établir
\begin{equation}\label{eq14}
\forall y\in Y, \forall v\in V, \sup_{\omega\in\Omega}\sup_{k\in G_o}\Nm{\rho(\omega^{-1}\cdot k\cdot y\cdot \omega)}\geq\frac{1}{c\cdot c_2} \Nm{v}.
\end{equation}
Soit~$V^\vee$ le dual algébrique de~$V$, et notons~$B$ sa boule unité.
D'après les hypothèses sur~$\Nm{-}$ 
nous avons~$\Nm{v}=\sup_{\phi\in B}\abs{\phi(v)}$. La formule qui précède équivaut donc à la suivante.
\begin{equation}\label{eq15}
\forall y\in Y, \forall v\in V, \sup_{\omega\in\Omega}\sup_{\phi\in B}\sup_{k\in G_o}\abs{\phi}\left(\rho(\omega^{-1}\cdot k\cdot y\cdot \omega)\right)\geq\frac{1}{c\cdot c_2} \Nm{v}.
\end{equation}
D'après la Proposition~\ref{Prop33}, la fonction~$\omega\mapsto {\phi}\left(\rho(\omega^{-1}\cdot k\cdot y\cdot \omega)\right)$ appartient à~$C_H(\Ad_\rho)$. Appliquant la Proposition~\ref{Prop32}, il sort
\begin{equation}\label{eq16}
 \sup_{\omega\in\Omega}\abs{\phi}\left(\rho(\omega^{-1}\cdot k\cdot y\cdot \omega)\right)
 	\geq
\frac{1}{c_1}
	\abs{
		\pi_{\kk}
		\left(
			\omega
				\mapsto
			{\phi}\left(\rho(\omega^{-1}\cdot k\cdot y\cdot \omega)\right)
		\right)
		}.
\end{equation}
D'après la Proposition~\ref{Prop34}, le membre de droite de~\eqref{eq16} vaut
\begin{equation}\label{eq17}
\frac{1}{c_1}\abs{\phi}(\pi_{\z}(k\cdot y)(v)).
\end{equation}
D'après la Proposition~\ref{prop35}, il existe une constante positive et inversible~$c_3$ telle que
\begin{equation}\label{eq18}
\sup_{k\in G_o}\abs{\phi}(\pi_{\z}(k\cdot y)(v))\geq\frac{1}{c_3}\abs{\phi}(\pi_{\z}(\theta(o)\cdot y)(v)).
\end{equation}
Par conséquent, nous avons établi, combinant~\eqref{eq16}, \eqref{eq17} et~\eqref{eq18},
\begin{equation*}
\forall y\in Y, \forall v\in V, 
\sup_{\omega\in\Omega}\sup_{k\in G_o}\abs{\phi}\left(\rho(\omega^{-1}\cdot k\cdot y\cdot \omega)\right)
\geq
\frac{1}{c_1 c_3}\sup_{\phi\in B}\abs{\phi}(\pi_{\z}(\theta(o)y)(v)),
\end{equation*}
d'où, considérant la borne supérieure relative aux~$\phi$ dans~$B$,
\begin{equation}\label{eq19}
\forall y\in Y, \forall v\in V, 
\sup_{\omega\in\Omega}\sup_{\phi\in B}\sup_{k\in G_o}\abs{\phi}\left(\rho(\omega^{-1}\cdot k\cdot y\cdot \omega)\right)
\geq
\frac{1}{c_1 c_3}\sup_{\phi\in B}\abs{\phi}(\pi_{\z}(\theta(o)y)(v)).
\end{equation}
Ainsi, pour démontrer~\eqref{eq15}, il suffit d'établir
\begin{equation}\label{eq20}
\frac{1}{c_1 c_3}\sup_{\phi\in B}\abs{\phi}(\pi_{\z}(\theta(o)y)(v))
\geq
\Nm{v}\frac{1}{c\cdot c_2}.
\end{equation}
D'après la Proposition~\ref{Prop34}, la fonction~$g\mapsto \pi(\pi_{\z}(g)(v))$ est régulière sur~$G$ en la variable~$g$, invariante sous l'action de~$H$ par conjugaison. Or, pour~$h$ dans~$H_o$, et~$y^{-1}\cdot o$ dans~$Y^{-1}\cdot o$, nous avons
\begin{equation*}
h\theta(y^{-1}o)h^{-1} = h\theta(y^{-1}o)
\end{equation*}
car~$G_\theta$ contient~$h^{-1}$, et
\begin{equation*}
h\theta(y^{-1}o) = \theta(h y^{-1}o)
\end{equation*}
car~$\theta$ est équivariante. Sur~$Y^{-1}\cdot o$, la fonction~$y^{-1}o\mapsto\abs{\phi}(\pi_{\z}(\theta(y^{-1}o))(v))$ est donc invariante à gauche sous~$H_o$.

D'après la Proposition~\ref{Prop36}, la fonction~$p\mapsto\abs{\phi}(\pi_{\z}(\theta(p))(v))$ est convexe sur~$\Ik(G)$. Par conséquent la fonction~$g\mapsto \sup_{\phi\in B}\abs{\phi}(\pi_{\z}(g)(v))$
est convexe sur~$\Ik(G)$, et sa restriction à~$Y^{-1}\cdot o$ est invariante à gauche sous~$H_o$. D'après la Proposition~\ref{Prop38}, on a, pour tout~$y$ de~$Y$, 
\begin{equation*}
\sup_{\phi\in B} \abs{\phi}(\pi_{\z}(y^{-1}o)(v))
\geq
\inf_{\gamma\in C}
\sup_{\phi\in B}\abs{\phi}(\pi_{\z}(\gamma)(v)).
\end{equation*}
Or, d'après la Proposition~\ref{Prop39}, nous avons
\begin{equation*}
\forall v\in V, \forall\gamma\in C, 
\sup_{\phi\in B}\abs{\phi}(\pi_{\z}(\gamma)(v))\geq \Nm{v}/c_4.
\end{equation*}
Ce qui démontre bien la formule~\eqref{eq20}, avec la constante~$c=
\frac{c_1c_2c_3}{c_4}$.
\end{proof}

\appendix
Dans cette annexe, nous faisons quelques rappels sur  les espaces analytiques et les immeubles. Nous y démontrons notamment (Thé\-o\-rè\-me~\ref{TheoF1})
un résultat de décomposition sur les immeubles et (Théorème~\ref{TheoF3}) un résultat sur la convexité logarithmique des fonctions régulières sur l'immeuble, une fois plongé dans l'espace analytique.

\section{Normes ultramétriques}\label{sectionC}\label{sectionC1}
Soit~$V$ un espace vectoriel sur un corps ultramétrique~$\kk$ de valeur absolue notée~$\abs{-}$. La to\-pologie de~$V$, topologie produit relative à une base de~$V$, est intrinsèque: les automorphismes de changement de base de~$\kk^{\dim(V)}$ sont des homéomorphismes. Nous appelons~\emph{norme} sur~$V$ une application~$V\to\R$ telle que l'application~$(x,y)\mapsto \Nm{x-y}$ définisse une distance compatible à la topologie. Cette norme est dite~\emph{$(\kk,\abs{-})$-homogène}, ou simplement~\emph{homogène}, lorsque toute homothétie de facteur~$\lambda$ agit sur les distances d'un facteur~$\abs{\lambda}$:
\addtocounter{equation}{1}
\begin{equation}\label{eq22}
\text{
pour tout~$\lambda$ dans~$\kk$ et~$v$ dans~$V$, nous avons~$\Nm{\lambda\cdot v}=\abs{\lambda}\cdot\Nm{v}$.}
\end{equation}
On dit que la norme~$\Nm{-}$ est \emph{ultramétrique} si
\begin{equation}\label{eq23}
\forall x,y\in V, \Nm{x+y}\leq\max\left\{\Nm{x\vphantom{y}};\Nm{y}\right\}.
\end{equation}

Nous dirons que deux normes~$\Nm{-}$ et~$\Nm{-}'$ sont \emph{comparables}, ou~\emph{équivalentes}, s'il existe une constante positive et inversible~$C$ telle que pour tout~$v$ dans~$V$ , nous
ay\-ons~$\Nm{v}\leq C\Nm{v}'$ et~$\Nm{v}'\leq C\Nm{v}$.

Remarquons que~$\abs{-}$ et~$\abs{-}^2$ sont deux normes sur~$\kk$ qui ne sont pas comparables, sauf si~$\kk$ est discret. \emph{A contrario} deux normes homogènes sur~$V$ sont toujours comparables. Nous ne l'appliquerons qu'à des corps~$\kk$ localement compacts, auquel cas c'est immédiat. Voir~\cite{TheseV} pour le cas général.

\emph{Dorénavant les normes seront supposées homogènes.}

Étant donné une norme sur~$V$, son dual algébrique acquiert une norme duale: à toute forme~$\kk$-linéaire sur~$V$ on associe  sa  norme
 en tant qu'application linéaire de~$V$ vers~$\kk$. Autrement dit on pose
$\Nm{\phi}=\NM{\phi}=\sup_{\Nm{v}\leq 1}\Nm{\phi(v)}$.

\begin{lemme}\label{LemmeC1} Soit~$V$ un espace vectoriel normé de dimension finie non nulle sur~$\kk$, et doit~$B$ la boule unité de son dual. Alors on a l'inégalité
\[
\forall v\in V, \Nm{v}\geq\sup\abs{\phi}(v),
\]
et il n'y a égalité, simultanément pour tout~$v$ de~$V$, que si et seulement si~$\Nm{-}$ est ultramétrique et si les valeurs, dans~$\R$, prises par~$\Nm{-}$ sont celles prises par~$\abs{-}$.
\end{lemme}
\begin{proof}
L'inégalité résulte de ce que pour tout~$\phi$ dans~$B$, nous avons~$\Nm{\phi}\leq 1$ et de ce que, par définition de la norme triple,~$\Nm{\phi(v)}\leq\NM{\phi}\cdot\Nm{v}$.

Le sens direct de l'équivalence découle de~\cite[II §1, Prop. 4 (p. 26)]{Wei74}.

Dans le sens réciproque, on vérifie directement que le membre de droite est une norme ultramétrique et ne prend que des valeurs prises par~$\abs{-}$.
\end{proof}

\section{Espaces analytiques d'après Berkovich}\label{sectionD}
Le but de cette section est tout d'abord de rappeler la construction des espaces analytiques de Berkovich, et de l'espace analytifié d'une variété algébrique. La référence exhaustive standard est~\cite{Ber90}.
L'autre but est d'étendre la propriété de convexité du polygone de Newton au fonctions analytiques sur restreinte à l'appartement d'un tore~$T$, une fois plongé dans l'espace analytique de~$T$ ou d'un groupe algébrique~$G$ contenant~$T$.
\subsection{Semi-normes multiplicatives homogènes}\label{sectionD1}
Soit~$\kk$ un corps local muni
d'une valeur absolue ultramétrique~$\abs{-}$. Sur une~$\kk$-algèbre commutative unifère~$A$, on appelle \emph{semi-norme multiplicative} une application\emph{non constante}~$\Nm{-} : A \xrightarrow{} \R_{\geq 0}$ 
qui soit multiplicative, c.-à-d. telle que
\begin{equation}\label{eq24}
\forall f,g\in A,~\Nm{fg}=\Nm{f}\cdot\Nm{g}
\end{equation}
et vérifiant l'inégalité triangulaire
\begin{equation}\label{eq25}
\forall f,g\in A,~\Nm{f+g}\leq\Nm{f}+\Nm{g}.
\end{equation}
Une telle application envoie l'unité~$1_A$ sur~$1$ et l'élément nul~$0_A$ sur~$0$. Elle est dite~\emph{$(k,\abs{-})$-homogène}, ou simplement \emph{homogène} lorsque
\begin{equation}\label{eq26}
\forall\lambda\in\kk, \forall f\in A,~\Nm{\lambda\cdot f}\leq \abs{f}\cdot\Nm{f}.
\end{equation}
Par multiplicativité, il revient au même d'imposer que~$\Nm{-}$ étende~$\abs{-}$, au sens où l’on a~$\Nm{\lambda\cdot 1_A} = \abs{\lambda}$.

\subsection{Analytification}\label{sectionD2}
Soit~$V$ une variété algébrique affine sur~$\kk$. Suivant
V. Berkovich, \cite[1.5.1]{Ber90}, nous appellerons \emph{espace analytique associé à}~$V$ l'espace topologique~$V^\an$ formé de l'ensemble des semi-normes multiplicatives
homogènes sur l'algèbre~$\kk[V ]$ des fonctions régulières sur~$V$, pour la topologie de la convergence simple.
Ce sont aussi les espaces topologiques sous-jacents à certains espaces analytiques que V. Berkovich définit en~\cite[3.1]{Ber90} (cf.~\cite[3.4.2]{Ber90}). Pour toute fonction régulière~$f$ sur~$V$, nous notons~$\abs{f}$ a fonction réelle~$x\mapsto x(f)$ sur~$V^\an$.  Par définition, la topologie sur~$V^\an$ est la plus grossière pour laquelle, pour toute fonction régulière~$f$, la fonction~$\abs{f}:V^\an\rightarrow \R_{\geq0}$ est continue.

La construction de l'espace analytique associé est fonctorielle, et covariante,
de la catégorie~$\Aff_\kk$ des variétés affines sur~$\kk$ dans celle des espaces topologiques. En effet tout
morphisme~$A \to A'$ d'algèbres permet, par composition, de produire une seminorme multiplicative~$A \rightarrow \R_{\geq0}$ à partir d'une semi-norme multiplicative homogène sur~$A' \rightarrow \R_{\geq0}$. Cette opération est bien sûr compatible à la convergence simple.
Pour tout morphisme~$\Phi:V\rightarrow V'$ entre variétés algébriques, nous notons~$\Phi^\an$ l'application continue correspondante~$V^\an\rightarrow V'^\an$.

\subsection{Des tores déployés analytiques \ldots}\label{sectionD3}

Soit~$T$ un tore déployé sur~$\kk$ et notons~$X(T)=\Hom(T,\GL(1))$ son groupe des caractères. On identifie l'algèbre~$\kk[T ]$ des fonctions régulières sur~$T$ à l'algèbre de groupe~$\kk[X(T )]$. Tout
caractère~$\chi: T \rightarrow\GL(1)$, définit, par composition une application additive
\[
Y(T)=\Hom(\GL(1),T)\xrightarrow{}\Hom(\GL(1),\GL(1))\simeq\Z.
\]
Par linéarité, chaque caractère définit une forme linéaire sur l'espace vectoriel réel~$\Lambda=Y(T)\tens\R$ que l'on notera~$\lambda\mapsto\langle\chi,\lambda\rangle$. Pour tout~$\lambda$ dans~$\Lambda$, l'application qui envoie une fonction régulière~$f=\sum_{\chi\in X(T)}a_\chi\cdot\chi$ dans~$\kk[T]$ sur
\begin{equation}\label{eq27}
\max_{\chi\in X(T)}\abs{a_\chi}\cdot\abs{\varpi}^{\langle\chi,\lambda\rangle},
\end{equation}
(où~$\abs{\varpi}$ est la valeur absolue d'une uniformisante~$\varpi$ de~$\kk$) définit clairement une norme homogène sur~$\kk[T]$; cette norme est multiplicative d'après~\cite[2.1, p.~21]{Ber90}.  La formule~\eqref{eq27} induit donc une application, que nous noterons~$\lambda \mapsto \theta_T (\lambda)$, de~$\Lambda$ dans~$T^\an$.

Lorsque l'on fixe~$f=\sum_{\chi\in X(T)}a_\chi\cdot\chi$ dans~$\kk[T]$, et 
que l'on fait varier le paramètre~$\lambda$, le logarithme
\[
\log\left(\abs{f}(\theta_T(\lambda))\right)
=\max_{\chi\in X(T)}
\left(
\log\abs{a_\chi}+\log\abs{\varpi}\cdot\langle\chi,\lambda\rangle
\right)
\]
est le maximum d'un nombre fini de fonctions affines en~$\lambda$. C'est donc une fonction convexe sur~$\Lambda$, et c'est en particulier une fonction continue. Pour toute fonction régulière~$f$ sur~$T$, la composition de~$\abs{f}$ avec~$\theta_T$, qui est logarithmiquement
convexe, est continue. Donc l'application~$\theta_T$ est continue par définition de la topologie sur~$T^\an$.

\subsubsection{}\label{sectionD31}
Le groupe~$T (\kk)$ agit sur~$T$ par translation. Il agit par transport de structure sur~$\kk[T ]$ puis sur~$T^\an$. Concrètement, un élément~$\mu$ de~$T (\kk)$ agit sur~$\kk[T ]$ en envoyant la fonction régulière
\[
x\mapsto\sum_{\chi\in X(T)} a_\chi\cdot \chi(x)
\]
sur la fonction régulière
\[
x\mapsto
\sum_{\chi\in X(T)} a_\chi\cdot \chi(x\mu^{-1})
=
\sum_{\chi\in X(T)} a_\chi\cdot \chi(\mu^{-1})\chi(x).
\]
Soit~$\lambda(\mu)$ l'élément de~$\Lambda$ tel que l'on ait
\[
\langle\chi,\lambda(\mu)\rangle=\log_{\abs{\varpi}}(\chi(\mu))
\]
pour tout~$\chi$ de~$X(T)$. Alors l'action de~$\mu$ envoie la norme~$\theta_T (\lambda)$ sur~$\theta_T (\lambda-\lambda(\mu))$.

\subsection{\ldots aux groupes algébriques affines.}\label{sectionD4}

Soit~$\Phi: T \to G$ un morphisme de
variétés algébriques affines du tore T dans un groupe algébrique affine G. Composant~$\theta_T : \Lambda \rightarrow T^\an$ par $\Phi^\an: T^\an \rightarrow G^\an$, nous obtenons une application continue
de~$\Lambda$ dans~$G^\an$. En outre, pour toute fonction régulière~$f$ sur~$G$, la fonction
réelle~$\abs{f}\circ\Phi^\an \circ\theta_T$ sur~$\Lambda$ s'identifie à~$\abs{f \circ\Phi}\circ\theta_T$, et est par conséquent logarithmiquement
convexe.

Notons que l'action à droite du groupe~$G(\kk)$ sur la variété~$G$ induit par une fonctorialité une action de~$G(\kk)$ sur~$G^\an$. En particulier nous pouvons définir, pour tout~$g$ dans~$G(\kk)$, l'application $(\Phi^\an \circ \theta_T )\cdot g$ translatée de~$\Phi^\an \circ\theta_T$ par~$g$.

\emph{Nous notons~$\kk_s$ une extension algébrique séparablement close de~$\kk$.}
\subsubsection{}\label{sectionD41}
Nous allons généraliser cette construction aux éléments de~$G(\kk_s)$.
Pour toute extension séparable finie~$\kk'$ de~$\kk$, on a un tore déployé~$T_{\kk'}=T\tens_\kk\kk'$ sur~$\kk'$, d'algèbre~$\kk'[T_{\kk'}]=\kk[T ] \tens_\kk\kk'$ isomorphe à~$\kk'[X(T)]$. Par restriction des normes de~$\kk'[T_{\kk'}]$ à~$\kk[T]$ on construit une application~${T_{\kk'}}^\an\to T^\an$. La formule~\eqref{eq27} s'étend à~$\kk[T_{\kk'}]$ et définit encore une norme~\emph{multiplicative} (\cite[2.1 (p.~21)]{Ber90}) homogène.
Nous obtenons ainsi une application~$\theta_{T_{\kk'}}:\Lambda\to {T_{\kk'}}^\an$ dont la composée avec l'application~${T_{\kk'}}^\an\to T^\an$ redonne l'application~$\theta_T$. De surcroît cette extension est compatibles aux morphismes d'extensions~$\kk\to\kk'\to\kk''$.

Nous en tirons une conséquence. Si~$g$ est un élément dans~$G(\kk')$, faisons agir~$g$ à droite sur le foncteur des points de~$G$ restreint à la catégorie des~$\kk'$-agèbres: pour une~$\kk'$-algèbre~$A$, l'élément~$g$ agit par translation à droite sur le groupe~$G(A)$, \emph{via} son image par~$G(\kk')\to G(A)$. Il correspond une action, disons~$a_g$, de~$g$ sur~$\kk'[G]$. D'où, par composition, un morphisme~$\kk[G]\to\kk'[G]\xrightarrow{a_g}\kk'[G]\xrightarrow{\Psi_{\kk'}}\kk'[T]$, où~$\Psi_{\kk'}$ est le morphisme de~$\kk'$-algèbres déduit de~$\Phi$. Ainsi, pour chaque~$\lambda$ dans~$\Lambda$, de la norme multiplicative homogène correspondante~$\kk[T]\to\R_{\geq0}$, par~\eqref{eq27}, on déduit, par composition, une norme multiplicative homogène~$\kk[G]\to\R_{\geq0}$.
Nous la noterons~$(\Phi^\an\circ\theta_T)g$ l'application de~$\Lambda$ dans~$G^\an$ ainsi obtenue.

Cette construction est manifestement compatible aux extensions: si~$\kk'$ est une extension finie de~$\kk'$ et~$g''$ l'image de~$g$ dans~$G(\kk'')$, alors~$(\Phi^\an\circ\theta_T)g=(\Phi^\an\circ\theta_T)g''$. Elle ne dépend donc que de l'image de~$g$ dans~$G(\kk_s)$, peu importe le morphisme~$\kk'\to\kk_s$. Autrement dit cette construction ne dépend que de l'image de~$g$ par~$G(\kk_s)\to G^\an$.

Pour~$f$ dans~$\kk'[T]$, la formule~\eqref{eq27} définit encore une fonction convexe de~$\lambda$. Il en résulte que pour~$f$ dans~$\kk[G]$ et~$g$ dans~$G(\kk')$, la composée~$\abs{f}\circ((\Phi\circ\theta_T)g)$ est une fonction convexe sur~$\Lambda$. Ainsi l'application~$(\Phi\circ\theta_T)g$ est continue.

\subsubsection{}\label{sectionD42} Étendons maintenant cette construction aux éléments~$p$ de~$G^\an$.

Soit un point~$p$ de~$G^\an$; ce point est dans l'adhérence de l'image de~$G(\kk_s)$ dans~$G^\an$. Nous allons alors construire l'application~$(\Phi^\an\circ\theta_T )\cdot p$ comme limite simple de fonctions
de la forme~$(\Phi^\an\circ\theta_T )\cdot g$, avec~$g$ dans~$G(\kk_s)$, lorsque l'image de~$g$ dans~$G^\an$ tend vers~$p$.

Montrons que cette limite simple existe et est unique.

\begin{proof}
Notons~$G_{\kk_s}$ le groupe groupe algébrique affine sur~$\kk_s$ d'algèbre~$\kk[G]\tens\kk_s$, notons~$T_{\kk_s}$ son tore sur~$\kk_s$ déployé d'algèbre~$\kk[T]\tens\kk_s$, et~$\Phi_{\kk_s}$ le morphisme~$G_{\kk_s}\to T_{\kk_s}$ issu de~$\Phi$. Soit~$\lambda$ dans~$\Lambda$, soit~$f$ dans~$\kk[G]$ et, pour tout~$g$ dans~$G(\kk_s)$, notons~$f\circ(\Phi_{\kk_s}\cdot g)=\sum_{\chi\in X(T)}a_\chi(g)\chi$ la fonction de~$\kk_s[X(T)]$ qui s'obtient en translatant~$\Phi_{\kk_s}:G_{\kk_s}\to T_{\kk_s}$ par~$g$ puis en composant par~$f$. Tout revient à montrer que lorsque l'image
de~$g$ converge dans~$G^\an$, le nombre réel
\[
\max_{\chi\in X(T)}\abs{a_\chi(g)}\cdot\abs{\varpi}^{\langle\chi,\lambda\rangle}
\]
tend vers une valeur limite.

Par définition,~$\abs{h(g )}$ tend vers une valeur limite pour toute fonction régulière~$h$ sur~$G$ définie sur~$\kk$. Il suffit donc de montrer que sauf pour un nombre fini de caractères~$\chi$, les fonctions 
$g\mapsto a_\chi(g)$ sont nulles et que, pour tout caractère~$\chi$ de~$T$, la formule~$g\mapsto a_\chi(g)$ définit une fonction régulière. Cela résulte de ce que l'action de~$G(\kk_s)$ sur~$\kk_s[G]$ est union de sous-espaces stables sous~$\mathrm{Aut}(\kk_s/\kk)$ de dimension finie, et que l'action de~$G(\kk_s)$ sur un tel sous-espce provient d'une représentation de~$G$ définie sur~$\kk$, \cite[1. \S1 1.9]{Bor91}.
\end{proof}

Comme une limite simple de fonctions convexes est convexe, pour tout~$f$
dans~$\kk[G]$, la composée~$\abs{f}\circ((\Phi^\an\circ\theta_T )\cdot p)$ est une fonction convexe sur~$\Lambda$. En particulier c'est une fonction continue. Par conséquent~$(\Phi^\an\circ\theta_T )\cdot p$ est continue

\subsection{Un Critère}\label{sectionD5} Pour toute extension finie~$\kk'$ de~$\kk$, notons~$G_{\kk'}$ la variété algébrique affine sur~$\kk'$ associée à l'algèbre de type fini~$\kk[G]\tens\kk'$, et~${G_{\kk'}}^\an$ l'espace analytique correspondant. Notons que~$G(\kk')$ est naturellement un \emph{groupe} algébrique affine sur~$\kk'$: son foncteur de points s'écrit~$G_{\kk'}(A)=G(A)$ pour une~$\kk'$-algèbre~$A$.

L'application~$\kk[G]\to\kk[G]\tens\kk'$ induit, par restriction des semi-normes, une application continue~${G_{\kk'}}^\an\to G^\an$. Ces applications sont manifestement compatibles aux morphismes d'extensions.

Notons que le morphisme~$\Phi:T\to G$ induit un morphisme~$\Phi_{\kk'}:T_{\kk'}\to G_{\kk'}$, d'où une application continue~${\Phi_{\kk'}}^\an:{T_{\kk'}}^\an\to {G_{\kk'}}^\an$. Le relèvement~$\Lambda\to{T_{\kk'  }}^\an$ de~$\theta_T:\Lambda\to T^\an$, obtenu en étendant la formule~\eqref{eq27}, est lui aussi compatibles aux extensions.
\begin{proposition} \label{propD1}\label{PropD1}
Soit~$\theta:\Lambda\to G$ une application continue. Supposons que pour toute extension finie~$\kk'$ de~$\kk$, il existe~$\theta_{\kk'}:\Lambda\to {G_{\kk'}}^\an$ telle que
\begin{enumerate}[label=\alph*.]
\item \label{D11} $\theta$ soit la composée de~$\theta_{\kk'}$ avec l'application~${G_{\kk'}}^\an\to G^\an$ ci-dessus;
\item \label{D12}	pour tout~$\mu$ dans~$T(\kk')$, l'action de~$T(\kk')$ par translation à gauche sur la fonction~$\theta_{\kk'}$ commute correspond à son action sur~$\Lambda$:
\begin{equation}
\forall\mu\in T(\kk'),~\theta_{\kk'}(x)\cdot\Phi(\mu)=\theta_{\kk'}(x-\lambda(\mu))
\end{equation}
\end{enumerate}
Alors~$\theta$ est l'application~$(\Phi^\an\circ\theta_T)\cdot p$ où~$p$ est le point~$\theta(0)$. 
\end{proposition}
\begin{proof}
Par limite simple on peut supposer que~$p$ est dans l'image
de~$G(\kk_s)$, puis quitte à considérer une extension finie, que~$p$ appartient à~$G(\kk)$, quitte à faire agir~$p$ à droite, que~$p$ est l'élément neutre.

Par continuité de~$\theta$ et~$\theta_T$ il suffit de montrer~$\theta(\lambda) = (\Phi^\an\circ\theta_T (\lambda))$ pour un ensemble
dense de~$\lambda$ dans~$\Lambda$.

Or l'orbite dans~$\Lambda$ de~$0$ sous l'action de~$T (\kk_s)$ est dense.
Cela provient de la description de cette action pour tout corps local contenu dans~$\kk_s$, grâce à la section~\ref{sectionD31} et au fait que quitte à considérer des racines de l'uniformisante d'ordre premier à la caractéristique, on peut approcher tout nombre réel positif par la valeur absolue d'éléments de~$\kk_s$.

Enfin~$\theta$ et~$(\Phi^\an \circ \theta_T (\lambda))$ concordent en~$0$ et vérifient la loi de transformation (cf. section~\ref{sectionD31}).
\end{proof}
Des sections~\ref{sectionD31} et~\ref{sectionD4} on déduit que cette Proposition~\ref{propD1} a pour corollaire le suivant.
\begin{corollaire}\label{coroD2}
Soit~$\theta:\Lambda\to G^\an$ une application telle que dans la Proposition~\ref{propD1}. Alors pour toute fonction régulière~$f$ sur~$G$, la fonction réelle~$\abs{f}\circ\theta$ est logarithmiquement convexe.
\end{corollaire}
\section{Immeubles euclidiens de Bruhat-Tits}

\subsection{Avant-propos} L'\emph{immeuble euclidien} de Bruhat-Tits d'un groupe algébrique semisimple~$G$ sur un corps local ultramétrique~$\kk$ est un analogue, dans le contexte ultramétrique, de l'espace riemannien symétrique des sous-groupes compacts maximaux associé à un groupe de Lie semi-simple réel~$L$, qui s'écrit~$L/K$, pour un sous-groupe compact maximal~$K$ de~$L$. Dans le contexte ultramétrique, l'analogue naïf du théorème de Cartan ne vaut plus:~$G(\kk)$ ne contient en général pas, à conjugaison près, d'unique sous-groupe compact maximal. Toutefois, du
moins pour les groupe déployés, les énoncés 5.3.1 et 5.3.3 de~\cite{Ber90} restituent \emph{a
posteriori} cette facette de l'analogie entre immeubles et espaces symétriques.

\subsection{Propriétés}
Soit~$\kk$ un corps local muni d'une valeur absolue ultramétrique~$\abs{-}$, et soit~$G$ un groupe algébrique semi-simple sur~$\kk$. L'\emph{immeuble euclidien de Bruhat-Tits} de~$G$ sur~$\kk$, ou plus simplement «~immeuble~», désigne un certain espace métrique~$\Ik(G)$ muni d'une action fidèle proprement continue du groupe topologique~$G(\kk)$ (à gauche, par isométries). L'immeuble est uniquement défini à unique isométrie près. (cf~\cite[2.1]{Tit79}). Le stabilisateur
de l'immeuble est réduit au centre de~$G$. Les stabilisateurs des points de~$\Ik(G)$ sont des sous-groupes dits \emph{parahoriques} de~$G(\kk)$; 
ils sont compacts et ouverts dans~$G(\kk)$: suivant~\cite[Introduction]{BT84} (voir aussi~\cite[3.4.1]{Tit79}), ce sont des groupes de la forme~$G(\mathscr{O}_\kk)$ pour certaines formes entières de~$G$ sur l'anneau des entiers ultramétriques~$\mathscr{O}_\kk$ («~schémas
en groupes plats prolongeant G~»).

L'immeuble~$\Ik(G)$ admet une famille distinguée de parties, appelées \emph{appartements}, réunissant les propriétés suivantes.
\begin{enumerate}[label=\alph*.]
\item \label{a.} L'ensemble des appartements est stable sous l'action de~$G(\kk)$, et l'action de~$G(\kk)$ sur l'ensemble des appartements est transitive.
\item \label{b.} Les appartements sont isométriques à un espace vectoriel euclidien. En particulier, pour tout appartement~$A$, et tout~$g$ dans~$G(\kk)$, l'isométrie~$A\xrightarrow{} g A$ est une application affine.
\item \label{c.} L'immeuble~$\Ik(G)$ est réunion de ses appartements, et tout point a un voisinage formé d'une réunion finie d'appartements.
\item \label{d.} Le stabilisateur dans~$G(\kk)$ d'un appartement donné agit via un réseau du grou\-pe d'isométries de cet appartement.
\item \label{e.} Deux points quelconques de~$\Ik(G)$ sont contenus dans un appartement commun \cite[7.14.18]{BT72}.
\item \label{f.} L'immeuble~$\Ik(G)$ un espace de Hadamard : l'inégalité CAT(0) est satisfaite \cite[3.2]{BT72}.
\item \label{g.} L'immeuble~$\Ik(G)$ admet une structure polysimpliciale~$G(\kk)$-invariante, dont un polysimplexe est domaine fondamental et intersection d'appartements. Le stabilisateur d'un appartement donné préserve un pavage issu d'une structure polysimpliciale qui s'étend en une structure polysimpliciale invariante
sur~$\Ik(G)$.
\end{enumerate}

\subsection{Avertissement}
Nous nous reposons sur~\cite{RTW09} pour la construction
de l'immeuble de Bruhat-Tits. Les hypothèses de travail de ces auteurs sont
énoncées en~\cite[1.3.4]{RTW09}, numéro qui, par ailleurs, indique explicitement que
ces hypothèse sont vérifiées si le corps de base~$\kk$ est un corps local, ce qui est
le cas considéré ici. Une autre référence dans le cas des corps locaux est~\cite{Tit79}.
Cette dernière, reposant sur l'exposé axiomatique~\cite{BT72}, indique, en~\cite[1.5]{Tit79},
quelques réserves sur la satisfiabilité des hypothèses de~\cite{BT72}, qui sont ramenées,
en ce qui concerne~\cite{Tit79}, aux propriétés 1.4.1 et 1.4.2 de~\cite{Tit79}. La
suite~\cite{BT84} de l'exposé~\cite{BT72}, suite postérieure à la référence~\cite{Tit79}, démontre
que les hypothèses de~\cite{BT72} sont satisfaites « pour tout groupe réductif sur un
corps de valuation discrète hensélien à corps résiduel parfait » (Introduction,
page 9).

Nous esquissons la construction de l'immeuble indiquée dans~\cite{RTW09}. Que
cette construction vérifie les propriétés indiquées plus haut résulte, pour certaines
de ces propriétés, de la construction même, qui procède par analyse-synthèse.
Pour les autres propriétés cela résulte d'une part de la satisfiabilité des hypothèse
de travail de~\cite{BT72} pour les corps locaux pour laquelle nous venons d'indiquer
des références ; d'autre part de certaines conclusions de~\cite{BT72}; enfin, dans le cas
des corps locaux notamment, de la référence~\cite{Tit79}.

\subsection{Construction}\label{sectionE4} Comme tout appartement A de l'immeuble~$\Ik(G)$ est une partie génératrice, on peut construire~$\Ik(G)$ comme quotient de~$G(k) \times A$. Les relations par lesquelles on quotiente sont engendrées par celles qui définissent le stabilisateur de chaque point de~$A$, et celles qui déterminent l'action sur~$A$ du stabilisateur
de~$A$. C'est l'approche utilisée dans~\cite[1.3]{RTW09}.

Ne discutons que du modèle de A muni de l'action de son stabilisateur.
Fixons un tore déployé maximal~$T$ de~$G$, et posons~$A = \Lambda = \Hom(\GL(1),T ) \tens \R$.
Le normalisateur~$N(T )(\kk)$ de~$T$ dans~$G(\kk)$ agit par transport de structure sur~$\Lambda = \Hom(\GL(1),T )\tens\R$, et le noyau de cette action est le centralisateur~$C(T )(\kk)$
de~$T$ dans~$G(\kk)$. Le groupe quotient~$N(T )(\kk)/C(T )(\kk)$ est un groupe fini, le groupe de Weyl sphérique de~$G$ relatif à~$T$. Alors le stabilisateur de l'appartement~$A$ dans~$G(\kk)$
est~$N(T )(\kk)$ et l'action de~$N(T )(\kk)$ est une action affine
\begin{itemize}
\item dont la partie linéaire est l'action précédente, 
\item et pour laquelle~$T (\kk)$, qui est contenu dans~$N(T )(\kk)$, agit par~\ref{sectionD5}.
\end{itemize}
\section{{Convexit\'{e}}}\label{sectionF}
\subsection{{Notion de convexit\'{e}}}
Rappelons qu'un~\emph{appartement} de~$\Ik(G)$ est l'image d'une partie de la forme~$\{g\}\times\Lambda$ par l'application~$G(\kk)\times\Lambda\rightarrow \Ik(G)$. Les appartements sont permutés transitivements sous l'action de~$G(\kk)$. En outre, le stabilisateur d'un appartement~$A$ agit sur~$A$ de manière affine et cette action contient l'action additive d'un réseau vectoriel de~$\Lambda$ (voir \cite[1.2, 1.3]{Tit79}).

\subsubsection{}\label{sectionF11}
En particulier, si~$F$ est un parallélotope fondamental de ce réseau,~$F$ est une partie bornée qui rencontre tout orbite de~$G(\kk)$ rencontrant cet appartement, ce qui est le cas de toute orbite de~$G(\kk)$. L'immeuble~$\Ik(G)$ contient donc une partie génératrice compacte.

Rappelons que deux points quelconques~$x$ et~$y$ de~$\Ik(G)$ sont contenus dans un appartement commun~(\cite[7.14.18]{BT72}). on peut donc définir le segment~$[x;y]$ joignant~$x$ à~$y$ dans~$A$. Comme le stabilisateur d'un appartement~$A$ agit de manière affine sur~$A$, ni le segment~$[x;y]$ ni la structure affine sur~$[x;y]$ ne dépendent de l'appartement choisi (cf.~\cite[2.2.1]{Tit79}).

\label{convexité}
Une fonction réelle continue sur~$\Ik (G)$ dont la restriction à tout segment
est \emph{con\-ve\-xe} (resp. \emph{affine}), sera dite convexe (resp. affine). Il revient au même
de dire que la restriction à chaque appartement est convexe (resp. affine),
relativement à la structure affine de cet appartement.

Bien évidemment, les fonctions affines et les fonctions dont le logarithme est
affine sont convexes. En outre, toute fonction réelle qui s'écrit comme borne
supérieure ou limite simple de fonctions convexes est convexe.

\label{convexite2}
Une partie de~$\Ik(G)$ sera dite convexe si elle contient tout segment dont elle contient les extrémités.

\subsection{Une décomposition de l'immeuble}
Soit~$k$ un sous-groupe compact de~$G(\kk)$. Le groupe compact~$k$ agit par isométries sur l'immeuble~$\Ik(G)$.
Notons~$\Ik(G)^k$ le lieu fixe de l'action de~$k$, et~$Z_G(k)$ le centralisateur de~$k$ dans~$G$.
Alors~$\Ik(G)^k$ est stable sous l'action du centralisateur~$Z_G(k)(\kk)$ de~$k$ dans~$G(\kk)$.
Nous allons montrer le premier énoncé suivant.

\begin{theoreme}\label{TheoF1}\label{propF1}
Soit~$k$ un sous-groupe compact de~$G(\kk)$ dont l'adhérence
de Zariski dans~$G$ est un sous-groupe fortement réductif, et soit~$Z_G (k)$ le centralisateur de~$k$ dans~$G$. Alors il existe une partie compacte non vide~$C$ de~$\Ik(G)^k$ qui est génératrice pour l'action de~$Z G (k)(\kk)$: on a~$\Ik(G)^k=Z_G (k)(\kk)\cdot C$.
\end{theoreme}
Commençons par un lemme. 

\begin{lemme}\label{LemmeF2}\label{lemmeF2}\footnote{N.d.É.: Bien que l'on se réfère à~\cite{PR94}, ce lemme est basé sur le travail~\cite{Richardson} de Richardson, algébrique, dont c'est une variante, et conséquence, ultramétrique.}
Soit~$k$ un sous-groupe compact de~$G(\kk)$ dont l'adhérence
de Zariski dans~$G$ est un sous-groupe fortement réductif dans~$G$,  et soit~$Z_G (k)$ le centralisateur de~$k$ dans~$G$.

Pour tout compact~$C$ de~$G(\kk)$, il existe un compact~$C'$ de~$G(k)$ tel que le transporteur
\begin{equation}
T (k, C) = \left\{g \in G(k)\middle|gkg^{-1} \subseteq C\right\}
\end{equation}
de~$k$ dans~$C$ s'écrive~$C'\cdot Z_G (k)(\kk)$.

\end{lemme}
\begin{proof}
 Tout d'abord ce transporteur est fermé. Pour chaque~$x$ dans~$k$ l'application~$g \mapsto g xg^{-1}$ est continue, l'image inverse de~$C$ est fermée; ce transporteur est l'intersection de ces images inverses, des fermés. Tout revient donc à montrer que~$T (k,C )$ est le
saturé par~$Z_G (k)(\kk)$ d'une partie relativement compacte. Comme~$T (k,C )$ est manifestement invariant par~$Z_G (k)(\kk)$, il suffit de montrer qu'il est contenu dans le saturé d'une partie compacte.

Remarquons que, pour la topologie de Zariski,~$k$ est noethérien. Par con-
séquent, pour cette topologie, il est topologiquement de type fini. En particulier, il existe une partie finie~$\{x_1 ; \ldots ; x_n \}$ topologiquement génératrice de~$k$ pour la topologie de Zariski. Le centralisateur de cette partie est le centralisateur de~$k$, c'est-à-dire~$Z_G (k)$.

Appliquant~\cite[Theorem 16.4]{Richardson}, on montre que l'application
\[
g\mapsto (gx_1g^{-1};\ldots;gx_ng^{-1})
\]
induit une immersion fermée de~$G/Z_G(k)$ dans~$G^n$. Par conséquent l'application correspondante
\[
\phi:\left(G/Z_G(k)\right)(\kk)\to G(\kk)^n
\]
est~\emph{propre}. En particulier l'image inverse de~$C^n$ dans~$\left(G/Z_G(k)\right)(\kk)$ est compacte. Or cette image inverse de~$C^n$ contient l'image de~$T(k,C)$ par~$G(\kk)\to\left(G/Z_G(k)\right)(\kk)$.

Il suffit donc de montrer que tout compact de~$\left(G/Z_G(k)\right)(\kk)$ rencontre l'image par~$\phi$ de~$G(\kk)$ en l'image d'un compact. D'après~\cite{PR94}, l'application~$\phi$ est ouverte, et les orbites de~$G(\kk)$ dans~$\left(G/Z_G(k)\right)(\kk)$ sont toutes ouvertes; elles sont donc aussi fermées. En particulier l'image~$\phi(G(\kk))$ est fermée, et intersecte donc tout compact en un compact. Il suffit de montrer que tout compact de~$G(\kk)/Z(\kk)$ est l'image d'un compact. D'après la propriété de Borel-Lebesgue, il suffit de travailler localement (c.à.d. montrer que tout ouvert assez petit est contenu dans l'image d'un compact). Or~$\phi$ est ouverte et~$G(\kk)$ est localement compact.
\end{proof}

Avant de démontrer le Théorème~\ref{TheoF1}, rappelons quelques faits, dont certains
sont bien connus.
\begin{enumerate}[label=\alph*.]
\item \label{a.} Il existe un compact~$F$ de~$\Ik(G)$ rencontrant toute orbite de~$G(\kk)$ (cf.~\ref{sectionF11}).
\item \label{b.} Pour toute partie bornée non vide~$P$ de~$\Ik(G)$, le stabilisateur de~$P$ dans~$G(\kk)$ est un sous-groupe compact et ouvert (cf.\cite[3.2]{Tit79}, \cite[Introduction]{BT72}).
En particulier, pour tout point~$p$ de~$\Ik (G)$, le fixateur de~$p$
dans~$G(k)$ est un sous-groupe compact et ouvert.
\item \label{c.} Le fixateur commun à tous les éléments d'une partie bornée non vide est
compact et ouvert (cf. \emph{supra}).
\item \label{d.} Pour tout sous-groupe ouvert~$U$ de~$G(\kk)$, le lieu fixe de~$U$ dans~$G$ est une partie compacte de~$\Ik(G)$.
\item \label{e.} Tout sous-groupe compact de~$G(\kk)$ est contenu dans un sous-groupe compact maximal.
\item \label{f.} Les sous-groupes compacts maximaux de~$G(\kk)$ sont ouverts.
\item \label{g.} Ils forment un nombre fini de classes de conjugaison (\cite[3.3.3]{BT72}).
\item \label{h.} Tout sous-groupe compact ouvert de~$G(\kk)$ est contenu dans un nombre fini de sous-groupe compact maximaux.
\end{enumerate}
\begin{proof}[Démonstration du Théorème~\ref{TheoF1}]
D'après le point~\ref{a.}, il existe un compact~$F$ de~$\Ik (G)$ rencontrant toute orbite de~$G(\kk)$. D'après le point~\ref{b.}, le stabilisateur commun à tous les points de~$F$ est un sous-groupe \emph{compact et ouvert} de~$G(\kk)$. Notons-le~$K_F$.

Pour tout point~$f$ de~$F$ le stabilisateur de~$f$, disons~$K_f$, est également un sous-groupe compact et ouvert de~$G(\kk)$ (point~\ref{b.} ci-dessus). Par construction ces groupes contiennent~$K_F$. Appliquant le point~\ref{h.}, il s'ensuit que l'ensemble~$E_F=\{K_f|f\in F\}$ de sous-groupes compacts et ouverts de~$G(\kk)$ est un ensemble \emph{fini}.

Pour~$K$ dans~$E_F$, notons~$F_K = \{ f \in F |K_f = K \}$. Par construction de~$E_F$, le compact~$F$ s'écrit comme l'union finie~$F = \bigcup\{F_K |K \in EF \}$. Notons que chaque~$F_K$, étant
contenu dans~$F$, est borné.

Soit~$k$ le groupe compact mentionné dans l'énoncé du théorème. Pour~$g$
dans~$G(\kk)$ et~$f$ dans~$F_K$, le point~$g\cdot f$ de~$\Ik (G)$ est fixé par~$k$ si et seulement si~$g^{-1}kg$ est contenu dans le stabilisateur de~$f$, c'est-à-dire dans~$K$.  Autrement dit~$g^{-1}$ est
dans le transporteur~$T (k, K )$ de~$k$ dans~$K$. D'après le Lemme~\ref{lemmeF2}, il existe un compact~$C_K$ de~$G(\kk)$ tel que~$T(k,K)$ s'écrive~$C_K\cdot Z_G(k)$. Par conséquent, l'élément~$g\cdot f$ ci-dessus appartient à la partie~$Z_G(k)(\kk)\cdot {C_K}^{-1}\cdot F_K$ du lieu fixe~$\Ik(G)^k$ de~$k$ agissant sur l'immeuble~$\Ik(G)$.

On a montré que tout point de la forme~$g\cdot f$, avec~$g$ dans~$G(\kk)$ et~$f$ dans~$F_K$ appartient en fait à~$Z_G(k)(\kk)\cdot {C_K}^{-1}\cdot F_K$.

Comme~$F$ rencontre toute orbite de~$G(\kk)$ dans~$\Ik (G)$, tout point~$p$ de~$\Ik (G)$ s'écrit~$f\cdot g$ avec~$f$ dans~$F$ et~$g$ dans~$G(\kk)$.  Ainsi tout point fixe de k appartient
à~$Z_G(k)(\kk)\cdot {C_K}^{-1}\cdot F_K$ pour un certain~$K$ dans~$F$.  Par conséquent,~$\Ik (G)^k$ est contenu dans~$\bigcup_{K\in E_F}Z_G(k)(\kk)\cdot {C_K}^{-1}\cdot F_K$.

Comme~$E_F$ est fini et que les~$F_K$ et~$C_K$ sont bornés, la partie~$\bigcup_{K\in E_F}Z_G(k)(\kk)\cdot {C_K}^{-1}\cdot F_K$ est bornée, donc son adhérence est compacte. Remarquons que~$\Ik(G)$ est fermé et invariant sous l'action de~$Z_G(k)(\kk)$. Par conséquent
\[
C=\overline{\bigcup_{K\in E_F}{C_K}^{-1}\cdot F_K}\cap \Ik(G)^k
\]
est un compact de~$\Ik(G)$ tel que~$\Ik(G)^k=CZ_(G)(k)(\kk)$.

\end{proof}
\subsection{Plongement analytique et Convexité des fonctions régulières sur l'immeuble} Nous basant sur~\cite{RTW09}, démontrons l'énoncé suivant.
\begin{theoreme}\label{TheoF3}
Soit~$\kk$ un corps local muni d'une valeur absolue ultramétrique, soit~$G$ un groupe algébrique semi-simple sur~$\kk$ et fixons un tore
algébrique déployé maximal~$T$ dans~$G$. Notons~$\Phi : T \rightarrow G$ le morphisme
d'inclusion, $\theta_T : \Lambda \rightarrow T^\an$ l'application~\eqref{eq23},~$\phi^\an : T^\an \rightarrow G^\an$ l'application correspondante, et~$\Ik(G)$ l'immeuble de Bruhat-Tits de~$G$ sur~$\kk$. 

Pour tout point~$p$ de~$G^\an$, on note~$((\Phi^\an \circ \theta_T )\cdot p) : \Lambda \rightarrow G^\an$ l'application définie dans la
section~\ref{sectionD4}, et pour tout~$g$ dans~$G(\kk)$ , on note~$g((\Phi^\an \circ \theta_T )\cdot p)$ l'application
translatée.

Alors il existe un point~$p$ de~$G^\an$ tel que l'application de~$G(\kk) \times \Lambda$ donnée
par
\begin{equation}\label{eq29}
(g,\lambda)\mapsto g\cdot ((\Phi^\an\circ \theta_T)\cdot p)(\lambda)
\end{equation}
passe au quotient en une application équivariante à gauche
\[
\theta:\Ik(G)\rightarrow G^\an.
\]

En outre on peut supposer que le stabilisateur à droite~$G_p$ de~$p$ dans~$G(k)$ est compact et ouvert.
\end{theoreme}

Dans le cas déployé, cet énoncé résulte de la construction explicite de V. Ber\-ko\-vich, dans \cite[5.3]{Ber90} et du Théorème 5.4.2 qui s'ensuit.

Dans le cas général, notre référence est~\cite{RTW09}, dont l'approche est différente, et repose sur les propriétés de fonctorialité des immeubles par des
extensions, non nécessairement algébriques, de corps ultramétriques. Pour pouvoir démontrer l'énon\-cé~\ref{TheoF3}, nous allons utiliser le critère~\ref{propD1}.
\begin{proof}
Soit~$\Theta:\Ik(G)\times\Ik(G)\to G^\an$ l'application de~\cite[2.3]{RTW09}.
Il suit de~\cite[Proposition~2.12]{RTW09} que l'application~$\Theta$ est continue et vérifie
\[
	\forall x,y\in \Ik(G),
		\forall g,h \in G(\kk),
				\Theta(gx,hy)=h\Theta(x,y)g^{-1}.
\]
Soit~$o$ dans~$\Ik(G)$ et notons~$\Theta_o$ l'application~$x\mapsto\theta(o,x)$. Alors~$\Theta_o$  est une application continue et équivariante de~$\Ik(G)$ dans~$G^\an$. Montrons que~$\Theta_o$ répond à
l'énoncé.

Par équivariance, il suffit donc de montrer qu'elle s'écrit sous la forme~\eqref{eq29} sur un seul appartement, par exemple~$\Lambda$, de~$\Ik (G)$.

Il suffit donc de montrer que la restriction de~$\Theta_o$ à~$\Lambda$ vérifie le critère~\ref{PropD1}. L'application~$\Theta_o$ est bien continue.
Soit~$\kk'$ une  une extension finie de~$\kk$, et soit~$\Theta_{\kk'}$ l'application composée issue du coin supérieur gauche du carré commutatif de \cite[Proposition 2.12 (ii).]{RTW09}.
Par commutativité, l'application~$(\Theta_{\kk'})_o:x\mapsto \Theta_{\kk'}(o,x)$ de~$\Lambda$ dans~$G_{\kk'}^\an$ répond  à la condition~\ref{D11} de~\ref{PropD1}.
Appliquant la proposition 2.12 de~\cite{RTW09} au corps~$\kk'$, nous obtenons que l'application~$(\Theta_{\kk'})_o$ est équivariante à gauche sur~$\Lambda$.
La seconde condition de la Proposition~\ref{PropD1} découle ainsi de la description de l'identité de l'action de~$T(\kk')$ sur~$\Lambda$, pris comme appartement (cf. section~\ref{sectionE4}), avec l'action donnée en section~\ref{sectionD31}.

\end{proof}
En appliquant le Corollaire~\ref{coroD2}, on en déduit ceci.
\begin{corollaire}\label{coroF4}
Soit~$\Theta$ une application~$\Ik(G) \rightarrow G^\an$ telle que dans le
Théorème~\ref{TheoF3}. 

Alors pour toute fonction régulière~$f$ sur~$G$, la fonction réelle~$\abs{f}  \circ\Theta$ est logarithmiquement convexe sur~$\Ik(G)$.
\end{corollaire}

%% file: 4-these-VI/Bombay-Lemma-A.tex
\section*{Introduction}
In this article, we show how, along the lines of~\cite{EMSGAFA}, to combine
\begin{enumerate}
\item \label{i1}application of Dani-Margulis linearisation method, in the form given by D.~Kleinbock and G.~Tomanov in~\cite{KT}; 
\item \label{i2}with the geometric results of \cite{Lemma} and \cite{Lemmap} (or rather their combination in Theorem~\ref{SBombay})
\end{enumerate}
in order to prove a more general and precise variant of the main result of~\cite{EMSGAFA} encompassing the $S$-arithmetic setup (Theorem~\ref{T1}, \ref{T2} and \ref{T3} below). In particular, together with~\cite{Lemma} and~\cite{KT}, this article contains a complete alternative proof of the main result of~\cite{EMSGAFA}. 

The result from~\cite{KT} that we will use is a consequence of \cite{KT}~Theorem~9.3 which is slightly more general than an application that~\cite{KT} already worked out, namely~Theorem~9.4. In order to obtain this generalisation 
we will need to check that the lemmas from~\cite{KT} allow to apply results from~\cite{KT} to the case we consider. This is done in section \ref{preliminaries} and \ref{good}. Section~\ref{lattices} recalls some of the notations of~\cite{KT} that we will use in order to express how we will use Theorems~9.3 and 9.4 of~\cite{KT}.

The results established concern non-divergence of certain sequences of translated measures. They are generalised in~\cite{0-RZ} which determines the limits of such sequences. Nonetheless the reference~\cite{0-RZ} relies on the non-divergence proved here.

\section{Non-divergence relative to semisimple $S$-Arithmetic lattices} 

Let us fix once for all a \emph{finite} set~$S$ of places of the field~$\Q$ of rational numbers. Excluding this section,~$S$ will be assumed to contain the archimedean place. Denote the completion of~$\Q$ at~$v$ by~$\Q_v$ and endow the product~$\Q_S=\prod_{v\in S}\Q_v$ with the ``maximum'' norm 
$$\abs{(a_v)_{v\in S}}_S=\max\{\abs{a_v}|v\in S\}.$$ 
so the balls for the product metric below are exactly products of balls of the factors~$\Q_v$ 
$$d_S((a_v)_{v\in S},(b_v)_{v\in S})=\abs{(a_v)_{v\in S}-(b_v)_{v\in S}}_S=\max\{\abs{a_v-b_v}|v\in S\}.$$ We fix a natural integer~$m$ and endow the group~$GL(m,\Q_S)$ with its natural metric topology. We consider a \emph{smooth reductive} closed algebraic subgroup~$G$ of~$GL(m)$ over~$\Q$, and we endow~$G(\Q_S)$ with the induced topology. 
Define~$\Z_S$ to be the localised ring~$\Z[1/p_1;\ldots;1/p_k]$, where~$p_1$,\dots,$p_k$ are the primes corresponding to the finite places in~$S$. We denote by~$G(\Z_S)$ the intersection~$GL(m,\Z_S)\cap G(\Q_S)$, and for any algebraic subgroup~$H$ in~$G$, by~$H(\Z_S)$ the intersection~$H(\Q_S)\cap G(\Z_S)$.

Recall that~$G(\Z_S)$ is a lattice in~$G(\Q_S)$ if and only if~$G$  has no non constant character~$G\to GL(1)$ defined over~$\Q$ (\cf~\cite{BH-C}). Of course, this holds for any reductive subgroup~$H$ of~$G$ instead of~$G$. We prove the following.
\begin{theorem}\label{T1}
 Assume that~$G(\Z_S)$ is a lattice in~$G(\Q_S)$, and consider a reductive subgroup~$H$ of~$G$ such that $H(\Z_S)$ is also a lattice in~$H(\Q_S)$. Let us denote by~$\mu_H$ the direct image by 
\begin{equation}\label{muH}
H(\Z_S)\sous H(\Q_S)\to G(\Z_S)\sous G(\Q_S)
\end{equation} 
of the $H(\Q_S)$-invariant probability on~$H(\Z_S)\sous H(\Q_S)$.

 For any~$g$ in~$G(\Q_S)$ we write $\mu_{Hg}$ for the direct image of the probability~$\mu_H$ by the right action of~$g$ on~$G(\Z_S)\sous G(\Q_S)$.

 Then the family~$\left(\mu_{Hg}\right)_{g\in G}$ is narrow if and only if the centraliser of~$H(\Q_S)$ in~$G(\Q_S)$ has compact image in~$G(\Z_S)\sous G(\Q_S)$.
\end{theorem}
Recall that a family~$(\mu_i)_{i\in I}$ of bounded positive measures on a locally compact space~$X$ is said to be narrow if for every~$\epsilon>0$, there is some compact~$K_\epsilon$ in~$X$ such that 
\begin{equation}\label{narrow}\forall i\in I, \mu_i(X\smallsetminus K_\epsilon)<\epsilon.\end{equation}
Equivalently this family is relatively compact in the space of probabilities on~$X$.

Theorem \ref{T1} will actually follows from this more general one.
\begin{theorem}\label{T2}
 Let~$Y_S$ be a subset of $G(\Q_S)$ satisfying Theorem~\ref{SBombay}.

Then the family~$\left(\mu_{Hy}\right)_{y\in Y_S}$ is narrow.
\end{theorem}
Let us show how Theorem \ref{T1} implies Theorem \ref{T2}.
\begin{proof}
Let~$Z$ be the centraliser of~$H$ in~$G$. We first show that the assumption on~$Z$ is necessary. For any compact subset~$C$ of~$H(\Q_S)$ and any point~$x$ in the image of~(\ref{muH}), we have $xCz=xzC$ for any $z$ in~$Z(\Q_S)$. Consequently, as~$xz$ go to infinity in Alexandroff compactification of~$G(\Z_S)\sous G(\Q_S)$,~$xCz$ will uniformly go to infinity. Choosing~$C$ to have positive, as we may, $\mu_H$~measure, we contradict the claimed narrowness, namely formula~(\ref{narrow}) for~$\epsilon=\mu_H(C)$.

Let us now prove the sufficiency of the assumption on~$Z$: there is a compact subset~$C$ of~$Z(\Q_S)$ such that $Z(\Q_S)=Z(\Z_S)\cdot C$. From Theorem~\ref{SBombay}, the~$Y_S$ is such that~$G(\Q_S)=Z(\Q_S)Y_S$. Consequently we get~$G(\Q_S)=Z(\Z_S)\cdot C Y_S$.

We now remark that for any~$y$ in~$Z(\Z_S)$, one has~$\mu_{Hy}=\mu_H$. Indeed both sides are direct image of the same measure under the same map, because the composition of (\ref{muH}) with the action of~$y$ result in the map (\ref{muH}) itself: indeed, for~$h$ in~$H(\Q_S)$, 
$$G(\Z_S)H(\Q_S)y=~G(\Z_S)yH(\Q_S)=~G(\Z_S)H(\Q_S)$$
as~$y$ centralises~$H$ and belongs to~$G(\Z_S)$ respectively.

Consequently the families~$\left(\mu_{Hg}\right)_{g\in G}$ and~$\left(\mu_{Hg}\right)_{g\in CY_S}$ have the same members. But applying Proposition~\ref{SC} and then Theorem~\ref{T2}, we prove that the latter family~$\left(\mu_{Hg}\right)_{g\in CY_S}$ is narrow, whence the desired result.
\end{proof}

As we may write~$H(\Q_S)=\prod_{v\in S} H(\Q_v)$, and choose below for~$f$ the characteristic function of a fundamental domain (\cf~\cite{BH-C}) of~$H(\Z_S)$ in~$H(\Q_S)$ Theorem~\ref{T2} is a particular case of the following.
\begin{theorem}\label{T3}
 Assume that~$G(\Z_S)$ is a lattice in~$G(\Q_S)$. 

 For each~$v$ in~$S$, let~$G_{\Q_v}$ denote the group on~$\Q_v$ obtained by base change from~$\Q$ to~$\Q_v$. We consider, for each~$v$ in~$S$, 
 a reductive closed algebraic subgroup~$H_v$ of~$G_{\Q_v}$, and we write~$H=\prod_{v\in S} H_v(\Q_v)$.
 Let~$\mu_H$ be some Haar measure on~$H$, let~$f$ be a positive bounded (measurable) $\mu_H$-summable real function on~$H$ and let us denote by~$\mu_f$ the direct image of $f\cdot\mu_H$ by~$$H\to G(\Z_S)\sous G(\Q_S).$$ For any~$g$ in~$G(\Q_S)$ we write~$\mu_{f\cdot g}$ for the direct image of the measure~$\mu$ by the right action of~$g$ on~$G(\Z_S)\sous G(\Q_S)$.

 Let~$Y_S$ be a subset of~$G(\Q_S)$ satisfying Theorem~\ref{SBombay}.

 Then the family~$\left(\mu_{y}\right)_{y\in Y_S}$ is narrow.
\end{theorem}
Actually Theorem \ref{T3} will follow from the following \latin{a priori} weaker corollary.
\begin{corollary}\label{C3}
 There exits at least one positive essentially non zero bounded measurable $\mu_H$-summable function~$f$ on~$H$ such that the conclusion of Theorem~\ref{T3} holds for every subset~$Y_S$ as in Theorem~\ref{SBombay}.
\end{corollary}
Let us see why Theorem \ref{T3} follows from Corrolary \ref{C3}.
\begin{proof}
 We will say that a function (resp. measure)~$f$ \emph{dominates} a function (resp. measure)~$g$ if there is a constant~$C$ such that~$f\leq g$. For measures, this inequality is an inequality of functions on the Borel algebras. We will say that a family of measures~$(\mu_i)_{i\in I}$ is \emph{dominated} by an other family~$(\nu_i)_{i\in I}$ on the same index set, if there is a constant~$C$ such that uniformly relative to~$i$ in~$I$, such that~$\mu_i\leq C\nu_i$.

 \emph{Barycenter} will means barycentre \emph{with positive coefficients}. One readily sees on formula~(\ref{narrow}) that
\begin{equation}\label{dom}
 \text{any family dominated by a narrow family is narrow,}
\end{equation}
\begin{equation}\label{bary}
 \text{the family of barycentres of the elements of a narrow family is narrow,}
\end{equation}
\begin{equation}\label{joint}
 \text{the disjunction of finitely many narrow families is a narrow family,}
\end{equation}
\begin{equation}\label{sous}
 \text{a subfamily extracted from a narrow family is narrow,}
\end{equation}
\begin{equation}\label{clos}
 \text{the family of limits elements of a narrow family is narrow.}
\end{equation}
Note that in the last property, we can equivalently use narrow or vague convergence. Indeed these are the same on narrow families, hence, by (\ref{sous}), on any converging subfamily. We will say a sequence~$((\mu_{(i,n)})_{i\in I})_{n\in\N}$ of families of measures uniformly converges to a family~$(\mu'_i)_{i\in I})$ if the mass~$\abs{\mu_{(i,n)}(1)-\mu'_i(1)}$ of~$\abs{\mu_{(i,n)}-\mu'_i}$ converges to~$0$, unifomrly with repect to~$i$.One shows that
\begin{equation}\label{suite}
 \text{If each family~$((\mu_{(i,n)})_{i\in I})$ is narrow, then the familly~$(\mu'_{i})_{i\in I}$ is narrow.}
\end{equation}
Indeed, for each positive invertible $\epsilon$, equation~(\ref{narrow}) is satisfied for~$(\mu'_{i})_{i\in I}$ if~$n$ is such that~(\ref{narrow}) it is satisfied~$((\mu_{(i,n)})_{i\in I})$ with~$\epsilon/2$  and~$\abs{\mu_{(i,n)}(1)-\mu'_i(1)}<\epsilon/2$ for every~$i$ in~$I$.

Let~$E$ denote the set of positive bounded measurable $\mu_H$-summable functions~$f$ on~$H$ such that the conclusion of Theorem~\ref{T3} holds for every subset~$Y_S$ as in Theorem~\ref{SBombay}.

Applying property~(\ref{dom}) we see that a bounded measurable $\mu_H$-summable positive function belongs to~$E$ as soon as the characteristic function of its support belongs to~$E$. Using properties~(\ref{joint}) and~(\ref{bary}) and (\ref{sous}), we see that~$E$ is invariant under barycentre, hence, by property (\ref{dom}), under positive linear combination. As every bounded measurable $\mu_H$-summable positive function is a monotone upper limit of functions with compact support,  it is enough, by~(\ref{suite}), to prove~$E$ contain the characteristic functions of compact subsets.

We now claim that~$E$ is stable under the action of~$H(\Q_S)$ by right translation. By Proposition~\ref{SC}, for any~$h$ in~$H(\Q_S)$ we can replace~$Y_S$ by~$\{h\}\cdot Y_S$. But Corollary~\ref{C3} for the translated function~$f\cdot h$ and the subset~$Y_S$ is equivalent to is the Corollary~\ref{C3} for the original function~$f\cdot h$ and the translated subset~$Y_S$, whence the claim.

Consequently, using Borel-Lebesgue criterion, invariance by linear combinations  and~(\ref{dom}),~$E$ will contain every characteristic function of a compact subset as soon as it contains the characteristic function of an open subset.

Let~$f$ be a function such that in Corollary~\ref{C3}. Then there exist a point~$h$ in~$H(\Q_s)$ such that~$f$ is essentially zero on no neighbourhood of~$h$. Possibly translating~$f$, we may assume that~$h$ is the neutral element. Let~$C'$ be a symmetric bounded (measurable) neighbourhood of~$h$; then $(f\mu_H)(C')>0$. Set~$C=C'\cdot C'$, and let~$f_C$ denote the convolution of~$f_C$ with the characteristic function of~$C$. Then one easily shows that, for any~$c$ in~$C'$, $f_{C}(c)\geq (f\mu_H)(C')>\epsilon$. Hence~$f_C$ dominate the characteristic function of~$C'$, and \latin{a fortiori} the characteristic function of the nonempty interior of~$C'$.

We will be done showing that~$f_C$ belongs to~$E$. Let~$\mu_C$ be the restriction of~$\mu_H$ to~$C$. Then~$f_C$ is proportional to the convolution of~$f$ with the probability~$\frac{\mu_C}{\mu(C)}$, as~$\mu(C)$ is finite and nonzero. But the latter convolution is a limit of barycentres of translates of~$f$, hence by (\ref{bary}) and (\ref{clos}), belong to~$E$. Consequently, by~(\ref{dom}),~$f_C$ belongs to~$E$, and we are done.
\end{proof}

 Note that Theorem~\ref{T3} is immediate if~$G(\Z_S)\sous G(\Q_S)$ is compact. From Borel and Harish-Chandra, this is the case if and only if~$G$ is anisotropic over~$\Q$. Note also that when~$S$ does not contain the archimedean place, then~$G(\Z_S)$ defines a lattice in~$G(\Q_S)$ if and only if~$G$ is anisotropic over~$\R$. In such a case~$G$ is also anisotropic over~$\Q$.

 From now on we assume that~$G$ is isotropic over~$\Q$. Hence~$S$ contains the archimedean place.
 We now turn to the statement of Theorem~\ref{SBombay} and the proof of Corollary~\ref{C3}.

\section{A $S$-adic combination of archimedean and $p$-adic stability}

We state a direct consequence of \cite{Lemma} and \cite{Lemmap}. Compare with Theorem \ref{T3} for some of the notations.
\begin{theorem}\label{theo}\label{SBombay}
Consider, for each~$v$ in~$S$, a closed reductive subgroup~$H_v$ of $G_{\Q_v}$ and let $H=\prod_{v\in S} H_v(\Q_v)$.
Let~$Z$ be the centraliser of~$H$ in~$G$. Then there exists a subset~$Y_S$ of~$G(\Q_S)$, which is closed for the metric topology, and such that 
 \begin{enumerate}
  \item on the one hand we have $G(\Q_S)=Y_S\cdot Z(\Q_S)$,\label{condition1}
  \item on the other hand, given\label{condition2}
  \begin{itemize}
   \item for each~$v$ in~$S$, a finite dimensional $\Q_v$-linear representation~$\rho_v:G_v\to GL(V_v)$,
   \item for each~$v$ in~$S$, a subset~$\Omega_v$ of~$H$ such that every matrix coefficient of some of the~$\rho_v$ that cancels on~$\Omega_v$ actually cancels on~$H_v$,
   \item for each~$v$ in~$S$, a $\Q_v$-homogeneous norm~$\Nm{-}_v$ on~$V_v$, 
  \end{itemize}
  there exists a constant~$c>0$ such that, 
  \begin{equation}\label{formuleS}
  \forall y\in Y, \forall (x_v)_{v\in S}\in V, \prod_{v\in S}\sup_{\omega_v\in\Omega_v}\Nm{\rho(y\cdot\omega_v)(v)}\geq\left.\left(\prod_{v\in S}\Nm{x_v}\right)\right/c.\end{equation}
 \end{enumerate}
\end{theorem}
\begin{proof} For any~$v$ in~$S$ we get a subset~$Y_v$ of~$G(\Q_S)$ by applying
\begin{itemize}
\item Theorem~$1$ of \cite{Lemma} (relative to~$H_v$) if~$v$ is archimedean;
\item Theorem~$1.1$ of \cite{Lemmap} (relative to~$H_v$) if~$v$ is ultrametric;
\end{itemize}
and let~$Y$ be the subset of~$G(\Q_S)$ made of the product of the~$Y_v$.

 By construction, we have~$G(\Q_v)=Z(\Q_v)\cdot Y_v$ for any~$v$ in~$S$, whence $G(\Q_S)=Z(\Q_S)\cdot Y_S$.
On the other hand, for every~$v$ in~$S$, we have, for some invertible positive constant~$c_v$,
\begin{equation}\label{formulev}
\forall y\in Y_v, \forall x_v\in V_v, \sup_{\omega\in\Omega_v}\Nm{\rho_v(y_v\cdot\omega_v)(x_v)}_v\geq\Nm{x_v}_v/c_v.
\end{equation}
Taking the product over~$v$ in~$S$, this yields formula (\ref{formuleS}) with~$c=\prod_{v\in S} c_v$.
\end{proof}

\begin{prop}\label{SC}
 If~$Y_S$ is a subset of $G(\Q_S)$ satisfying Theorem~\ref{SBombay}, then, for any nonempty compact subset~$C$ of~$G(\Q_S)$, we may replace, in Theorem~\ref{SBombay},~$Y_S$ by the subset~$C\cdot Y_S$.
\end{prop}
\begin{proof}
 We first note that~$Y_S$ being closed and~$C$ being compact, $C\cdot Y_S$ is closed in~$G(\Q_S)$. then that $C Y_S Z(\Q_S)=C G(\Q_S)=G(\Q_S)$, and finally that for all $(a,y)$ in $C\times Y_S$, for all $(x_v)_{v\in S}$ in $V$,
\begin{equation}
\left.\left(\prod_{v\in S}\Nm{x_v}\right)\right/ c 
\leq \prod_{v\in S}\sup_{\omega_v\in\Omega_v}\Nm{\rho(y\cdot\omega_v)(v)}
\leq \prod_{v\in S}\sup_{\omega_v\in\Omega_v}\NM{a^{-1}}_v\Nm{\rho(ay\cdot\omega_v)(v)}
\end{equation}
where~$\NM{a^{-1}}_v$ is the opertor norm of~$a$ acting on~$V_v$. As~$C$ is compact,~$\NM{a^{-1}}_v$ is bounded on~$C$, for any~$v$ in~$S$, by some~$\Lambda_v$. Putting~$\Lambda=\prod_{v\in S}\Lambda_v$, we get
$$\left.\left(\prod_{v\in S}\Nm{x_v}\right)\right/(c\Lambda) 
\leq \prod_{v\in S}\sup_{\omega_v\in\Omega_v}\Nm{\rho(ay\cdot\omega_v)(v)},$$
whence theorem \ref{SBombay} for $C\cdot Y_S$ with constant~$c\Lambda$ (depending on~$\rho_v$ and~$\Omega_v$).
\end{proof}

\section{Preliminary Lemmas}\label{preliminaries}
We first prove some lemmas that will allow us to adapt the proof of Theorem 9.3 in~\cite{KT} to our needs.
\begin{lemma}\label{ratio}
 Let~$p$ be a prime number, let~$n$ be a natural integer and let~$\Phi$ be a finite dimensional linear subspace of~$\Q_p(T_1,\ldots,T_n)$. Assume that every~$\phi$ is defined at the origin~$0$ of~${\Q_p}^n$. Then there exists 
\begin{itemize}
\item an arbitrarily small compact neighbourhood~$U$ of~$0$ in~$\Q_p^n$
\item and a polynomial~$P$ in~$\Q_p[T_1,\ldots,T_n]$
\end{itemize}
  such that
\begin{enumerate}
\item $P\cdot \Phi$ is included in~$\Q_p[T_1,\ldots,T_n]$\label{lem1}
\item and~$\abs{P(x)}=1$ for any~$x$ in~$U$.\label{lem2}
\end{enumerate}
\end{lemma}
\begin{proof}
 Choose a finite basis of~$\Phi$ and let~$P$ the lowest common denominator of the elements of~$\Phi$. As~$P$ is a common denominator, property~\ref{lem1} is satisfied. By minimality,~$P$ only cancels where some of the element of~$\Phi$ is undefined. In particular~$P(0)$ is nonzero, so we can renormalise~$P$ in such a way that~$P(0)=1$. Then~$\abs{P}$ is a real continuous function on~$\Q_p^n$ which takes value~$1$ at~$0$. But~$A$ is an isolated value of~$\abs{-}$, and \latin{a fortiori} of~$\abs{P}$. Hence~$\abs{P}=1$ holds on a some neighbourhood of~$0$. Let~$U$ be such a neighbourhood, we can assume to be arbitrarily small, and in particular compact. Then~$P$ and~$U$ answer the lemma.
\end{proof}
\begin{prop}\label{pparam}
 Let~$p$ be a prime number, let~$n$ be a natural integer and let~$H$ be a algebraic closed subgroup of~$GL(n)$ over $\Q_p$, for some natural integer~$n$. 

 Then there exist natural integers~$M$ and~$d$ and a continuous open map~$\Theta:{\Z_p}^d\to H(\Q_p)$ sending~$0$ to the neutral element~$e$ of~$H$ and such that for any $\Q_p$-linear form~$\phi$ on~$M_n(\Q_p)$, the real map~$\abs{\Phi\circ\Theta}$ can be written~$\abs{P}$ for some~$P$ in $\Q_p[T_1,\ldots,T_d]$ of degree at most~$M$.
\end{prop}
\begin{proof}
 Note that~$H$ is a linear algebraic group. Recall that~$H$ is a unirational variety: there exists an integer~$d$ and a dominant birational map~$u:\A^d_{\Q_p}\to H$. Such a map is generically submersive, and as~${\Q_p}^d$ is Zariski dense in~$\A^d_{\Q_p}$, there is point~$x$ in~${\Q_p}^d$ at which~$u$ is defined and submersive. Possibly translating by~$-x$ we may assume that~$x=0$, and possibly left translating by $u(0)^{-1}$, we may assume that~$u(0)=e$.

 As~$u$ is submersive at~$0$, the corresponding map~$u(\Q_p):{\Q_p}^d\to H(\Q_p)$ is open on a neighbourhhod of~$0$. As the map~$u(\Q_p)$ is rational, for any $\Q_p$-linear form~$\phi$ on~$M_d(\Q_p)$, the map~$\phi\circ u$ is a rational function. For varying~$\phi$, these rational functions generate a finite dimensional linear subspace~$\Phi$ of~$\Q_p(T_1,\ldots,T_d)$ made of rational functions which are defined at~$0$. Applying the lemma we find a polynomial~$P$ such that~$P\Phi$ is a finite dimensional linear space of polynomials. Let~$M$ be the maximum degree of these polynomials. The lemma also gives a neighbourhood~$U$ of~$0$, which we choose to be sufficiently so that~$u(\Q_p)$ is open on~$U$, such that~$\abs{P}=1$. Then, writing~$\abs{\phi}=\abs{P\phi}/\abs{P}$, and using the fact that~$\abs{P}=1$ on~$U$, we conclude that for any $\Q_p$-linear form~$\phi$ on~$M_d(\Q_p)$, $\abs{\Phi\circ u(\Q_p)}$ can be written~$\abs{P}$ for some~$P$ in $\Q_p[T_1,\ldots,T_d]$ of degree at most~$M$. Then we can construct~$\Theta$ such as in the proposition by composing by~$u(\Q_p)$ any linear open immersion~${\Z_p}^d\to U$. If~$R$ is the positive radius of a closed ball  contained in~$U$, then $x\mapsto p^{-k}x$ whenever~$k$ is an integer that is greater than~$\log_p(R)$.
\end{proof}

\begin{prop}
 Let~$\Phi$ be a rational open map ${\Z_p}^d \to H(\Q_s)$. Then the direct image of the Haar measure of~${\Z_p}^d$ is absolutely continuous with respect to the Haar measure of~$H(\Q_S)$.
\end{prop}
\begin{proof}
 One needs to show that the inverse image, say~$I$, of a null set is a null set. As~$\Phi$ is open, it is generically submersive, for the Zariski topology. As the singular locus of~$\Phi$ is a strict subvariety, it is contained in the zero set of a nonzero polynomial, hence is a null set. Consequently it will be enough to show that the intersection of~$I$ with the regular locus of~$\Phi$ has measure zero. It will be enough to check this locally. But locally,~$\Phi$ is \emph{analytically conjugated} to a linear map~${\Z_p}^d\to{\Z_p}^{\dim(H)}$ (we mean formally conjugated \latin{via} locally invertible convergent power series on the considered neighbourhood.) We recall that the Haar measure on~${\Z_p}^d$ and~$H(\Q_p)$ comes from a differential volume form. Consequently their image measure in~${\Z_p}^d$, (resp.~$\to{\Z_p}^{\dim(H)}$) will be associated to differential volume forms, hence will have locally bounded above and below density with respect to the Haar measures. Thus we can consider the case of a linear map with respect to the Haar measures~${\Z_p}^d\to{\Z_p}^{\dim(H)}$, in which case the proposition can be check directly.
\end{proof}
Applying Radon-Nykodym theorem, we deduce the following.
\begin{corollary}\label{densite}
 Let~$\Phi$ be a rational open map ${\Z_p}^d \to H(\Q_s)$. Then the direct image of the Haar measure of~${\Z_p}^d$ has an essentially nonzero integrable density with respect to the Haar measure of~$H(\Q_S)$, and it dominates a measure with essentially nonzero bounded density.
\end{corollary}

\section{Spaces of lattices, Mahler's criterion}\label{lattices}
In order to use~\cite{KT}, we fix here some notations.
We refer to~\cite{KT}, section~8 for the details.
Recall that~$S$ is assume to contain the archimedean place. Consequently,~$\Z_S$ defines a discrete cocompact ring in~$\Q_S$. A \emph{lattice} in~$\Q_S^m$ will be discrete $\Z_S$-submodule $\Lambda$ of rank~$m$. Equivalently~$\Lambda$ is a $\Z_S$-module generated by a basis of ${\Q_S}^m$ over~$\Q_S$. 

We identify the space of lattices in~${\Q_S}^m$ with~$GL(m,\Z_S)\sous GL(m,\Q_S)$ \latin{via} the map~
$$GL(m,\Z_S)g\mapsto {\Z_S}^mg.$$

Let~$\lambda_v$ denote the standard Lebesgue measure on~$\R$ if~$v$ is archimedean, or the additive Haar measure on~$\Q_p$ such that $\lambda_v(\Z_p)=1$ if~$v$ is finite and associated with the prime~$p$. We denote by~$\lambda_S$ the product Haar measure on~$\Q_S$. Then $\Z_S$ has covolume~$1$ in~$\Q_S$ with respect to~$\lambda_S$. We endow ${\Q_S}^m$ with the product Haar measure. Then lattices have finite covolume, and for any lattice written as~${\Z_S}^mg$, this covolume is~$\abs{\det(g)}_S$.

 As $G(\Z_S)$ is a lattice in $G(\Q_S)$, we know, from Borel and Harish-Chandra criterion that~$\det(G)=\{1\}$. Consequently~$G$ is contained in~$SL(m)$, and the map~$GL(\Z_S)g\mapsto {\Z_S}^mg$ identify~$G(\Z_S)\sous G(\Q_S)$ with a space of lattices in~${\Q_S}^m$ of covolume one.

 When~$L$ is a discrete $\Z_S$-submodule of~$\Q_S$, write the subspace~$\Q_SL$ generated by~$L$ as a sum $\Q_SL =\oplus_{v\in S} (\Q_SL)\tens\Q_v$, identify each $(\Q_SL)\tens\Q_v$ with the corresponding subspace of~${\Q_v}^m$, and endow $\Q_SL$ with the Haar measure which is the product measure of, for each factor~$(\Q_SL)\tens\Q_v$, the Haar measure induced
\begin{itemize}                                                                                                                                                                                                                                                                                                                                                                \item by the induced euclidean structure from $\R^m$ on $(\Q_SL)\tens\R$ if $v$ is archimedean;                                                                                                                                                                                                                                                                                                                                               
\item the Haar measure on~$(\Q_SL)\tens\Q_v$ giving volume one to~$((\Q_SL)\tens\Q_v)\cap \Z_p^m$ if $v$ is finite and associated with a prime~$p$.
\end{itemize}
We then define the covolume~$\cov(L)$ of~$L$ as the covolume of~$L$ inside~$\Q_S L$ with respect to the above Haar measure.

As a particular case, if~$L$ has rank one and generator~$x$, then~$\cov(L)=\Nm{x}_S$ where we define~$\Nm{(x_v)_{v\in S}}_S=\prod_{v\in S}\Nm{x_v}_v$ where~$\Nm{-}_v$ denote
\begin{itemize}                                                                                                                                                                                                                                                                                                                                                                \item the standard euclidean norm if $v$ is archimedean;                                                                                                                                                                                                                                                                                                                                               
\item the homogeneous norm unit sphere~${\Z_p}^m\smallsetminus p{\Z_p}^m$ if $v$ is finite and associated with a prime~$p$.
\end{itemize}
We define the systole function~$\sys:G(\Z_S)\sous G(\Q_S)\to\R$ as~$g\mapsto \inf\left\{\Nm{x}_S|x\in {\Z_S}^mg\smallsetminus\{0\}\right\}.$ Note that this infimum is actually a minimum, because any lattice in ${\Q_S}^m$, being discrete intersect any closed ball in~${\Q_S}^m$, which is compact, along a finite set.

We recall that Mahler's compactness criterion asserts that a subset~$E$ in~$G(\Z_S)\sous G(\Q_S)$ is bounded if and only if the systole function is bounded away from zero on this subset. Consequently, a family of measures~$(\mu_i)_{i\in I}$ on~$G(\Z_S)\sous G(\Q_S)$ is narrow if and only if
\begin{equation}\label{Mahler}
 \forall \epsilon>0, \exists \epsilon', \forall i \in I, \mu_i( \{x \in G(\Z_S)\sous G(\Q_S)~|~\sys(x)\leq \epsilon' \}) < \epsilon.
\end{equation}

\section{A ``good'' parametrisation}\label{good}

Recall that \cite{KT}, section 3, define, on a subset~$U$ metric space~$(X,d)$ with nowhere zero Borel measure~$\mu$, a real nonzero Borel function to be \emph{$(C,\alpha)$-good} on~$U$, for some positive invertible constants~$C$ and~$\alpha$, if for any open ball~$B$ in~$U$ one has
\begin{equation}\label{eqCA}
 \forall \epsilon>0, \mu\left(\left\{x\in B\left| \abs{f(x)}<\epsilon\cdot\sup_{b\in B}\abs{f(b)} \right.\right\}\right)\leq C\epsilon^\alpha \mu(B).
\end{equation}
In other words ``we can effectively contol, in any ball, the relative time a good function is relatively small''.

Consider the class of functions~$E(n,\Lambda)$ from~\cite{EMSGAFA}, Definition~1, namely the class of functions~$\R\to\C$ made of a linear combination of~$t\mapsto t^l\exp(\lambda t)$, with a natural integer $l$ such that $l\leq n$ and a complex parameter~$\lambda$ such that~$\abs{\lambda}\leq \Lambda$. Let us prove that for any bounded subset~$U$ of~$\R$, there exists is some positive invertible constants~$C$ and~$\alpha$ such that for any (nonzero) function $f$ in~$E(n,\Lambda)$ ,~$\abs{f}$ is $(C,\alpha)$-good on~$U$ (with respect to the Lebesgue measure).
\begin{proof}
It will be enough to check property (\ref{eqCA}) for any open ball of bounded radius in~$\R$, that is an open interval of bounded length.

Desired result is the content of Corollary 2.10 from~\cite{EMSGAFA}, provided one can effectively bound involved constant~$M$ by a polynomial in the~$\epsilon$ of the statement. From the use of Corollary~2.9  in the proof of Corollary~2.10, the subset
$$\left\{x\in B\left| \abs{f(x)}<1/M\cdot\sup_{b\in B}\abs{f(b)} \right.\right\}$$
is a union of at most~$n^2-1$ subintervals. Assume that~$M>1$ so that any of these intervals can be~$B$ itself. So each of these subintervals~$J$ has a boundary point in~$B$, and consequently~$\sup_{b\in J}\abs{f(b)}=\epsilon\sup_{b\in B}\abs{f(b)}$. Applying Corollary~2.10 on each of these subintervals, we conclude that~$M$ convene for~$\epsilon$ and~$M>1$, then~$M^2$ convene for~$\epsilon^2$.

This yields the desire bound and finishes the proof.
\end{proof}
Using Lemma 3.2 from \cite{KT}, we deduce that for the any class~$E(m,n,\Lambda)$ of multivariate functions $\R^m\to\C$, and for any bounded subset~$U$ in~$\R^m$, there exists some positive invertible constants~$C$ and~$\alpha$ such that for any (nonzero) function $f$ in~$E(m,n,\Lambda)$ ,~$\abs{f}$ is $(C,\alpha)$-good on~$U$ for the Lebesgue measure. (Note that any bounded subset of~$\R^m$ is contained in a product of bounded subsets of~$\R$)

In situation of Lemma~\ref{ratio}, given~$\Phi$, and~$U$ as in the Lemma,  the Lemmas 3.2 and 3.4 of \cite{KT} implies there exists constants~$C$ and~$\alpha$ such that for any nonzero function~$\phi$ in~$\Phi$,~$\abs{\phi}$ is $(C,\alpha)$-good on~$U$.

\begin{prop}\label{param}
Let~$\rho:G\to Sl(n)$ be a representation, for some natural integer~$n$. Let~$\Nm{-}_S$ be some norm on~${\Q_S}^n$. Consider, for every place~$v$ in~$S$, let~$\Theta_v$ denote
\begin{itemize}
\item a generically submersive function (\cf \cite{Lemma}, A.3) of some class~$E_G(m,n,\Lambda)$ if~$v$ is archimedean; 
\item a function~${\Z_p}^{d_v}$ as in Proposition~\ref{pparam} if~$v$ is a finite associated with some prime~$p$.
\end{itemize}

 Let~$\Theta$ denote the product map~$\Theta_S((\lambda_v)_{v\in S}\mapsto \prod_{v\in S}\Theta_v(\lambda_v)$.

 Then for any bounded subset~$U$ in the domain of~$\Theta$, there exists positive and invertible constants~$C$ and~$\alpha$ such that for any~$x$ in~${\Q_S}^n$, and any~$g$ in~$G(\Q_S)$
 the function~$\Nm{\rho(\Theta g)\cdot x}_S$ is~$(C,\alpha)$-good on~$U$.
\end{prop}
Note that the submersivity assumption may always be satisfied. It follows, for exemple from~\cite{Lemma},~A.3 and A.4.
\begin{proof}
We first note that, possibly translating~$x$, we may ignore the element~$g$.

Using lemma 3.1 (d) in \cite{KT}, we may replace~$\Nm{-}$ by any comparable norm. Consequently, we may assume that~$\Nm{-}_S$ may be written~$\prod_{v\in S}\Nm{-}_v$, with~$\Nm{-}_v$ be the maximum of~$\abs{-}_v$ applied to the coordinates.

Using Lemma 3.3 of~\cite{KT}, we may assume that~$S$ contains only one element, say~$v$. 

Using Lemma 3.1 (c) in \cite{KT}, we may replace the function~$\Nm{\Theta\cdot x}_S$ by the absolute value of some coordinate. Note that coordinates of~$\Nm{\Theta\cdot x}_S$ are matrix coefficients of~$\Theta$. 

Recall that the class~$E_G(m,n,\Lambda)$ is made of functions into~$G$ whose matrix coefficients in the adjoint representation of~$G$ belongs to~$E(m,n,\Lambda)$. As~$G$ is semisimple, its adjoint representation is a closed immersion. It follows that any algebraic regular function on~$G$, and in particular the matrix coefficients of~$\rho$, belong to some class~$E(m',n',\Lambda)'$. Then for archimedean~$v$, the proposition follows from the preceding discussion.

If~$v$ is a finite place, then from Proposition~\ref{pparam} and definition~(\ref{eqCA}) we may assume~$\Theta$ is a polynomial map. It then follows from the use of Lemmas 3.1, 3.3 and 3.4 of \cite{KT} as in the proof of Theorem 9.4 of \cite{KT}.
\end{proof}

%
%
%
%
%
%
%
%

\section{Proof}
Let~$n=2^m$ and consider the exterior algebra representation $\rho:GL(m)\to GL(n)$ (written in some basis). As~$G$ is contained in~$SL(m)$, its image under $\rho$, and \latin{a fortiori} the image of~$H$ is contained in~$SL(n)$. Indeed,~$\det\circ \rho$ defines a character, hence is trivial on~$SL(m)$, the latter being simple.

As in proof of Theorem 9.4 in~\cite{KT}, there is suitable norm on~${\Q_S}^n$ such that for any discrete $\Z_S$-submodule~$\Delta$ of~${\Q_S}^n$, there is some~$x$ in~${\Q_S}^n$, such that for any~$g$ in~$GL(m,\Q_S)$, 
$$\cov(\Delta g)=\Nm{x\rho(g)}_S.$$

We consider a map~$\Theta$ as in Proposition~\ref{param}. Then there are some positive and invertible constants~$C$ and~$\alpha$ such that for any discrete $\Z_S$-submodule~$\Delta$ of~${\Q_S}^n$, and any~$g$ in~$G(\Q_S)$, the function~$\cov(\Delta \Theta g)$ is $(C,\alpha)$-good. In particular condition (i) of Theorem 9.3 of \cite{KT} is satisfied for the map~$\Theta g$. Let~$Y_S$ be as in Theorem~\ref{SBombay}, and let~$c$ be constant as in Theorem~\ref{SBombay} relative to~$\rho$ and the image of the~$\Theta_v$ (at a finite place~$v$, $\Theta_v$, being open, has Zariski dense image). Then for any~$y$ in~$Y_S$, the translated map~$\Theta\cdot y$ satisfies the condition (i) and (ii) of Theorem 9.3. Of course we may assume~$c<1$.

The other conditions of Theorem 9.3 are satisfied for the same reason as in the proof of Theorem  9.4. We conclude that there is some (effectively computable, \cf \cite{KT}~9.3) constant~$C'$ such that for any~$y$ in~$Y_S$ for any ball~$B$ in~$\prod_{v\in S} {\Q_v}^{d_v}$ of some radius~$R$ at a centre~$x$ such that the ball of radius $3^mR$ at the same centre is contained in the domain of~$\Theta$, we have
$$\forall \epsilon<c \mu(\{b\in B|\sys({\Z_S}^m\Theta(b)y)<\epsilon c)\leq C'\epsilon^\alpha,$$
where~$\mu$ is a Haar measure on the domain of~$\Theta$. By virtue of formula~(\ref{Mahler}), this implies that if~$\mu_B$ denote the direct image of~$\mu$ in~$G(\Z_S)\sous G(\Q_S)$, then the family of translates~$(\mu_B\cdot y)_{y\in Y_S}$ is narrow.

According Corollary~\ref{densite} and~\cite{Lemma}~A.4, the measures~$\mu_B$ is absolutely continuous, Hence dominate a measure with positive bounded non essentially zero density. Using property~(\ref{dom}) concludes the proof of Corollary~\ref{C3}.

\backmatter